\documentclass[11pt]{article}

\PassOptionsToPackage{dvipsnames}{xcolor}
\usepackage{amssymb,amsmath,bm}
\usepackage{amsmath}
\usepackage{mathrsfs}
\usepackage{slashed}
\usepackage{amssymb}
\usepackage{amsthm}
\usepackage{graphicx,overpic}
\usepackage{svg}
\usepackage{epstopdf}
\usepackage{enumerate}
\usepackage{comment}
\usepackage{indentfirst}
\usepackage{setspace}  

\usepackage{textcomp}
\usepackage{enumerate}      
\usepackage{graphicx} 
\usepackage{caption}
\usepackage{subcaption}

\usepackage{tabu}

\usepackage[export]{adjustbox}

\usepackage{mathrsfs}

\usepackage{url}

\usepackage{hyperref}

\def\psl{\mathrm{PSL}(2,\mathbb{R})}
\def\pgl{\mathrm{PGL}(2,\mathbb{R})}

\def\SL{\mathrm{SL}(2,\mathbb{R})}
\def\GL{\mathrm{GL}(2,\mathbb{R})}
\def\sl{\mathfrak{sl}_2(\mathbb{R})}
\def\fund{\pi_1(\Sigma)}
\def\hom{\mathrm{Hom}\big(\fund, \psl\big)}

\def\univcover{\widetilde{\SL}}
\def\Hyp{\mathrm{Hyp}}
\def\Par{\mathrm{Par}}
\def\Parp{\mathrm{Par}^+}
\def\Parm{\mathrm{Par}^-}
\def\Ell{\mathrm{Ell}}

\def\R{\mathcal{R}(\Sigma)}
\def\Rn{\mathcal{R}_n(\Sigma)}
\def\Rns{\mathcal{R}_n^s(\Sigma)}

\def\W{\mathcal{W}(\Sigma)}
\def\Wn{\mathcal{W}_n(\Sigma)}

\def\HP{\mathrm{HP}}

\def\HPns{\mathrm{HP}_n^s}

\def\PE{\mathrm{PE}}
\def\PP{\mathrm{PP}}
\def\nonabel{\mathrm{NA}^s_n(\Sigma)}

\def\tphicone{\widetilde{\phi(c_1)}}
\def\tphictwo{\widetilde{\phi(c_2)}}

\def\ev{\mathrm{ev}}
\def\tr{\mathrm{tr}}

\def\decomp{\displaystyle\bigcup_{k=1}^{-\chi(\Sigma)} \Sigma_k}
\def\apdecomp{\displaystyle \bigg(\bigcup_{i=1}^{g+p-2} P_i\bigg)\cup\bigg(\bigcup_{j=1}^g T_j\bigg)}

\def\torus{\Sigma_{1,1}}
\def\pants{\Sigma_{0,3}}
\def\torustwo{\Sigma_{1,2}}
\def\pantstwo{\Sigma_{0,4}}

\def\parpp
{\begin{bmatrix}
    1 & 1 \\
    0 & 1
\end{bmatrix}}

\def\parmp
{\begin{bmatrix}
    1 & -1 \\
    0 & 1
\end{bmatrix}}

\def\parpm
{\begin{bmatrix}
    -1 & -1 \\
    0 & -1
\end{bmatrix}}

\def\parmm
{\begin{bmatrix}
    -1 & 1 \\
    0 & -1
\end{bmatrix}}

\def\part
{\begin{bmatrix}
    1 & t \\
    0 & 1
\end{bmatrix}}

\def\abcd
{\begin{bmatrix}
    a & b \\
    c & d
\end{bmatrix}}

\def\inverse
{\begin{bmatrix}
    d & -b \\
    -c & a
\end{bmatrix}}

\def\pqrs
{\begin{bmatrix}
    p & q \\
    r & s
\end{bmatrix}}

\def\xdiag
{\begin{bmatrix}
    x & 0 \\
    0 & x^{-1}
\end{bmatrix}}

\def\ell
{\begin{bmatrix}
    \cos\frac{\theta}{2} & \sin\frac{\theta}{2} \\
    -\sin\frac{\theta}{2} & \cos\frac{\theta}{2}
\end{bmatrix}}

\def\elltheta
{\begin{bmatrix}
    \cos\theta & \sin\theta \\
    -\sin\theta & \cos\theta
\end{bmatrix}}

\def\xell
{\begin{bmatrix}
    x & y \\
    -y & x
\end{bmatrix}}

\def\sldiag
{\begin{bmatrix}
    1 & 0 \\
    0 & -1
\end{bmatrix}}

\def\slupper
{\begin{bmatrix}
    0 & 1 \\
    0 & 0
\end{bmatrix}}

\def\sllower
{\begin{bmatrix}
    0 & 0 \\
    1 & 0
\end{bmatrix}}

\def\path{\{\phi_t\}\interval}
\def\interval{_{t\in [0,1]}}

\newtheorem{theorem}{Theorem}[section]
\newtheorem{proposition}[theorem]{Proposition}
\newtheorem{lemma}[theorem]{Lemma}
\newtheorem{corollary}[theorem]{Corollary}

\newtheorem{remark}[theorem]{Remark}

\theoremstyle{definition}

\begin{document}
	
\title{\textbf{Connected components of the space of
type-preserving representations}}
\author{Inyoung Ryu and Tian Yang}
\date{}
\maketitle

\begin{abstract}
We complete the characterization of the connected components of the space of type-preserving representations of a punctured surface group into $\psl$. 
We show that the connected components are indexed by the relative Euler classes and the signs of the images of the peripheral elements satisfying a generalized Milnor-Wood inequality; and when the surface is a punctured sphere, there are additional connected components consisting of ``totally non-hyperbolic" representations. As a consequence, we count the total number of the connected components of the space of type-preserving representations. 
\end{abstract}


\tableofcontents

\section{Introduction}

For a connected, closed and oriented surface  $\Sigma,$
the $\psl$-\emph{representation space} is the space
of group homomorphisms $\phi:\pi_1(\Sigma)\to \psl$ of the fundamental group of $\Sigma$ into the Lie group $\psl,$ endowed with the compact-open topology. The \emph{Euler class} $e(\phi)$ of such  $\phi$ is the Euler class of the associated flat $\psl$-bundle over $\Sigma,$ which is an integer under the identification $\mathrm{H}^2(\Sigma;\pi_1(\psl))\cong \mathbb Z.$  It was proved in \cite{milnor, wood} that the Euler class $e(\phi)$ satisfies the Milnor-Wood inequality 
$$\chi(\Sigma)\leqslant e(\phi)\leqslant -\chi(\Sigma),$$
where $\chi(\Sigma)$ is the Euler characteristic of $\Sigma.$ Goldman, in the seminal work \cite{goldman}, proved that if the genus of $\Sigma$ equals $g,$ then the $\psl$-representation space of $\Sigma$ has exactly $4g-3$ connected components, one for each integer $n$ in between $\chi(\Sigma)$ and $-\chi(\Sigma)$ consisting of the representations with Euler class equal to $n.$ He further proved that the two components respectively corresponding to $n=\pm\chi(\Sigma)$ consist of discrete and faithful representations, whose quotients by the $\psl-$conjugations give two copies of the Teichm\"uller space of $\Sigma$, one for each orientation of $\Sigma.$ 

The goal of this article is to understand the space of $\psl$-representations for punctured surfaces. Now let $\Sigma$ be a connected, oriented punctured surface, that is, a surface obtained from a connected, closed  and oriented surface with finitely many points removed. As the fundamental group of $\Sigma$ is a free group of rank $1-\chi(\Sigma),$ the space of all the $\psl$-representations of $\pi_1(\Sigma),$ endowed with the compact-open topology, is homeomorphic to $\psl^{1-\chi(\Sigma)}$ and is connected. Therefore, it is natural to consider the subspaces of the $\psl$-representations with suitable constraints. Natural constraints include the \emph{hyperbolic boundary condition} where all the peripheral elements (elements represented by circles around a puncture) of $\pi_1(\Sigma)$ are required to be sent to hyperbolic elements of $\psl,$ and the  \emph{parabolic boundary condition} where all the peripheral elements of $\pi_1(\Sigma)$ are sent to parabolic elements of $\psl.$ For a representation $\phi:\pi_1(\Sigma)\to\psl$ satisfying either the hyperbolic boundary condition or the parabolic boundary condition,  a \emph{relative Euler class} $e(\phi)$ was defined in \cite{goldman} which is an integer and still satisfies the Milnor-Wood inequality $\chi(\Sigma)\leqslant e(\phi)\leqslant -\chi(\Sigma).$ See Section \ref{relE} for more details.

Let $\mathcal W(\Sigma)$ be the space of the $\psl$-representations of $\pi_1(\Sigma)$ that satisfy the hyperbolic boundary condition. Then similar to the closed surface case, Goldman\,\cite{goldman} proved that the connected components of $\mathcal W(\Sigma)$ are indexed by the relative Euler classes, ie., each component of $\mathcal W(\Sigma)$ corresponds to an integer $n$ in between $\chi(\Sigma)$ and $-\chi(\Sigma)$ and consists of representations with the relative Euler class equal to $n.$
\\

The situation becomes a lot more subtle when considering representations with the parabolic boundary condition. For instance, as observed in \cite[Theorem 1.3]{yang}, for the four-puncture sphere $\Sigma_{0,4}$ the space of representations with the parabolic boundary condition of relative Euler class $\pm 1$ has five connected components each, and the spaces of  irreducible representations with the parabolic boundary condition of relative Euler class $0$ has six connected components.  This shows that for representations with the parabolic boundary condition, the relative Euler classes alone do not characterize the connected components. 

As the main focus of this paper, we call a representation 
 $\phi:\pi_1(\Sigma)\to\psl$ satisfying the parabolic boundary condition a \emph{type-preserving representation,} and denote by $\mathcal R(\Sigma)$  the space of all the type-preserving representations. As observed by Kashaev\,\cite{kashaev}, besides the relative Euler class, the connected components of $\mathcal R(\Sigma)$ depend on another quantity, the \emph{sign} of the representation which assigns each puncture of $\Sigma$  a sign $+1$ or $-1.$ See Section \ref{Sign} for more details. Kashaev\,\cite{kashaev} then conjectured that the space $\mathcal R_n^s(\Sigma)$ of type-preserving representations with the relative Euler class $n$ and the sign $s,$ if non-empty, is connected. However, there was no prediction on when the space $\mathcal R_n^s(\Sigma)$ is non-empty.

In Theorem \ref{thm_main1} and Theorem \ref{thm_main2} below, we find a sufficient and necessary condition for the space $\mathcal R_n^s(\Sigma)$ to be non-empty, and prove the connectedness of each non-empty $\mathcal R_n^s(\Sigma),$ confirming Kashaev's Conjecture and counting  the number of connected components of $\R.$ In Theorem \ref{thm_main1} we consider punctured surfaces with a positive genus, and in Theorem \ref{thm_main2} we consider the punctured spheres.  To state the results, for an $s\in\{\pm 1\}^p,$ let $p_+(s)$ be the number of $+1$'s and let $p_-(s)$ be the number of $-1$'s in the components of $s.$

 \begin{theorem}\label{thm_main1}
        Let $\Sigma=\Sigma_{g,p}$ be a connected and oriented punctured surface with genus $g\geqslant 1$ and $p\geqslant 1$ punctures.  Let $n\in \mathbb{Z}$ and $s\in \{\pm 1\}^p.$ Then the space $\mathcal R_n^s(\Sigma)$ of type-preserving representations with relative Euler class $n$ and sign $s$  is non-empty if and only if the pair $(n,s)$ satisfies the following inequality: 
        $$\chi(\Sigma) + p_+(s)\leqslant n\leqslant -\chi(\Sigma) - p_-(s).$$
        Moreover, each non-empty $\mathcal R_n^s(\Sigma)$ above is connected.
    \end{theorem}

\begin{theorem}\label{thm_main2}
        Let $\Sigma = \Sigma_{0, p}$ be a punctured sphere with $p \geqslant 3$ punctures. Let $n\in \mathbb{Z}$ and $s\in \{\pm 1\}^p.$ 
        Then the space $\mathcal R_n^s(\Sigma)$ of type-preserving representations with relative Euler class $n$ and sign $s$ is non-empty if and only if the pair $(n,s)$ satisfies one of the following conditions: 
       \begin{enumerate}[(1)]

           \item $\chi(\Sigma) + p_+(s)\leqslant n\leqslant -\chi(\Sigma) - p_-(s)$, 

            \item $n = 0,$ and either $p_-(s) = 1$ or $p_+(s)=1$, and

           \item $n = 1$ and $p_-(s) = 0,$ or $n = -1$ and $p_+(s) = 0.$

      \end{enumerate}
       Moreover, each non-empty $\mathcal R_n^s(\Sigma)$ above is connected.
    \end{theorem}

In the rest of this paper, we will call the inequality in Theorem \ref{thm_main1} and in Theorem \ref{thm_main2} (1) the \emph{generalized Milnor-Wood inequality.}

\begin{remark} 
The representations in (2) and (3) of Theorem \ref{thm_main2} are the ``super-maximal" representations in \cite{deroin_tholozan} that send every element of $\pi_1(\Sigma)$ represented by a simple closed curve in $\Sigma$ to a non-hyperbolic element of $\psl,$ where in (2) all the representation are abelian and in (3) all the representations are irreducible.   
\end{remark}

For $n\in \mathbb Z,$ let $\mathcal R_n(\Sigma)$ be the space of type-preserving representations of $\pi_1(\Sigma)$ of relative Euler class $n.$ Then as a consequence of Theorem \ref{thm_main1} and Theorem \ref{thm_main2}, we are able to count the number of connected components of each $\mathcal R_n(\Sigma)$ and $\mathcal R(\Sigma).$

\begin{corollary}\label{thm_main3} Let $\Sigma=\Sigma_{g,p}.$ 
    \begin{enumerate}[(1)]  
    \item If $g\geqslant 1,$ then for each $n\in\mathbb Z,$
             the number of connected components of $\mathcal R_n(\Sigma)$ equals
            $$\displaystyle\sum_{k = \max\{n+p+\chi(\Sigma), 0\}}^{\min\{n-\chi(\Sigma), p\}}\binom{p}{k}. $$
            
    \item If $g=0$ and $n=0,$ then the number of connected components of  $\mathcal R_n(\Sigma)$ equals
            $$\displaystyle\sum_{k = \max\{n+p+\chi(\Sigma), 0\}}^{\min\{n-\chi(\Sigma), p\}}\binom{p}{k} + 2p.$$
            
   \item If $g=0$ and $n=\pm 1,$ then the number of connected components of  $\mathcal R_n(\Sigma)$ equals
            $$\displaystyle\sum_{k = \max\{n+p+\chi(\Sigma), 0\}}^{\min\{n-\chi(\Sigma), p\}}\binom{p}{k} + 1.$$

    \end{enumerate}
    
    \end{corollary}

       \begin{corollary}\label{total} Let $\Sigma = \Sigma_{g,p}.$ 
       \begin{enumerate}[(1)]
           \item  If $g\geqslant 1,$ then $\R$ has in total
        $$\displaystyle\sum_{n=\chi(\Sigma)}^{-\chi(\Sigma)} 
        \displaystyle\sum_{k = \max\{n+p+\chi(\Sigma), 0\}}^{\min\{n-\chi(\Sigma), p\}} \binom{p}{k}$$
        connected components, one for each pair $(n,s)\in \mathbb Z\times \{\pm 1\}^p$ satisfying the generalized Milonr-Wood inequality.
        \item If $g=0,$ then $\R$ has in total $$\displaystyle\sum_{n=\chi(\Sigma)}^{-\chi(\Sigma)} 
        \displaystyle\sum_{k = \max\{n+p+\chi(\Sigma), 0\}}^{\min\{n-\chi(\Sigma), p\}} \binom{p}{k} + 2p + 2$$
        connected components, one for each pair $(n,s)\in\mathbb Z\times \{\pm 1\}^p$ satisfying the conditions in Theorem \ref{thm_main2}. 
       \end{enumerate}
        \end{corollary}

\begin{corollary} Let  $\Sigma = \Sigma_{g,p}.$ Then the space of type-preserving $\psl$-characters of $\Sigma$ has the same number of connected components as in Corollary \ref{total}, indexed by the pairs  $(n,s)\in\mathbb Z\times \{\pm 1\}^p$ satisfying the conditions in Theorem \ref{thm_main1} and Theorem \ref{thm_main2}. 
    \end{corollary}

\begin{remark}
In \cite{B,CB-M}, the same result was obtained for the five-puncture sphere and the two-puncture torus with a different method, using Kashaev's lengths coordinates\,\cite{kashaev} as in \cite{yang,MPY}. 
\end{remark}

    To prove Theorem \ref{thm_main1} and Theorem \ref{thm_main2}, we need to consider a larger space $\mathrm{HP}(\Sigma)$ consisting of $\psl$-representations with the \emph{mixed boundary condition,} which send each peripheral element of $\pi_1(\Sigma)$ to either a hyperbolic or a parabolic element of $\psl.$ Then both $\R$ and $\W$ are subspaces of $\mathrm{HP}(\Sigma).$ Similar to the type-preserving representations, for each representation $\phi$ in $\mathrm{HP}(\Sigma),$ the \emph{sign} of $\phi$ assigns  a sign $+1$ or $-1$ to each parabolic puncture, and a $0$ to each hyperbolic puncture of $\Sigma.$  See Section \ref{Sign} for more details. 
 For an $s\in\{+1,0,-1\}^p,$ we let $p_+(s),$ $p_0(s)$ and $p_-(s)$ respectively be the number of $+1$'s, $0$'s and $-1$'s in the components of $s.$

    \begin{theorem}\label{thm_main4} 
        Let $\Sigma = \Sigma_{g, p}.$ Let $n\in \mathbb{Z}$ and $s\in \{+1,0, -1\}^p$ with $p_0(s)\geqslant 1.$ Then the space $\HP_n^s(\Sigma)$ of representations with the mixed boundary condition with relative Euler class $n$ and sign $s$ is non-empty if and only if the pair $(n,s)$ satisfies the generalized Milnor-Wood inequality: 
        $$\chi(\Sigma) + p_+(s)\leqslant n\leqslant -\chi(\Sigma) - p_-(s).$$
       Moreover, each non-empty  $\HP_n^s(\Sigma)$ above is connected.
    \end{theorem}

\begin{remark}
In \cite{Mondello}, similar results for the relative $\psl$-character varieties with a non-zero relative Euler class were obtained using the techniques from Higgs bundles; and in \cite{KPW}, related results were obtained for representations into higher rank Lie groups.
In a subsequent paper \cite{RY}, we will address the problem of characterizing the connected components of spaces of $\psl$-representations which allow the elliptic boundary condition.
\end{remark}

This article is organized as follows. 
    In Section 2, we recall the relative Euler class and the sign of type-preserving representations and representations with the mixed boundary condition, and state the related results.
    In Section 3, we prove the main results for surfaces with Euler characteristic $-1$. 
    A genericity property is established in Section 4, which will be used
    in the proof of Theorem \ref{thm_main1} and Theorem \ref{thm_main4} for surfaces with Euler characteristic $-2$ in Section 5, and for the general punctured surfaces in Section 6. Finally, in Section 7, we address the extra components in Theorem \ref{thm_main2}.   
\\

\noindent\textbf{Acknowledgments.} 
The authors would like to thank Francis Bonahon, Tushar Pandey, Sara Maloni and Gabriele Mondello for helpful discussions, Nicolas Tholozan and Inkang Kim for bringing their attentions to \cite{Mondello} and \cite{KPW}. They are also grateful to Bill Goldman for showing interests in this work. The second author is supported by NSF Grants DMS-2203334.

\section{Preliminary}

Let $\Sigma = \Sigma_{g,p}$ be a punctured surface with genus $g$ and $p$ punctures, and let $\fund$ be the fundamental group of $\Sigma$. 
    A \emph{peripheral element} of $\fund$ 
    is an element represented by a closed curve freely homotopic to a circle around a puncture of $\Sigma$. Notice that each puncture of $\Sigma$ determines a conjugacy class of $\pi_1(\Sigma)$ consisting of the peripheral elements around it; and through out this paper, we choose and fix a set of representatives $\{c_1,\dots,c_p\}$ of the conjugacy classes of the peripheral elements, called the \emph{preferred peripheral elements}, one for each puncture of $\Sigma.$ 
    
    A representation $\phi: \pi_1(\Sigma)\to\psl$ is \emph{type-preserving} if for each $i\in\{1,\dots,p\},$ $\phi(c_i)$ is a parabolic element of $\psl.$ Let $\hom$ be the space of $\psl$-representations of $\pi_1(\Sigma)$ into $\psl$ endowed with the compact-open topology, and denote by $\R$ be the subspace of $\hom$ consisting of the type-preserving representations. Similarly, a representation $\phi: \pi_1(\Sigma)\to\psl$  is with the \emph{hyperbolic boundary condition} if for each $i\in\{1,\dots, p\},$  $\phi(c_i)$ is a hyperbolic element of $\psl;$ and is with the \emph{mixed boundary condition} if for each $i\in\{1,\dots, p\},$ $\phi(c_i)$ is either a hyperbolic or a parabolic element of $\psl.$ We denote by $\W$ be the subspace of $\hom$ consisting of representations with the hyperbolic boundary condition, and denote by $\mathrm{HP}(\Sigma)$ be the subspace of $\hom$ consisting of representations with the mixed boundary condition. Then both $\R$ and $\W$ are subspaces of $\mathrm{HP}(\Sigma).$

The purpose of this section is to recall the definition and basic properties  of the relative Euler classes and the signs of $\psl$-representations with these various boundary conditions, and results of Goldman in \cite{goldman}, which are the key ingredients in our main results. 
In the rest of the paper, we will use the notation $g$ or $\pm A$ for $\psl$-elements, $A$ for $\SL$-elements, and $\widetilde{g}$ or $\widetilde{A}$ for the elements in the universal cover $\univcover$. 

   \subsection{The relative Euler classes}\label{relE}

    For a closed surface $\Sigma$, each representation 
    $\phi\in \hom$ 
    determines  and is the holonomy representation of a flat principal $\psl-$bundle over $\Sigma$. 
    The Euler class of this bundle is an obstruction class in  $\mathrm{H}^2\big(\Sigma; \pi_1\big(\psl\big)\big)$ that measures the non-triviality of the principal bundle.
    This obstruction class defines the Euler class of the representation $\phi$, which can be considered as an integer under the isomorphism $\mathrm{H}^2\big(\Sigma; \pi_1\big(\psl\big)\big)\cong \mathbb{Z}.$ 
    For a punctured surface $\Sigma$, as $\mathrm{H}^2\big(\Sigma; \pi_1\big(\psl\big)\big) = 0,$ all flat principal bundles over $\Sigma$ are trivial. Therefore, boundary conditions are needed to obtain a nontrivial invariant of $\psl-$representations of $\pi_1(\Sigma).$
    For each representation $\phi \in \HP(\Sigma)$, 
    the corresponding flat bundle admits a well-defined special trivialization (see \cite[Section 3]{goldman}). The obstruction to extend this special trivialization on $\partial \Sigma$ to a trivialization on $\Sigma$ defines the \emph{relative Euler class} of $\phi$, which is an obstruction class in 
    $\mathrm{H}^2\big(\Sigma, \partial \Sigma ; \pi_1\big(\psl\big)\big)$ and can be considered as an integer under the isomorphism $\mathrm{H}^2\big(\Sigma, \partial \Sigma ; \pi_1\big(\psl\big)\big)\cong \mathbb{Z}$.
    \\

    For more details, let us recall that elements of $\psl$ are classified as hyperbolic, parabolic and elliptic as follows: For a non-trivial $\pm A$ in $\psl,$ let $A$ be one of its lifts in $\SL.$ Then $\pm A$ is hyperbolic if $|\tr(A)|>2,$ parabolic if $|\tr(A)|=2,$ and elliptic if $|\tr(A)|<2.$ 
We respectively let $\Hyp,$ $\Par$ and $\Ell$ be the spaces of hyperbolic, parabolic and elliptic elements of $\psl$. Then $\Hyp$ and $\Ell$ are connected subspaces of $\psl$; and $\Par$ has two connected components $\Parp$ and $\Parm$ which are respectively the $\psl$-conjugacy classes of $\pm \parpp$ and $\pm \parmp$. We call a parabolic element \emph{positive} if it lies in $\Par^+$, and \emph{negative} if it lies in $\Par^-$.
    
  Similarly, non-central elements of the universal covering $\univcover$ of $\psl$ are also classified as \emph{hyperbolic}, \emph{parabolic}, or 
    \emph{elliptic} according to the type of their projections to $\psl$. A parabolic element of $\univcover$ is \emph{positive} or \emph{negative} if its projection to $\psl$ is respectively so. As a convention, we defined the \emph{trace} of an element in $\univcover$ to be the trace of its projection to $\SL.$ Then a non-central element in $\univcover$ is hyperbolic, parabolic, or elliptic if its trace is greater than, equal to, or less than $2$ in the absolute value.

    As $\psl$ is homeomorphic to the unit tangent bundle of the hyperbolic plane $\mathbb H^2$ which is an open solid torus, its universal covering 
    group $\univcover$ is  homeomorphic to $\mathbb{R}^3$, with the group of
    deck transformations 
    $\pi_1(\psl)\cong \mathbb{Z}$. 
The pre-image of each of $\Hyp$, $\Par^\pm$ and $\Ell$ in $\univcover$ have infinitely many connected components, indexed by the integers $\mathbb{Z}$. We denote these components as follows: Recall that the center of  $\univcover$ is  infinite cyclic consisting of the lifts of $\pm\mathrm I.$  Let $\widetilde{\exp}: \sl\to \univcover$ be the exponential map of $\univcover$, and let $z \doteq \widetilde{\exp}\begin{bmatrix}
        0 & \pi \\
       -\pi & 0 
    \end{bmatrix}$. Then $z$ is a generator of the center of $\univcover$ and for any $n\in\mathbb Z,$ the lift $ \widetilde{\exp}
    \begin{bmatrix}
        0 & n\pi \\
       -n\pi & 0 
    \end{bmatrix} $ of $\pm\mathrm{I}$  in $\univcover$ is equal to $z^n.$ 
(In \cite{goldman}, the generator $z$  is chosen as
    $\widetilde{\exp}
    \begin{bmatrix}
        0 & -\pi \\
       \pi & 0 
    \end{bmatrix}$, which is the inverse of ours.
    As a result, the relative Euler class defined there  
    is the negative of ours.) 
    Let  $\Hyp_0$ be the space of hyperbolic elements of $\univcover$ lying in the image  $\widetilde{\exp}(\sl)$ of the exponential map; 
    and for any $n\in\mathbb Z$, let  $\Hyp_n \doteq z^n\Hyp_0$. 
    Then $\{\Hyp_n\}_{n\in \mathbb Z}$ are the connected components of the space of hyperbolic elements in $\univcover$; and two hyperbolic elements in $\univcover$ are conjugate if and only if they lie in the same connected component $\Hyp_n$  for some $n\in\mathbb Z$ and have the same trace.
    Similarly, let $\Par^+_0$ and $\Par^-_0$ respectively be the spaces of positive and negative parabolic elements of $\univcover$ lying in the image  of the exponential map, and let  $\Par_0
    = \Par^+_0\cup \Par^-_0.$ For any $n\in\mathbb Z,$ let $\Par^+_n = z^n\Par^+_0$ and $\Par^-_n = z^n\Par^-_0,$ and let  $\Par_n \doteq \Par^+_n\cup \Par^-_n=z^n\Par_0.$ Then $\{\Par^+_n\}_{n\in \mathbb Z}$ and $\{\Par^-_n\}_{n\in \mathbb Z}$ are respectively the connected components of the space of positive and negative parabolic elements in $\univcover$; and two parabolic elements in $\univcover$ are conjugate if and only if they lie in the same connected component $\Par^+_n$ or $\Par^-_n$ for some $n\in\mathbb Z.$  Finally, for $n\in \mathbb{Z}$, we denote by $\overline{\Hyp_n}$ the closure of $\Hyp_n$ in $\univcover$, which is a union of 
    $\Hyp_n,$ $\Par_n$ and $\{z^n\}$.

Unlike hyperbolic and parabolic elements, all elliptic elements of $\univcover$ lie in the image of the exponential map; and we need to index their connected components differently. For $n>0$, let  $\Ell_n$ be the subspace of  $\univcover$  consisting of elements conjugate to $\widetilde{\exp}\begin{bmatrix}
        0 & \theta \\
       -\theta & 0 
    \end{bmatrix}$ for some $\theta$ in $((n-1)\pi, n\pi),$ and for $n<0$, let  $\Ell_n$ be the subspace of  $\univcover$ consisting of elements conjugate to $\widetilde{\exp}\begin{bmatrix}
        0 & \theta \\
       -\theta & 0 
    \end{bmatrix}$ for some $\theta$ in $(n\pi, (n+1)\pi).$
    Then $\{\Ell_n\}_{n\in \mathbb Z\setminus\{0\}}$ are the connected components of the space of elliptic  elements in $\univcover;$ and two elliptic elements in $\univcover$ are conjugate if and only if both lie in the same connected component $\Ell_n$ and have the same trace. See Figure \ref{fig: Universal_cover_Ell_colored2}.

\begin{figure}[hbt!]
        \centering
        \begin{overpic}[width=1.0\textwidth]{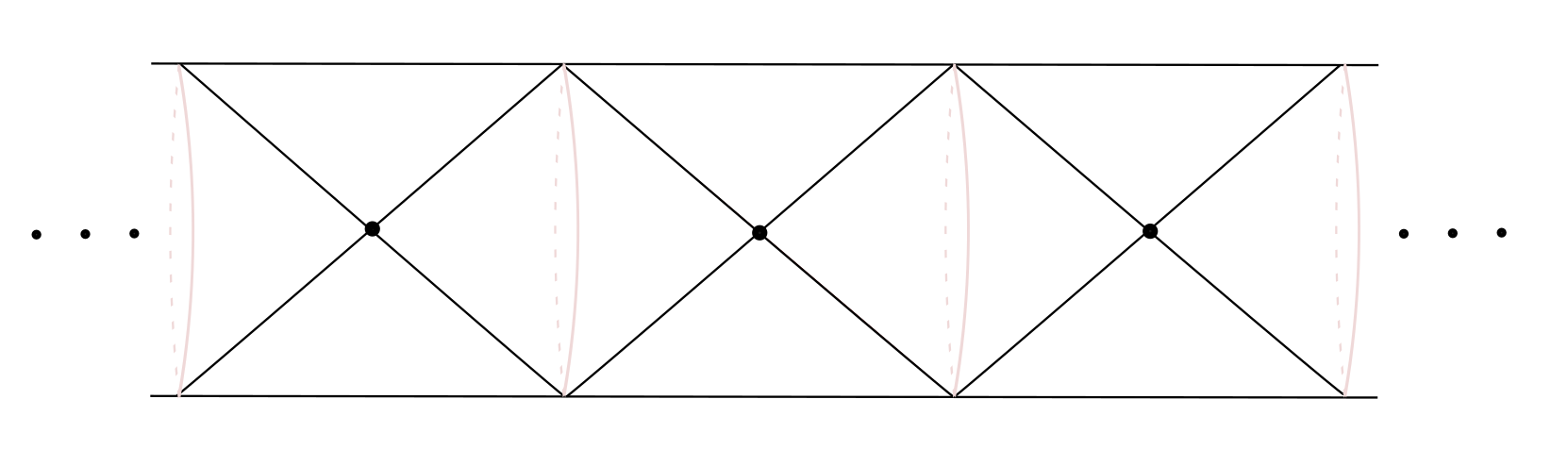}

            \put(22.0,17){\textcolor{black}{$z^{-1}$}}

            \put(48,17){\textcolor{black}{$\mathrm{I}$}}

            \put(72.5,17){\textcolor{black}{$z$}}

        \put(21,6){\textcolor{OliveGreen}{$\Hyp_{-1}$}}

            \put(46,6){\textcolor{OliveGreen}{$\Hyp_{0}$}}
            
            \put(71,6){\textcolor{OliveGreen}{$\Hyp_{1}$}}

            \put(14,23.5){\textcolor{blue}{$\Par^-_{-1}$}}

            \put(27,23.5){\textcolor{blue}{$\Par^+_{-1}$}}

            \put(39,23.5){\textcolor{blue}{$\Par^-_0$}}

            \put(53,23.5){\textcolor{blue}{$\Par^+_0$}}

            \put(64,23.5){\textcolor{blue}{$\Par^-_1$}}

            \put(78,23.5){\textcolor{blue}{$\Par^+_1$}}

        \put(10,14){\textcolor{red}{$\Ell_{-2}$}}
        
        \put(33,14){\textcolor{red}{$\Ell_{-1}$}}

        \put(59,14){\textcolor{red}{$\Ell_1$}}

        \put(84,14){\textcolor{red}{$\Ell_2$}}
        \end{overpic}
        \caption{\label{fig: Universal_cover_Ell_colored2} The universal covering $\univcover$}
    \end{figure}

    The relative Euler class of a representation $\phi$ in $\HP(\Sigma)$ can be defined as follows.
    Let 
    $$\fund = \big\langle a_1, b_1, \cdots, a_g, b_g, c_1, \cdots, c_p \big|
    [a_1, b_1] \cdots [a_g, b_g] c_1\cdots c_p
    \big\rangle$$ be a presentation of $\pi_1(\Sigma)$, where $c_1,\cdots, c_p$ are the preferred peripheral elements of $\pi_1(\Sigma)$.  
    Since for each $i\in\{1,\dots, p\},$ the image $\phi(c_i)$ of the peripheral element $c_i$ is either hyperbolic or parabolic, there exists a unique lift 
    $\widetilde{\phi(c_i)}$  of $\phi(c_i)$ in $\univcover$ that lies in $\overline{\Hyp_0}$.  
 For $j \in \{1,\cdots, g\}$, choose arbitrarily lifts 
    $\widetilde{\phi(a_j)}$ and $\widetilde{\phi(b_j)}$ of $\phi(a_j)$ and $\phi(b_j)$ 
    in $\univcover$; and 
    note that the commutator 
    $[\widetilde{\phi(a_j)}, \widetilde{\phi(b_j)}]$ is independent of the choice of the lifts. 
    Since $$[\phi(a_1), \phi(b_1)] \cdots [\phi(a_g), \phi(b_g)] \phi(c_1)\cdots\phi(c_p)=\pm\mathrm{I},$$ 
    the product $[\widetilde{\phi(a_1)}, \widetilde{\phi(b_1)}] \cdots [\widetilde{\phi(a_g)}, \widetilde{\phi(b_g)}] \widetilde{\phi(c_1)}\cdots\widetilde{\phi(c_p)}$ in $\univcover$ projects to 
    $\pm\mathrm{I}$ in $\psl$ under the covering map, hence lies in the center of $\univcover,$ i.e., there exists an $n\in\mathbb Z$ such that $$[\widetilde{\phi(a_1)}, \widetilde{\phi(b_1)}] \cdots [\widetilde{\phi(a_g)}, \widetilde{\phi(b_g)}] \widetilde{\phi(c_1)}\cdots\widetilde{\phi(c_p)} = z^n.$$
   The \emph{relative Euler class} $e(\phi)$ of the representation $\phi\in\mathrm{HP}(\Sigma)$ is defined to as
   $e(\phi)=n.$


A key property of the relative Euler class is that it defines an integer valued continuous function on $\HP(\Sigma),$ hence is a constant on each connected subspace of $\HP(\Sigma).$   As a consequence, it provides a necessary condition for two representations to lie in the same connected component of a subspace of $\HP(\Sigma)$: If  $\phi$ and $\psi$ are connected by a path in $\HP(\Sigma)$, then 
     $e(\phi) = e(\psi)$. In \cite{goldman}, Goldman proved that for the subspace $\W$ of $\HP(\Sigma)$ consisting of representations with the hyperbolic boundary condition, this necessary condition is also sufficient. As a consequence, the connected components of $\W$ are characterized by the relative Euler class. 
     
     \begin{theorem}\label{thm_goldman}\cite[Theorem 3.3]{goldman}
        Let $\Sigma = \Sigma_{g,p}$ be a connected, oriented punctured surface with genus $g$ and $p$ punctures. Then for each $n\in\mathbb Z,$ the space $\mathcal W_n(\Sigma)=\{\phi\in\W\ |\ e(\phi)=n\}$ is non-empty if and only if $n$ satisfies the Milnor-Wood inequality $$\chi(\Sigma)\leqslant n\leqslant -\chi(\Sigma).$$
        Moreover, each non-empty $\mathcal W_n(\Sigma)$ above is connected. 
    \end{theorem}

    In the rest of this subsection, we list several results of \cite{goldman}  that will be used in the proofs of our main results. Recall that a $\psl$-representation is \emph{reducible} if there exists a line in $\mathbb{R}^2$ that is invariant under the action of  $\phi\big(\pi_1(\Sigma)\big)$, or equivalently, $\phi$ is conjugate to a representation into the 
    Borel subgroup of $\psl$ consisting of the projection of the upper-triangular matrices. A representation is \emph{irreducible} if it is not reducible. 
    The following proposition determines the relative Euler class of reducible representations in $\HP(\Sigma)$, where in the original statement $\phi$ is assumed to lie in $\W$ but the proof works verbatim here.

    \begin{proposition}\cite[Proposition 3.6]{goldman}\label{prop_reducible}
        If $\phi\in \HP(\Sigma)$ is reducible, then $e(\phi)=0$. In particular, if $\phi$ is abelian, then $e(\phi)=0.$
    \end{proposition}


    The relative Euler class satisfies the following additivity property.
    
    \begin{proposition}\label{prop_additivity}\cite[Proposition 3.7]{goldman}
        Let $\Sigma = \Sigma_{g,p},$ let $\gamma$ be a separating simple closed curve on $\Sigma,$ and let $\Sigma_1$ and $\Sigma_2$ respectively be the two components of the complement $\Sigma\setminus \gamma.$ 
     Suppose a representation $\phi\in \HP(\Sigma)$ maps $[\gamma]\in \pi_1(\Sigma)$ to a hyperbolic or parabolic element of $\psl.$ Then $\phi|_{\pi_1(\Sigma_1)}\in \mathrm{HP}(\Sigma_1)$, $\phi|_{\pi_1(\Sigma_2)}\in \mathrm{HP}(\Sigma_2)$, and 
        $$e(\phi) = e\big(\phi|_{\pi_1(\Sigma_1)}\big) + e(\phi|_{\pi_1(\Sigma_2)}).$$
    \end{proposition}


    In the proof of the main results, the additivity of the relative Euler class allows us to reduce the argument to proper subsurfaces of $\Sigma$. To this end, we need to consider a
    \emph{maximal dual-tree decomposition}, which is a decomposition
    $\Sigma = \displaystyle\decomp$  of $\Sigma$ satisfying the following properties:
    \begin{enumerate}[(1)]
        \item Each $\Sigma_k$ in the decomposition is a subsurface of Euler characteristic $-1,$ 
        \item two subsurfaces are either disjoint or intersect at  
        boundary components, and 
        \item  the dual graph of the decomposition is a tree.
        \end{enumerate}
    Here, the dual graph is a graph 
    whose vertices correspond to subsurfaces in the decomposition, 
    and whose edges correspond to the common 
    boundary components of the adjacent pairs of subsurfaces.  We call the common boundary components in the decomposition the \emph{decomposition curves}; and by abuse of notation, we also call the elements in $\pi_1(\Sigma)$ represented by these decomposition curves the \emph{decomposition curves}.
 The dual graph being a tree and the subsurfaces being either the one-hole torus or the three-hole sphere  imply that in a maximal dual-tree decomposition of $\Sigma_{g,p},$ there are exactly $g$ one-hole tori and $g+p-2$ three-hold sphere; or equivalently, the tree has exactly $g$ univalent vertices and $g+p-2$ trivalent vertices. 

   The  following graph theoretical lemma allows us to further reduce the argument to subsurfaces of Euler characteristic $-1$ or $-2$. 
    \begin{lemma}
    \cite[Lemma 3.9]{goldman}\label{lem_graph}
        Let $\Sigma$ be a surface with Euler characteristic $\chi(\Sigma)\leqslant -2$, together with a maximal dual-tree decomposition $\Sigma = \decomp$. For some $n\in\mathbb Z$ and $s\in\{-1,0,+1\}^p,$ let $\phi, \psi\in \HPns(\Sigma)$  
        be representations that map all the decomposition curves to hyperbolic elements of $\psl$. Then there exists a sequence of representations 
        $\phi_1, \cdots, \phi_N$ in $\HPns(\Sigma)$ such that $\phi_1=\phi,$ $\phi_N=\psi,$ and for each $j \in\{0,\cdots, N-1\}$, $\phi_j$ and $\phi_{j+1}$ are related as follows: 
        After a possibly reindexing, there exists a connected subsurface 
        $\Sigma_1\cup \Sigma_2$ such that $e(\phi_j|_{\pi_1(\Sigma_1\cup \Sigma_2)}) = e(\phi_{j+1}|_{\pi_1(\Sigma_1\cup \Sigma_2)})$ and for all $k\in \{3,\cdots, -\chi(\Sigma)\}$, 
        $e(\phi_j|_{\pi_1(\Sigma_k)}) =
        e(\phi_{j+1}|_{\pi_1(\Sigma_k)})$.
    \end{lemma}

\subsection{Path-lifting property}\label{PLP}

    As another important ingredient in the proof, a smooth map $f: X\to Y$ between two smooth manifolds $X$ and $Y$ is said to satisfy the
    \emph{path-lifting property} if, 
    for every $x\in X$ and path $\{y_t\}_{t\in [0,1]}$ with $f(x) = y_0$, there exists a non-decreasing surjective map (a re-parametrization of the path) $\tau: [0,1]\to [0,1]$ and a path 
    $\{x_s\}_{s\in [0,1]}$ starting at $x_0 = x$ such that $f(x_s) = y_{\tau(s)}$ for $0\leqslant s\leqslant 1$. 
    The following lemma provides a sufficient condition for a map to satisfy the path-lifting property.
    \begin{lemma}\cite[Lemma 1.4]{goldman}\label{lem_plp}
        Let $X, Y$ be smooth manifolds, and let $f: X\to Y$ be a smooth map. Suppose that 
        (a) $f$ is a submersion, and 
        (b) $f^{-1}(y)$ is connected for every $y\in Y$.
        Then $f$ satisfies the path-lifting property. Furthermore, 
        if $Y$ is path-connected, then $X$ is path-connected.
    \end{lemma}


  One  example of the maps that we need to satisfy the path-lifting property is the following  \emph{character map}
    $\chi: \SL\times\SL\to \mathbb{R}^3$  defined  by 
    $$
    \chi(A,B) \doteq 
    \big(
    \mathrm{tr}(A), 
    \mathrm{tr}(B), 
    \mathrm{tr}(AB)\big).
    $$
 Since $\pi_1(\Sigma_{0,3}) = \langle c_1, c_2, c_3\ |\ c_1c_2c_3\rangle \cong \mathbb{F}_2$ is a free group of rank $2,$ a representation $\phi:\pi_1(\Sigma_{0,3})\to \SL$ is determined by the images $\phi(c_1)$ and $\phi(c_2)$ in $\SL$. 
    As a consequence, 
    $$\mathrm{Hom}\big(\pi_1(\Sigma_{0,3}),  \SL\big)\cong \SL\times\SL$$
    under the identification of $\phi$ with the pair $\big(\phi(c_1), \phi(c_2)\big),$ on which the character map $\chi$ is defined.

    Similarly, the character of $\psl$-representations can be defined under suitable boundary conditions. Recall that the trace of an element in $\univcover$ is defined as the trace of its projection to $\SL$. 
    Let $\overline{\Hyp}=\mathrm{Hyp}\cup\mathrm{Par}\cup\{\pm\mathrm{I}\}$ be the subspace of $\psl$ consisting of hyperbolic elements, parabolic elements and the identity element.
    If $\phi:\pi_1(\Sigma_{0,3})\to \psl$ maps the peripheral elements $c_1, c_2$ of $\pi_1(\Sigma_{0,3})$ into $\overline{\Hyp}$, 
    then there exists a unique lift $\tphicone$ of $\phi(c_1)$ and a unique lift $\tphictwo$ of $\phi(c_2)$ in $\overline{\Hyp}_0= \Hyp_0
    \cup 
    \Par_0
    \cup 
    \{\mathrm I\}$. 
    We define the \emph{character} of such $\phi$ as the triple
    $$
    \chi(\phi) \doteq 
    \Big(
    \mathrm{tr}\big(\tphicone\big), 
    \mathrm{tr}\big(\tphictwo\big), 
    \mathrm{tr}\big(\tphicone\tphictwo\big)\Big).
    $$
    In particular, 
    if $\phi$ lies in $\HP(\Sigma_{0,3})$ with $e(\phi) =n$, then by the definition of the relative Euler class we have
    $\widetilde{\phi(c_1)}
    \widetilde{\phi(c_2)}
    \widetilde{\phi(c_3)}
    = z^n$ for the unique lifts
    $\widetilde{\phi(c_1)},$ $
    \widetilde{\phi(c_2)},$ $ \widetilde{\phi(c_3)}$ of $\phi(c_1)$, $\phi(c_2)$, $\phi(c_3)$ in $\overline{\Hyp_0}$. 
  Therefore,  
    the product $\widetilde{\phi(c_1)}
    \widetilde{\phi(c_2)}
    = z^n \widetilde{\phi(c_3)}^{-1}
    $ lies in $\overline{\Hyp_n}=
    z^n\overline{\Hyp_0}$.
    Observe that, if $n$ is even,  then the trace of an element of $\overline{\Hyp_n}$ lies in $[2,\infty);$ and if $n$ is odd, then the trace of an element of $\overline{\Hyp_n}$ lies in $(-\infty, -2].$ 
    As a consequence, the character $\chi(\phi)$ lies in $[2,\infty)^3$ when $e(\phi)$ is even, and lies in $[2,\infty)^2\times (-\infty, -2]$ when $e(\phi)$ is odd. 
    \\

    \begin{proposition}
    \cite[Proposition 4.1, Theorem 4.3]{goldman}
    \label{prop_char}
        Let $\chi: \SL\times\SL\to \mathbb{R}^3$ be the character map, and let $\kappa: \mathbb{R}^3 \to \mathbb{R}$ 
        be the polynomial map defined by
        $$
        \kappa(x,y,z) \doteq x^2 + y^2 + z^2 - xyz - 2 .
        $$
        Then 
        \begin{enumerate}[(1)]
            \item 
        For a pair $(A,B)$ in $\SL\times\SL$, $\kappa\big(\chi(A,B)\big) = \tr[A,B]$. Furthermore, if the projection of $(A,B)$ to $\psl\times \psl$ lies in $\overline{\Hyp}\times \overline{\Hyp}$,
        then $\kappa\big(\chi(A,B)\big) = 2$ if and only if the representation $\phi:\pi_1(\Sigma_{0,3})\to \SL$ determined by the pair $(A,B)$ is reducible.
        \item  If 
        $\kappa(x,y,z)\neq 2$
        and $\chi^{-1}(x,y,z)$ is non-empty, 
        $\chi^{-1}(x,y,z)$ in $\SL\times \SL$ 
        consists of one orbit of the $\GL$-conjugation, which is the union of two orbits of the $\SL$-conjugation.
        \end{enumerate} 
    \end{proposition}
    
    \begin{lemma}\cite[Theorem 4.3, Lemma 4.5 (b)]{goldman}
    \label{lem_plpchar}
        Let 
        $$\Omega 
        \doteq
        \{(A,B)\in \SL\times\SL\ |\ 
        [A,B]\neq \mathrm I
        \}.$$ Then the character map
        $$\chi: \Omega\to \mathbb{R}^3
        \setminus \big([-2,2]^3\cap \kappa^{-1}([-2,2])\big)
        $$ is surjective, and satisfies the path-lifting property.
    \end{lemma}


    Another example of the maps satisfying the path-lifting property is the following \emph{lifted commutator} $\widetilde{R}: \psl \times \psl \to \univcover$
    defined for any pair of elements  $(\pm A,\pm B)$ of $\psl$ by 
    $$\widetilde{R}(\pm A, \pm B) = \widetilde{A} \widetilde{B} \widetilde{A}^{-1} \widetilde{B}^{-1},$$ where 
    $\widetilde{A}$ and $\widetilde{B}$ are respectively arbitrary lifts of $\pm A$ and $\pm B$ in $\univcover$. The lifted commutator map is closely related to the $\psl$-representations of $\pi_1(\Sigma_{1,1})=\langle a,b,c\ |\ [a,b]\cdot c\rangle\cong \mathbb{F}_2$, where the peripheral element $c^{-1}$ equals the commutator of the two generators $a$ and $b.$

        \begin{theorem}\cite[Theorem 7.1]{goldman}\label{thm_liftcommu}
            The image of the lifted commutator 
            $\widetilde{R}: \psl\times \psl \to \univcover$ equals 
            $$\mathcal{I} = 
            \{\mathrm I\}
             \cup
            \mathrm{Hyp}_0
            \cup 
            \mathrm{Par}_0
            \cup 
            \mathrm{Ell}_{-1}
            \cup
            \mathrm{Ell}_{1}
            \cup 
            \mathrm{Par}_{-1}^{+}
             \cup 
            \mathrm{Par}_{1}^{-}
            \cup
            \mathrm{Hyp}_{- 1}
            \cup
            \mathrm{Hyp}_{1}.
            $$
            Moreover, for 
            each $C$ in the image $\mathcal{I}$, the fiber 
            $\widetilde{R}^{-1}(C)$ is connected.
        \end{theorem}

\begin{figure}[hbt!]
        \centering
        \begin{overpic}[width=0.8\textwidth]{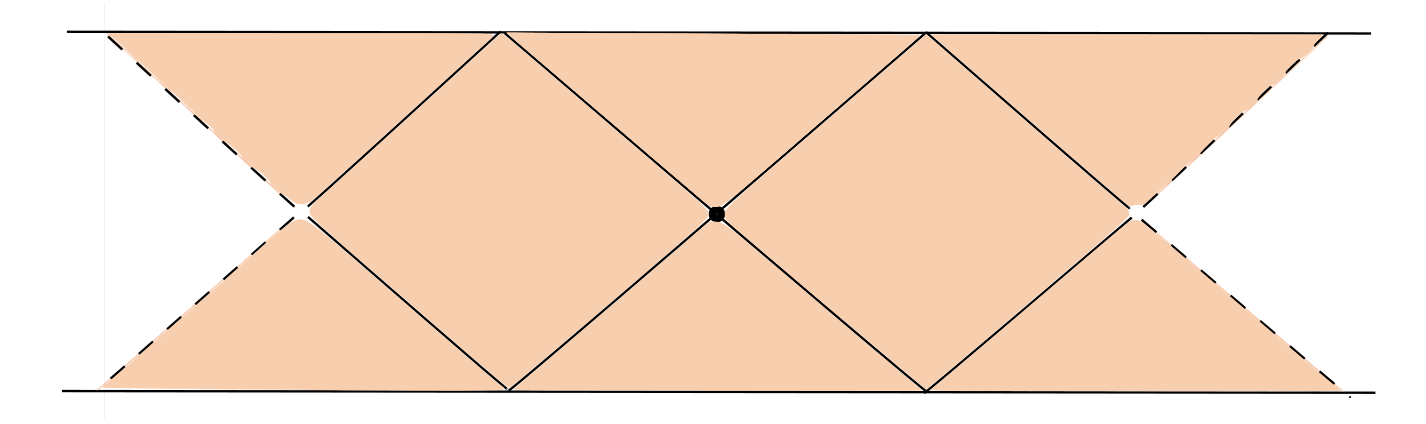}


            \put(49.8,17){\textcolor{black}{$\mathrm{I}$}}


        \put(19,4){\textcolor{OliveGreen}{$\Hyp_{-1}$}}

            \put(48,4){\textcolor{OliveGreen}{$\Hyp_{0}$}}
            
            \put(77,4){\textcolor{OliveGreen}{$\Hyp_{1}$}}


            \put(24,25){\textcolor{blue}{$\Par^+_{-1}$}}

            \put(39,25){\textcolor{blue}{$\Par^-_0$}}

            \put(56.5,25){\textcolor{blue}{$\Par^+_0$}}

            \put(69,25){\textcolor{blue}{$\Par^-_1$}}


        
        \put(33,14){\textcolor{red}{$\Ell_{-1}$}}

        \put(63,14){\textcolor{red}{$\Ell_1$}}

        \end{overpic}
        \caption{\label{fig: evimage} The image of the lifted commutator  and the lifted product maps}
    \end{figure}
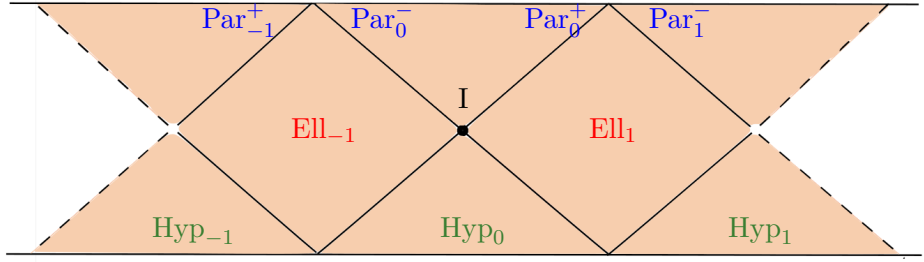

\begin{lemma}\label{lem_evTsubm}
    Let 
    $R: \SL\times \SL
    \to \SL$ defined by $R(A,B) = ABA^{-1}B^{-1}$. Then its differential $dR$ is surjective at the point $(A,B)\in \SL\times \SL$ if and only if $A$ and $B$ do not commute. As a consequence, the restriction of the lifted commutator $$\widetilde{R}: \widetilde{R}^{-1}(\mathcal{I} \setminus \{\mathrm I\}) \to \mathcal{I} \setminus 
            \{\mathrm I\}$$ 
    is a submersion.
\end{lemma}

A proof of Lemma \ref{lem_evTsubm} is included in the Appendix. As a consequence of Lemma \ref{lem_plp}, Theorem \ref{thm_liftcommu} and Lemma \ref{lem_evTsubm}, we have the following Proposition \ref{prop_plpliftcommu}.

        \begin{proposition}\label{prop_plpliftcommu}
            The lifted commutator 
            $\widetilde{R}: \widetilde{R}^{-1}(\mathcal{I} \setminus \{\mathrm I\}) \to \mathcal{I} \setminus 
            \{\mathrm I\}$ satisfies the path-lifting property.
        \end{proposition}

    In Section 3, we will define the \emph{lifted product} $\ev: \overline{\Hyp}\times \overline{\Hyp} \to \univcover$, which is closely related to the $\psl$-representations of $\pi_1(\Sigma_{0,3})=\langle c_1,c_2,c_3\ |\ c_1c_2c_3\rangle\cong \mathbb{F}_2$.
    To show its path-lifting property in Theorem \ref{thm_PLP}, we need the following Lemma \ref{lem_evPsubm}, whose proof is included in the Appendix. 
\begin{lemma}\label{lem_evPsubm}

    Let $m: \SL\times \SL
    \to \SL$ be defined by $(A,B)\mapsto AB$.
    Let $U$ be an open subset of $\SL$, and let 
    $\Par$ be the subset of $\SL$ consisting of the parabolic matrices.
    \begin{enumerate}[(1)]
        \item The restrictions of $m$ to 
        $U\times U$, $\Par\times U$, and $U\times \Par$ are submersions. 
        \item At $(A,B)\in \Par\times \Par$, the differential $dm_{(A,B)}$ of the restriction   
    $m: \Par\times \Par
    \to \SL$ is surjective if and only if $A$ and $B$ do not commute. 
    \end{enumerate}
\end{lemma}
\begin{remark}
    Recall that the set $\Par$ is previously defined as the set of parabolic elements of $\psl$. In Lemma \ref{lem_evPsubm}, the set $\Par\subset \SL$ refers to the inverse image of $\Par\subset \psl$ under the covering map $\SL\to \psl$; here we abuse notation and continue to denote this set by $\Par$. In the rest of the paper, $\Par$ always refers to the subset of $\psl$ consisting of parabolic elements.
\end{remark}



   \subsection{The signs}\label{Sign}

    For $\Sigma = \Sigma_{g,p}$ with $c_1, \cdots, c_p$ the preferred peripheral elements of $\fund.$  The \emph{sign} of a representation $\phi\in\HP(\Sigma)$ is a $p$-tuple $s(\phi)=(s_1,\dots,s_p)$ in $\{-1,0,+1\}^p$  with 
    \begin{equation*}
   s_i = \left\{
    \begin{array}{rl}
       +1,  &  \text{if } \phi(c_i) \text{ is positive parabolic}\\
        0, & \text{if } \phi(c_i) \text{ is hyperbolic}\\ 
        -1, &  \text{if } \phi(c_i) \text{ is negative parabolic},
    \end{array}\right.
    \end{equation*}
$i\in\{1,\dots,p\}.$
    Since the type of an element of $\psl$ is invariant under $\psl-$conjugation, the sign of $\phi$ is independent on the choice of the preferred peripheral elements $c_i$'s. In the rest of this paper, for an $s\in \{-1,0,+1\}^p$, we respectively let  $p_+(s),$ $p_0(s)$ and $p_-(s)$ be the number of $1$'s, $0$,s and $-1$'s in the components of  $s$.
    
   Similar to the relative Euler class, the sign also defines an integer valued continuous function on $\HP(\Sigma),$ hence is a constant on each connected subspace of $\HP(\Sigma).$   As a consequence, it provides a necessary condition for two representations to lie in the same connected component of a subspace of $\HP(\Sigma)$: If  $\phi$ and $\psi$ are connected by a path in $\HP(\Sigma)$, then  $s(\phi) = s(\psi)$.
   \bigskip

   The following lemma provides a convenient way to determine the sign of a parabolic element.  

    \begin{lemma}\label{lem_offdiag} 
        For $n\in \mathbb{Z}$ and a parabolic element $\widetilde{A}\in \mathrm{Par}_n$ of $\univcover,$ let  $s(\widetilde{A})\in\{\pm\}$ be its sign, i.e., $s(\widetilde{A})=+$ if $\widetilde{A}\in \mathrm{Par}^+_n,$ and $s(\widetilde{A})=-$ if $\widetilde{A}\in \mathrm{Par}^-_n.$ Let $A$ be the projection of $\widetilde{A}$ to $\SL$, and for $i,j\in\{1,2\}$ let $a_{ij}$ be the $(i,j)$-entry of $A.$ Then either $a_{12}\neq0$ or $a_{21}\neq 0.$  Furthermore,
        
        \begin{enumerate}[(1)]
       \item If $n$ is even, then $s(\widetilde{A})=sgn(a_{12})$ if $a_{12}\neq 0,$ and $s(\widetilde{A})=-sgn(a_{21})$ if $a_{21}\neq 0.$
       
    \item If $n$ is odd, then $s(\widetilde{A})=-sgn(a_{12})$ if $a_{12}\neq 0,$  and $s(\widetilde{A})=sgn(a_{21})$ if $a_{21}\neq 0.$
    \end{enumerate}
    \end{lemma}

    \begin{proof}  Recall that a parabolic element  of $\psl$ is \emph{positive} if it is conjugate to $\pm\parpp,$ and is \emph{negative} if it is conjugate to $\pm \parmp;$ and a parabolic element in the universal covering $\widetilde{\SL}$ of $\psl$ is \emph{positive} or \emph{negative} if its projection to $\psl$ is so. 
    Therefore, by the definition, if $\widetilde{A}$ is positive, then $A$ is conjugate to $\parpp$ when $n$ is even, and is conjugate to $-\parpp$ when $n$ is odd. This means $A = \pm \abcd\parpp\abcd^{-1}
        = \pm \begin{bmatrix}
                1-ac & a^2 \\
                -c^2 & 1+ac
          \end{bmatrix}$ for some matrix $\abcd$ in $\SL,$ again, $+$ when $n$ is even and $-$ when $n$ is odd.  Note that $ad - bc = 1$ implies that at least one of $a$ and $c$ is nonzero. As a consequence, in the case that $n$ is even, if $a\neq 0,$ then $a_{12}=a^2>0,$ and if $c\neq 0,$ then $a_{21}=-c^2<0;$ and in the case that $n$ is odd, if $a\neq 0,$ then $a_{12}=-a^2<0,$ and if $c\neq 0,$ then $a_{21}=c^2>0.$  Similarly, if $A$ is negative, then $A$ is conjugate to $\parmp$ when $n$ is even, and is conjugate to $-\parmp$ when $n$ is odd; and we have $A =\pm \abcd\parmp\abcd^{-1}
        = \pm \begin{bmatrix}
                1+ac & -a^2 \\
                c^2 & 1-ac
          \end{bmatrix}.$ As a consequence, in the case that $n$ is even, if $a\neq 0,$ then $a_{12}=-a^2<0,$ and if $c\neq 0,$ then $a_{21}=c^2>0;$ and in the case that $n$ is odd,  if $a\neq 0,$ then $a_{12}=a^2>0,$ and if $c\neq 0,$ then $a_{21}=-c^2<0.$
    \end{proof} 
    
The following proposition describes the behavior of the relative Euler class and the sign under the $\pgl\setminus\psl$-conjugations.

    \begin{proposition}\label{prop_pglpsl} 
    Let $h$ be an element of $\pgl\setminus \psl$; and for a representation $\phi\in\HP(\Sigma),$ let $h \phi h^{-1}:\pi_1(\Sigma)\to\psl$ be  defined by 
    $h \phi h^{-1}(c)=h \phi(c)h ^{-1}$
 for each $c\in\pi_1(\Sigma).$  Then $h \phi h^{-1}\in\HP(\Sigma)$ with the relative Euler class $e(h \phi h^{-1})=-e(\phi)$ and the sign $s(h \phi h^{-1})=-s(\phi).$    
    \end{proposition}

\begin{proof} Let $J=\pm \begin{bmatrix}
        0 & 1 \\
        1 & 0
    \end{bmatrix}\in\pgl\setminus\psl$. As every element of $\pgl\setminus \psl$ has the form $Jg$ for some  $g$ in $\psl,$ up to a $\psl$-conjugation of $\phi$, we can assume that $h=J.$ Then we observe that for $\pm A = \pm \abcd\in\psl,$ 
    $\pm JAJ^{-1}
    = \pm \begin{bmatrix}
        d & c \\
        b & a
    \end{bmatrix}
    $, and hence $|\tr(\pm A)|=|\pm JAJ^{-1}|=|a+d|$, and $\pm A$ and $\pm JAJ^{-1}$ share the same type. As a consequence, $h\phi h^{-1}\in \HP(\Sigma).$ Furthermore, if both $\pm A$ and $\pm JAJ^{-1}$ are parabolic, then by Lemma \ref{lem_offdiag} they have the opposite sign. This implies that the sign $s(h \phi h^{-1})=-s(\phi).$ The relationship of the relative Euler classes is included in \cite[Proof of Proposition 4.6]{goldman}. 
    \end{proof}

    \subsection{Some other results of Goldman}\label{Other}

In this subsection, we recall the known results that is needed for the proof in this paper. The two lemmas below will be needed for punctured surfaces of Euler characteristic $-2$, namely, the two-hole torus $\Sigma_{1,2}$ and the four-hole sphere $\Sigma_{0,4}.$

    \begin{lemma}\cite[Lemma 9.3]{goldman}\label{lem_inthyptorus}
    	Let $\Sigma = \Sigma_{1,2}$
        with $c_1, c_2$ the peripheral elements of $\fund$. Let $d=(c_1c_2)^{-1}$ be the element of $\pi_1(\Sigma)$ which is represented by a simple closed curve that separates $\Sigma$ into a one-hole torus $T$  and a three-hole sphere $P$.
        Let $\phi\in \hom$ such that $\phi\big(\pi_1(P)\big)$ and $\phi\big(\pi_1(T)\big)$ are non-abelian. Then,
        there exists a path $\{\phi_t\}_{t\in [0,1]}$ in $\hom$
         such that:
         \begin{enumerate}[(1)]
            \item $\phi_0 = \phi$,
            \item for all $t\in [0,1]$, $\phi_t|_{\pi_1(P)}$ and $\phi_t|_{\pi_1(T)}$ are non-abelian,
            \item for all $t\in [0,1]$ and $i\in\{1,2\},$ $\phi_t(c_i)$ is conjugate to $\phi(c_i)$, and
            \item $\phi_1(d)$ is hyperbolic.
            \end{enumerate}
    \end{lemma}

    \begin{lemma}\cite[Lemma 9.6]{goldman}\label{lem_inthyppants}
    Let $\Sigma = \Sigma_{0,4}$
        with $c_1, c_2, c_3, c_4$ the  peripheral elements in $\fund$. Let $d=(c_1c_2)^{-1}$ be the element of $\pi_1(\Sigma)$ which is represented by a simple closed curve that separates $\Sigma$ into two  three-hole spheres $P_1$ and $P_2$.
        Let $\phi\in \hom$ such that $\phi\big(\pi_1(P_1)\big)$ and $\phi\big(\pi_1(P_2)\big)$ are non-abelian. 
        Suppose in addition  that there exists an $i\in \{1,2,3,4\}$ such that $\phi(c_i)$ is hyperbolic.
        Then there exists a path $\{\phi_t\}_{t\in [0,1]}$ in $\hom$ such that:
        \begin{enumerate}[(1)]
            \item $\phi_0 = \phi$,
            \item for all $t\in [0,1]$, $\phi_t|_{\pi_1(P_1)}$ and $\phi_t|_{\pi_1(P_2)}$ are non-abelian,
            \item for all $t\in [0,1]$ and $i\in\{1,2,3,4\}$,  $\phi_t(c_i)$ is conjugate to $\phi(c_i)$, and 
       \item $\phi_1(d)$ is hyperbolic.
       \end{enumerate}
    \end{lemma}
    \begin{remark}
        In each of Lemma \ref{lem_inthyptorus} and Lemma \ref{lem_inthyppants}, Condition (2) is not stated explicitly in \cite{goldman}, however it is nevertheless satisfied, as the proofs rely on Lemma \ref{lem_plpchar} and Lemma \ref{prop_plpliftcommu}.
    \end{remark}

     \begin{lemma}\cite[Lemma 1.3]{goldman}\label{lem_conjugacypath}
            Let $\{\pm A_t\}\interval$ and $\{\pm B_t\}\interval$ 
            be piecewise smooth paths in $\psl \setminus \{\pm \mathrm I\}$ such that 
            for each $t\in [0,1]$, the elements $\pm A_t$ and $\pm B_t$ are conjugate. 
            Suppose that each path is transverse to the subset $\Par$ of $\psl \setminus \{\pm \mathrm I\}$. 
            Then there exists a continuous path 
            $\{g_t\}_{t\in [0,1]}$ in $\psl$ such that $\pm B_t = g_t(\pm A_t)g_t^{-1}$. 
        \end{lemma}


Very similar to Lemma \ref{lem_conjugacypath}, we will also need the following Lemma \ref{lem_conjugacypath_const}. 
    
    \begin{lemma}\label{lem_conjugacypath_const}
        Let $\pm A$ be an element of $\psl\setminus \{\pm \mathrm I\}$,
        and let $\{\pm A_t\}\interval$ 
        be a continuous path in $\psl$ with $\pm A_0 = \pm A$. 
        Suppose that for each $t\in [0,1]$, the elements $\pm A_t$ and $\pm A$ are conjugate. 
        Then there exists a continuous path 
        $\{g_t\}_{t\in [0,1]}$ in $\psl$ with $g_0 = \pm \mathrm{I}$ such that $\pm A_t = g_t(\pm A)g_t^{-1}$. 
    \end{lemma}
\begin{proof}
  
    If $\pm A$ is  hyperbolic or elliptic, then by applying Lemma \ref{lem_conjugacypath} to the path $\{\pm A_t\}\interval$ and the constant path $\pm B_t\equiv \pm A$, we obtain a path 
    $\{h_t\}_{t\in [0,1]}$ in $\psl$ such that for all $t\in [0,1]$, $\pm A_t = h_t(\pm A)h_t^{-1}$. Since $\pm A = \pm A_0 = h_0(\pm A)h_0^{-1}$, $h_0$ commutes with $\pm A$. Then by letting $g_t \doteq h_th_0^{-1}$ for each $t\in [0,1]$, 
    we define a path
    $\{g_t\}\interval$ in $\psl$ with $g_0 = \pm \mathrm{I}$, such that 
    $g_t(\pm A)g_t^{-1}
    = h_th_0^{-1}(\pm A)h_0h_t^{-1}
    = h_t(\pm A)h_t^{-1}
    = \pm A_t$. 
    \smallskip
    
    If $\pm A$ is parabolic, then up to conjugation we can assume that  $\pm A 
    = \pm 
    \begin{bmatrix}
        1 & s\\
        0 & 1
    \end{bmatrix}$ for some $s\in \{\pm 1\}$. 
 Recall that for each $t\in [0,1]$, as an orientation preserving isometry of $\mathbb{H}^2$, $\pm A_t$ has a unique fixed point $x_t$ on $\partial \mathbb{H}^2= \mathbb{R}\cup\{\infty\}$.
    As $\{\pm A_t\}\interval$ is a continuous path in $\psl$, $\{x_t\}\interval$ defines a continuous path in $\partial \mathbb{H}^2 $ starting from  $x_0 = \infty$ which is the fixed point of $\pm A_0 = \pm\begin{bmatrix}
        1 & s \\
        0 & 1
    \end{bmatrix} $. 
    For each $t\in [0,1]$, we let
    $\sigma_t \doteq \pm \dfrac{1}{\sqrt{1+x_t^2}}
    \begin{bmatrix}
        x_t & -1 \\
        1 & x_t
    \end{bmatrix}$.
    Then $\sigma_0 = \pm \mathrm I$; and for each $t\in [0,1],$
   $\sigma_t^{-1}(\pm A_t)\sigma_t$ fixes $\infty\in \partial \mathbb{H}^2$. Thus,
    $\sigma_t^{-1}(\pm A_t)\sigma_t = 
    \pm 
    \begin{bmatrix}
        1 & \xi_t\\
        0 & 1
    \end{bmatrix}$ for some $\xi_t\in \mathbb{R}\setminus\{0\}$. Since $\{\sigma_t^{-1}(\pm A_t)\sigma_t\}\interval$ is a continuous path in $\psl$ starting from $\sigma_0^{-1}(\pm A_0)\sigma_0 = \pm \begin{bmatrix}
        1 & s\\
        0 & 1
    \end{bmatrix}$, $\{\xi_t\}\interval$ is a continuous path in $\mathbb R\setminus \{0\}$ starting from $\xi_0=s$. As a consequence,  $sgn(\xi_t) = sgn(s)$ for all $t\in [0,1]$. 
    Let $h_t \doteq \pm \begin{bmatrix}
        \sqrt{s\xi_t} & 0 \\
        0 & \sqrt{s\xi_t}^{-1}
    \end{bmatrix}$, and let $g_t \doteq \sigma_th_t$. This defines a continuous path $\{g_t\}\interval$ such that $g_0 = \sigma_0h_0 = \pm \mathrm I$, and  
    $\pm A_t 
    = \sigma_t \bigg(\pm \begin{bmatrix}
        1 & \xi_t\\
        0 & 1
    \end{bmatrix}\bigg) \sigma_t^{-1}
    = g_t(\pm A)g_t^{-1}$.
    \end{proof}



\section{Evaluation maps and surfaces with Euler characteristic $-1$}

    In this section, we define and study the path-lifting properties of the evaluation maps, which will be used to prove the main results for surfaces with Euler characteristic $-1$ in this section and for the general surfaces later.

        \subsection{One-hole torus}
        
     The fundamental group of the one-hole torus  $\Sigma_{1,1}$  is a free group of rank two with a presentation 
        $\langle a,b,c\ |\ [a,b]\cdot c\rangle,$ where $c$ is the preferred peripheral element; and a $\psl$-representation $\phi$ of $\pi_1(\Sigma_{1,1})$ can be identified with 
        a pair $(\phi(a),\phi(b)) \in \psl\times \psl$ with  $c$ mapped  to 
        $[\phi(c_1),\phi(c_2)]^{-1}$.  
        The \emph{evaluation map} 
        $$\ev:\mathcal R(\Sigma_{1,1})\to\univcover$$
        is the restriction of the lifted commutator map $\widetilde{R}: \psl\times\psl\to \univcover$ to $\mathcal R(\Sigma_{1,1})$ recalled in Section \ref{Other} and defined by sending the pair $(g_1,g_2)$ 
        to the commutator $[\widetilde{g_1}, \widetilde{g_2}]$, where $\widetilde{g_1}$ and $\widetilde{g_2}$ are arbitrary lifts of $g_1$ and $g_2$  in $\univcover$. 

        \begin{theorem}\label{thm_torus}
           The space $\mathcal{R}(\Sigma_{1,1})$  has four connected components $\mathcal R_{-1} ^{-1} (\Sigma_{1,1}),$ $\mathcal R_{0} ^{+1} (\Sigma_{1,1}),$ $\mathcal R_{0} ^{-1} (\Sigma_{1,1})$ and $\mathcal R_{1} ^{+1} (\Sigma_{1,1}).$
        \end{theorem}
        
        \begin{proof} By Theorem \ref{thm_liftcommu}, the image of the evaluation map  $\ev:\mathcal R(\Sigma_{1,1})\to \widetilde{\SL}$ lies in the union of the four pieces $\mathrm{Par}_{-1}^+,$ $\mathrm{Par}_0^-,$ $\mathrm{Par}_0^+$ and $\mathrm{Par}_{1}^-;$ and  the continuity of $\ev$ implies that each connected component of $\mathcal{R}(\Sigma_{1,1})$ is mapped into one of these pieces. We will prove that the pre-image of each of these four pieces is non-empty and connected, from which the result follows.
        
 To show that these spaces are non-empty, under the identification of the $\psl$-representations of $\pi_1(\Sigma_{1,1})$ and pairs of elements of $\psl,$ we have  
         $\ev\Bigg(\pm \begin{bmatrix}
        \frac{5}{3} & \frac{4}{3} \\
        \frac{4}{3} & \frac{5}{3}
        \end{bmatrix},
         \pm \begin{bmatrix}
        2 & 0 \\
        0 & \frac{1}{2}
        \end{bmatrix}
        \Bigg)\in \mathrm{Par}_{-1}^{+},$ 
       $\ev\Bigg(
        \pm \begin{bmatrix}
        2 & 0 \\
        0 & \frac{1}{2}
        \end{bmatrix},
        \pm \begin{bmatrix}
         1 & 1 \\
        0 & 1
        \end{bmatrix}
        \Bigg)\in \mathrm{Par}_{0}^{+},$ 
        $\ev\Bigg( \pm \begin{bmatrix}
        1 & 1 \\
        0 & 1
        \end{bmatrix},
        \pm \begin{bmatrix}
        2 & 0 \\
        0 & \frac{1}{2}
        \end{bmatrix}
        \Bigg) \in  \mathrm{Par}_{0}^{-}$,  and
         $\ev\Bigg(
        \pm \begin{bmatrix}
        2 & 0 \\
        0 & \frac{1}{2}
        \end{bmatrix}, 
        \pm \begin{bmatrix}
        \frac{5}{3} & \frac{4}{3} \\
        \frac{4}{3} & \frac{5}{3}
        \end{bmatrix}
        \Bigg)\in \mathrm{Par}_{1}^{-}.$

        Next we show that the pre-image under $\ev$ of each of these pieces $\mathrm{Par}_{-1}^+,$ $\mathrm{Par}_0^-,$ $\mathrm{Par}_0^+$ and $\mathrm{Par}_{1}^-$ is connected. For $\phi, \psi\in \mathcal{R}(\Sigma_{1,1})$ with $\ev(\phi)$ and $\ev(\psi)$ in the same piece, we will find a path in the pre-image  connecting $\phi$ and $\psi.$ 
        Since  $\ev(\phi)$ and $\ev(\psi)$ are lifts of parabolic elements that have the same sign, i.e., conjugate parabolic elements, 
       there is a $g\in \psl$ such that $[\psi(a), \psi(b)] = g[\phi(a), \phi(b)]g^{-1};$ and for any paths $\{g_t\}_{t\in[0,1]}$ in $\psl$ connecting $\pm \mathrm I$ and $g,$ the path $\{g_t\phi g_t^{-1}\}_{t\in[0,1]}$ connects $\phi$ and $g\phi g^{-1}$. Since for all $t\in[0,1],$ the commutator $g_t[\phi(a),\phi(b)]g_t^{-1}$ of $\phi_t$ is parabolic, the curve $\{\phi_t\}$ stays in the pre-image of one of the four pieces. 
        On the other hand, since the commutator  $[g\phi g^{-1}(a), g\phi g^{-1}(b)]=g[\phi(a),\phi(b)]g^{-1}=[\psi(a),\psi(b)]$ and the relative Euler class $e(g\phi g^{-1})=e(\phi)=e(\psi)$,  we have
        $\ev(g\phi g^{-1}) = \ev(\psi)\in\widetilde{\SL}.$ Hence $g\phi g^{-1}$ and $\psi$ are in the same fiber of the evaluation map $\ev,$ which by Theorem \ref{thm_liftcommu} is connected; and there is a path $\{\psi_t\}$ in the fiber connecting $g\phi g^{-1}$ and $\psi.$
        Finally, the path $\{\phi_t\}$ connecting $\phi$ and $g\phi g^{-1}$ followed by the path $\{\psi_t\}$ connecting $g \phi g^{-1}$ and $\psi$ connects $\phi$ and $\psi$ in the pre-image of one of the four pieces. This completes the proof.
        \end{proof}
        
        \begin{remark} We remark that the components $\mathcal R_{-1} ^{-1} (\Sigma_{1,1})$ and $\mathcal R_{1} ^{+1} (\Sigma_{1,1})$ consist of discrete faithful representations, and the components $\mathcal R_{0} ^{+1} (\Sigma_{1,1})$ and $\mathcal R_{0} ^{-1} (\Sigma_{1,1})$ consist of reducible representations. (cf. Proposition \ref{prop_char}(a).)
        \end{remark}
        
        \begin{remark} 
        Theorem \ref{thm_torus} implies that for $n\in \mathbb Z$ and $s\in\{\pm 1\},$ the space $\mathcal R_n^s(\Sigma_{1,1})$ is non-empty if and only if $(n,s)$ satisfies the generalized Milnor-Wood inequality 
        $$\chi(\Sigma_{1,1})+ p_+(s) \leqslant n \leqslant -\chi(\Sigma_{1,1}) - p_-(s).$$ 
        I.e., Theorem \ref{thm_main1} holds for the one-hole torus $\Sigma_{1,1}.$
        \end{remark}

    \subsection{Three-hole sphere}

    The fundamental group of the three-hole sphere $\Sigma_{0,3}$ is a free group of rank two with a presentation 
        $$\pi_1(\Sigma_{0,3})=\langle c_1,c_2,c_3\ |\  c_1c_2c_3 \rangle, $$ where $c_i$'s are the preferred peripheral elements; and a $\psl$-representation $\phi$ of $\pi_1(\Sigma_{0,3})$ can be identified with 
        a pair $\big(\phi(c_1),\phi(c_2)\big) \in \psl\times \psl$ with $c_3$ mapped to 
        $\big(\phi(c_1)\phi(c_2)\big)^{-1}$. In this subsection, we define  the evaluation maps $\ev: \mathcal R(\Sigma_{0,3})\to \widetilde{\SL} $
     and 
     $\ev: \mathrm{HP} (\Sigma_{0,3})\to \widetilde{\SL}$, and prove that they satisfy the path-lifting property, using which we will prove the main results for $\Sigma_{0,3}$ in Subsections \ref{subsection_HP} and \ref{subsection_TP}.

\subsubsection{The evaluation map and the path-lifting properties}\label{subsection_PLP}
        Let $\overline{\Hyp}=\Hyp\cup\Par\cup\{\pm\mathrm{I}\}\subset \psl,$ and for $n\in\mathbb Z$ let $\overline{\Hyp_n}=\Hyp_n\cup\Par_n\cup\{z^n\}\subset\univcover$.
        We consider the \emph{lifted product map}  
        $$\ev: \overline{\Hyp}\times \overline{\Hyp}\to \widetilde{\SL}$$
    defined by sending the pair $(g_1,g_2)$ to the product $\widetilde{g_1}
    \widetilde{g_2}$ of the unique lifts $\widetilde{g_1}$ and $
    \widetilde{g_2}$ of  $g_1$ and $g_2$  in $\overline{\Hyp_0}$.
    Recall that the product $\widetilde{g_1}\widetilde{g_2}$ in the universal cover $\widetilde{\SL}$ is defined as follows. For $i=1,2,$ let $\{g_{i,t}\}\interval$ be a path connecting $g_{i,0} = \pm\mathrm{I}$ and $g_{i,1} = g_i$ in the one-parameter subgroup of $\psl$ generated by $g_i,$ and let $\{\widetilde{g_{1,t}g_{2,t}}\}\interval$ be the lift of the path $\{g_{1,t}g_{2,t}\}\interval$ in $\psl$ to  $\widetilde{\SL}$ starting from $\mathrm I.$ Then  $\widetilde{g_1}\widetilde{g_2}$ is the endpoint $\widetilde{g_{1,1}g_{2,1}}$ of the lifted path.
    
    Let $\mathcal{R}(\Sigma_{0,3})$ be the space of type-preserving representations of $\pi_1(\Sigma_{0,3})$, and let $\mathrm{HP}(\Sigma_{0,3})$ be the space of $\psl-$representations of $\pi_1(\Sigma_{0,3})$ with the mixed boundary condition.
    Under the identification of $\psl$-representations of $\pi_1(\Sigma_{0,3})$ with pairs of elements of $\psl,$ both $\mathcal R(\Sigma_{0,3})$ and $\mathrm{HP}(\Sigma_{0,3})$ can be considered as a subspace of $\overline{\Hyp}\times \overline{\Hyp}.$ The \emph{evaluation maps} 
     $$\ev: \mathcal R(\Sigma_{0,3})\to \widetilde{\SL} $$
     and 
     $$\ev: \mathrm{HP} (\Sigma_{0,3})\to \widetilde{\SL}$$
     are respectively defined as the restrictions of the lifted product map.
    For a representation $\phi$ in $\HP(\Sigma_{0,3})$ with  $e(\phi)=n$, 
     let $\widetilde{\phi(c_i)},$  $i \in\{ 1,2,3\},$  be the unique lift of $\phi(c_i)$ in $\overline{\Hyp_0}$. Since $\ev(\phi)=\widetilde{\phi(c_1)}\widetilde{\phi(c_2)}= z^n\widetilde{\phi(c_3)}^{-1} \in \overline{\Hyp_n}$, 
     the image of $\phi$ under the map $\ev$ lies in $\overline{\Hyp_n}$.

        \begin{proposition}\label{prop_evimage}
            The image of the lifted product map $\ev: \overline{\Hyp}\times \overline{\Hyp}\to \univcover$ equals 

            $$\mathcal{I} = 
            \mathrm{Hyp}_{- 1}
            \cup 
            \mathrm{Par}_{-1}^{+}
            \cup
            \mathrm{Ell}_{-1}
            \cup
            \mathrm{Par}_0
            \cup 
            \{\mathrm I\}
            \cup 
            \mathrm{Hyp}_0
            \cup
            \mathrm{Ell}_{1}
            \cup 
            \mathrm{Par}_{1}^{-}
            \cup
            \mathrm{Hyp}_{1}.
            $$
        Moreover, we  have:
        \begin{enumerate}[(1)]
            \item $\ev(\Par\times \Par) = \mathcal{I}$,
            and 
            $\ev\big(\R\big)= \mathrm{Par}_{-1}^+\cup \mathrm{Par}_0^-\cup \mathrm{Par}_0^+\cup \mathrm{Par}_{1}^-.$
            
            \item $\Hyp_0\cup \Hyp_{s}\subset \ev(\Par^{sgn(s)}\times \Hyp)$ for $s\in\{\pm 1\}$.
            
            \item $\Hyp_{-1}\cup  \Hyp_0\cup\Hyp_1\cup \{\mathrm I\}\subset \ev(\Hyp\times \Hyp)$, and $  \ev\big(\mathcal{W}_n(\Sigma_{0,3})\big)= \Hyp_n$ for $n\in \{-1,0,1\}$.

        \end{enumerate}
        \end{proposition}
        
        See Figure \ref{fig: evimage}. To prove Proposition \ref{prop_evimage}, we need the following lemma.
        
        \begin{lemma}\label{lem_evimage}
            The subsets  $\ev\Big(\overline{\Hyp}\times \overline{\Hyp}\Big)$ and  $ \Par_1^+\cup \Par_{-1}^-\cup\{z^{\pm 1}\}$ of $\univcover$ are disjoint. 
        \end{lemma}
        \begin{proof}
            If 
            $\ev(g_1, g_2) = \widetilde{g_1}\widetilde{g_2}= z^n$ for some $g_1, g_2\in \overline{\Hyp}$ and $n\in \mathbb{Z}$, 
            then 
            $g_1g_2 =\pm  \mathrm{I}$, and hence $g_2 = g_1^{-1}$. As the inverse $\widetilde{g_1}^{-1}$ of $\widetilde{g_1}$ is the unique lift of $g_1^{-1}$ to $\overline{\Hyp_0}$, we have $\widetilde{g_1}\widetilde{g_2} 
            =\widetilde{g_1}\widetilde{g_1}^{-1}= \mathrm I$ in $\univcover$, hence $n = 0$. This implies that $\ev\Big(\overline{\Hyp}\times \overline{\Hyp}\Big)$ and $\{z^{\pm 1}\}$ are disjoint.
            
            Suppose that $\ev(g_1, g_2) \in \Par_1^+\cup \Par_{-1}^-$ for some $g_1, g_2\in \overline{\Hyp}$. 
            Then $g_1\neq\pm\mathrm I$ and $g_2\neq\pm\mathrm I$, 
            as otherwise $\ev(g_1,g_2) \in \overline{\Hyp_0}$. It follows that the pair $(g_1,g_2)$ lies in one of the following sets: $\Hyp\times \Hyp$, 
            $\Hyp\times \Par$, 
            $\Par\times \Hyp$, 
            or $\Par\times \Par$. We first derive a contradiction for the case $(g_1,g_2)\in \Hyp\times \Hyp$.

            For $i=1,2,$ let $\{g_{i,t}\}\interval$ be a path connecting $\pm\mathrm{I}$ and $g_i$ in the one-parameter subgroup of $\psl$ generated by $g_i,$ and let $\{\widetilde{g_{1,t}g_{2,t}}\}\interval$ be the lift of $\{g_{1,t}g_{2,t}\}_{t\in [0,1]}$ in $\widetilde{\SL}$ starting from $\mathrm{I}$ and ending at $\widetilde{g_1}\widetilde{g_2}$. Up to a $\psl-$conjugation, we can assume that $g_1=\pm \begin{bmatrix}
        e^\lambda & 0 \\
        0 & e^{-\lambda}
        \end{bmatrix}$ for some $\lambda>0$; and by a re-parametrization if necessary, we can assume that 
        $$g_{1,t}=\pm \begin{bmatrix}
        e^{\lambda t} & 0 \\
        0 & e^{-\lambda t}
        \end{bmatrix}$$
        for all $t\in [0,1].$ Similarly, we can assume that 
        $$g_{2,t}=\pm \abcd\begin{bmatrix}
        e^{\eta t} & 0 \\
        0 & e^{-\eta t}
        \end{bmatrix}\abcd^{-1} 
       $$
        for some $\eta>0$ and $\pm\abcd\in\psl.$ Then the projection $\{A_t\}\interval$ of the path $\{\widetilde{g_{1,t}g_{2,t}}\}\interval$  in $\univcover$ to $\SL$ equals 
          $$ A_t= 
        \begin{bmatrix}
               e^{\lambda t} (ad e^{\eta t} - bc e^{-\eta t})
                & abe^{\lambda t}(e^{-\eta t} - e^{\eta t})  \\
                cde^{-\lambda t}(e^{\eta t} - e^{-\eta t})
                &  e^{-\lambda t} (ad e^{-\eta t} - bc e^{\eta t})
            \end{bmatrix}.$$
            Since $\widetilde{g_{1,1}g_{2,1}} =\widetilde{g_1}\widetilde{g_2} \notin \overline{\Hyp_0}$, the paths $\{g_{1,t}\}\interval$ and $\{g_{2,t}\}\interval$ are not contained in a common one-parameter subgroup of $\psl$, as otherwise the path $\{\widetilde{g_{1,t}g_{2,t}}\}\interval$ is entirely contained in $\overline{\Hyp_0}$. As a consequence, for all $t\in (0,1]$, the lift of $g_{2,t}$ to $\SL$ is not a diagonal matrix. Hence either $ab\neq 0$ or $cd\neq 0$, which implies that $\widetilde{g_{1,t}g_{2,t}}\neq z^{\pm 1}$ for all $t\in [0,1]$. Moreover, for any $t$ such that $\widetilde{g_{1,t}g_{2,t}}\in \Par_{\pm 1}$, Lemma \ref{lem_offdiag} implies that the sign of $\widetilde{g_{1,t}g_{2,t}}
            $ equals $- sgn\big(abe^{\lambda t}(e^{-\eta t} - e^{\eta t})\big)
            = sgn(ab)$ when $ab\neq 0$; and equal $sgn\big(cde^{-\lambda t}(e^{\eta t} - e^{-\eta t})\big)
            = sgn(cd)$ when $cd\neq 0$. In particular, all $\widetilde{g_{1,t}g_{2,t}}$, $t\in (0, 1],$ have the same sign. 
            
            Notice that $\mathrm{Par}^{-}_{1}\cup \{z\}$ devides $\univcover$ into two connected components with $\mathrm{I}$ lying in one and $\Par^+_1$  in the other.
            Therefore, if $\widetilde{g_1}\widetilde{g_2} = \widetilde{g_{1,1}g_{2,1}}\in \Par^+_1$, then
           the path $\big\{\widetilde{g_{1,t}g_{2,t}}\big\}$ connecting $\mathrm{I}$ and $\widetilde{g_{1,1}g_{2,1}}\in \mathrm{Par}^+_1$ must intersect
    $\mathrm{Par}^{-}_{1}\cup \{z\}$ at some time $t_0\in (0,1)$; see Figure \ref{fig: Lemma_3.2_colored2.png}.
     Since $\widetilde{g_{1,t_0}g_{2,t_0}}\neq z$,
     $\widetilde{g_{1,t_0}g_{2,t_0}}\in \Par^-_1$, which implies that $\widetilde{g_{1,t_0}g_{2,t_0}}$ and $\widetilde{g_{1,1}g_{2,1}}$ are parabolic elements of $\univcover$ of opposite signs. However, as shown above, $\widetilde{g_{1,t_0}g_{2,t_0}}$ and $\widetilde{g_{1,1}g_{2,1}}$ have the sign, which is a contradiction. 
     Similarly, $\mathrm{Par}^{+}_{-1}\cup \{z^{-1}\}$ devides $\univcover$ into two connected components with $\mathrm{I}$ lying in one  and $\Par^-_{-1}$  in the other; and 
    if $\widetilde{g_{1,1}g_{2,1}}\in \Par^-_{-1}$, then
     there exists a $t_0\in (0,1)$ such that $\widetilde{g_{1,t_0}g_{2,t_0}}\in \Par^+_{-1}$, and one can derive a contradiction in the same way.  This concludes that there is no $(g_1, g_2)\in \Hyp\times \Hyp$ such that $\ev(g_1, g_2)\in \Par^+_1\cup \Par^-_{-1}$.
     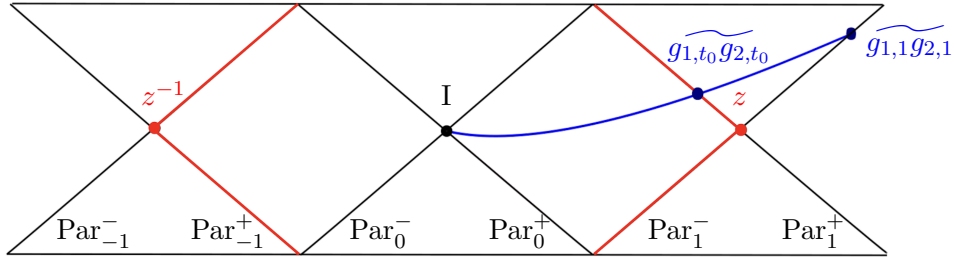
\begin{figure}[hbt!]
        \centering
        \begin{overpic}[width=0.8\textwidth]{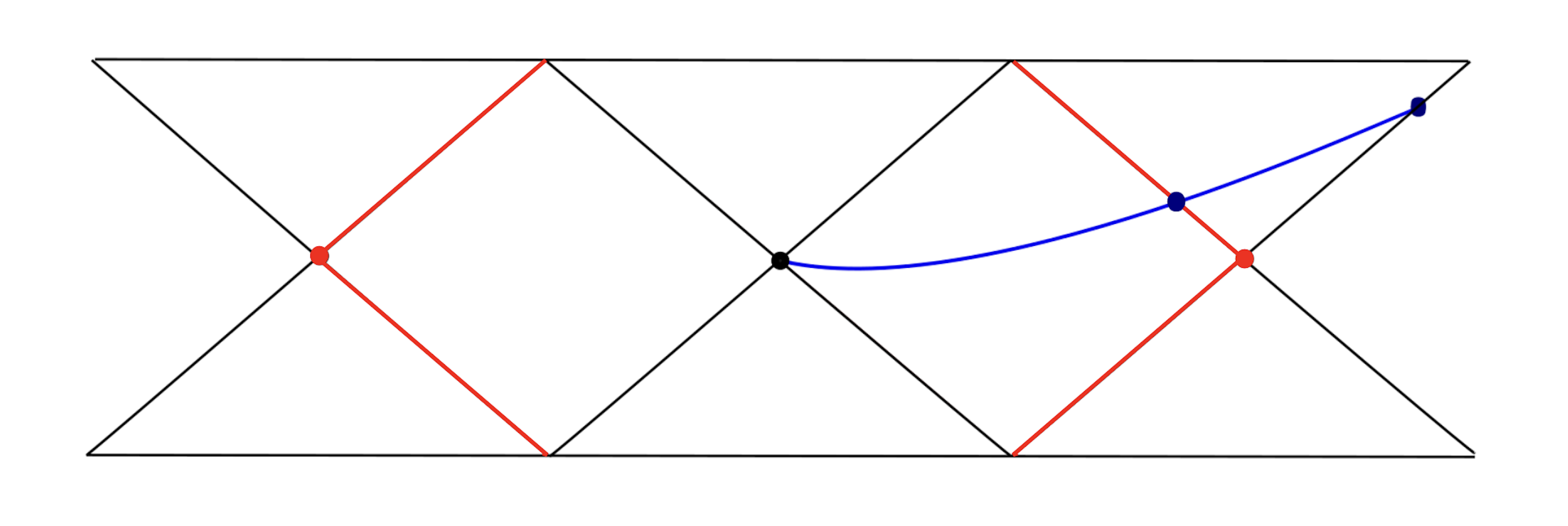}

            \put(19.0,19.5){\textcolor{red}{$z^{-1}$}}

            \put(49.2,19.5){\textcolor{black}{$\mathrm{I}$}}

            \put(78.5,19.5){\textcolor{red}{$z$}}

            \put(72,24.5){\textcolor{blue}{$\widetilde{g_{1,t_0}g_{2,t_0}}$}}
             \put(92,25){\textcolor{blue}{$\widetilde{g_{1,1}g_{2,1}}$}}


            

            \put(10.5,6){\textcolor{black}{$\Par^-_{-1}$}}

            \put(24,6){\textcolor{black}{$\Par^+_{-1}$}}

            \put(40,6){\textcolor{black}{$\Par^-_0$}}

            \put(54,6){\textcolor{black}{$\Par^+_0$}}

            \put(70,6){\textcolor{black}{$\Par^-_1$}}

            \put(83.5,6){\textcolor{black}{$\Par^+_1$}}

        


        \end{overpic}
        \caption{\label{fig: Lemma_3.2_colored2.png} The path $\widetilde{g_{1,t}g_{2,t}}$ with endpoint lying in $\Par^+_1$.}
    \end{figure}
     
    For the other cases, we compute the path $\big\{\widetilde{g_{1,t}g_{2,t}}\big\}_{t\in[0,1]}$ as follows:
    \\
   If $(g_1,g_2) \in \Par\times \Hyp$, then we can assume for all $t\in [0,1]$ that 
        $$g_{1,t}=\pm \abcd\begin{bmatrix}
        1 & s_1 t \\
        0 & 1
        \end{bmatrix}\abcd^{-1}\quad\text{and}\quad g_{2,t}=\pm \begin{bmatrix}
        e^{\lambda t} & 0 \\
        0 & e^{-\lambda t}
        \end{bmatrix}$$
        for some $s_1\in \{\pm 1\}$, $\lambda >0$, and $\pm\abcd\in\psl.$  Then the projection $A_t$ of $\widetilde{g_{1,t}g_{2,t}}$ to $\SL$ equals
          $$ A_t= 
        \begin{bmatrix}
               (1-s_1act)e^{\lambda t}
                & s_1a^2te^{-\lambda t}  \\
                -s_1c^2te^{\lambda t} 
                &  (1+s_1act)e^{-\lambda t}
            \end{bmatrix}.$$
        Since  either $a\neq 0$ or $c\neq 0$ as $\pm \abcd\in \psl$, $A_t\neq \pm \mathrm{I}$ and hence $\widetilde{g_{1,t}g_{2,t}}\neq z^{\pm 1}$ for all $t\in [0,1]$.
        \\
        
        If  $(g_1,g_2) \in \Hyp\times \Par$, then we can assume for all $t\in [0,1]$ that 
        $$g_{1,t}=\pm  \begin{bmatrix}
        e^{\lambda t} & 0 \\
        0 & e^{-\lambda t}
        \end{bmatrix}\quad\text{and}\quad g_{2,t}=\pm \abcd\begin{bmatrix}
        1 & s_2 t \\
        0 & 1
        \end{bmatrix}\abcd^{-1}$$
        for some $s_2\in \{\pm 1\}$, $\lambda >0$, and $\pm\abcd\in\psl.$  Then the projection $A_t $ of $\widetilde{g_{1,t}g_{2,t}}$ to $\SL$ equals
          $$ A_t= 
        \begin{bmatrix}
               (1-s_2act)e^{\lambda t}
                & s_2a^2te^{\lambda t}  \\
                -s_2c^2te^{-\lambda t} 
                &  (1+s_2act)e^{-\lambda t}
            \end{bmatrix}.$$
        Again, since either $a\neq 0$ or $c\neq 0$, $A_t\neq \pm \mathrm{I}$ and hence $\widetilde{g_{1,t}g_{2,t}}\neq z^{\pm 1}$ for all $t\in [0,1]$.
        \\
        
        If  $(g_1,g_2) \in \Par\times \Par$, then we can assume for all $t\in [0,1]$ that 
        $$g_{1,t}=\pm \begin{bmatrix}
        1 & s_1 t \\
        0 & 1
        \end{bmatrix}\quad\text{and}\quad g_{2,t}=\pm \abcd\begin{bmatrix}
        1 & s_2 t \\
        0 & 1
        \end{bmatrix}\abcd^{-1}$$
        for some $s_1, s_2\in \{\pm 1\}$ and $\pm\abcd\in\psl.$  Then the projection $A_t$ of  $\widetilde{g_{1,t}g_{2,t}}$ to $\SL$ equals
          $$ A_t= 
        \begin{bmatrix}
                1 - s_2act - s_1s_2c^2t^2
                & (s_1  + s_2 a^2)t + s_1s_2ac t^2  \\
                -s_2c^2t 
                & 1 + s_2act
            \end{bmatrix}.$$
            Since $\widetilde{g_{1,1}g_{2,1}}=\widetilde{g_1}\widetilde{g_2}  \not\in \overline{\Hyp_0}$, the paths $\{g_{1,t}\}\interval$ and $\{g_{2,t}\}\interval$ are not contained in a common one-parameter subgroup of $\psl$.
            This implies $c\neq 0$, and hence $-s_2c^2t\neq 0$ for all $t>0$. Therefore, $A_t\neq\pm\mathrm I$ and hence $\widetilde{g_{1,t}g_{2,t}}\neq z^{\pm 1}$ for all $t\in (0,1]$.

            In all the cases above, if $\widetilde{g_1}\widetilde{g_2} = \widetilde{g_{1,1}g_{2,1}}\in \Par^-_{-1}\cup \Par^+_1,$ by the same argument as in the $\Hyp\times\Hyp$ case, 
            there exists a $t_0\in (0,1)$ such that
            $\widetilde{g_{1,t_0}g_{2,t_0}}\in \Par_{-1}^+\cup\Par_{1}^-,$ and $\widetilde{g_{1,t_0}g_{2,t_0}}$ and  $\widetilde{g_{1,1}g_{2,1}}$ are elements in $\Par_{-1}$ or $\Par_1$ of the opposite signs.
            On the other hand, Lemma \ref{lem_offdiag} implies that $\widetilde{g_{1,t_0}g_{2,t_0}}$ and $\widetilde{g_{1,1}g_{1,1}}$ have the same sign, which is a contradiction. This concludes that 
    $\ev\Big(\overline{\Hyp}\times \overline{\Hyp}\Big)$ and $\Par_1^+\cup \Par_{-1}^-$ are disjoint.
        \end{proof}
        \begin{proof}[Proof of Proposition \ref{prop_evimage}]

We first prove that $\ev\Big(\overline{\Hyp}\times \overline{\Hyp}\Big)\subset \mathcal{I}$. 
            By Lemma \ref{lem_evimage} and the connectedness of $\overline{\Hyp}\times\overline{\Hyp},$ the image $\ev\Big(\overline{\Hyp}\times \overline{\Hyp}\Big)$ is contained in the connected component of the complement of $\Par_1^+\cup \Par_{-1}^-\cup \{z^{\pm 1}\}$ in $\univcover$ that contains $\mathrm{I},$ which is $\mathcal I.$ The other direction that $\mathcal I\subset \ev\Big(\overline{\Hyp}\times\overline{\Hyp}\Big)$ follows from (1) below. 
            \\

    For (1),
             we consider $$g_1=\pm \begin{bmatrix}
        1 & s_1 \\
        0 & 1
        \end{bmatrix}\quad\text{and}\quad g_2 = 
        \pm 
        \begin{bmatrix}
        1 & 0\\
        \lambda & 1
        \end{bmatrix}
        \begin{bmatrix}
        1 & s_2 \\
        0 & 1
        \end{bmatrix}
        \begin{bmatrix}
        1 & 0\\
        \lambda  & 1
        \end{bmatrix}^{-1}$$ for $s_1,s_2\in \{\pm 1\}$ and 
        $\lambda \geqslant 0.$ 
         Then $\ev(g_1,g_2) = \widetilde{g_1}\widetilde{g_2}$ projects to the matrix 
         $$A_\lambda = \begin{bmatrix}
                1 - s_2\lambda - s_1s_2\lambda^2
                & s_1  + s_2 + s_1s_2\lambda  \\
                -s_2\lambda^2 
                & 1 + s_2\lambda
            \end{bmatrix}
           $$ in $\SL$ with $\tr(A_\lambda) = 2 - s_1s_2\lambda^2$.  
           If $s_1 = s_2 = 1$, then we have
            $A_\lambda = \begin{bmatrix}
                1 - \lambda - \lambda^2
                & 2 + \lambda  \\
                -\lambda^2 
                & 1 + \lambda
            \end{bmatrix}$, and  
            as $\lambda$ varies from $0$ to $+\infty,$ the trace $\tr(A_\lambda)=2-\lambda^2$ decreases from $2$ to $-\infty$.
            Moreover, for $\lambda = 0$, $\widetilde{g_1}\widetilde{g_2}\in \Par_0\cup \{\mathrm{I}\}$ as $g_1$ and $g_2$ lie in a common one-parameter subgroup; and by Lemma \ref{lem_offdiag}, $\widetilde{g_1}\widetilde{g_2}\in \Par^+_0$ as the projection  $A_0 = \begin{bmatrix}
            1 & 2 \\
            0 & 1\end{bmatrix}$. 
            For $\lambda\in(0,2),$ we have $\tr(A_\lambda)\in(-2,2)$ and hence  $\widetilde{g_1}\widetilde{g_2}\in \Ell_1$ which is adjacent to $\Par_0^+;$ for $\lambda = 2$, we have $\tr(A_2)=-2$ and  
            hence  $\widetilde{g_1}\widetilde{g_2}\in\Par^-_1$ which is adjacent to $\Ell_{1};$ and for $\lambda\in (2,\infty),$ we have $\tr(A_\lambda)<-2$ and hence  $\widetilde{g_1}\widetilde{g_2}\in \Hyp_1$ which is adjacent to $\Par^-_1$. Therefore, the family $ \{\widetilde{g_1}\widetilde{g_2}\ |\ \lambda\in [0,\infty)\}
            $ 
            in $\univcover$ lies in
            $\Par^+_0\cup \Ell_1\cup \Par^-_1\cup \Hyp_1$, with their traces taking all the possible values in $(-\infty, 2]$.
            Recall that two non-central  elements in $\univcover$ are conjugate if and only if they have the same trace and are contained in one and the same of $\Hyp_n$, $\Par^+_n$, $\Par^-_n$ or $\Ell_n,$ for some $n\in \mathbb{Z}$.
            Therefore, for any $\widetilde{C} \in \Par^+_0\cup \Ell_1\cup \Par^-_1\cup \Hyp_1$, there exists a $\lambda\geqslant 0$ where the corresponding $\widetilde{g_1}\widetilde{g_2}$ is conjugate to $\widetilde{C}$ by some element $\widetilde{h}\in \univcover$. Let  $h$ be the projection of $\widetilde{h}$ to $\psl$. Then  $\widetilde{h}\widetilde{g_i}\widetilde{h}^{-1}$ is the unique lift of $hg_ih^{-1}$ in $\Par_0$ for $i \in \{1,2\},$ and  we have $\ev(hg_1h^{-1}, hg_2h^{-1}) = \widetilde{h}\widetilde{g_1}\widetilde{h}^{-1}\cdot \widetilde{h}\widetilde{g_2}\widetilde{h}^{-1} = \widetilde{C}$. As a consequence, $\Par^+_0\cup \Ell_1\cup \Par^-_1\cup \Hyp_1$ is contained in $\ev(\Par\times \Par)$. 
            Similarly, 
            If $s_1 = s_2 = -1$, then we have
            $A_\lambda = \begin{bmatrix}
                1 +\lambda - \lambda^2
                & -2 + \lambda  \\
                \lambda^2 
                & 1 -\lambda
            \end{bmatrix}$, and as $\lambda$ varies from $0$ to $+\infty,$ the trace $\tr(A_\lambda)=2-\lambda^2$ decreases from $2$ to $-\infty$.
            Moreover, for $\lambda = 0$, Lemma \ref{lem_offdiag}  shows that $\widetilde{g_1}\widetilde{g_2}\in \Par^-_0$ as the projection $A_0= \begin{bmatrix}
            1 & -2 \\
            0 & 1
           \end{bmatrix}$. Then by the same argument as in the previous case,  the family $\{\widetilde{g_1}\widetilde{g_2}\ |\ \lambda\in [0,\infty)\}$ in $\univcover$ lies in
            $\Hyp_{-1}\cup\Par^+_{-1}\cup\Ell_{-1}\cup\Par^-_0$, with their traces taking all the possible values in $(-\infty, 2]$, and $\Hyp_{-1}\cup\Par^+_{-1}\cup\Ell_{-1}\cup\Par^-_0$ is contained in $\ev(\Par\times \Par)$. 
            Finally, if $s_1\neq s_2$, then $A_\lambda =  \begin{bmatrix}
                1 - s_2\lambda +\lambda^2
                & -\lambda  \\
                -s_2\lambda^2 
                & 1 + s_2\lambda
            \end{bmatrix}$, with $\tr(A_\lambda)=2+\lambda^2$. 
            It follows that $\widetilde{g_1}\widetilde{g_2} = \mathrm{I}$ when $\lambda = 0$, and lies in $\Hyp_0$ for all $\lambda >0$ with its trace taking all the possible values in $(2,\infty)$.
            They by the same argument as in the previous two cases, $\Hyp_0$ is contained in $\ev(\Par\times \Par)$.  
            Putting all the three cases togeher, we have that the set $\mathcal I$ lies in the image $\ev(\Par\times \Par).$ In particular, this implies that $\mathcal I\subset\ev\Big(\overline{\Hyp}\times\overline{\Hyp}\Big).$
\\
            
            For (2), 
             we consider $$g_1
        = \pm \begin{bmatrix}
        1 & 0\\
        2s  & 1
        \end{bmatrix}\begin{bmatrix}
            1 & s \\
            0 & 1
        \end{bmatrix}
        \begin{bmatrix}
        1 & 0\\
        2s  & 1
        \end{bmatrix}^{-1}\quad\text{and}\quad g_2=
        \pm 
        \begin{bmatrix}
            e^{\lambda} & 0 \\
        0 & e^{-\lambda}
        \end{bmatrix}
        $$ for $s \in \{\pm 1\}$ and 
        $\lambda \in \mathbb{R}.$ 
         Then $(g_1, g_2)\in \Par^{sgn(s)}\times \Hyp$ for all $\lambda \neq 0$, and $\ev(g_1,g_2)= \widetilde{g_1}\widetilde{g_2}$ projects to the matrix $A_\lambda= 
            \begin{bmatrix}
        -e^{\lambda} & se^{-\lambda} \\
        -4se^{\lambda} & 3e^{-\lambda}
        \end{bmatrix}
           $ in $\SL$ with 
            $\tr(A_\lambda) = 3e^{-\lambda}-e^\lambda.$ As $\lambda$ increases from $0$ to $+\infty$, the trace $\tr(A_\lambda)$ takes all the possible values in $(-\infty, 2]$. 
            Moreover, for $\lambda = 0$, $\widetilde{g_1}\widetilde{g_2}\in \Par_0$ as $g_2 = \pm \mathrm I$; and Lemma \ref{lem_offdiag} shows that $\widetilde{g_1}\widetilde{g_2}\in \Par^{sgn(s)}_0$ as the projection $A_0 = \begin{bmatrix} 
        -1 & s \\
        -4s & 3 
        \end{bmatrix}.$ 
For $\lambda\in(0,\ln 3),$ we have $\tr(A_\lambda)\in(-2,2)$ and hence  $\widetilde{g_1}\widetilde{g_2}\in \Ell_{s}$ which is adjacent to $\Par_0^{sgn(s)};$ for $\lambda = \ln 3$, we have $\tr(A_2)=-2$ and  
            hence  $\widetilde{g_1}\widetilde{g_2}\in\Par^{-sgn(s)}_{s}$ which is adjacent to $\Ell_{s};$ and for $\lambda\in (\ln 3,\infty),$ we have $\tr(A_\lambda)<-2$ and hence  $\widetilde{g_1}\widetilde{g_2}\in \Hyp_{s}$ which is adjacent to $\Par^{-sgn(s)}_{s}$. 
            Thus, $\widetilde{g_1}\widetilde{g_2}\in \Hyp_{s}$ for all $\lambda > \ln 3$; and for any $\widetilde{C}\in \mathrm{Hyp}_{s},$ one can choose $\lambda\in (\ln 3, \infty)$ such that $\widetilde{g_1}\widetilde{g_2}$ is conjugate to $\widetilde{C}$ by some element $\widetilde{h}$ in $\widetilde{\SL}.$ Then for the projection $h$ of $\widetilde{h}$ to $\psl,$  $\ev(h\widetilde{g_1}h^{-1}, h\widetilde{g_2} h^{-1}) = \widetilde{h}\widetilde{g_1}\widetilde{h}^{-1}\cdot \widetilde{h}\widetilde{g_2} \widetilde{h}^{-1}=\widetilde{C},$ and hence $\Hyp_{s}\subset \ev(\Par^{sgn(s)}\times \Hyp)$. 
        Similarly, as $\lambda$ decreases from $0$ to $-\infty,$ the trace $\tr(A_\lambda)$ increases from 2 to $+\infty$.
        It follows that $\widetilde{g_1}\widetilde{g_2}\in \Hyp_0$ for all $\lambda >0$ with $\tr(A_\lambda)$ taking all the possible values in $(2,\infty)$,
        and consequently $\Hyp_0\subset \ev(\Par^{sgn(s)}\times \Hyp)$. This concludes that $\Hyp_{-1}\cup \Hyp_0\cup \Hyp_1$ lies in the image $\ev(\Par^{sgn(s)}\times \Hyp).$
\\

           For (3), to show $\Hyp_{\pm 1}\subset \ev(\Hyp\times \Hyp)$,
             we consider $$g_1=\pm \begin{bmatrix}
            e^{\lambda} & 0 \\
        0 & e^{-\lambda}
        \end{bmatrix}\quad\text{and}\quad g_2
        = \pm \begin{bmatrix}
        -1 & 1\\
        -2  & 1
        \end{bmatrix}
        \begin{bmatrix}
            e^{\lambda} & 0 \\
        0 & e^{-\lambda}
        \end{bmatrix}
        \begin{bmatrix}
        -1 & 1\\
        -2  & 1
        \end{bmatrix}^{-1}
        $$ for
        $\lambda \in \mathbb{R}.$ 
         Then $(g_1, g_2)\in \Hyp\times \Hyp$ for all $\lambda \neq 0$, 
         and $\ev(g_1,g_2)= \widetilde{g_1}\widetilde{g_2}$ projects to the matrix $ 
            A_\lambda = \begin{bmatrix}
        2 -e^{2\lambda} & -1+e^{2\lambda} \\
        2e^{-2\lambda}-2 & -e^{-2\lambda}+2
        \end{bmatrix}
           $ in $\SL$ with $\tr(A_\lambda) = 4 - 2\cosh(2\lambda).$ 
           As $\lambda$ varies from $0$ to $+\infty,$ $\tr(A_\lambda)$ take all the possible values in $(-\infty, 2].$ 
           Let $\lambda_0=\frac{1}{2}\cosh^{-1}(3)$ so that $\tr(A_{\lambda_0})=-2$.
           Then by Lemma \ref{lem_offdiag}, $\widetilde{g_1}\widetilde{g_2}\in \Par^-_1$ as its projection $A_{\lambda_0} = \begin{bmatrix}
        2 -e^{2\lambda_0} & -1+e^{2\lambda_0} \\
        2e^{-2\lambda_0}-2 & -e^{-2\lambda_0}+2
        \end{bmatrix}$. Thus, for all $\lambda > \lambda_0,$ $\widetilde{g_1}\widetilde{g_2}\in \Hyp_1$  as it is adjacent to $\Par_1^-;$ and for any $\widetilde{C}\in \mathrm{Hyp}_1,$ one can choose $\lambda \in (\lambda_0, \infty)$ such that $\widetilde{g_1}\widetilde{g_2}$ is conjugate to $\widetilde{C}$ by some element $\widetilde{h}$ in $\widetilde{\SL}.$ Then for the projection $h$ of $\widetilde{h}$ to $\psl,$  $\ev(h\widetilde{g_1}h^{-1}, h\widetilde{g_2} h^{-1}) = \widetilde{h}\widetilde{g_1}\widetilde{h}^{-1}\cdot \widetilde{h}\widetilde{g_2} \widetilde{h}^{-1}=\widetilde{C},$ and hence  $\Hyp_1\subset \ev(\Hyp\times \Hyp)$. 
        Similarly, as $\lambda$ varies from $0$ to $-\infty,$ $\tr(A_\lambda)$ take all the possible values in $(-\infty, 2].$
           Moreover, for $\lambda = -\lambda_0$,
           Lemma \ref{lem_offdiag} shows  that $\widetilde{g_1}\widetilde{g_2}\in \Par^+_{-1}$ as its projection $A_{-\lambda_0} = \begin{bmatrix}
        2 -e^{-2\lambda_0} & -1+e^{-2\lambda_0} \\
        2e^{2\lambda_0}-2 & -e^{
        2\lambda_0}+2
        \end{bmatrix}$; and by the same argument as above, $\widetilde{g_1}\widetilde{g_2}\in \Hyp_{-1}$ for all $\lambda < -\lambda_0,$ 
           and $\Hyp_{-1}\subset \ev(\Hyp\times \Hyp)$. 
           
        To show $\Hyp_0\subset \ev(\Hyp\times \Hyp)$,
             we consider 
             $$g_1= g_2=
        \pm 
        \begin{bmatrix}
            e^{\lambda} & 0 \\
        0 & e^{-\lambda}
        \end{bmatrix}
        $$ for  
        $\lambda > 0.$ Then $\ev(g_1,g_2)= \widetilde{g_1}\widetilde{g_2}$ projects to the matrix $A_\lambda = \begin{bmatrix}
        e^{2\lambda} & 0 \\
        0 & e^{-2\lambda}
        \end{bmatrix}
           $ in $\SL$ with $\tr(\widetilde{g_1}\widetilde{g_2}) = 2\cosh(2\lambda).$ 
           As $\lambda$ varies from $0$ to $+\infty,$ the trace $\tr(\widetilde{g_1}\widetilde{g_2})$ increases from 2 to $+\infty$.
            It follows that $\widetilde{g_1}\widetilde{g_2}\in \Hyp_0$ for all $\lambda >0$ with $\tr(\widetilde{g_1}\widetilde{g_2})$ taking all the possible values in $[2,\infty)$;
        and by the same argument as above, $\Hyp_0\subset \ev(\Hyp\times \Hyp)$.
           With $\ev(g, g^{-1}) = \mathrm{I}$ for any element $g\in \Hyp$, this concludes  that the set $\Hyp_{-1}\cup
        \Hyp_0\cup \Hyp_{1}\cup \{\mathcal I\}$ lies in the image $\ev(\Hyp\times\Hyp)$.
        \end{proof}

    Let $e: \HP(\Sigma_{0,3})\to \mathbb{Z}$ be the relative Euler class map and let $s: \HP(\Sigma_{0,3})\to \{-1,0,+1\}^3$ be the sign map defined respectively by sending $\phi$ to the relative Euler class $e(\phi)$ and the sign $s(\phi)$ of $\phi.$
    For $n\in \mathbb{Z}$ and $s\in \{-1,0,+1\}^3$, we let $\HPns(\Sigma_{0,3}) = e^{-1}(n)\cap s^{-1}(s)$ be the subspace of $\HP(\Sigma_{0,3})$ consisting of representations with the relative Euler class $n$ and the sign $s.$
\begin{theorem} \label{thm_PLP}
For $n\in\mathbb Z$ and $s=(s_1,s_2,0)$ 
with $s_1,s_2\in \{-1,0,1\}$ 
so that $\mathrm{HP}_n^s(\Sigma_{0,3})$ is non-empty, the restriction 
$$\ev: \mathrm{HP}_n^s(\Sigma_{0,3})\to\Hyp_n$$
of the evaluation map to $\HPns(\pants)$ satisfies the path-lifting property. Moreover, for each $\widetilde{C}$ in $\Hyp_n$, the fiber $\ev^{-1}\big(\widetilde{C}\big)$ is connected in $\HPns(\Sigma_{0,3})$.
\end{theorem}
    
\begin{proof}  By Lemma \ref{lem_plp} and Lemma \ref{lem_evPsubm}, it suffices to show that for all $\widetilde{C}\in \Hyp_n$, the fiber $\ev^{-1}(\widetilde{C})$ is connected in $\HPns(\Sigma_{0,3}).$ The case that $s_1 = s_2 = 0$ is contained in \cite[Proposition 4.6]{goldman}; and we only need to consider the remaining cases.
\\


    We first consider the case that $s_1\neq 0$ and $s_2 = 0$. Let $\widetilde{C}\in \Hyp_n$, and let $\phi_1, \phi_2$ be representations in the fiber $\ev^{-1}(\widetilde{C})$. Recall that for each $j \in \{ 1,2\}$, the character of $\phi_j$ is defined by the triple $$
    \chi(\phi_j) = 
    \Big(
    \mathrm{tr}\big(\widetilde{\phi_j(c_1)}\big), 
    \mathrm{tr}\big(\widetilde{\phi_j(c_2)}\big), 
    \mathrm{tr}\big(\widetilde{\phi_j(c_1)}\widetilde{\phi_j(c_2)}\big)\Big)
    $$
    in $\mathbb{R}^3$, where $\widetilde{\phi_j(c_1)}$ and 
    $\widetilde{\phi_j(c_2)}$ are respectively the lifts of $\phi_j(c_1)$ and $\phi_j(c_2)$ in $\overline{\Hyp_0}$, and their traces are defined as the traces of their projections to $\SL$. Since 
    $\widetilde{\phi_j(c_1)}, \widetilde{\phi_j(c_2)}\in \overline{\Hyp_0}$, we have  $\tr\big(\widetilde{\phi_j(c_1)}\big), \tr\big(\widetilde{\phi_j(c_2)}\big) \in [2,\infty)$; and since $s_1\neq 0$ and $s_2=0,$   
    we have $\chi(\phi_1) = (2,y_1,z)$ and 
    $\chi(\phi_2) = (2,y_2,z)$ with $y_1, y_2 \in (2,\infty)$ and $z = \tr \widetilde{C}$. Choose arbitrarily a $y\in (2,\infty)\setminus z$; and for $j\in\{1,2\},$
    define the straight line segments
    $\big\{(2,y_{j,t},z)\big\}\interval$ in $\mathbb{R}^3$, where $y_{j,t}= (1-t)y_j+ty$ connecting $y_j$ and $y$.  
    
    To construct paths in $\HPns(\pants)$, we use
    the path-lifting property of the character map $\chi:\SL\times \SL\to \mathbb{R}^3$ in Lemma \ref{lem_plpchar}. 
    For $j\in \{1,2\}$, let $A_{j,0}$ and $B_{j,0}$ respectively
    be the projections of $\widetilde{\phi_j(c_1)}$ and $\widetilde{\phi_j(c_2)}$ to $\SL$, noting that the image of $(A_{j,0}, B_{j,0})$ under the character map $\chi$ equals $\chi(\phi_j) = (2,y_j, z)$.
    Since $\phi_j(c_1)\in \Par^{sgn(s_1)}$ and $\phi_j(c_2)\in \Hyp$, $A_{j,0}$ and $B_{j,0}$ do not commute; and since
    $|z| = \big|\tr\widetilde{C}\big|>2$, the paths $\big\{(2,y_{j,t},z)\big\}\interval$ lie in
    $\mathbb{R}^3
        \setminus \big([-2,2]^3\cap \kappa^{-1}([-2,2])\big)$, where $\kappa$ is as defined in Proposition \ref{prop_char}. 
        Then by Lemma \ref{lem_plpchar}, 
        there exist paths $\{(A_{j,t}, B_{j,t})\}\interval$ of non-commuting pairs of elements of $\SL$ starting from  $(A_{j,0}, B_{j,0})$, having 
        $\chi(A_{j,t}, B_{j,t}) = (2, y_{j,t},z)$ for all $t\in [0,1]$. 
        Let $\pm A_{j,t}$ and $\pm B_{j,t}$ respectively be the projections of 
        $A_{j,t}$ and $B_{j,t}$ to $\psl$, and define $\phi_{j,t}:\pi_1(\pants)\to \psl$ by letting $\big(\phi_{j,t}(c_1), \phi_{j,t}(c_2)\big) = (\pm A_{j,t}, \pm B_{j,t})$. Then for each $j \in \{1,2\}$, the path $\{\phi_{j,t}\}\interval$ of $\psl-$representations starting from $\phi_{j,0} = \phi_j$ satisfies $\chi(\phi_{j,t}) = \chi(A_{j,t}, B_{j,t}) = (2,y_{j,t}, z)$ for all $t\in [0,1]$, which implies $\chi(\phi_{1,1}) = \chi(\phi_{2,1})=(2,y,z)$. 
        %
        Since $\tr A_{j,t} = 2$ for all $t\in [0,1]$, the path $\{\phi_{j,t}(c_1)\}\interval$ in $\psl$ is contained in $\Par\cup \{\pm \mathrm{I}\}$,
        which never passes $\pm\mathrm{I}$ as otherwise $A_{j,t} = \mathrm{I}$ for some $t\in [0,1]$ and would commute with $B_{j,t}$. Since $\phi_{j,0}(c_1)$ lies in $\Par^{sgn(s_1)}$ which is a connected component of $\Par$, we have
        $\phi_{j,t}(c_1)\in \Par^{sgn(s_1)}$
        for all $t\in [0,1]$; since $y_j, y\in (2,\infty)$,
        $y_{j,t} = (1-t)y_j+ty \in (2,\infty)$, and hence $\phi_{j,t}(c_2)\in \Hyp$ for all $t\in [0,1]$; and since $|z|>2,$ $\phi_{j,t}(c_3) = \big(\phi_{j,t}(c_1)\phi_{j,t}(c_2)\big)^{-1}\in \Hyp$ for all $t\in [0,1]$. This implies that 
        the path $\{\phi_{j,t}\}\interval$ lies in $\HP(\Sigma_{0,3})$, and more precisely, in  $s^{-1}(s).$  Moreover, the relative Euler classes $e(\phi_{j,t}) = e(\phi_{j}) = n$ for all $t\in [0,1]$, hence $\{\phi_{j,t}\}\subset e^{-1}(n)\cap s^{-1}(s) =  \HPns(\pants)$.
        
        We now find a path connecting $\phi_{1,1}$ and $ \phi_{2,1}$ in $\HPns(\Sigma_{0,3})$. 
        Since both $(A_{1,1}, B_{1,1})$ and $(A_{2,1}, B_{2,1})$ are in $\chi^{-1}(2,y,z)\subset \SL\times \SL$ with $y \neq z,$ and $\kappa(2, y, z) = (y-z)^2+ 2 \neq 2$, Proposition \ref{prop_char} implies that these pairs of elements of $\SL$ are $\GL$-conjugate.
        It follows that they are $\SL$-conjugate as otherwise by Proposition \ref{prop_pglpsl},
        $s(\phi_{2,1}) = -s(\phi_{1,1})\neq s(\phi_{1,1})$.
        Let $P\in\SL$ be such that 
        $\big(A_{2,1}, B_{2,1}\big)
        = 
        \big(PA_{1,1}P^{-1}, PB_{1,1}P^{-1}\big),$ and let $g=\pm P$ be the projection of $P$ to $\psl.$ Then we have $\phi_{2,1}=g \phi_{1,1} g^{-1}.$ 
        Let $\{g_t\}\interval$ be a path in $\psl$ connecting $g_0 =\pm \mathrm I$ and $g_1 = g$. Then
        the path $\{g_t\phi_{1,1} g_t^{-1}\}\interval$ connects $\phi_{1,1}$ to $\phi_{2,1}$ within $\HPns(\Sigma_{0,3})$, as both the relative Euler class and the sign remain constant under $\psl$-conjugations of $\phi_{1,1}$. 
        We have constructed the paths 
        $\{\phi_{1,t}\}$ connecting $\phi_1$ and $\phi_{1,1}$, $\{g_t\phi_{1,1} g_t^{-1}\}$ connecting $\phi_{1,1}$ and $\phi_{2,1},$ and $\{\phi_{2,1-t}\}$ connecting $\phi_{2,1}$ and $\phi_2$, and the composition of these three paths with a suitable re-parametrization defines a path $\{\phi'_t\}_{t\in [1,2]}$ in $\HPns(\pants)$ connecting $\phi'_1 = \phi_1$ and $\phi'_2 = \phi_2$. 
        
        Next, using the path $\{\phi'_t\}_{t\in [1,2]}$ in  $\HPns(\pants)$, we will construct a path $\{\phi_t\}_{t\in [1,2]}$ in $\ev^{-1}(\widetilde{C})$ connecting $\phi_1$ and $\phi_2$.
        For each $t\in [1,2]$, since $\mathrm{tr}\big(\widetilde{\phi'_t(c_1)}\widetilde{\phi'_t(c_2)}\big) = z = \tr\widetilde{C}$, the hyperbolic element $\phi'_t(c_1)\phi'_t(c_2)$ is $\psl$-conjugate to the projection $\pm C$ of $\widetilde{C}$ to $\psl$. 
        Guaranteed by  Lemma \ref{lem_conjugacypath_const}, let $\{h_t\}_{t\in [1,2]}$ be a continuous path in $\psl$ such that 
        $\pm C = h_t\big(\phi'_t(c_1)\phi'_t(c_2)\big)h_t^{-1}$ for all $t\in [1,2]$. Since $\phi'_1, \phi'_2\in \ev^{-1}(\widetilde{C})$, for $j\in\{1,2\},$ 
        $h_j\big(\pm C\big)h_j^{-1}=h_j\big(\phi'_j(c_1)\phi'_j(c_2)\big)h_j^{-1}= \pm C$, and hence $h_j$ commutes with $\pm C$. This implies that $h_1^{-1}h_t\big(\phi'_t(c_1)\phi'_t(c_2)\big)h_t^{-1}h_1 = h_1\big(\pm C\big) h_1^{-1}= \pm C$ for all $t\in [1,2]$; and by replacing $h_t$ by  $h_1^{-1}h_t$ if necessary, we can assume that $h_1 = \pm \mathrm I$. 
        Then the path $\{h_t\phi'_th_t^{-1}\}_{t\in [1,2]}$ connects  $h_1\phi'_1h_1^{-1} = \phi_1$ to $h_2\phi'_2h_2^{-1} = 
         h_2\phi_2h_2^{-1}$. Moreover, since $\phi'_t$ lies in $\HPns(\pants)$ for all $t\in [1,2]$, its $\psl$-conjugation $h_t\phi'_th_t^{-1}$ also lies in $\HPns(\pants)$ as shown in the previous paragraph. Thus, the evaluation map $\ev$ is defined at $h_t\phi'_th_t^{-1}$ for all $t\in [1,2]$ with $\ev\big(h_t\phi'_th_t^{-1}\big)=\widetilde{C}$ which is the unique lift of $\pm C$ to $\Hyp_n$, i.e., $\{h_t\phi'_th_t^{-1}\}_{t\in[1,2]}\subset \ev^{-1}(\widetilde C)$.
        It remains to find a path in $\ev^{-1}(\widetilde{C})$ connecting $h_2\phi_2h_2^{-1}$ and $\phi_2$. Let $\{h_{2,t}\}\interval$ be a path connecting $h_{2,0} = h_2$ to $h_{2,1} = \pm\mathrm{I}$ in the one-parameter subgroup of $\psl$ generated by $h_2$. Since $h_2$ and $\pm C$ commute, $\pm C$ lies in this one-parameter subgroup, and we have $h_{2,t}
        \phi_2(c_1)
        \phi_2(c_2)
        h_{2,t}^{-1} = h_{2,t}
       \big(\pm C\big)
        h_{2,t}^{-1} = \pm C$ for all $t\in [0,1]$. 
        As a consequence, the path 
        $\{h_{2,t}
        \phi_2
            h_{2,t}^{-1}\}_{t\in [0,1]}$ connects $h_2\phi_2h_2^{-1}$ to $\phi_2$ within $\HPns(\pants)$ with $\ev(h_{2,t}
        \phi_2
            h_{2,t}^{-1}) = \widetilde{C}$ for all $t\in [0,1]$, i.e., $\{h_{2,t}
        \phi_2
            h_{2,t}^{-1}\}_{t\in [0,1]}\subset \ev^{-1}(\widetilde{C})$. 
            Finally, the composition of 
            $\{h_t\phi'_th_t^{-1}\}_{t\in[1,2]}$ and  
            $\{h_{2,t}
        \phi_2
            h_{2,t}^{-1}\}_{t\in [0,1]}$ with a suitable re-parametrization defines a path $\{\phi_t\}_{t\in [1,2]}$ in 
            $\ev^{-1}(\widetilde{C})$ connecting $\phi_1$ and $\phi_2$. This concludes that  the fiber $\ev^{-1}(\widetilde{C})$ is connected in $\HPns(\pants)$. 
        \\

       Lastly, we consider the remaining cases of $s$.  In the case that $s_1 = 0$ and $s_2 \neq 0$, 
        the representations $\phi_1, \phi_2\in \ev^{-1}(\widetilde{C})$ have characters $\chi(\phi_1) = (x_1,2,z)$ and 
    $\chi(\phi_2) = (x_2,2,z)$ with $x_1, x_2 \in (2,\infty)$. Similar to the previous case, we choose $x\in (2,\infty)\setminus \{z\} $ to construct a straight line segment 
    $\big\{(x_{j,t}, 2,z)\big\}\interval,$ where $\{x_{j,t}\}_{t\in[0,1]}$ connects $x_j$ and $x$; and we use Lemma \ref{lem_plpchar} to obtain the path $\{\phi_{j,t}\}\interval$ of $\psl$-representations with $\chi(\phi_{j,1}) = (x,2,z)$. The rest of the proof follows verbatim the previous case. In the case that $s_1\neq 0$ and $s_2\neq 0$, we have $\chi(\phi_1) = \chi(\phi_2) = (2,2,z)$ with $|z|>2,$ and $\kappa(2, 2, z) = (z-2)^2 + 2 \neq 2$. 
    Then Proposition \ref{prop_char} and Proposition \ref{prop_pglpsl} imply that there exists a $g\in \psl$ such that $\phi_2 = g\phi_1 g^{-1}$. Let $\{g_t\}\interval$ be a path in $\psl$ connecting $g_0 =\pm \mathrm I$ and $g_1 = g$. Then the path $\{g_t\phi_1 g_t^{-1}\}\interval$ connects $\phi_1$ and $\phi_2$ within $\HPns(\Sigma)$. The rest of the proof follows  the previous case.
    This concludes that for all $s = (s_1,s_2,0)$ and all $\widetilde{C}\in \Hyp_n$, the fiber $\ev^{-1}(\widetilde{C})$ in $\HPns(\Sigma_{0,3})$ is connected. 
\end{proof}

 \subsubsection{Representations with mixed boundary condition}\label{subsection_HP}
        Let $n\in \mathbb{Z}$, and $s\in \{-1,0,+1\}^3$.
        Note that if $s\in\{\pm 1\}^3$, then $\mathrm{HP}^s_n(\Sigma_{0,3}) = \mathcal R^s_n(\Sigma_{0,3})$; and if $s=(0,0,0),$ then $\mathrm{HP}^s_n(\Sigma_{0,3}) = \mathcal W_n(\Sigma_{0,3})$. The connected components of 
        $\mathcal W(\Sigma_{0,3})$ are described in Theorem \ref{thm_goldman}. Here, we consider the cases of $s\in \{-1,0,+1\}^3$ with  $p_0(s)\geqslant 1$ in the following Theorem \ref{thm_pants2} and Corollary \ref{cor_evimage}, and describe the connected components of $\mathcal R(\Sigma_{0,3})$  in Theorem \ref{thm_pants}.

\begin{theorem}\label{thm_pants2} For $n\in\mathbb Z$ and $s \in\{-1,0,+1\}^3,$ let $\mathrm{HP}_n^s(\Sigma_{0,3})=e^{-1}(n)\cap s^{-1}(s).$
\begin{enumerate}[(1)]
\item If, up to a permutation of the peripheral elements, $s=(s_1,s_2,0)$ with $s_1, s_2\in\{\pm 1\},$ then $\mathrm{HP}_n^s(\Sigma_{0,3})$ is non-empty if and only if $n=\frac{1}{2}(s_1+s_2).$

\item If, up to a permutation of the peripheral elements, $s=(s_1,0,0)$ with $s_1\in\{\pm 1\},$ then $\mathrm{HP}_n^s(\Sigma_{0,3})$ is non-empty if and only if $n=0$ or $n=s_1.$ 
\end{enumerate}
Moreover, all the non-empty spaces above are connected. 
\end{theorem}

        \begin{proof} For (1), under the identification of the $\psl$-representations of $\pi_1(\Sigma_{0,3})$ and pairs of elements of $\psl,$ we have
         $\Bigg(
        \pm \begin{bmatrix}
        0 & -\frac{1}{5} \\
        5 & 2
        \end{bmatrix},
        \pm 
        \begin{bmatrix}
         1 &  -1\\
        0 & 1
        \end{bmatrix}
        \Bigg)\in \mathrm{HP}_{-1}^{(-1,-1,0)}
            (\Sigma_{0,3})
            ,
            $ 
       $\Bigg(
        \pm \begin{bmatrix}
        1 & 1 \\
        0 & 1
        \end{bmatrix},
        \pm \begin{bmatrix}
         1 & 0\\
         1 & 1
        \end{bmatrix}
        \Bigg)\in \mathrm{HP}_0^{(+1,-1,0)}
            (\Sigma_{0,3})
            $,
            and 
         $\Bigg(
        \pm \begin{bmatrix}
        2 & 5\\
        -\frac{1}{5} & 0
        \end{bmatrix},
        \pm 
        \begin{bmatrix}
         1 & 0 \\
        -1 & 1
        \end{bmatrix}
        \Bigg)\in \mathrm{HP}_1^{(+1,+1,0)}
            (\Sigma_{0,3}).$ This shows that $\mathrm{HP}_{\frac{1}{2}(s_1+s_2)}^{(s_1, s_2,0)}
            (\Sigma_{0,3})$ is non-empty. 
            
            Next we show that if $\mathrm{HP}_n^{(s_1, s_2,0)}
            (\Sigma_{0,3})$ is non-empty for  $s_1,s_2\in\{\pm 1\},$ then $n = \frac{1}{2}(s_1+s_2).$ For $\phi\in \mathrm{HP}_n^{(s_1, s_2,0)}
            (\Sigma_{0,3}),$ let $\widetilde{\phi(c_1)}$ and $\widetilde{\phi(c_2)}$ respectively be the lifts of $\phi(c_1)$ and $\phi(c_2)$ in $\mathrm{Par_0}\subset\widetilde{\SL}$. 
            Then the evaluation map sends $\phi$ to  $\ev(\phi) = \widetilde{\phi(c_1)}\widetilde{\phi(c_2)}
            \in \mathrm{Hyp}_n.$
For $i=1,2,$ let $\{g_{i,t}\}$ be a path connecting $\pm\mathrm{I}$ and $\phi(c_i)$ in the one-parameter subgroup of $\psl$ generated by $\phi(c_i),$ and let $\{\widetilde{g_{1,t}g_{2,t}}\}$ be the lift of $\{g_{1,t}g_{2,t}\}$ in $\widetilde{\SL}$ starting from $\mathrm I.$
        Up to a $\psl$-conjugation of $\phi$, we can assume that $\phi(c_1)=\pm \begin{bmatrix}
        1 & s_1 \\
        0 & 1
        \end{bmatrix};$ and by a re-parametrization if necessary, we can assume that 
        $$g_{1,t}=\pm \begin{bmatrix}
        1 & s_1 t \\
        0 & 1
        \end{bmatrix}$$
        for all $t\in [0,1].$ Similarly, we can also assume that 
        $$g_{2,t}=\pm \abcd\begin{bmatrix}
        1 & s_2 t \\
        0 & 1
        \end{bmatrix}\abcd^{-1}$$
        for some $\pm\abcd\in\psl.$  Then $\widetilde{g_{1,t}g_{2,t}}$ projects to the matrix
          $$A_t= 
            \begin{bmatrix}
                1 - s_2act - s_1s_2c^2t^2
                & (s_1  + s_2 a^2)t + s_1s_2ac t^2  \\
                -s_2c^2t 
                & 1 + s_2act
            \end{bmatrix}
           $$ in $\SL$ with $\tr(A_t)=2-s_1s_2c^2t^2.$ Since $\phi$ is non-abelian (as the peripheral elements of $\pi_1(\Sigma_{0,3})$ are sent to elements of different types in $\psl$), we have $c\neq 0.$
           \\

If $s_1=s_2,$ then we have $\tr(A_t)=2-c^2t^2<2,$ where the strict inequality comes from that $c\neq 0.$ Since $\widetilde{\phi(c_1)}\widetilde{\phi(c_2)}$ is a hyperbolic element, we have $
\tr(A_1)<-2<0.$ As a consequence, $\widetilde{\phi(c_1)}\widetilde{\phi(c_2)}\in\mathrm{Hyp}_{\pm 1}$ and $n=\pm 1.$  To further determine $n,$ let us look at the intersection of the path $\{\widetilde{g_{1,t}g_{2,t}}\}$ with $\mathrm{Par}_{\pm 1}.$ Since $\tr(A_t)=2-c^2t^2$ is monotonically decreasing in $t$ and is strictly less than $-2$ at the endpoint, there is a unique time $t_0\in(0,1)$ such that $
\tr(A_{t_0})=-2,$ i.e., $\widetilde{g_{1,t_0}g_{2,t_0}}\in \mathrm{Par_{\pm 1}}.$ More precisely, $\widetilde{g_{1,t_0}g_{2,t_0}}\in \mathrm{Par}_{- 1}^{+}\cup\mathrm{Par}_{1}^{-},$ as the former is the component of $\Par_{-1}$ adjacent to $\Ell_{-1}$ and $\Hyp_{-1},$ and the latter is the component of $\Par_{1}$ adjacent  $\Ell_{1}$ and $\Hyp_{1}.$ Then by Lemma \ref{lem_offdiag}, since $c\neq 0,$ the sign of $\widetilde{g_{1,t_0}g_{2,t_0}}$ equals $sgn(-s_2c^2t_0)=-sgn(s_2),$ and 
$\widetilde{g_{1,t_0}g_{2,t_0}}\in \mathrm{Par}^{-sgn(s_2)}_{s_2}$. 
As a consequence, $\widetilde{\phi(c_1)}\widetilde{\phi(c_2)}\in \mathrm{Hyp}_{s_2},$ which among $\mathrm{Hyp}_{\pm 1}$ is the one adjacent to $\mathrm{Par}_{s_2}^{-sgn(s_2)},$ and $n=s_2=\frac{1}{2}(s_1+s_2).$ 

If $s_1\neq s_2,$ then we have $\tr(A_t)
=2+c^2t^2>2,$ where the strict inequality again comes from that $c\neq 0.$  This shows that the path $\{\widetilde{g_{1,t}g_{2,t}}\}$ 
never intersects $\bigcup_{k \in \mathbb{Z}}\mathrm{Par}_k,$
and hence stays in $\mathrm{Hyp}_0,$ including the endpoint $\widetilde{\phi(c_1)}\widetilde{\phi(c_2)}.$  As a consequence, $n=0=\frac{1}{2}(s_1+s_2).$

\bigskip

       For (2), under the identification of the $\psl$-representations of $\pi_1(\Sigma_{0,3})$ and pairs of elements of $\psl,$ we have that 
         $\Bigg(
          \pm 
        \begin{bmatrix}
         5 & -4 \\
        4 & -3
        \end{bmatrix},
        \pm \begin{bmatrix}
        \frac{1}{2} & 0 \\
        0 & 2
        \end{bmatrix}
        \Bigg)\in \mathrm{HP}_{-1}^{(-1,0,0)}
            (\Sigma_{0,3}),
            $ 
       $\Bigg(
        \pm \begin{bmatrix}
         1 & 0\\
         1 & 1
        \end{bmatrix},
         \pm \begin{bmatrix}
        2 & 0 \\
        0 & \frac{1}{2}
        \end{bmatrix}
        \Bigg)\in \mathrm{HP}_0^{(-1,0,0)}
            (\Sigma_{0,3}),$ 
        $\Bigg(
        \pm \begin{bmatrix}
         1 & 1 \\
        0 & 1
        \end{bmatrix},
        \pm \begin{bmatrix}
       \frac{1}{2} & 0 \\
        0 &   2
        \end{bmatrix}
        \Bigg) \in \mathrm{HP}_0^{(+1,0,0)}
            (\Sigma_{0,3})$,  and 
         $\Bigg(
        \pm \begin{bmatrix}
         -3 & 4 \\
        -4 & 5
        \end{bmatrix},
         \pm \begin{bmatrix}
        2 & 0 \\
        0 & \frac{1}{2}
        \end{bmatrix}
        \Bigg)\in \mathrm{HP}_1^{(+1,0,0)}
            (\Sigma_{0,3}).$
            This shows that $\mathrm{HP}_0^{(s_1,0,0)}
            (\Sigma_{0,3})$ and $\mathrm{HP}_{s_1}^{(s_1,0,0)}
            (\Sigma_{0,3})$ are non-empty. 

Next we show that if $\mathrm{HP}_n^{(s_1, 0,0)}
            (\Sigma_{0,3})$ is non-empty for $s_1\in\{\pm 1\},$ then $n = 0$ or $n=s_1.$ Since $\mathrm{HP}_0^{(s_1, 0,0)}
            (\Sigma_{0,3})$ is non-empty for both $s_1=+1$ and $s_1=-1,$ it suffices to show that if $\mathrm{HP}_n^{(s_1, 0,0)}
            (\Sigma_{0,3})$ is non-empty for $s_1\in\{\pm 1\}$ and $n\in\{\pm 1\},$ then $n=s_1.$ For $\phi\in \mathrm{HP}_n^{(s_1,0,0)}
            (\Sigma_{0,3})$ with $n\in\{\pm 1\},$ let $\widetilde{\phi(c_1)}$ be the lift of $\phi(c_1)$ in $\mathrm{Par}_0\subset\widetilde{\SL}$ and let  $\widetilde{\phi(c_2)}$ and $\widetilde{\phi(c_3)}$ respectively be the lifts of  $\phi(c_2)$ and $\phi(c_3)$ in $\mathrm{Hyp}_0\subset\widetilde{\SL}.$ 
            Then the evaluation map sends $\phi$ to  $\widetilde{\phi(c_1)}\widetilde{\phi(c_2)}
            \in \mathrm{Hyp}_n.$
            Similar to part (1), for $i=1,2,$ let $\{g_{i,t}\}$ be a path connecting $\pm\mathrm{I}$ and $\phi(c_i)$ in the one-parameter subgroup of $\psl$ generated by $\phi(c_i),$ and let $\{\widetilde{g_{1,t}g_{2,t}}\}$ be the lift of $\{g_{1,t}g_{2,t}\}$ in $\widetilde{\SL}$ starting from $\mathrm I$ and ending with $\widetilde{\phi(c_1)}\widetilde{\phi(c_2)}.$  Up to conjugation, we can assume that $\phi(c_2)=\pm \begin{bmatrix}
        e^\lambda & 0 \\
        0 & e^{-\lambda}
        \end{bmatrix}$ for some $\lambda>0$; and by a re-parametrization if necessary, we can assume that 
        $$g_{2,t}=\pm \begin{bmatrix}
        e^{\lambda t} & 0 \\
        0 & e^{-\lambda t}
        \end{bmatrix}$$
        for all $t\in [0,1].$ Similarly, we can also assume that 
        $$g_{1,t}=\pm \abcd\begin{bmatrix}
        1 & s_1 t \\
        0 & 1
        \end{bmatrix}\abcd^{-1}$$
        for some $\pm\abcd\in\psl.$  Then $\widetilde{g_{1,t}g_{2,t}}$ projects down to an $\SL$-matrix
          $$ A_t= 
        \begin{bmatrix}
               (1-s_1act)e^{\lambda t}
                & s_1a^2te^{-\lambda t}  \\
                -s_1c^2te^{\lambda t} 
                &  (1+s_1act)e^{-\lambda t}
            \end{bmatrix}$$
            with $\tr(A_t)=2\cosh (\lambda t)-2s_1act \sinh(\lambda t).$ 
           Since $\widetilde{\phi(c_1)}\widetilde{\phi(c_2)}\in \mathrm{Hyp}_{\pm 1},$ $
           \tr(A_1)<-2;$ and since $\tr(A_0) = \tr (\mathrm I)=2>-2,$ there exist some $t_0\in (0,1)$ such that $
           \tr(A_{t_0})=-2,$ i.e., $\widetilde{g_{1,t_0}g_{2,t_0}}\in \mathrm{Par}_{-1}^+\cup\Par_1^-.$  By Lemma \ref{lem_offdiag}, for all such $t_0,$ if $a\neq 0,$ then the sign of $\widetilde{g_{1,t_0}g_{2,t_0}}$ equals $sgn\big(-s_1a^2t_0e^{-\lambda t_0}\big)=-sgn(s_1);$ and if $c\neq 0,$  then the sign of $\widetilde{g_{1,t_0}g_{2,t_0}}$ equals $sgn\big(-s_1c^2t_0e^{\lambda t_0}\big)=-sgn(s_1).$ 
           In any case, 
$\widetilde{g_{1,t_0}g_{2,t_0}}\in \mathrm{Par}^{-sgn(s_1)}_{s_1}$. 
           As a consequence, $\widetilde{\phi(c_1)}\widetilde{\phi(c_2)}\in \mathrm{Hyp}_{s_1},$ which among $\mathrm{Hyp}_{\pm 1}$ is the one adjacent to $\mathrm{Par}_{s_1}^{-sgn(s_1)},$ and $n=s_1.$ 
        \bigskip
        
        For the connectedness, let $(n,s)$ be a pair in (1) or (2) where the corresponding $\HPns(\pants)$ is nonempty.
        By Corollary \ref{cor_evimage} below, the image of $\HPns(\Sigma_{0,3})$ under the evaluation map is $\Hyp_n$, which is connected. Then by Theorem \ref{thm_PLP}, the domain $\mathrm{HP}_n^s(\Sigma_{0,3})$ of $\ev: \mathrm{HP}_n^s(\Sigma_{0,3})\to\Hyp_n$ is connected, completing the proof.
        \end{proof}
        By Theorem \ref{thm_pants2} (1), (2) and Proposition \ref{prop_evimage}, we have the following corollary, which is needed in the proofs of Theorem \ref{thm_pants2} (3), and of the main results in Section 5 and Section 6.
        
    \begin{corollary}\label{cor_evimage}
    For $n\in\mathbb Z$ and $s \in\{-1,0,+1\}^3,$ let $\mathrm{HP}_n^s(\Sigma_{0,3})=e^{-1}(n)\cap s^{-1}(s).$
        \begin{enumerate}[(1)] 
            \item For $s_1, s_2\in \{\pm 1\},$ $\ev\Big(\HP_{_{\frac{1}{2}(s_1+s_2)}}^{(s_1,s_2,0)}(\Sigma_{0,3})\Big)= \Hyp_{\frac{1}{2}(s_1+s_2)}$.

            \item For $s_1\in\{\pm 1\},$ $\ev\big(\HP_{0}^{(s_1,0,0)}(\Sigma_{0,3})\big)= \Hyp_0$
             and 
             $\ev\big(\HP_{s_1}^{(s_1,0,0)}(\Sigma_{0,3})\big)= \Hyp_{s_1}$.

             \item For $s_2\in\{\pm 1\},$ $\ev\big(\HP_{0}^{(0,s_2,0)}(\Sigma_{0,3})\big)= \Hyp_0$
             and 
             $\ev\big(\HP_{s_2}^{(0,s_2,0)}(\Sigma_{0,3})\big)= \Hyp_{s_2}$.


        \end{enumerate}
    \end{corollary}
    \begin{proof}
        
        
         For each pair $(n,s)$ where $s_3 = 0$ and $\HPns(\pants) \neq \emptyset$, we have $\ev\big(\HPns(\pants)\big)\subset \Hyp_n$. Here, for $(n,s)$ in (1) and (2), we find a representation $\phi\in \HPns(\pants)$ for each $\widetilde{C}\in \Hyp_n$ such that $\ev(\phi) = \widetilde{C}$, and conclude that $\Hyp_n\subset \ev\big(\HPns(\pants)\big)$.       
         \\
        
        For (1), let $\widetilde{C}\in \Hyp_{\frac{1}{2}(s_1+s_2)}$. By Proposition \ref{prop_evimage}, there exists a $(g_1, g_2)\in \Par\times \Par$ such that $ \ev(g_1,g_2) = \widetilde{C}$. 
        Define a representation $\phi\in \HP(\pants)$ by setting $\phi(c_1) = g_1$ and $\phi(c_2) = g_2$. Since $\ev(\phi) = \widetilde{C}\in \Hyp_{\frac{1}{2}(s_1+s_2)}$, 
        $\phi\in \HP_{_{\frac{1}{2}(s_1+s_2)}}^{(s'_1,s'_2,0)}(\Sigma_{0,3})$ for some $s'_1, s'_2\in \{\pm 1\}$. If $s'_1 = s_1$ and $s'_2 = s_2$, then $\phi$ is a desired representation. Otherwise, since Theorem \ref{thm_pants2} implies $\frac{1}{2}(s_1+s_2) = \frac{1}{2}(s'_1+s'_2)$, we have $s'_1 = -s_1, s'_2 = -s_2$ and $e(\phi) = 0$. Let $\pm C\in \Hyp$ be a projection of $\widetilde{C}\in \Hyp_0$ to $\psl$. 
        Then $\pm C = g\bigg(\pm
        \begin{bmatrix}
            e^{\lambda t} & 0 \\
            0 & e^{-\lambda t}
        \end{bmatrix}
        \bigg)g^{-1}$ for some $\lambda >0$ and $g\in \psl$. 
         Let $h=g\bigg(\pm \begin{bmatrix}
        1 & 0 \\
        0 & -1
    \end{bmatrix}\bigg)g^{-1}\in\pgl\setminus\psl$. 
    By Proposition \ref{prop_pglpsl}, 
    the $\pgl$-conjugation $h\phi h^{-1}$ of $\phi$ is in $\HP^{(s_1, s_2, 0)}_{0}$.
    Moreover, $(hg_1h^{-1})(hg_2h^{-1}) =  \pm C$ and hence $\ev(h\phi h^{-1}) = \widetilde{C}$, as it is a unique lift of $\pm C$ to $\Hyp_0$.
    
        For (2), we first let $\widetilde{C}\in \Hyp_0$. Then by Proposition \ref{prop_evimage}, there exists a $(g_1, g_2)\in \Par\times \Hyp$ such that $ \ev(g_1,g_2) = \widetilde{C}$. 
        Define a representation $\phi\in \HP(\pants)$ by setting $\phi(c_1) = g_1$ and $\phi(c_2) = g_2$. Since $\ev(\phi) = \widetilde{C}\in \Hyp_{0}$, 
        $\phi\in \HP_0^{(s'_1,0,0)}(\Sigma_{0,3})$ for some $s'_1\in \{\pm 1\}$. Letting $\pm C \in \Hyp$ and $h\in \pgl\setminus \psl$ as above, we obtain a desired representation $\phi$ if $s'_1 = s_1$, or $h\phi h^{-1}$ if $s'_1 = -s_1$. 
        For $\widetilde{C}\in \Hyp_{s_1}$, again by Proposition \ref{prop_evimage} there exists a $\phi\in \HP^{(s'_1,0,0)}_{s_1}(\pants)$ for some $s'_1\in \{\pm 1\}$ such that $\ev(\phi) = \widetilde{C}$. Then Theorem \ref{thm_pants2} implies $s'_1 = s_1$.

        For (3), as it differs from (2) by a permutation of the pepheral elements $c_1$ and $c_2$, it follows directly from (2).
    \end{proof}

   \begin{remark}\label{rmk_MWpants} Theorem \ref{thm_pants2} together with Theorem \ref{thm_goldman} imply that, if $s\notin \{\pm 1\}^3,$  then $\mathrm{HP}_n^s(\Sigma_{0,3})$ is non-empty and connected if and only if $(n,s)$ satisfies the generalized Milnor-Wood inequality 
$$\chi(\Sigma_{0,3})+ p_+(s) \leqslant n \leqslant -\chi(\Sigma_{0,3}) - p_-(s).$$  I.e., Theorem \ref{thm_main4} holds for the three-hole sphere $\Sigma_{0,3}.$
   \end{remark}

    \subsubsection{Type-preserving representations}\label{subsection_TP}

    Finally, we describe the connected components of the space $\mathcal R (\Sigma_{0,3})$ of type-preserving representations. 
    Let $e: \mathcal R(\Sigma_{0,3})\to \mathbb{Z}$ be the relative Euler class map and  $s: \mathcal R(\Sigma_{0,3})\to \{\pm 1\}^3$ be the sign map defined respectively by sending $\phi$ to the relative Euler class $e(\phi)$ and the sign $s(\phi)$ of $\phi.$
    For $n\in \mathbb{Z}$, let 
    $\mathcal R_n(\Sigma_{0,3}) = e^{-1}(n)$.
   Then by Proposition \ref{prop_evimage}, $\mathcal R_n(\Sigma_{0,3})$ is non-empty if and only if $n\in \{-1,0,1\}.$

    \begin{theorem}\label{thm_pants} For $n\in \mathbb{Z}$ and $s\in\{\pm 1\}^3,$  let $\mathcal R^s_n(\Sigma_{0,3}) = e^{-1}(n)\cap s^{-1}(s)$. 
    
    \begin{enumerate}[(1)]

\item $\mathcal R_{-1}^s(\Sigma_{0,3})$ is non-empty if and only if  $s=(-1,-1,-1).$ 

\item $\mathcal R_{0}^s(\Sigma_{0,3})$ is non-empty if and only if,  up to a permutation of the peripheral elements, $s=(+1,-1,-1)$ or $(+1,+1,-1).$

\item $\mathcal R_{1}^s(\Sigma_{0,3})$ is non-empty if and only if $s=(+1,+1,+1).$

 \end{enumerate}
Moreover, all the non-empty spaces above are connected. 
\end{theorem}

For the proof of Theorem \ref{thm_pants}, we need the following lemma, which will also be needed in the proof of Theorem 5.6.
    \begin{lemma}\label{lem_R0pants}
        The space $\mathcal{R}_0(\Sigma_{0,3})$ consists of abelian representations.
    \end{lemma}
    \begin{proof}
        By Proposition \ref{prop_reducible},  all abelian  type-preserving  representations are in $\mathcal{R}_0(\Sigma_{0,3})$.
        Let $\phi \in \mathcal{R}_0(\Sigma_{0,3})$. For the lifts  $\widetilde{\phi(c_1)}, 
        \widetilde{\phi(c_2)}\in \Par_0$ of $\phi(c_1), 
        \phi(c_2)\in \psl$, $e(\phi) = 0$ implies that $\widetilde{\phi(c_1)}
        \widetilde{\phi(c_2)}\in \Par_0$, and hence $\chi(\phi) = (2,2,2)$. 
        By proposition \ref{prop_char} (1), $\phi$ is reducible, and hence up to a $\psl-$conjugation is a representation into the Borel subgroup of $\psl$. That is, up to conjugation, 
        $$\phi(c_1) = 
        \pm \begin{bmatrix}
            1 & t_1\\
            0 & 1
        \end{bmatrix} \quad \text{and}\quad \phi(c_2) = 
        \pm \begin{bmatrix}
            1 & t_2\\
            0 & 1
        \end{bmatrix}$$ 
        for some $t_1, t_2\in \mathbb{R}\setminus \{0\}$, which  implies that $\phi$ is abelian.
    \end{proof}

        \begin{proof}[Proof of Theorem \ref{thm_pants}] 

        For (1) and (3), under the identification of the 
        $\psl$-representations of $\pi_1(\Sigma_{0,3})$ and pairs of elements of $\psl$, we have  
        $\Bigg(\pm \begin{bmatrix}
            3 & -2 \\
            2 & -1
        \end{bmatrix},
         \pm \begin{bmatrix}
            1 & 0 \\
            2 & 1
        \end{bmatrix}
        \Bigg)\in 
        \mathcal R_{-1}^{(-1,-1,-1)} (\Sigma_{0,3})$ and
        $\Bigg(\pm \begin{bmatrix}
            -1 &  2\\
            -2 & 3
        \end{bmatrix},
         \pm \begin{bmatrix}
            1 & 2 \\
            0 & 1
        \end{bmatrix}
        \Bigg)\in 
        \mathcal R_{1}^{(+1,+1,+1)} (\Sigma_{0,3}),$ hence $\mathcal R_{-1}^{(-1,-1,-1)} (\Sigma_{0,3})$ and $\mathcal R_{1}^{(+1,+1,+1)} (\Sigma_{0,3})$ are non-empty. Next, we will show that any $\phi\in \mathcal R_{\pm 1}^s(\Sigma_{0,3})$ must be in one of these two spaces.  By abuse of notations, we let $\widetilde{\phi(c_1)}$ and $\widetilde{\phi(c_2)}$ respectively be both the lifts of $\phi(c_1)$ and $\phi(c_2)$ in $\mathrm{Par}_0,$ and their projections to $\SL.$ Since $e(\phi)=\pm 1,$ we can up to conjugation assume that $\widetilde{\phi(c_1)}\widetilde{\phi(c_2)}=-\begin{bmatrix}
        1 & -\epsilon \\
        0 & 1
        \end{bmatrix},$ where $\epsilon = \pm 1. $ Let $\ev: \mathcal R(\Sigma_{0,3})\to \widetilde{\SL}$ be the evaluation map.  Then the sign of $\ev(\phi) = \widetilde{\phi(c_1)}\widetilde{\phi(c_2)}$ equals $-sgn(\epsilon),$
        hence $\ev(\phi)\in \mathrm{Par}_{\pm 1}^{-sgn(\epsilon)}.$ By Lemma \ref{lem_evimage}, the only possibility is $\ev(\phi)\in \mathrm{Par}_{\epsilon}^{-sgn(\epsilon)}$, which implies that  $e(\phi)=\epsilon.$ Now let 
        $\widetilde{\phi(c_1)}=\abcd.$ Then $\widetilde{\phi(c_2)}=\begin{bmatrix}
        -d &  \epsilon d + b \\
        c & -\epsilon c-a 
        \end{bmatrix}.$
     Since $\tr\widetilde{\phi(c_1)}=a+d=2$ and $\tr \widetilde{\phi(c_2)}=-\epsilon c -(a+d)=2,$ we have $c= -4\epsilon.$ 
        Then by Lemma \ref{lem_offdiag}, both
        $\phi(c_1)$ and $\phi(c_2)$ have sign $-sgn(c)=-sgn(-4\epsilon)=sgn(\epsilon)$, and $\phi(c_3)$ has sign $sgn(\epsilon)$ as it is the sign of $\big(\widetilde{\phi(c_1)}\widetilde{\phi(c_2)}\big)^{-1}$. 
        Therefore, $\phi\in\mathcal R_{\epsilon}^{(\epsilon,\epsilon,\epsilon)}(\Sigma_{0,3}).$ 
\\

For (2), if $\phi\in\mathcal R_0^s(\Sigma_{0,3}),$ then Lemma \ref{lem_R0pants} implies that $\phi$ is abelian, and hence is conjugate to a representation into the one-parameter subgroup generated by $\pm\parpp.$ That is, up to a $\psl$-conjugation, 
$(\phi(c_1),\phi(c_2))=\Bigg(\pm  \begin{bmatrix}
        1 & t_1 \\
        0 & 1
        \end{bmatrix} ,\pm  \begin{bmatrix}
        1 & t_2 \\
        0 & 1
        \end{bmatrix}\Bigg)$ for some $t_1,t_2\in\mathbb R\setminus \{0\}.$  Then $\phi(c_3)=\big(\phi(c_1)\phi(c_2)\big)^{-1}=\pm \begin{bmatrix}
        1 & -(t_1+t_2) \\
        0 & 1
        \end{bmatrix};$
        and by Lemma \ref{lem_offdiag} we have $s(\phi)=(sgn(t_1), sgn(t_2), - sgn(t_1+t_2)),$ which can take all the possible values of $\{\pm 1\}^3$ except $(+1,+1,+1)$ and $(-1,-1,-1).$ Therefore, $\mathcal R_0^s(\Sigma_{0,3})$ is non-empty if and only if $s\neq (+1,+1,+1)$ and $s\neq (-1,-1,-1).$ 
\\

Finally, we prove the connectedness. For the spaces in (1) and (3), we claim that any $\phi_1$ and $\phi_2$ in $\mathcal R_{-1}^{(-1,-1,-1)} (\Sigma_{0,3}),$ or in $\mathcal R_{1}^{(+1,+1,+1)} (\Sigma_{0,3}),$ are conjugate. Indeed, for $i\in\{1,2\},$ define representations $\widetilde{\phi_i}:\pi_1(\Sigma_{0,3})\to \SL$ by $\widetilde{\phi_i}(c_1)=\widetilde{\phi_i(c_1)}$ and $\widetilde{\phi_i}(c_2)=\widetilde{\phi_i(c_2)}.$
        Notice that $\widetilde{\phi_1}$ and $\widetilde{\phi_2}$ are not $\SL$-lifting of $\phi_1$ and $\phi_2$ as they have odd relative Euler class and do not lift to $\SL.$ We have the character  $\chi(\widetilde{\phi_i})=(2,2,-2)$ for $i\in\{1,2\}.$ By Proposition \ref{prop_char}(b), since $\kappa(2,2,-2)=18\neq 2,$ the $\SL$-representations  $\widetilde{\phi_1}$ and  $\widetilde{\phi_2}$ are $\GL$-conjugate. Since the signs $s(\phi_1)=s(\phi_2),$ we have $s\big(\widetilde{\phi_1}(c_1)\big)=s\big(\widetilde{\phi_2}(c_1)\big),$ which by Proposition \ref{prop_pglpsl} implies that $\widetilde{\phi_1}$ and $\widetilde{\phi_2}$ are indeed $\SL$-conjugate. Let $A\in\SL$ so that $\widetilde{\phi_2}=A\widetilde{\phi_1}A^{-1},$ and let $g=\pm A$ be its projection in $\psl.$ Then we have $\phi_2=g \phi_1 g^{-1}.$
        Let $\{g_t\}_{t\in[0,1]}$ be any path in $\psl$ connecting $\pm\mathrm{I}$ and $g$. Then the path $\{g_t\phi_1 g_t^{-1}\}$ connects $\phi_1$ and $\phi_2$ in $\mathcal R_{-1}^{(-1,-1,-1)} (\Sigma_{0,3}),$ or in $\mathcal R_{1}^{(+1,+1,+1)} (\Sigma_{0,3}).$ 
        
For the spaces in (2), let $s\in\{\pm 1\}^3\setminus \{(+1,+1,+1), (-1,-1,-1)\},$ and let $\phi_1$ and $\phi_2$ be two representations in $\mathcal R_{0}^s(\Sigma_{0,3}).$ For $i\in\{1,2\}$  let $g_i$ be an element of $\psl$ such that $$(g_i\phi_i(c_1)g_i^{-1}, g_i\phi_i(c_1)g_i^{-1})=\Bigg(\pm  \begin{bmatrix}
        1 & a_i \\
        0 & 1
        \end{bmatrix} ,\pm  \begin{bmatrix}
        1 & b_i \\
        0 & 1
        \end{bmatrix}\Bigg)$$ for some $a_i,b_i\in\mathbb R\setminus \{0\}.$  and let $\{g_{i,t}\}_{t\in[0,1]}$ be the path in $\psl$ connecting $\pm\mathrm{I}$ and $g_i.$   Then path $\{g_{i,t}\phi_ig_{i,t}^{-1}\}$ connects $\phi_i$ and $g_i\phi_ig_i^{-1}$ in $\mathcal R_0^s(\Sigma_{0,3}).$ To connect $g_1\phi_1g_1^{-1}$ and $g_2\phi_2g_2^{-1},$ we consider the path $\{\psi_t\}_{t\in [1,2]}$ given by $$(\psi_t(c_1),\psi_t(c_2))= \Bigg(\pm  \begin{bmatrix}
        1 & (2-t)a_1+(t-1)a_2 \\
        0 & 1
        \end{bmatrix} ,\pm  \begin{bmatrix}
        1 & (2-t)b_1+(t-1)b_2 \\
        0 & 1
        \end{bmatrix}\Bigg).$$ 
        Since $s(\phi_1)=s(\phi_2)=s,$ we have $sgn(a_1)=sgn(a_2)$, $sgn(b_1)=sgn(b_2)$ and $sgn(a_1+b_1)=sgn(a_2+b_2).$ As a consequence, for all $t\in [1,2],$ we have $sgn\big((2-t)a_1+(t-1)a_2\big)=sgn(a_1),$ $sgn\big((2-t)b_1+(t-1)b_2\big)=sgn(b_1),$ and $-sgn\big((2-t)a_1+(t-1)a_2 +(2-t)b_1+(t-1)b_2\big)=-sgn(a_1+b_1)$, which implies that $s(\psi_t)=s.$ Therefore, $\{\psi_t\}_{t\in [1,2]}$ is  path in $\mathcal R_0^s(\Sigma_{0,3})$ connecting $g_1\phi_1g_1^{-1}$ and $g_2\phi_2g_2^{-1}.$ As a consequence, $\mathcal R_0^s(\Sigma_{0,3})$ is connected.
\end{proof}

\begin{remark} Theorem \ref{thm_pants} implies that Theorem \ref{thm_main2} holds with none of the pairs $(n,s)$ satisfying the generalized Milnor-Wood inequality 
and all the components are exceptional, where $\mathcal R_{-1} ^{(-1,-1,-1)} (\Sigma_{0,3})$ and $\mathcal R_{1} ^{(+1,+1,+1)} (\Sigma_{0,3})$ consist of discrete faithful representations  and the other six consist of abelian representations.
\end{remark}
\section{A genericity property}

    Let $\Sigma = \Sigma_{g,p}$ be a surface with genus $g$ and $p\geqslant 1$ punctures such that the Euler characteristic $\chi(\Sigma)=2-2g-p\leqslant -2$. 
     We will consider a maximal dual-tree decomposition of $\Sigma$, that is, a pants decomposition of $\Sigma$ whose dual graph is a tree.
    Each vertex of this dual graph corresponds to a subsurface homeomorphic to either $\torus$ or $\pants$.
    The union of all subsurfaces in this decomposition that are homeomorphic to $\pants$ is a connected subsurface of $\Sigma$ with genus zero, as otherwise the dual graph would contain a loop.
    Therefore, there are exactly $g+p-2$ subsurfaces in the maximal dual-tree decomposition homeomorphic to $\pants$, which we denote by $P_i$ for $i\in \{1, \cdots, g+p-2\}$; and there are exactly $g$ subsurfaces in the decomposition homeomorphic to $\torus$, which we denote by $T_j$ for $1\leqslant j\leqslant g$.
    By abuse of notations, we let $P_i$'s and $T_j$'s be both the subsurfaces of $\Sigma$ in the decomposition and the corresponding vertices in the dual graph. 
    For convenience, we use a maximal dual-tree decomposition $\Sigma = \apdecomp$ with an additional condition that the dual graph of the decomposition restricted to the subsurface $\displaystyle\bigcup_{i=1}^{g+p-2} P_i$ is a path, i.e., a tree where each vertex is adjacent to at most two edges. 
     We call such a decomposition an \emph{almost-path decomposition}. See Figure \ref{fig:APD_general}. For each $i\in\{1,\dots,g+p-3\}$, we connect the vertices $P_i$ and $P_{i+1}$ by an edge to form a path; and for each $j\in\{1,\dots,g\}$, we connect the vertices $T_j$ and $P_j$ by an edge. In the special case that $p = 1$, we have $p+g -2 = g - 1$, and we connect both $T_{g-1}$ and $T_g$ to $P_{g-1}$. This gives the dual graph of an almost-path decomposition. Throughout this paper, we will use this fixed almost-path decomposition. 
        \begin{figure}[hbt!]
        \centering
        \begin{overpic}[width=1.0\textwidth]{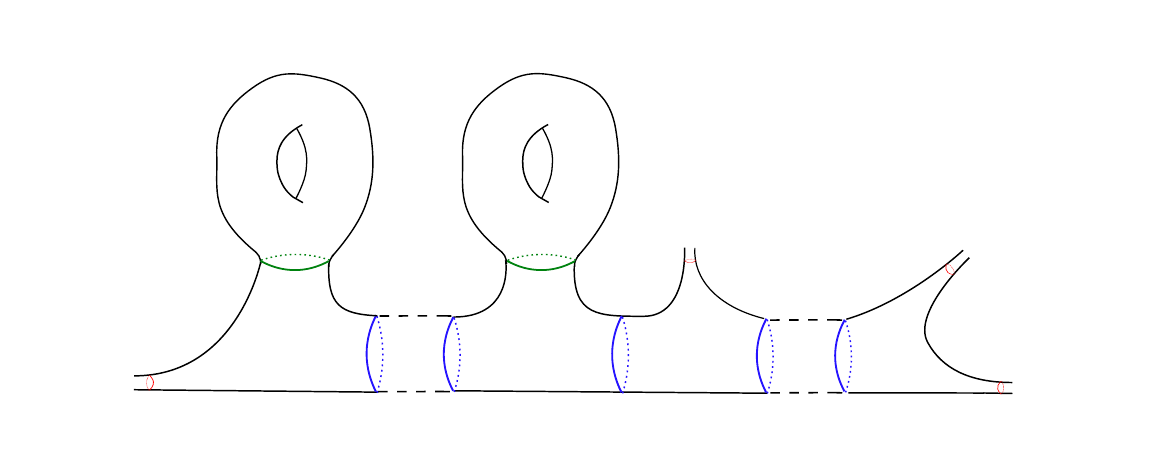}
            \put(9,5.5){\textcolor{red}{$c_1$}}
            \put(58,19){\textcolor{red}{$c_2$}}
            \put(83,17.5){\textcolor{red}{$c_{p-1}$}}
            \put(87,5){\textcolor{red}{$c_{p}$}} 

        \put(24,34){\textcolor{OliveGreen}{$T_1$}}
             \put(45,34){\textcolor{OliveGreen}{$T_g$}}
            \put(24,13.5){\textcolor{OliveGreen}{$d'_1$}}
            \put(45,13.5){\textcolor{OliveGreen}{$d'_g$}}

            \put(24,3){\textcolor{blue}{$P_1$}}
            \put(45,3){\textcolor{blue}{$P_g$}}
            \put(57,3){\textcolor{blue}{$P_{g+1}$}}
            \put(77,3)
            {\textcolor{blue}{$P_{g+p-2}$}}
            \put(29,8){\textcolor{blue}{$d_1$}}
            \put(39,8){\textcolor{blue}{$d_{g-1}$}}
            \put(49.5,8)
            {\textcolor{blue}{$d_g$}}
            \put(59.8,8){\textcolor{blue}{$d_{g+1}$}}
            \put(72,8){\textcolor{blue}{$d_{g+p-3}$}}
 \end{overpic}
        \caption{\label{fig:APD_general} Almost-path decomposition of 
        $\Sigma = \Sigma_{g,p}$.}
    \end{figure}
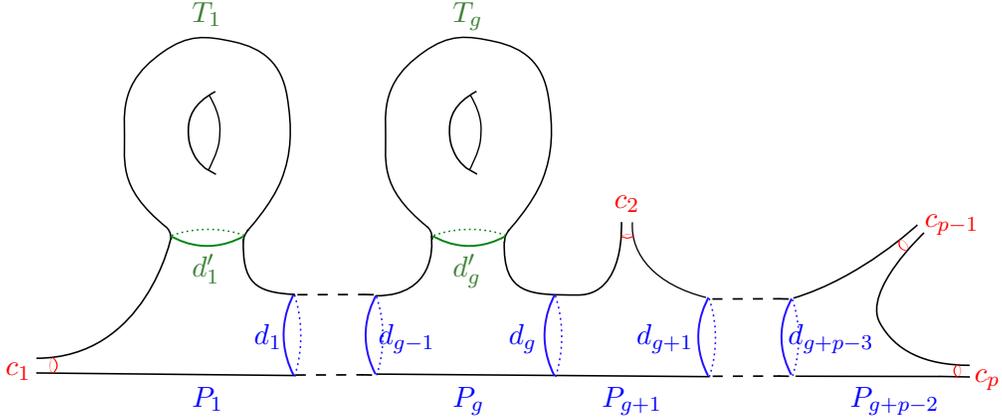
   
    For $j\in\{1,\dots,g\}$, let $a_j, b_j$ be the generators of $\pi_1(T_j)$. In the rest of this paper, we will use the following presentation 
    $$\pi_1(\Sigma) = \big\langle 
     a_1, b_1,a_2, b_2,\cdots, a_g, b_g,
     c_1, \cdots, c_p\ \big|\ c_1\cdot [a_1, b_1][a_2, b_2]\cdots [a_g, b_g]\cdot
     c_2\cdots c_p  \big\rangle$$ of the fundamental group of $\Sigma$, 
     where $c_1,\cdots, c_p$ are the preferred peripheral elements. 
    For each $j\in\{1,\dots,g\}$,  let 
    $$d'_j =  [a_j, b_j]$$
    be the element in $\pi_1(\Sigma)$ represented by a common boundary of $P_j$ and $T_j$; and for $i\in\{1,\dots,g + p - 3\}$, let $d_i$ be the element of $\pi_1(\Sigma)$ represented by a common boundary of $P_i$ and $P_{i+1}$, defined as 
    \begin{equation*}
  d_i  = \left\{
    \begin{array}{rcl}
      (c_1d'_1)^{-1} & \text{if } & i=1,\\
     \big(d_{i-1}^{-1}d'_i\big)^{-1} & \text{if } &  2\leqslant i \leqslant g,\\
    \big(d_{i-1}^{-1}c_{i-g+1}\big)^{-1} & \text{if }&  g+1\leqslant  i\leqslant g+p-3.
    \end{array} \right.
\end{equation*}
    By abuse of notations, we let $d'_1,\cdots, d'_g$ and $d_1,\cdots, d_{g+p-3}$ be both the elements of $\pi_1(\Sigma)$ and their representatives, and call them the \emph{decomposition curves} of the almost-path decomposition.

    For $n\in \mathbb{Z}$ and $s\in \{-1,0,+1\}^p$, we denote by $\nonabel$ the subspace of $\HPns(\Sigma)$ consisting of representations whose restriction to  each $\pi_1(P_i),$ $i \in\{ 1,\cdots, g+p-2\}$,  and each $\pi_1(T_j),$ $j \in\{ 1,\cdots, g\}$, is non-abelian. If $\phi$ lies in $\nonabel$, then we can apply Lemma \ref{lem_inthyptorus} and Lemma \ref{lem_inthyppants} respectively to the restrictions $\phi|_{\pi_1(T_j\cup P_j)}$ and $\phi
|_{\pi_1(P_i\cup P_{i+1})}$
    for all $i\in \{1,\cdots, g+p-3\}$ and  $j\in \{1,\cdots, g\}$. This will be used in the proof of Theorem \ref{thm_general}, which is the main result of this paper for the general surfaces. 
    The goal of this section is to prove the following  Theorem \ref{thm_nonabel}  which provides  a necessary and sufficient condition on the triple $(\Sigma, n,s)$ for which every representation in $\HPns(\Sigma)$ can be continuously deformed into $\nonabel.$ 

        \begin{theorem}\label{thm_nonabel}
        Let $\Sigma = \Sigma_{g,p}$ be a punctured surface with $\chi(\Sigma)\leqslant -2$, together with the chosen almost-path decomposition $\apdecomp$. For $n\in \mathbb{Z}$ and $s\in \{-1,0,+1\}^p$,
        suppose that one of the following conditions holds:
        \begin{enumerate}[(1)]
            \item $g\geqslant 1$,
            \item $n\neq 0$,
            \item $p_0(s)\geqslant 1$, and
            \item $p_+(s)\neq 1$ and $p_-(s)\neq 1$.
        \end{enumerate}
        Then for any $\phi \in \HPns(\Sigma)$, 
        there is a path 
        $\{\phi_t\}_{t\in [0,1]}$ 
         in $\HPns(\Sigma)$ with $\phi_0 = \phi$ and $\phi_1\in \nonabel$.
\end{theorem}

Theorem \ref{thm_nonabel} will be proved in Subsection \ref{Pf} with the following  outline: 
For the three cases corresponding to Conditions (1), (2), and (3), the theorem will be proved using Proposition \ref{prop_NAhyp}, Proposition \ref{prop_NAtorus}, and Proposition \ref{prop_NAsphere1}, respectively. These propositions are respectively stated and proved in Subsections \ref{4.1}, \ref{4.2.1}, and \ref{4.2.2}. When only Condition (4) holds, we will deform each representation $\phi\in \mathcal{R}^s_0(\Sigma_{0,p})\setminus \mathrm{NA}^s_0(\Sigma_{0,p})$ into $ \mathrm{NA}^s_0(\Sigma_{0,p})$, using Proposition \ref{prop_NAsphere1} if $\phi|_{\pi_1(P_i)}$ is non-abelian for some $1\leqslant i\leqslant p-2$, and using Proposition \ref{prop_NAsphere2} if $\phi|_{\pi_1(P_i)}$ is abelian for all $1\leqslant i\leqslant p-2$.

\begin{lemma}\label{lem_nonabelpants}
        Let $\pi_1(\pants) = \langle c_1, c_2\rangle$,
        and let $\phi$ be an abelian $\psl$-representation of $\pi_1(\pants)$ with $\phi(c_1)\neq \pm\mathrm{I}$ and $\phi(c_2)\neq \pm\mathrm{I}$. 
        Then there is a one-parameter subgroup $\{h_t\}_{t\in \mathbb{R}}$ of $\psl$ with $h_0 = \pm \mathrm I$, such that $\phi(c_1)$ and $h_t \phi(c_2) h_t^{-1}$ do not commute for all $t\neq 0$.
    \end{lemma}

    \begin{proof}
        Recall that every non-trivial element $g$ of $\psl$ acts on $ \mathbb{H}^2$: Each $g\in \Hyp$ fixes a unique geodesic $\mathrm{Fix}(g)$ in $\mathbb{H}^2$, each $g\in \Par$ fixes a unique ideal point $\mathrm{Fix}(g)$ on $\partial \mathbb{H}^2$, and each $g\in \Ell$ fixes a unique point $\mathrm{Fix}(g)$ in $\mathbb{H}^2$. 
        Since $\phi(c_1)$ and $\phi(c_2)$ are non-trivial and commute, we have
        $\mathrm{Fix}\big(\phi(c_1)\big)=\mathrm{Fix}\big(\phi(c_2)\big)$.
            As $\psl$ acts transitively on $\mathbb{H}^2$, we can choose $h\in \psl$ 
            such that $h\cdot \mathrm{Fix}\big(\phi(c_2)\big)\neq \mathrm{Fix}\big(\phi(c_1)\big)$. Then we have $\mathrm{Fix}\big(h\phi(c_2)h^{-1}\big) = h\cdot \mathrm{Fix}\big(\phi(c_2)\big)\neq\mathrm{Fix}\big(\phi(c_1)\big)$; and as a consequence 
            $h\phi(c_2)h^{-1}$ and $\phi(c_1)$ do not commute. 
            Let $\{h_t\}_{t\in\mathbb{R}}$ be the one-parameter subgroup of $\psl$ generated by $h$, with $h_0 = \pm\mathrm{I}$.  Since $\mathrm{Fix}(h_t) = \mathrm{Fix}(h)\neq \mathrm{Fix}\big(\phi(c_2)\big)$ for $t\neq 0$, 
            $h_t\cdot\mathrm{Fix}\big(\phi(c_2)\big)\neq \mathrm{Fix}\big(\phi(c_2)\big) = \mathrm{Fix}\big(\phi(c_1)\big)$. As a consequence, 
            $h_t\phi(c_2)h_t^{-1}$ and $\phi(c_1)$ do not commute for all $t\neq 0$.
    \end{proof}

    \begin{lemma}\label{lem_nonabeltorus}
            Let $\Sigma = \Sigma_{g,p}$ with $\chi(\Sigma)\leqslant -2$, and let $\apdecomp$ be the chosen almost-path decomposition of $\Sigma$. For $n\in \mathbb{Z}$ and $s\in \{-1,0,+1\}^p$,
            let $\phi\in \HPns(\Sigma)$ such that $\phi |_{\pi_1(T_j)}$ is abelian for some $j\in \{1,\cdots, g\}$. 
            Then there exists a path 
            $\{\phi_t\}_{t\in [0,1]}$ in 
            $\HPns(\Sigma)$ with  $\phi_0 = \phi$, 
            $\big(\phi_1(a_j), \phi_1(b_j)\big)
            = \bigg(\pm\mathrm{I}, 
        \pm \begin{bmatrix}
            e & 0 \\
            0 & e^{-1}
        \end{bmatrix}\bigg)$ for the generators $a_j, b_j$ of $\pi_1\big(T_j\big)$, 
        and 
        $\phi_1|_{\pi_1(\Sigma \setminus T_j)} = \phi|_{\pi_1(\Sigma \setminus T_j)}$.
    \end{lemma}
        \begin{proof}
        Let $j\in \{1,\cdots, g\}$ such that $\phi |_{\pi_1(T_j)}$ is abelian.
        Then $\phi(a_j)$ and $\phi(b_j)$ lie in a common one-parameter subgroup $S = \{g_t\}_{t\in\mathbb{R}}$ of $\psl$, and hence the pair $\big(\phi(a_j), \phi(b_j)\big)$ lies in $S\times S$. 
        Since $S\times S$ is connected, there exists a path in $S\times S$ connecting $\big(\phi(a_j), \phi(b_j)\big)$ to $(\pm \mathrm{I}, \pm \mathrm{I})$, consisting of commuting pairs of elements of $\psl$.
        The composition of this path and the path 
        $\bigg\{\bigg(\pm\mathrm{I}, 
        \pm \begin{bmatrix}
            e^t & 0 \\
            0 & e^{-t}
        \end{bmatrix}\bigg)\bigg\}\interval$ 
        connects $\big(\phi(a_j), \phi(b_j)\big)$ and 
        $\bigg(\pm\mathrm{I}, 
        \pm \begin{bmatrix}
            e & 0 \\
            0 & e^{-1}
        \end{bmatrix}\bigg)$ in $\psl\times\psl$. Under the identification of $\psl$-representations of $\pi_1(T_j)$ and pairs of elements of $\psl$ and a suitable re-parametrization, this path of pairs corresponds to a path $\{\phi_t\}\interval$ of abelian representations of $\pi_1(T_j)$ starting from $\phi_0 = \phi$. 
        As we have $\phi_t(d'_j) = [\phi_t(a_j), \phi_t(b_j)] = \pm\mathrm{I}$ for all $t\in [0,1]$, we can extend the path $\{\phi_t\}\interval$ to $\pi_1(\Sigma)$ by setting $\phi_t|_{\pi_1(\Sigma \setminus T_j)} = \phi|_{\pi_1(\Sigma \setminus T_j
        )}$. 
        Since $\phi_t(c_i) = \phi(c_i)$ for all $t\in [0,1]$ and $i\in \{1,\cdots, p\}$, 
        the signs $s(\phi_t) = s(\phi)$ and the relative Euler classes $e(\phi_t) = e(\phi)$ for all $t\in [0,1]$. As a consequence, $\{\phi_t\}\interval$ lies in $\HPns(\Sigma)$. 
    \end{proof}

\subsection{Representations with at least one hyperbolic boundary component}\label{4.1}
    In this subsection, we will prove Theorem \ref{thm_nonabel} for representations with the sign $s$ satisfying $p_0(s)\geqslant 1$, which is stated as Proposition \ref{prop_NAhyp} below. 

    \begin{proposition}\label{prop_NAhyp}
        Let $\Sigma = \Sigma_{g,p}$ with $\chi(\Sigma) \leqslant - 2$ and with the chosen almost-path decomposition $\Sigma = \apdecomp$. Let $n\in \mathbb{Z}$, and let $s\in \{-1,0,+1\}^p$ with $p_0(s)\geqslant 1$. Then for every representation $\phi$ in $\mathcal{\HP}^s_n(\Sigma)$, there exists a path 
        in $\HP^s_n(\Sigma)$ connecting $\phi$ to a representation $\psi$
        in $\nonabel$.
    \end{proposition}


    \begin{proof}
        

The outline of the proof is as follows:
      We will first construct a path in $\HPns(\Sigma)$ connecting $\phi$ to a representation $\rho\in \HPns(\Sigma)$, such that for all $j \in \{1, \cdots, g\}$, the restriction 
      $\rho|_{\pi_1(T_j)}$ is non-abelian. 
      Next, we will construct a path in $\HPns(\Sigma)$ connecting $\rho$ to $\psi$ in $\nonabel$, that is,  each restriction 
      $\psi|_{\pi_1(P_i)}$, $i\in \{1,\cdots, g+p-2\}$, and $\psi|_{\pi_1(T_j)}$, $j\in \{1,\cdots, g\}$,
       is non-abelian. Then the composition of these two paths connects $\phi$ to $\psi\in \nonabel$.
\\

     The path connecting $\phi$ to $\rho$ is constructed as follows. 
     If $g = 0$, then there is no one-hole torus in the almost-path decompositon. Hence, $\phi$ automatically satisfies the condition for $\rho$, and we can let $\rho = \phi$. If otherwise that $g\geqslant 1$, then we will show by induction that for all $j\in \{0,\cdots, g\}$, there is a path in $\HPns(\Sigma)$ connecting $\phi$ to a representation $\phi_j\in \HPns(\Sigma)$, such that for all $k\in \{1,\cdots, j\}$, $\phi_j|_{\pi_1(T_k)}$ is non-abelian. Then for all $k\in \{1,\cdots, g\}$, $\phi_g|_{\pi_1(T_k)}$ is non-abelian, and we let $\rho = \phi_g$. For the base case $j = 0$, the statement automatically holds, as we can let $\phi_0 = \phi$ connected by a constant path. 
     Now assume that, for some
     $j\in \{0, \cdots, g-1\}$, there is a path $\{\phi_t\}_{t\in [0,j]}$ connecting $\phi$ to $\phi_j$, where $\phi_j$ satisfies the condition above. 
     We will find a path $\{\phi_t\}_{t\in [j,j+1]}$ connecting $\phi_j$ to a representation $\phi_{j+1}$ such that for all $k\in \{1,\cdots, j+1\}$, $\phi_{j+1}|_{\pi_1(T_k)}$ is non-abelian. 
     
     If $\phi_j|_{\pi_1(T_{j+1})}$ is also non-abelian, 
    then we let $\phi_{j+1} = \phi_{j}$, and let $\{\phi_t\}_{t\in [j,j+1]}$ be the constant path. If $\phi_j|_{\pi_1(T_{j+1})}$ is abelian, then we construct the path $\{\phi_t\}_{t\in [j,j+1]}$ as follows. 
     We first construct a continuous one-parameter family $\{\phi_t\}_{t\in \mathbb{R}}$ of representations containing $\phi_j$ at $t = j$, such that for all $t\neq j$ and $k\in \{1,\cdots, j+1\}$,  $\phi_t|_{\pi_1(T_{k})}$ is non-abelian.
     As  $\phi_j|_{\pi_1(T_{j+1})}$ is abelian, by Lemma \ref{lem_nonabeltorus}, we can assume that
         $\phi_j(a_{j+1}) = \pm\mathrm{I}$, and 
         $\phi_j(b_{j+1})= 
        \pm \begin{bmatrix}
            e & 0 \\
            0 & e^{-1}
        \end{bmatrix}$.
        For each $t\in \mathbb{R}$, 
        we define $\phi_t$ as follows: Let
        $\big(\phi_t(a_{j+1}), \phi_t(b_{j+1})\big)
        \doteq\bigg(\pm 
        \begin{bmatrix}
            1 & t-j \\
            0 & 1
        \end{bmatrix}, \pm \begin{bmatrix}
            e & 0 \\
            0 & e^{-1}
        \end{bmatrix}\bigg)$, 
        and 
        $\big(\phi_t(a_l),  \phi_t(b_l)\big) \doteq \big(\phi_j(a_l),  \phi_j(b_l)\big)$ for each $
        l\in \{1,\cdots, g\}\setminus\{j+1\}$.
        Since $p_0(s)\geqslant 1,$ there is an $m\in\{1,\dots, p\}$ such that $\phi(c_m)$ is hyperbolic; and for each $k\in \{1,\cdots, p\}\setminus\{m\}$, we let $\phi_t(c_k) \doteq \phi_j(c_k)$. Then the values $\{\phi_t(a_j),\phi_t(b_j)\}_{j\in\{1,\dots, g\}}$ and $\{\phi_t(c_k)\}_{k\in\{1,\dots,p\}\setminus\{m\}}$ determine the representation $\phi_t.$ 
        For all $t\in \mathbb{R}$ and $k\in \{1,\cdots, j\}$, we have $\phi_t|_{\pi_1(T_k)} = \phi_j|_{\pi_1(T_k)}$, which is non-abelian. 
        Moreover, for all $t\neq j$, $\phi_t(a_{j+1})$ and $\phi_t(b_{j+1})$ are elements of different types in $\psl$ which do not commute with each other, hence $\phi_t|_{\pi_1(T_{j+1})}$ is non-abelian. Next, we will find a sub-path of $\{\phi_t\}_{t\in \mathbb{R}}$ that is contained in $\HPns(\Sigma)$. 
        Since $\phi_j(c_m)$ lies in the open subset $\Hyp$ of $\psl$, there exists a $t_0>0$ such that for all $t\in [j,j+t_0]$, $\phi_t(c_m)\in \Hyp$. 
            Then for all $t\in [j,j+t_0]$ and for all $k\in \{1,\cdots, p\}$, $\phi_t(c_k)$ and $\phi_j(c_k)$ are 
            both hyperbolic or both parabolic with the same sign. 
            Therefore, the signs $s(\phi_t) = s(\phi_j)$ and the relative Euler classes $e(\phi_t) = e(\phi_j)$ for all $t\in [j,j+t_0]$.
            By a re-parametrization,
            this defines
            the path $\{\phi_t\}_{t\in [j,j+1]}$ in $\HPns(\Sigma)$ starting from $\phi_j$ such that 
            for each $k\in \{1, \cdots, j+1\}$,
            $\phi_{j+1}|_{\pi_1(T_k)}$ is non-abelian. 
\\

    Next, let $m_0\in\{1,\cdots, g+p-2\}$ such that $c_m\in \pi_1(P_{m_0}).$ Then $m_0 = 1$ when $m=1$, $m_0  = m + g - 1$ when
    $2\leqslant m\leqslant p-1$, and $m_0 = g + p -2$ when $m = p$ and $p \neq 1$. We will construct a path connecting $\rho$ to a representation $\mu\in  \HPns(\Sigma)$ such that each restriction $\mu|_{\pi_1(P_i)}$, $i\in \{1,\dots, m_0-1\}$, and $\mu|_{\pi_1(T_j)}$, $j\in \{1,\dots, g\}$, is non-abelian. 
 Recall that for all $j\in \{1,\cdots, g\}$, $\rho|_{\pi_1(T_j)}$ is non-abelian.
     Hence, if $m_0 = 1$, $\rho$ automatically satisfies the condition for $\mu$, and we can let $\mu = \rho$. If otherwise that $m_0\geqslant 2$, then we will show by induction that for all $i\in \{0,\cdots, m_0 - 1\}$, there is a path in $\HPns(\Sigma)$ connecting $\rho$ to a representation $\rho_i\in \HPns(\Sigma)$
     such that each restriction  $\rho_i|_{\pi_1(P_{k})}$, $k\in \{1,\cdots, i\}$, and each $\rho_i|_{\pi_1(T_j)},$ $j \in\{ 1,\cdots, g\}$, is non-abelian. Then $\rho_{m_0 - 1}$ satisfies the condition for $\mu$, and we let $\mu = \rho_{m_0 - 1}$. For the base case $i = 0$, the statement automatically holds, as we can let $\rho_0 = \rho$ connected by a constant path. 
     Now assume that, for some
     $i\in \{0, \cdots, m_0 - 2\}$, there is a path $\{\rho_t\}_{t\in [0,i]}$ connecting $\rho$ to $\rho_i$, where $\rho_i$ satisfies the condition above. 
     We will find a path $\{\rho_t\}_{t\in [i,i+1]}$ connecting $\rho_i$ to a representation $\rho_{i+1}$ such that each 
     $\rho_{i+1}|_{\pi_1(P_k)}$, $k\in \{1,\cdots, i+1\}$,
     and each $\rho_{i+1}|_{\pi_1(T_j)},$ $j \in\{ 1,\cdots, g\}$, is non-abelian. 
     
     If $\rho_i|_{\pi_1(P_{i+1})}$ is also non-abelian, 
    then we let $\rho_{i+1} = \rho_i$, and let $\{\rho_t\}_{t\in [i,i+1]}$ be the constant path. If $\rho_i|_{\pi_1(P_{i+1})}$ is abelian, then we construct the path $\{\rho_t\}_{t\in [i,i+1]}$ as follows. 

     First, we construct a continuous one-parameter family $\{\rho_t:\pi_1(P_{i+1})\to\psl\}_{t\in \mathbb{R}}$ of representations of $\pi_1(P_{i+1})$ that contains $\rho_i|_{\pi_1(P_{i+1})}$ at $t = i$, and with $\rho_t$ non-abelian for all $t\neq i$. 
     For this construction, we will apply Lemma \ref{lem_nonabelpants} to $\rho_i|_{\pi_1(P_{i+1})}$, with the two generators $c'_1, c'_2$ of $\pi_1(P_{i+1})$ chosen as follows: If $i = 0$, then we let $c'_1 = c_1$; and if $i\geqslant 1$, then we let $c_1' = d_i$. Also, if $i\leqslant g-1$, then we let $c'_2 = d'_{i+1}$, and if $i\geqslant g$, then we let $c'_2 = c_{i-g+2}$ (cf. Figure \ref{fig:APD_general}).  
    We claim that $\rho_i(c_1')\neq \pm \mathrm I$ and $\rho_i(c_2')\neq \pm \mathrm I$. Indeed, if $i = 0$, then since $\rho_i\in \HP(\Sigma)$, we have $\rho_i(c_1')=\rho_i(c_1)\in \Hyp\cup \Par$; and if $i\geqslant 1$, then 
    since $d_i$ is the common boundary component of $P_i$ and $P_{i+1}$, and since $\rho_i|_{\pi_1(P_i)}$ is non-abelian, we have
    $\rho_i(c_1')=\rho_i(d_i)\neq \pm \mathrm I$.  Similarly, if $i\leqslant g-1$, then since $\rho_i|_{\pi_1(T_{i+1})}$ is non-abelian, we have
    $\rho_i(c_2')=\rho_i(d'_{i+1}) = \big[\rho_i(a_{i+1}), \rho_i(b_{i+1})\big]\neq \pm \mathrm I$; and if $i \geqslant g$, then since $\rho_i\in \HP(\Sigma)$, we have $\rho_i(c_2')=\rho_i(c_{i-g+2})\in \Hyp\cup \Par$. As $\rho_i|_{\pi_1(P_{i+1})}$ is abelian with $\rho_i(c_1')\neq \pm \mathrm I$ and $\rho_i(c_2')\neq \pm \mathrm I$,
        by Lemma \ref{lem_nonabelpants}, there is a one-parameter subgroup $\{h_t\}_{t\in \mathbb{R}}$ of $\psl$ with $h_{i} = \pm\mathrm{I}$ such that for all $t\neq i$, $\rho_i(c_1')$ and $ h_t\rho_i(c_2')h_t^{-1}$ do not commute. We define
        the representation
        $\rho_t$ 
        by setting $\big(\rho_t(c_1'), \rho_t(c_2')\big)\doteq \big(\rho_i(c_1'), h_t \rho_i(c_2') h_t^{-1}\big)$.
        
        Next, by abuse of notations, we extend  $\{\rho_t:\pi_1(P_{i+1})\to\psl\}_{t\in \mathbb{R}}$ to a family of representations $\{\rho_t:\pi_1(\Sigma)\to\psl\}_{t\in \mathbb{R}}$ such that for all $t\neq i$, each restriction $\rho_t|_{\pi_1(P_{k})}$, $k\in \{1,\cdots, i+1\}$, and $\rho_t|_{\pi_1(T_j)},$ $j \in\{1,\cdots, g\}$, is non-abelian. For each $t\in \mathbb{R}$, we extend $\rho_t$ from $\pi_1(P_{i+1})$ to $\pi_1(\Sigma)$ as follows: 
         Recall that if $i\leqslant g - 1$, then $c_2' = d_{i+1}'$. In this case, we let
        $\big(\rho_t(a_{i+1}), \rho_t(b_{i+1})\big) \doteq \big(h_t\rho_i(a_{i+1})h_t^{-1}, h_t\rho_i(b_{i+1})h_t^{-1}\big)$, 
        $\big(\rho_t(a_j), \rho_t(b_j)\big) \doteq \big(\rho_i(a_j), \rho_i(b_j)\big)$ for each $j\in \{1,\cdots, g\}\setminus \{i+1\}$, and 
        $\rho_t(c_k) \doteq \rho_i(c_k)$ for each $k\in \{1,\cdots, p\}\setminus \{m\}$. 
        Recall also that if $i\geqslant g$, then we have $c'_2 = c_{i-g+2}$. In this case, we let
        $\big(\rho_t(a_j), \rho_t(b_j)\big) \doteq \big(\rho_i(a_j), \rho_i(b_j)\big)$ for each $j\in \{1,\cdots, g\}$, and 
        $\rho_t(c_k) \doteq \rho_i(c_k)$ for each $k\in \{1,\cdots, p\}\setminus \{i-g+2,m\}$.
        Then the values $\{\rho_t(a_j),\rho_t(b_j)\}_{j\in\{1,\dots, g\}}$ and $\{\rho_t(c_k)\}_{k\in\{1,\dots,p\}\setminus\{m\}}$ determine the representation $\rho_t$, whose restriction to $\pi_1(P_{i+1})$ equals $\rho_t : \pi_1(P_{i+1})\to \psl$ defined above. 
        Then for all $t\in \mathbb{R}$ and
        $k\in \{1,\cdots, i\}$, we have $\rho_t|_{\pi_1(P_k)} = \rho_i|_{\pi_1(P_k)}$, which is non-abelian; and hence for all $t\neq i$ and $k\in \{1,\cdots, i+1\}$, $\rho_t|_{\pi_1(P_k)}$ is non-abelian. Moreover, for all $t\in \mathbb{R}$ and $j\in \{1,\cdots, g\}$, we have either $\rho_t|_{\pi_1(T_j)} = \rho_i|_{\pi_1(T_j)}$ or 
        $\rho_t|_{\pi_1(T_j)} = h_t\big(\rho_i|_{\pi_1(T_j)}\big)h_t^{-1}$, neither of which is abelian. 
        
        Finally, we will find a sub-path of $\{\rho_t\}_{t\in \mathbb{R}}$ that is contained in $\HPns(\Sigma)$. 
        Similar to the construction of $\{\phi_t\}_{t\in [j,j+1]}$ above, since $\rho_i(c_m)\in \Hyp$, there exists a $t_0>0$ such that for all $t\in [i,i+t_0]$, $\rho_t(c_m)\in \Hyp$. 
        Then the signs $s(\rho_t) = s(\rho_i)$ and the relative Euler classes $e(\rho_t) = e(\rho_i)$ for all $t\in [i,i+t_0]$.
            By a re-parametrization,
            this defines
        the path $\{\rho_t\}_{t\in [i,i+1]}$ in $\HPns(\Sigma)$ such that        each restriction $\rho_{i+1}|_{\pi_1(P_{k})}$, $k\in \{1,\cdots, i+1\}$, and $\rho_{i+1}|_{\pi_1(T_j)},$ $j \in\{1,\cdots, g\}$, is non-abelian.
\\

    It remains to construct a path connecting $\mu$ to $\psi\in \nonabel$. Recall that each restriction   $\mu|_{\pi_1(P_k)}$, $k\in \{1,\dots, m_0-1\}$, and $\mu|_{\pi_1(T_j)}$, $j\in \{1,\dots, g\}$, is non-abelian. We will show by induction that for all $i\in \{0, \cdots, g+p-1 - m_0\}$, there is a path in $\HPns(\Sigma)$ connecting $\mu$ to a representation $\mu_i\in \HPns(\Sigma)$, 
     such that each restriction  $\mu_i|_{\pi_1(P_k)}$, $k\in \{1,\cdots, m_0-1\}$, $\mu_i|_{\pi_1(P_{g+p-1-l})}$, $l\in \{1,\cdots, i\}$, and $\mu_i|_{\pi_1(T_j)},$ $j \in\{ 1,\cdots, g\}$, is non-abelian. Then $\mu_{g+p-1-m_0}\in \nonabel$, and we let $\psi = \mu_{g+p-1-m_0}$. For the base case $i = 0$, the statement automatically holds, as we can let $\mu_0 = \mu$ connected by a constant path. Now assume that, for some
     $i\in \{0, \cdots, g+p-2 - m_0\}$, there is a path $\{\mu_t\}_{t\in [0,i]}$ connecting $\mu$ to $\mu_i$, where $\mu_i$ satisfies the condition above. 
     We will find a path $\{\mu_t\}_{t\in [i,i+1]}$ connecting $\mu_i$ to a representation $\mu_{i+1}$ such that each restriction 
     $\mu_{i+1}|_{\pi_1(P_k)}$, $k\in \{1,\cdots, m_0-1\}$, $\mu_{i+1}|_{\pi_1(P_{g+p-1-l})}$, $l\in \{1,\cdots, i+1\}$, and $\mu_{i+1}|_{\pi_1(T_j)},$ $j \in\{ 1,\cdots, g\}$, is non-abelian. 
     
     If $\mu_i|_{\pi_1(P_{g+p-2-i})}$ is also non-abelian, 
    then we let $\mu_{i+1} = \mu_i$, and let $\{\mu_t\}_{t\in [i,i+1]}$ be the constant path. 
    
    Now suppose that $\mu_i|_{\pi_1(P_{g+p-2-i})}$ is abelian. If $i\leqslant g+p-3-m_0$, then we construct the path $\{\mu_t\}_{t\in [i,i+1]}$ as follows. To simplify the notation, we let $i^* \doteq g+p-2-i$. Similar to the construction of $\{\rho_t\}_{t\in \mathbb{R}}$ above, we will construct a continuous one-parameter family $\{\mu_t\}_{t\in \mathbb{R}}$ of representations containing $\mu_i$ at $t = i$ such that for all $t\neq i$, each restriction $\mu_t|_{\pi_1(P_k)}$, $k\in \{1,\cdots, m_0-1\}$, $\mu_t|_{\pi_1(P_{g+p-1-l})}$, $l\in \{1,\cdots, i+1\}$, and $\mu_t|_{\pi_1(T_j)},$ $j \in\{ 1,\cdots, g\}$, is non-abelian.
     For this construction, we will first apply Lemma \ref{lem_nonabelpants} to $\mu_i|_{\pi_1(P_{i^*})}$, with the two generators $c'_1, c'_2$ of $\pi_1(P_{i^*})$ chosen as follows: 
     If $i = 0$ and $p = 1$, then we let $c'_1 = d'_g$; if $i = 0$ and $p\geqslant 2$, then we let $c'_1 = c_p$; and if $i\geqslant 1$, then we let $c_1' = d_{i^*}$. Also, if $i^*\leqslant g$, then we let $c'_2 = d'_{i^*}$; and if $i^*\geqslant g+1$, then we  let $c'_2 = c_{i^*-g+1}$ (cf. Figure \ref{fig:APD_general}).
     By a similar argument as above, we can show that $\mu_{i}(c_1')\neq \pm \mathrm I$ and $\mu_{i}(c_2')\neq \pm \mathrm I$. Then
        by Lemma \ref{lem_nonabelpants}, there is a one-parameter subgroup $\{h_t\}_{t\in \mathbb{R}}$ of $\psl$ with $h_{i} = \pm\mathrm{I}$ such that for all $t\neq i$, $\mu_{i}(c_1')$ and $ h_t\mu_{i}(c_2')h_t^{-1}$ do not commute. For each $t\in \mathbb{R}$, we define
        the representation
        $\mu_t$ of $\pi_1(P_{i^*})$ 
        by setting $\big(\mu_t(c_1'), \mu_t(c_2')\big)\doteq \big(\mu_{i}(c_1'), h_t \mu_{i}(c_2') h_t^{-1}\big)$. We extend $\mu_t$ from $\pi_1(P_{i^*})$ to $\pi_1(\Sigma)$ as follows: 
        If $i^*\geqslant g+1$, then $c'_2 = c_{i^*-g+1}$. In this case, we let
        $\big(\mu_t(a_j), \mu_t(b_j)\big) \doteq \big(\mu_i(a_j), \mu_i(b_j)\big)$ for each $j\in \{1,\cdots, g\}$, and 
        $\mu_t(c_k) \doteq \mu_i(c_k)$ for each $k\in \{1,\cdots, p\}\setminus \{i^*-g+1,m\}$.
        If $i^*\leqslant g$, then $c_2' = d_{i^*}'$. In this case, we let
        $\big(\mu_t(a_{i^*}), \mu_t(b_{i^*})\big) \doteq \big(h_t\mu_i(a_{i^*})h_t^{-1}, h_t\mu_i(b_{i^*})h_t^{-1}\big)$, 
        $\big(\mu_t(a_j), \mu_t(b_j)\big) \doteq \big(\mu_i(a_j), \mu_i(b_j)\big)$ for each $j\in \{1,\cdots, g\}\setminus \{i^*\}$, and 
        $\mu_t(c_k) \doteq \mu_i(c_k)$ for each $k\in \{1,\cdots, p\}\setminus \{m\}$. 
        Then the values $\{\mu_t(a_j),\mu_t(b_j)\}_{j\in\{1,\dots, g\}}$ and $\{\mu_t(c_k)\}_{k\in\{1,\dots,p\}\setminus\{m\}}$ determine the representation $\mu_t$, whose restriction to $\pi_1(P_{i^*})$ equals $\mu_t:\pi_1(P_{i^*})\to \psl$ defined above. For each $t\in \mathbb{R}$, we have either $\mu_t|_{\pi_1(T_j)} = \mu_i|_{\pi_1(T_j)}$ or 
        $\mu_t|_{\pi_1(T_j)} = h_t\big(\mu_i|_{\pi_1(T_j)}\big)h_t^{-1}$ for all $j\in \{1,\cdots, g\}$,
        $\mu_t|_{\pi_1(P_k)} = \mu_i|_{\pi_1(P_k)}$ for all $k\in \{1,\cdots, m_0 - 1\}$,
        and if $i\geqslant 1$, 
        $\mu_t|_{\pi_1(P_{g+p-1-l})} = \mu_i|_{\pi_1(P_{g+p-1-l})}$ for all $l\in \{1,\cdots, i\}$. By the assumption of $\mu_i$, all these restrictions are non-abelian. Moreover, for all $t\neq i$, by the definition of $h_t$, $\mu_t|_{\pi_1(P_{i^*})}$ is non-abelian. 
        Finally, we will find a sub-path of $\{\mu_t\}_{t\in \mathbb{R}}$ in $\HPns(\Sigma)$. 
        Since $\mu_i(c_m)\in \Hyp$, there is a $t_0>0$ such that for all $t\in [i,i+t_0]$, $\mu_t(c_m)\in \Hyp$. 
        Then the signs $s(\mu_t) = s(\mu_i)$ and the relative Euler classes $e(\mu_t) = e(\mu_i)$ for all $t\in [i,i+t_0]$.
            By a re-parametrization,
            this defines
        the path $\{\mu_t\}_{t\in [i,i+1]}$ in $\HPns(\Sigma)$ such that        each restriction 
     $\mu_{i+1}|_{\pi_1(P_k)}$, $k\in \{1,\cdots, m_0-1\}$, $\mu_{i+1}|_{\pi_1(P_{g+p-1-l})}$, $l\in \{1,\cdots, i+1\}$, and $\mu_{i+1}|_{\pi_1(T_j)},$ $j \in\{ 1,\cdots, g\}$, is non-abelian. 
        
        If $i = g+p-2-m_0$, then as above, we will construct a continuous family $\{\mu_t\}_{t\in \mathbb{R}}$ of representations containing $\mu_i$ at $t = i$ such that for all $t\neq i$, each restriction $\mu_t|_{\pi_1(P_k)}$, $k\in \{1,\cdots, g+p-2\}$, and $\mu_t|_{\pi_1(T_j)},$ $j \in\{ 1,\cdots, g\}$, is non-abelian.
        To apply Lemma \ref{lem_nonabelpants} to $\mu_i|_{\pi_1(P_{m_0})}$, the two generators $c'_1, c'_2$ of $\pi_1(P_{m_0})$ are chosen as follows: 
     If $m_0 = 1$, then we let $c'_1 = d_1$; and if $m_0 \geqslant  2$, then we let $c'_1 = d_{m_0 - 1}$. Also, if $m_0 \leqslant g$, then we let $c'_2 = d'_{m_0}$; and if $g+1\leqslant m_0\leqslant g+p-3$, then we let $c'_2 = d_{m_0}$. If $m_0 = g+p-2$, then $m = p-1$ or $m = p$. In this case, let $c'_2$ be the one in $\{c_{p-1},c_p\}$ that is not $c_m$. 
     By the same argument as above, we can show that $\mu_{i}(c_1')\neq \pm \mathrm I$ and $\mu_{i}(c_2')\neq \pm \mathrm I$; and
        by Lemma \ref{lem_nonabelpants}, there is a one-parameter subgroup $\{h_t\}_{t\in \mathbb{R}}$ of $\psl$ with $h_{i} = \pm\mathrm{I}$ such that for all $t\neq i$, $\mu_{i}(c_1')$ and $ h_t\mu_{i}(c_2')h_t^{-1}$ do not commute. For each $t\in \mathbb{R}$, we define
        $\mu_t$ of $\pi_1(P_{m_0})$ 
        by setting $\big(\mu_t(c_1'), \mu_t(c_2')\big)\doteq \big(\mu_{i}(c_1'), h_t \mu_{i}(c_2') h_t^{-1}\big)$; and we extend $\mu_t$ from $\pi_1(P_{m_0})$ to $\pi_1(\Sigma)$ as follows: 
        If $m_0 \leqslant g+p-3$, then the decomposition curve $c'_2$ separates $\Sigma$ into two subsurfaces $\Sigma_1$ and $\Sigma_2$ with $P_{m_0}\subset \Sigma_1$. In this case, we let $\mu_t|_{\pi_1(\Sigma_1 \setminus P_{m_0})}\doteq \mu_i|_{\pi_1(\Sigma_1\setminus P_{m_0})}$, and $\mu_t|_{\pi_1(\Sigma_2)}\doteq h_t\big(\mu_i|_{\pi_1(\Sigma_2)}\big)h_t^{-1}$.
        If $m_0 = g+p-2$, then we let $\mu_t|_{\pi_1(\Sigma\setminus P_{m_0})}\doteq \mu_i|_{\pi_1(\Sigma\setminus P_{m_0})}$.
        This
     extends $\mu_t$ from $\pi_1(P_{m_0})$ to 
     $\pi_1(\Sigma)$, defining
     the continuous family $\{\mu_t\}_{t\in \mathbb{R}}$ on $\pi_1(\Sigma)$. For each $t\in \mathbb{R}$, we have either
        $\mu_t|_{\pi_1(P_k)} = \mu_i|_{\pi_1(P_k)}$ or $\mu_t|_{\pi_1(P_k)} = h_t\big(\mu_i|_{\pi_1(P_k)}\big)h_t^{-1}$ for all $k\in \{1,\cdots, g+p-2\}\setminus \{m_0\}$, and $\mu_t|_{\pi_1(T_j)} = \mu_i|_{\pi_1(T_j)}$ or 
        $\mu_t|_{\pi_1(T_j)} = h_t\big(\mu_i|_{\pi_1(T_j)}\big)h_t^{-1}$ for all $j\in \{1,\cdots, g\}$. By the assumption of $\mu_i$, all these restrictions are non-abelian. Moreover, for all $t\neq i$, by the definition of $h_t$, $\mu_t|_{\pi_1(P_{m_0})}$ is non-abelian. 
        Finally, we will find a sub-path of $\{\mu_t\}_{t\in \mathbb{R}}$ in $\HPns(\Sigma)$. 
        Since $\mu_i(c_m)\in \Hyp$, there is a $t_0>0$ such that for all $t\in [i,i+t_0]$, $\mu_t(c_m)\in \Hyp$. 
        Then the signs $s(\mu_t) = s(\mu_i)$ and the relative Euler classes $e(\mu_t) = e(\mu_i)$ for all $t\in [i,i+t_0]$.
            By a re-parametrization,
            this defines
        the path $\{\mu_t\}_{t\in [i,i+1]}$ in $\HPns(\Sigma)$ such that  $\mu_{i+1}\in \nonabel$. Since $\mu_{g+p-1-m_0}=\mu_{i+1}\in \nonabel$, we let $\psi = \mu_{g+p-1-m_0}$. This completes the proof.
         \end{proof}

    \subsection{Type-preserving representations}
   
    \subsubsection{Surfaces with positive genus}\label{4.2.1}
    In this subsection, we will prove Theorem \ref{thm_nonabel} for type-preserving representations of $\pi_1(\Sigma)$ where $\Sigma$ has positive genus, which is stated as Proposition \ref{prop_NAtorus} below. 
    \begin{proposition}\label{prop_NAtorus}
        Let $\Sigma = \Sigma_{g,p}$ with $g\geqslant 1$ and $\chi(\Sigma)\leqslant -2$, and with the chosen almost-path decomposition $\Sigma = \apdecomp$. 
        Let $n\in\mathbb{Z}$, and $s\in \{\pm 1\}^p$.
        Then for every type-preserving representation $\phi \in \Rns$, there exists a path 
        in $\Rns$ connecting $\phi$ to a representation
        $\psi \in \nonabel$.
    \end{proposition}
    \begin{proof}
    The outline of the proof is as follows:
      We will first construct a path in $\Rns$ connecting $\phi$ to a representation $\rho\in \Rns$ such that the restriction 
      $\rho|_{\pi_1(T_1)}$ is non-abelian. 
      Next, we will construct a path in $\Rns$ connecting $\rho$ to a representation $\mu\in \Rns$ such that for all $j \in \{1, \cdots, g\}$, the restriction 
      $\mu|_{\pi_1(T_j)}$ is non-abelian. 
      Finally, we will construct a path in $\Rns$ connecting $\mu$ to $\psi$ in $\nonabel$. Then the composition of these three paths connects $\phi$ to $\psi\in \nonabel$.
      \\

        The path $\{\phi_t\}_{t\in [0,1]}$ connecting $\phi$ to $\rho$ is constructed as follows. Since $g\geqslant 1$, there is a one-hole torus $T_1$ in the almost-path decomposition.
        If $\phi|_{\pi_1(T_1)}$ is non-abelian, then we let $\rho = \phi$ and let $\{\phi_t\}_{t\in [0,1]}$ be the constant path. 
        
        In the case that $\phi|_{\pi_1(T_1)}$  is abelian, for the convenience of the proof, we assume 
        the values of $\phi(a_1), \phi(b_1),$ $\phi(c_1)$ and $\phi(d_1)$ in $\phi\big(\pi_1(P_1\cup T_1)\big)$ as follows. 
        As $\phi|_{\pi_1(T_1)}$ is abelian, we have $\phi(d'_1) = \big[\phi(a_1), \phi(b_1)\big] = \pm\mathrm{I}$, 
        and hence $\phi(d_1)\phi(c_1) = \phi(d'_1)^{-1} = \pm\mathrm{I}$, i.e., $\phi(d_1) = \phi(c_1)^{-1}$. 
         As $\phi$ is type-preserving, $\phi(c_1)$ and $\phi(d_1)$ are parabolic elements of the opposite signs.
        If $\phi(c_1)\in \Parp$, then there exists a $g\in \psl$ such that 
        $g\phi(c_1)g^{-1} = \pm\parpp$. Let $\{g_t\}\interval$ be a path connecting $\pm \mathrm I$ to $g$ in $\psl$. Then the path $\{g_t\phi g_t^{-1}\}\interval$ in $\Rns$ connects $\phi$ to the representation $g\phi g^{-1}$ that maps $c_1$ to $\pm \parpp$ and $d_1$ to $\pm \parmp$. 
        Therefore, by replacing $\phi$ to $g\phi g^{-1}$ if necessary, we can assume that $\phi(c_1) = \pm\parpp$ and $\phi(d_1) = \pm\parmp$. 
        Similarly, if $\phi(c_1)\in \Par^-$, we can assume that $\phi(c_1) = \pm\parmp$ and $\phi(d_1) = \pm\parpp$. 
        Finally, by Lemma \ref{lem_nonabeltorus}, we can assume 
        that
         $\phi(a_1) = \pm\mathrm{I}$, and 
         $\phi(b_1)= 
        \pm \begin{bmatrix}
            e & 0 \\
            0 & e^{-1}
        \end{bmatrix}$. 
        
        We now construct a continuous one-parameter family $\{\phi_t\}_{t\in \mathbb{R}}$ of representations with $\phi_0 = \phi$, such that for all $t\neq 0$, $\phi_t|_{\pi_1(T_1)}$ is non-abelian. 
        For each $t\in \mathbb{R}$, 
        we define $\phi_t$ as follows: Let
        $\big(\phi_t(a_1), \phi_t(b_1)\big)
        \doteq\bigg(\pm 
        \begin{bmatrix}
            1 & t \\
            0 & 1
        \end{bmatrix}, \pm \begin{bmatrix}
            e & 0 \\
            0 & e^{-1}
        \end{bmatrix}\bigg)$, 
        $\big(\phi_t(a_j),  \phi_t(b_j)\big) \doteq \big(\phi(a_j),  \phi(b_j)\big)$ for each $
        j\in \{2,\cdots, g\}$,
        and $\phi_t(c_k) \doteq \phi(c_k)$ for each $k\in \{2,\cdots, p\}$. Then the values $\{\phi_t(a_j),\phi_t(b_j)\}_{j\in\{1,\dots, g\}}$ and $\{\phi_t(c_k)\}_{k\in\{2,\dots,p\}}$ determine the representation $\phi_t.$ 
        For all $t\neq 0$, $\phi_t(a_1)$ and $\phi_t(b_1)$ are elements of different types in $\psl$ which do not commute, hence $\phi_t|_{\pi_1(T_1)}$ is non-abelian. 
        
        Next, we find a sub-path of $\{\phi_t\}_{t\in \mathbb{R}}$ that is contained in $\Rns$. 
        By the definition of $d_1$, we have
        $$d_1 
        = (c_1[a_1, b_1])^{-1}
        = [a_2, b_2]\cdots [a_g, b_g]\cdot c_2\cdots c_p,$$
        hence we have
        \begin{eqnarray*}
            \phi_t(c_1) &=& \big(\big[\phi_t(a_1),\phi_t(b_1)\big]\cdots \big[\phi_t(a_g),\phi_t(b_g)\big]\cdot \phi_t(c_2)\cdots \phi_t(c_p)\big)^{-1}\\
            &=& \big(\big[\phi_t(a_1),\phi_t(b_1)\big]\big[\phi(a_2),\phi(b_2)\big]\cdots \big[\phi(a_g),\phi(b_g)\big]\cdot \phi(c_2)\cdots \phi(c_p)\big)^{-1}
            \\
            &=& \big(\big[\phi_t(a_1),\phi_t(b_1)\big]\phi(d_1)\big)^{-1}\\
            &=& \pm \begin{bmatrix}
                1 & (e^2 - 1)t \pm 1\\
                0 & 1
                \end{bmatrix},
        \end{eqnarray*}
        where the $\pm$ in the $(1,2)$-entry of the matrix is $+$ if $\phi(c_1)\in \Par^+$, and is $-$ if $\phi(c_1)\in \Par^-$. Let $t_0\in \big(0, \frac{1}{e^2-1}\big)$. Then, for all $t\in [0,t_0]$,
        we have respectively $sgn\big((e^2 - 1)t\pm 1\big) = \pm$, and $\phi_t(c_1)$ and $\phi(c_1)$ are parabolic elements of the same sign. 
        Then for all $t\in [0,t_0]$ and for all $k\in \{1,\cdots, p\}$, $\phi_t(c_k)$ and $\phi(c_k)$ are 
     parabolic with the same sign, and hence the signs $s(\phi_t) = s(\phi)$ and the relative Euler classes $e(\phi_t) = e(\phi)$ for all $t\in [0,t_0]$.
        By a re-parametrization,
        this defines
        the path $\{\phi_t\}_{t\in [0,1]}$ in $\Rns$ starting from $\phi$ such that 
        $\phi_1|_{\pi_1(T_1)}$ is non-abelian; and we let $\rho = \phi_1$.
        \\
        
        Next, we construct the path connecting $\rho$ to $\mu$. We will show by induction that for all $j\in \{1,\cdots, g\}$, there is a path in $\Rns$ connecting $\rho$ to a representation $\rho_j\in \Rns$, such that for all $k\in \{1,\cdots, j\}$, $\rho_j|_{\pi_1(T_k)}$ is non-abelian. Then for all $k\in \{1,\cdots, g\}$, $\rho_g|_{\pi_1(T_k)}$ is non-abelian, and we can let  $\mu = \rho_g$. 
        For the base case $j = 1$, since $\rho|_{\pi_1(T_1)}$ is non-abelian, we can let $\rho_1 = \rho$ connected by a constant path. 
     Now assume that, for some
     $j\in \{1, \cdots, g-1\}$, there is a path $\{\rho_t\}_{t\in [0,j]}$ connecting $\rho$ to $\rho_j$, where $\rho_j$ satisfies the condition above. 
     We will find a path $\{\rho_t\}_{t\in [j,j+1]}$ connecting $\rho_j$ to a representation $\rho_{j+1}$ such that for all $k\in \{1,\cdots, j+1\}$, $\rho_{j+1}|_{\pi_1(T_k)}$ is non-abelian. 
     
     If $\rho_j|_{\pi_1(T_{j+1})}$ is also non-abelian, 
    then we let $\rho_{j+1} = \rho_{j}$, and let $\{\rho_t\}_{t\in [j,j+1]}$ be the constant path. 
    
    If $\rho_j|_{\pi_1(T_{j+1})}$ is abelian, then we construct the path $\{\rho_t\}_{t\in [j,j+1]}$ as follows: We first construct a continuous one-parameter family 
     $\{\rho_t:\pi_1(\Sigma\setminus T_1)\to\psl\}_{t\in \mathbb{R}}$ of representations of $\pi_1(\Sigma\setminus T_1)$ that contains $\rho_j|_{\pi_1(\Sigma\setminus T_1)}$ at $t = j$, and with $\rho_t|_{\pi_1(T_k)}$ non-abelian for all $t\neq j$ and $k\in \{2,\cdots, j+1\}$. 
     As  $\rho_j|_{\pi_1(T_{j+1})}$ is abelian, by Lemma \ref{lem_nonabeltorus}, we can assume that
         $\rho_j(a_{j+1}) = \pm\mathrm{I}$, and 
         $\rho_j(b_{j+1})= 
        \pm \begin{bmatrix}
            e & 0 \\
            0 & e^{-1}
        \end{bmatrix}$.
        For each $t\in \mathbb{R}$, 
        we define $\rho_t$ as follows: Let
        $\big(\rho_t(a_{j+1}), \rho_t(b_{j+1})\big)
        \doteq\bigg(\pm 
        \begin{bmatrix}
            1 & t-j \\
            0 & 1
        \end{bmatrix}, \pm \begin{bmatrix}
            e & 0 \\
            0 & e^{-1}
        \end{bmatrix}\bigg)$, 
        $\big(\rho_t(a_l),  \rho_t(b_l)\big) \doteq \big(\rho_j(a_l),  \rho_j(b_l)\big)$ for each $
        l\in \{2,\cdots, g\}\setminus\{j+1\}$,
        and $\rho_t(c_k) \doteq \rho(c_k)$ for each $k\in \{1,\cdots, p\}$. Then the values $\{\rho_t(a_j),\rho_t(b_j)\}_{j\in\{2,\dots, g\}}$ and $\{\rho_t(c_k)\}_{k\in\{1,\dots,p\}}$ determine the representation $\rho_t$ of $\pi_1(\Sigma\setminus T_1)$. For all $t\in \mathbb{R}$ and $k\in \{2,\cdots, j\}$, we have $\rho_t|_{\pi_1(T_k)} = \rho_j|_{\pi_1(T_k)}$, which is non-abelian. Moreover, for all $t\neq j$, $\rho_t(a_{j+1})$ and $\rho_t(b_{j+1})$ are elements of different types in $\psl$ which do not commute, hence $\rho_t|_{\pi_1(T_{j+1})}$ is non-abelian.
        
        Next, by abuse of notations, we will extend the path $\{\rho_t:\pi_1(\Sigma\setminus T_1)\to\psl\}_{t\in [j,j+t_0]}$ for some $t_0>0$ to a path of representations $\{\rho_t:\pi_1(\Sigma)\to\psl\}_{t\in [j,j+t_0]}$ in $\Rns$.
        To this end, we consider the commutator map $R: \psl\times \psl\to \psl$ sending $(\pm A,\pm B)$ to $[\pm A,\pm B]$. 
             Since $\rho_j|_{\pi_1(T_1)}$ is non-abelian, $\rho_j(a_1)$ and $\rho_j(b_1)$ do not commute. 
             Then as the commutator map $
             \SL\times \SL\to \SL$ descends to $R$ under the natural projection $\SL\to \psl$, by Lemma \ref{lem_evTsubm}, the differential $dR$ is surjective at the point $\big(\rho_j(a_1), \rho_j(b_1)\big)$. 
             Therefore, by the inverse function theorem, 
             there exists a $t_0 >0$ and a path of non-commuting $\psl$-pairs $\big\{(\pm A_t,\pm B_t)\big\}_{t\in [j, j+t_0]}$, with $(\pm A_j,\pm B_j) = \big(\rho_j(a_1), \rho_j(b_1)\big)$ and $R(\pm A_t,\pm B_t) = [\pm A_t,\pm B_t] = \rho_t(d'_1)$ for all $t\in [j, j+t_0]$.
            For each $t\in [j,j+t_0]$,
            let $\big(\rho_t(a_1), \rho_t(b_1)\big) \doteq (\pm A_t,\pm B_t)$.
            This extends the path 
            $\{\rho_t:\pi_1(\Sigma\setminus T_1)\to\psl\}_{t\in [j,j+t_0]}$ to $\pi_1(\Sigma)$, and since $\pm A_{j+t_0}$ and $\pm B_{j+t_0}$ do not commute, $\rho_{j+t_0}|_{\pi_1(T_1)}$ is non-abelian. 
            For all $t\in [j,j+t_0]$, since $\rho_t(c_k) = \rho_j(c_k)$ for all $k\in \{1,\cdots, p\}$, 
            the signs $s(\rho_t) = s(\rho_j)$ and the relative Euler classes $e(\rho_t) = e(\rho_j)$.
            By a re-parametrization,
            this defines
            the path $\{\rho_t\}_{t\in [j,j+1]}$ in $\HPns(\Sigma)$ starting from $\rho_j$ such that 
            for each $k\in \{1, \cdots, j+1\}$,
            $\rho_{j+1}|_{\pi_1(T_k)}$ is non-abelian. 
            \\
            
            It remains to construct a path connecting $\mu$ to $\psi\in \nonabel$. We will show by induction that for all $i\in \{0, \cdots, g+p-2\}$, there is a path in $\Rns$ connecting $\mu$ to a representation $\mu_i\in \Rns$, 
     such that each of the restrictions  $\mu_i|_{\pi_1(P_{g+p-1-l})}$, $l\in \{1,\cdots, i\}$, and $\mu_i|_{\pi_1(T_j)},$ $j \in\{ 1,\cdots, g\}$, is non-abelian. Then $\mu_{g+p-2}\in \nonabel$, and we can let $\psi = \mu_{g+p-2}$. For the base case $i = 0$, the statement automatically holds, and we can let $\mu_0 = \mu$ connected by a constant path. Now assume that, for some
     $i\in \{0, \cdots, g+p-3\}$, there is a path $\{\mu_t\}_{t\in [0,i]}$ connecting $\mu$ to $\mu_i$, where $\mu_i$ satisfies the condition above. 
     We will find a path $\{\mu_t\}_{t\in [i,i+1]}$ connecting $\mu_i$ to a representation $\mu_{i+1}$ such that each restriction 
     $\mu_{i+1}|_{\pi_1(P_{g+p-1-l})}$, $l\in \{1,\cdots, i+1\}$, and $\mu_{i+1}|_{\pi_1(T_j)},$ $j \in\{ 1,\cdots, g\}$, is non-abelian. 
     
     If $\mu_i|_{\pi_1(P_{g+p-2-i})}$ is also non-abelian, 
    then we let $\mu_{i+1} = \mu_i$, and let $\{\mu_t\}_{t\in [i,i+1]}$ be the constant path. If $\mu_i|_{\pi_1(P_{g+p-2-i})}$ is abelian, then we construct the path $\{\mu_t\}_{t\in [i,i+1]}$ as follows. To simplify the notation, we let $i^* \doteq g+p-2-i$. We will construct a continuous one-parameter family $\{\mu_t:\pi_1(\Sigma\setminus T_1)\to\psl\}_{t\in \mathbb{R}}$ of representations containing $\mu_i$ at $t = i$ such that for all $t\neq i$, each restriction $\mu_t|_{\pi_1(P_{g+p-1-l})}$, $l\in \{1,\cdots, i+1\}$, and $\mu_t|_{\pi_1(T_j)},$ $j \in\{ 1,\cdots, g\}$, is non-abelian.
     For this construction, we will first apply Lemma \ref{lem_nonabelpants} to $\mu_i|_{\pi_1(P_{i^*})}$, with the two generators $c'_1, c'_2$ of $\pi_1(P_{i^*})$ chosen as follows: 
     If $i = 0$ and $p = 1$, then we let $c'_1 = d'_g$; if $i = 0$ and $p\geqslant 2$, then we let $c'_1 = c_p$; and if $i\geqslant 1$, then we let $c_1' = d_{i^*}$. Also, if $i^*\leqslant g$, then we let $c'_2 = d'_{i^*}$; and if $i^*\geqslant g+1$, then we  let $c'_2 = c_{i^*-g+1}$ (cf. Figure \ref{fig:APD_general}).
     By a similar argument as in the proof of Proposition \ref{prop_NAhyp}, we can show that $\mu_{i}(c_1')\neq \pm \mathrm I$ and $\mu_{i}(c_2')\neq \pm \mathrm I$. Then
        by Lemma \ref{lem_nonabelpants}, there is a one-parameter subgroup $\{h_t\}_{t\in \mathbb{R}}$ of $\psl$ with $h_{i} = \pm\mathrm{I}$ such that for all $t\neq i$, $\mu_{i}(c_1')$ and $ h_t\mu_{i}(c_2')h_t^{-1}$ do not commute. For each $t\in \mathbb{R}$, we define
        the representation
        $\mu_t|_{\pi_1(P_{i^*})}$ 
        by setting $\big(\mu_t(c_1'), \mu_t(c_2')\big)\doteq \big(\mu_{i}(c_1'), h_t \mu_{i}(c_2') h_t^{-1}\big)$. 
        
        We extend $\mu_t|_{\pi_1(P_{i^*})}$ to $\pi_1(\Sigma\setminus T_1)$ as follows: 
        If $i^*\geqslant g+1$, then $c'_2 = c_{i^*-g+1}$. In this case, we let
        $\big(\mu_t(a_j), \mu_t(b_j)\big) \doteq \big(\mu_i(a_j), \mu_i(b_j)\big)$ for each $j\in \{2,\cdots, g\}$, and 
        $\mu_t(c_k) \doteq \mu_i(c_k)$ for each $k\in \{1,\cdots, p\}\setminus \{i^*-g+1\}$.
        If $i^*\leqslant g$, then $c_2' = d_{i^*}'$. In this case, we let
        $\big(\mu_t(a_{i^*}), \mu_t(b_{i^*})\big) \doteq \big(h_t\mu_i(a_{i^*})h_t^{-1}, h_t\mu_i(b_{i^*})h_t^{-1}\big)$, 
        $\big(\mu_t(a_j), \mu_t(b_j)\big) \doteq \big(\mu_i(a_j), \mu_i(b_j)\big)$ for each $j\in \{2,\cdots, g\}\setminus \{i^*\}$, and 
        $\mu_t(c_k) \doteq \mu_i(c_k)$ for each $k\in \{1,\cdots, p\}$. 
        Then the values $\{\mu_t(a_j),\mu_t(b_j)\}_{j\in\{2,\dots, g\}}$ and $\{\mu_t(c_k)\}_{k\in\{1,\dots,p\}}$ determine the representation $\mu_t$, whose restriction to $\pi_1(P_{i^*})$ equals $\mu_t:\pi_1(P_{i^*})\to \psl$ defined above. For each $t\in \mathbb{R}$, we have either $\mu_t|_{\pi_1(T_j)} = \mu_i|_{\pi_1(T_j)}$ or 
        $\mu_t|_{\pi_1(T_j)} = h_t\big(\mu_i|_{\pi_1(T_j)}\big)h_t^{-1}$ for all $j\in \{2,\cdots, g\}$, and $\mu_t|_{\pi_1(P_{g+p-1-l})} = \mu_i|_{\pi_1(P_{g+p-1-l})}$ for all $l\in \{1,\cdots, i\}$. By the assumption of $\mu_i$, all these restrictions are non-abelian. Moreover, for all $t\neq i$, by the definition of $h_t$, $\mu_t|_{\pi_1(P_{i^*})}$ is non-abelian.

        Finally, we will extend the path $\{\mu_t:\pi_1(\Sigma\setminus T_1)\to\psl\}_{t\in [j,j+t_0]}$ for some $t_0>0$ to a path of representations $\{\mu_t:\pi_1(\Sigma)\to\psl\}_{t\in [j,j+t_0]}$ in $\Rns$. As in the construction $\{\rho_t\}_{t\in [j,j+t_0]}$ above,
        since $\mu_i|_{\pi_1(T_1)}$ is non-abelian, Lemma \ref{lem_evTsubm} implies that the differential of the commutator map $R$ at the point $\big(\mu_i(a_1), \mu_i(b_1)\big)$ is surjective; then by the inverse function theorem, there is a $t_0>0$ such that $\mu_t$ is extended from $\pi_1(\Sigma\setminus T_1)$ to $\pi_1(\Sigma)$ for all $t\in [i,i+ t_0]$. For all $t\in [i,i+t_0]$, since $\mu_t(c_k) = \mu_i(c_k)$  for all $k\in \{1,\cdots, p\}$, 
            the signs $s(\mu_t) = s(\mu_i)$ and the relative Euler classes $e(\mu_t) = e(\mu_i)$.
            By a re-parametrization,
            this defines
        the path $\{\mu_t\}_{t\in [i,i+1]}$ in $\Rns$ such that        each restriction 
     $\mu_{i+1}|_{\pi_1(P_{g+p-1-l})}$, $l\in \{1,\cdots, i+1\}$, and $\mu_{i+1}|_{\pi_1(T_j)},$ $j \in\{ 1,\cdots, g\}$, is non-abelian. 
     Since $\mu_{g+p-2}\in \nonabel$, we let $\psi = \mu_{g+p-2}$. This completes the proof.
    \end{proof}
    \subsubsection{Punctured spheres}\label{4.2.2}
    

    It remains to prove Theorem \ref{thm_nonabel} for a type-preserving representation $\phi$ of the fundamental group $\pi_1(\Sigma)$ of a punctured sphere $\Sigma=\Sigma_{0,p}$. In this subsection, we will consider the chosen almost-path decomposition $\Sigma = \displaystyle\bigcup^{p-2}_{i = 1} P_i$, with the decomposition curves  $d_i = c_{i+2}\cdots c_p = (c_1\cdots c_{i+1})^{-1}$ for $i\in \{1,\cdots, p-3\}$.
    We first consider the case that $\phi|_{\pi_1(P_i)}$ is non-abelian for some $i\in \{1,\cdots, p-2\}$ as stated in Proposition \ref{prop_NAsphere1}; and consider the case that $\phi |_{\pi_1(P_i)}$ is abelian for all $i\in \{1,\cdots, p-2\}$ in Proposition \ref{prop_NAsphere2} and subsection \ref{Pf}.

\begin{proposition}\label{prop_NAsphere1}
        Let $\Sigma = \Sigma_{0,p}$ with $p\geqslant 4$.
        For $n\in\mathbb{Z}$ and $s\in \{\pm 1\}^p$, let $\phi \in \Rns$ be a type-preserving representation such that $\phi|_{\pi_1(P_i)}$ is non-abelian for some $i\in \{1,\cdots, p-2\}$.
        Then there exists a path  in $\Rns$ connecting $\phi$ to 
        $\psi\in \nonabel$. 
    \end{proposition}
    \begin{proof}
    We will construct a path 
    $\{\phi_t\}_{t\in [0,1]}$ in $\Rns$ connecting $\phi$ to a representation $\mu\in \Rns$ with 
    $\mu|_{\pi_1(P_1)}$ non-abelian, followed by constructing a path 
 $\{\mu_t\}_{t\in [0,1]}$ in $\Rns$ connecting $\mu$ to a representation $\psi$ in $\nonabel$.  Then the composition of these two paths connects $\phi$ to $\psi\in \nonabel$.
\\
    
            If $\phi|_{\pi_1(P_1)}$ is non-abelian, then we let $\mu = \phi$ and let  $\{\phi_t\}_{t\in [0,1]}$ be the constant path. If $\phi|_{\pi_1(P_1)}$ is abelian, then we construct the path $\{\phi_t\}_{t\in [0,1]}$ as follows: By assumption, we can choose the largest integer $i_0\in\{2,\dots,p-2\}$  such that $\phi|_{\pi_1(P_{i_0})}$ is non-abelian. Let $\Sigma_0 \doteq \displaystyle\bigcup^{i_0 - 1}_{i = 1} P_i$.
            We will first construct a continuous one-parameter family $\{\phi_t:\pi_1(\Sigma_0) \to\psl\}_{t\in \mathbb{R}}$ 
            of representations with $\phi_0 = \phi|_{\pi_1(\Sigma_0)}$ such that  for all $t\neq 0$, $\phi_t|_{\pi_1(P_1)}$ is non-abelian.
     	   As $\phi$ sends $c_1$ and $c_2$ into $\Par$, we have  $\phi(c_1)\neq \pm \mathrm I$ and $\phi(c_2)\neq \pm \mathrm I$. By Lemma \ref{lem_nonabelpants}, 
	   there is a one-parameter subgroup $\{h_t\}_{t\in \mathbb{R}}$ of $\psl$ with $h_0 = \pm\mathrm{I}$ such that for all $t\neq 0$, $h_t\phi(c_1)h_t^{-1}$ and $\phi(c_2)$ do not commute. 
	   For each $t\in \mathbb{R}$, we define the representation $\phi_t$ of $\pi_1(P_1)$ by setting $\big(\phi_t(c_1), \phi_t(c_2)\big)\doteq \big(h_t\phi(c_1)h_t^{-1}, \phi(c_2)\big)$.
	   By letting $\phi_t(c_k) \doteq \phi(c_k)$ for each $k\in \{3,\cdots, i_0\}$, we extend the continuous family $\{\phi_t\}_{t\in \mathbb{R}}$ from $\pi_1(P_1)$ to $\pi_1(\Sigma_0)$.

         Next, we will extend the path $\{\phi_t:\pi_1(\Sigma_0)\to\psl\}_{t\in [0, t_0]}$ for some $t_0>0$ to a path of representations $\{\phi_t:\pi_1(\Sigma_0\cup P_{i_0})\to\psl\}_{t\in  [0, t_0]}$.
         To this end, we first claim that if $i_0 <p-2$, then $\phi(d_{i_0})$ is parabolic. Indeed, by the definition of $i_0$, $\phi|_{\pi_1(P_{i_0 + 1})}$ is abelian; and as elements of  $\pi_1(P_{i_0 + 1})$, $\phi(c_{i_0 + 2})$ and $\phi(d_{i_0})$ commute  and  lie in the same one-parameter subgroup. As $\phi(c_{i_0 + 2})\in \Par$, either $\phi(d_{i_0}) = \pm \mathrm I$ or $\phi(d_{i_0})\in \Par$. Since $d_{i_0}$ is also a boundary component of $P_{i_0}$ and $\phi|_{\pi_1(P_{i_0})}$ is non-abelian, we have $\phi(d_{i_0})\neq \pm\mathrm I$, and hence $\phi(d_{i_0})\in \Par$. 
         
         We now extend the path $\{\phi_t\}_{t\in [0, t_0]}$ to $\pi_1(\Sigma_0\cup P_{i_0})$ using the product map $m: \Par \times \Par\to \psl$ sending $(\pm A,\pm B)$ to $\pm AB$. 
         If $i_0 < p-2$, let $(\pm A_0,\pm B_0) \doteq (\phi(c_{i_0 +1}), \phi(d_{i_0}))\in \Par\times \Par$; and if $i_0 = p-2$, let $(\pm A_0,\pm B_0) \doteq (\phi(c_{p-1}), \phi(c_{p}))\in \Par\times \Par$. Then we have $m(\pm A_0,\pm B_0) = \phi(d_{i_0 - 1})$.
         Since $\phi|_{\pi_1(P_{i_0})}$ is non-abelian, $\pm A_0$ and $\pm B_0$ do not commute. 
         Then, as the product map $
             \SL\times \SL\to \SL$ descends to $m$ under the natural projection $\SL\to \psl$, by Lemma \ref{lem_evPsubm}, the differential $dm$ is surjective at the point $\big(\pm A_0, \pm B_0\big)$. 
             Therefore, by the inverse function theorem, there exists a $t_0 >0$ and a path of non-commuting pairs of parabolic elements $\big\{(\pm A_t,\pm B_t)\big\}_{t\in [0, t_0]}$ containing $\big(\pm A_0, \pm B_0\big)$ at $t = 0$ such that 
             $m(\pm A_t,\pm B_t) = \pm A_t B_t = \phi_t(d_{i_0 - 1})$ for all $t\in [0, t_0]$. 
             For each $t\in [0, t_0]$, we let $\ (\phi_t(c_{i_0 +1}), \phi_t(d_{i_0})) = (\pm A_t,\pm B_t)$ if $i_0<p-2$; and we let $\ (\phi_t(c_{p-1}), \phi_t(c_p)) = (\pm A_t,\pm B_t)$ if $i_0 = p-2$. 
             This extends the path $\{\phi_t\}_{t\in [0, t_0]}$ from $\pi_1(\Sigma_0)$ to $\pi_1(\Sigma_0\cup P_{i_0})$. 
             
             Next, we extend the path $\{\phi_t:\pi_1(\Sigma_0\cup P_{i_0})\to\psl\}_{t\in [0, t_0]}$ to $\pi_1(\Sigma)$. 
             If $i_0 = p-2$, then $\Sigma = \Sigma_0\cup P_{i_0}$, and the path $\{\phi_t\}_{t\in [0, t_0]}$ is already defined on $\pi_1(\Sigma)$. 
             If $i_0 < p-2$, then we extend $\{\phi_t\}_{t\in [0, t_0]}$ from $\pi_1(\Sigma_0\cup P_{i_0})$ to $\pi_1(\Sigma)$ as follows: Since the path $\{\phi_t(d_{i_0})\}_{t\in[0,t_0]}$ lies in a connected component $\Par^+$ or $\Par^-$ of $\Par$, for each $t\in [0,t_0]$, 
             $\phi_t(d_{i_0})$ and $\phi(d_{i_0})$ are conjugate. Then by Lemma \ref{lem_conjugacypath_const}, there exists a path $\{g_t\}_{t\in [0,t_0]}$ with $g_0 = \pm \mathrm{I}$ and such that $\phi_t(d_{i_0}) = g_t \phi(d_{i_0}) g_t^{-1}$. For $t\in [0,t_0]$, we extend the representation $\phi_t$ by defining $\phi_t|_{\pi_1(\Sigma - (\Sigma_0\cup P_{i_0}))}\doteq 
             g_t \big(\phi|_{\pi_1(\Sigma - (\Sigma_0\cup P_{i_0}))} \big)g_t^{-1}$. Then for all $t\in [0, t_0]$, we have: $\phi_t(c_1)=h_t\phi(c_1)h_t^{-1}$ which is conjugate to $\phi(c_1)$; $\phi_t(c_k)=\phi(c_k)$ for all $k\in \{2,\cdots, i_0\}$; $\phi_t(c_{i_0+1})$ is conjugate to $\phi(c_{i_0+1})$ as the path $\{\phi_t(c_{i_0+1})\}_{t\in[0,t_0]}$ lies in a connected component $\Parp$ or $\Parm$ which itself is a conjugacy class; and $\phi_t(c_k)=g_t\phi(c_k)g_t^{-1}$ for all $k\in\{i_0+2,\dots,p\}$, which is conjugate to $\phi(c_k)$.  Therefore the signs $s(\phi_t) = s(\phi)$ and the relative Euler classes $e(\phi_t) = e(\phi)$ for all $t\in [0,t_0]$.
            By a re-parametrization and letting $\mu = \phi_1$,
            this defines
            the path $\{\phi_t\}_{t\in [0,1]}$ in $\Rns$ connecting $\phi$ and $\mu$ with 
            $\mu|_{\pi_1(P_1)}$  non-abelian.
            \\
            
            It remains to construct a path connecting $\mu$ to $\psi\in \nonabel$. We will show by induction that for all $i\in \{0, \cdots, p-3\}$, there is a path in $\Rns$ connecting $\mu$ to a representation $\mu_i\in \Rns$
         such that $\mu_i|_{\pi_1(P_1)}$ and $\mu_i|_{\pi_1(P_{p-1-l})}$ for all $l\in \{1,\cdots, i\}$ are non-abelian. Then $\mu_{p-3}\in \nonabel$, and we can let $\psi = \mu_{p-3}$. For the base case $i = 0$, the statement automatically holds, and we let $\mu_0 = \mu$ connected by a constant path. Now assume that, for some
         $i\in \{0, \cdots, p-4\}$, there is a path $\{\mu_t\}_{t\in [0,i]}$ connecting $\mu$ to $\mu_i$, where $\mu_i$ satisfies the condition above. 
         We will find a path $\{\mu_t\}_{t\in [i,i+1]}$ connecting $\mu_i$ to a representation $\mu_{i+1}$ such that $\mu_{i+1}|_{\pi_1(P_1)}$ and $\mu_{i+1}|_{\pi_1(P_{p-1-l})}$ for all $l\in \{1,\cdots, i+1\}$ are non-abelian. 

     First, we will construct a continuous one-parameter family $\{\mu_t:\pi_1(\Sigma\setminus P_1) \to\psl\}_{t\in \mathbb{R}}$ 
            of representations containing $\mu_i|_{\pi_1(\Sigma\setminus P_1)}$ at $t = i$ such that for all $t\neq i$, $\mu_t|_{\pi_1(P_{p-2-i})}$ is non-abelian. 
            To simplify the notation, we let $i^* \doteq p-2-i$. 
         For this construction, we will apply Lemma \ref{lem_nonabelpants} to $\mu_i|_{\pi_1(P_{i^*})}$, with the two generators $c'_1, c'_2$ of $\pi_1(P_{i^*})$ chosen as follows: 
     If $i = 0$, then we let $(c'_1, c'_2) = (c_p, c_{p-1})$; and if $i\geqslant 1$, then we let $(c'_1, c'_2) = (d_{i^*}, c_{i^*+1})$. 
     Then for $i = 0$, as $\mu_i$ is type-preserving, $\mu_i(c'_1), \mu_i(c'_2)\in \Par$; and for $i\geqslant 1$, since $d_{i^*}$ is the common boundary component of $P_{i^*}$ and $P_{i^*+1}$, and since $\mu_i|_{\pi_1(P_{i^*+1})}$ is non-abelian, we have
    $\mu_i(c_1')=\mu_i(d_{i^*})\neq \pm \mathrm I$, and $\mu_i(c'_2)\in \Par$ as $\mu_i$ is type-preserving.
    Hence by Lemma \ref{lem_nonabelpants}, 
	there is a one-parameter subgroup $\{h_t\}_{t\in \mathbb{R}}$ of $\psl$ with $h_i = \pm\mathrm{I}$ such that for all $t\neq i$, $\mu_i(c'_1)$ and $ h_t\mu_i(c'_2)h_t^{-1}$ do not commute. For each $t\in \mathbb{R}$, we define the representation $\mu_t$ of $\pi_1(P_{i^*})$ by setting $\big(\mu_t(c'_1), \mu_t(c'_2)\big)\doteq \big(\mu(c'_1), h_t \mu(c'_2) h_t^{-1}\big)$.
	   By letting $\mu_t(c_k) \doteq \mu(c_k)$ for each $k\in \{3,\cdots, p\}\setminus \{i^*+1\}$, we extend the continuous family $\{\mu_t\}_{t\in \mathbb{R}}$ from $\pi_1(P_{i^*})$ to $\pi_1(\Sigma\setminus P_1)$. Then $\mu_t|_{\pi_1(P_{p-1-l})} = 
       \mu_i|_{\pi_1(P_{p-1-l})}$ for $l\in \{1,\cdots, i\}$, hence is non-abelian by the induction hypothesis on $\mu_i$.
       Moreover, for all $t\neq i$, by the definition of $h_t$, $\mu_t|_{\pi_1(P_{i^*})}$ is non-abelian. 
       
       Next, by abuse of notations, we extend the path $\{\mu_t:\pi_1(\Sigma\setminus P_1)\to\psl\}_{t\in [i, i+t_0]}$ for some $t_0>0$ to a path of representations $\{\mu_t:\pi_1(\Sigma)\to\psl\}_{t\in  [i, i+t_0]}$,
          using the product map $m: \Par \times \Par\to \psl$. 
         Let $(\pm A_i,\pm B_i) \doteq \big(\mu_i(c_1), \mu_i(c_2)\big)\in \Par\times \Par$, with $m(\pm A_i,\pm B_i) = \mu_i(d_1)$.
         Since $\mu_i|_{\pi_1(P_1)}$ is non-abelian, Lemma \ref{lem_evPsubm} implies that the differential of the product map $m$ at the point $\big(\mu_i(c_1), \mu_i(c_2)\big)$ is surjective; then by the inverse function theorem, there is a $t_0 >0$ and a path of non-commuting pairs of parabolic elements $\big\{(\pm A_t,\pm B_t)\big\}_{t\in [i, i+t_0]}$ containing $\big(\pm A_i, \pm B_i\big)$ at $t = i$, such that 
             $m(\pm A_t,\pm B_t) = \pm A_t B_t = \mu_t(d_1)$ for all $t\in [i, i+t_0]$. Moreover, each of the paths 
             $\{\pm A_t\}_{t\in [0,t_0]}$ and 
             $\{\pm B_t\}_{t\in [0,t_0]}$ lies in a connected component of $\Par$, i.e., either $\Par^+$ or $\Par^-$, hence for all $t\in [i,i+t_0]$, $\pm A_t$ is conjugate to $\pm A_0$ and $\pm B_t$ is conjugate to $\pm B_0$.
             For each $t\in [0, t_0]$, we let $\ (\mu_t(c_1), \mu_t(c_2)) = (\pm A_t,\pm B_t)$, which extends $\{\mu_t\}_{t\in [0, t_0]}$ from $\pi_1(\Sigma\setminus P_1)$ to $\pi_1(\Sigma)$. 
             Since $\pm A_{t_0}$ and $\pm B_{t_0}$ do not commute, $\mu_{t_0}|_{\pi_1(P_1)}$ is non-abelian; and since
             $\mu_t(c_k)$ and $\mu_i(c_k)$ are conjugate for all $t\in [i,i+t_0]$ and for all $k\in \{1,\cdots, p\}$, the signs $s(\mu_t) = s(\mu_i)$ and the relative Euler classes $e(\mu_t) = e(\mu_i)$ for all $t\in [i,i+t_0]$.
             By a re-parametrization,
            this defines
            the path $\{\mu_t\}_{t\in [i,i+1]}$ in $\Rns$ starting from $\mu$ such that 
            $\mu_{i+1}|_{\pi_1(P_1)}$ and $\mu_{i+1}|_{\pi_1(P_{p-1-l})}$ for all $l\in \{1,\cdots, i+1\}$ are non-abelian. 
            Since $\mu_{p-3}\in \nonabel$, we let $\psi = \mu_{p-3}$. This completes the proof.
    \end{proof}

 The case that $\phi |_{\pi_1(P_i)}$ is abelian for all $i\in \{1,\cdots, p-2\}$ is considered in following proposition.
 
     \begin{proposition}\label{prop_NAsphere2}
        Let $\Sigma = \Sigma_{0,p}$ with $p\geqslant 4$.
        For $s\in \{\pm 1\}^p$ with $p_+(s)\geqslant 2$ and $p_-(s)\geqslant 2$, let $\phi\in \mathcal{R}^s_0(\Sigma)$ 
        such that $\phi |_{\pi_1(P_i)}$ is abelian for all $i\in \{1,\cdots, p-2\}$.
        Then there exists a path  in $\mathcal{R}^s_0(\Sigma)$ connecting $\phi$ to a representation
        $\psi\in \mathrm{NA}^s_0(\Sigma)$. 
    \end{proposition}
    
    To prove Proposition \ref{prop_NAsphere2}, we need the following Lemma \ref{lem_abelsphere} and Lemma \ref{lem_claimforNAsphere2}.

\begin{lemma}\label{lem_abelsphere}
        Let $\Sigma = \Sigma_{0,p}$ with $p\geqslant 3$.
        Let $\phi$ be a type-preserving representation with sign $s\in \{\pm 1\}^p$ such that for all $i\in \{1,\cdots, p-2\}$, $\phi |_{\pi_1(P_i)}$ is abelian. Then there exists a path $\{\phi_t\}\interval$ in $\mathcal{R}^s_0(\Sigma)$ 
        connecting $\phi$ to an abelian representation $\phi_1$ such that the image $\phi_1(\fund)$ lies in the Borel subgroup of $\psl$. Consequently, $p_+(s)\geqslant 1$ and $p_-(s)\geqslant 1$. 
    \end{lemma}
    \begin{proof}
        Since $\phi|_{\pi_1(P_i)}$ is abelian for all $i\in \{1,\cdots, p-2\}$, by
        Proposition \ref{prop_reducible} and 
        Proposition \ref{prop_additivity}, 
        $\phi\in \mathcal{R}^s_0(\Sigma)$.  
        For each $i\in \{1, \cdots, p-2\}$, let $\Sigma_i\doteq \displaystyle\bigcup_{k = 1}^{i} P_k$.
        We will show by induction that there is a path $\{\phi_t: \pi_1(\Sigma_i)\to \psl\}\interval$ connecting $\phi|_{\pi_1(\Sigma_i)}$ to $\phi_1$ whose image $\phi_1\big(\pi_1(\Sigma_i)\big)$ lies in the Borel subgroup of $\psl$, along with a path $\{g_{k,t}\}\interval$ in $\psl$ for each $k\in \{1,\cdots, i\}$, such that $\phi_t|_{\pi_1(P_k)} = g_{k,t}\big(\phi|_{\pi_1(P_k)}\big)g_{k,t}^{-1}$ for all $t\in [0,1]$. 
        Then as $\Sigma = \Sigma_{p-2}$, this defines the desired path $\path$ in $\mathcal{R}^s_0(\Sigma)$.
        Finally, using $s = s(\phi_1)$, we will conclude that $p_+(s)\geqslant 1$ and $p_-(s)\geqslant 1$.
        \\
        
        As the base case of the induction, the path $\{\phi_t: \pi_1(P_1)\to \psl\}\interval$ is defined as follows. 
        Since $\phi(c_2)\in \phi\big(\pi_1(P_1)\big)$ is parabolic, there is a $g_1\in \psl$ such that $g_1\phi(c_2)g_1^{-1}$ is either $\pm \parpp$ or $\pm \parmp$. 
        Letting $\{g_{1,t}\}\interval$ be a path in $\psl$ connecting $\pm\mathrm{I}$ to $g_1$, 
         the path $\path$ defined by $\phi_t\doteq g_{1,t}\big(\phi|_{\pi_1(P_1)}\big) g_{1,t}^{-1}$ connects $\phi|_{\pi_1(P_1)}$ to $\phi_1 = g_1\big(\phi|_{\pi_1(P_1)}\big)g_1^{-1}$.
        As a $\psl$-conjugation of an abelian representation, $\phi_1$ is abelian on $\pi_1(P_1)$, and its image 
        lies in a one-parameter subgroup of $\psl$ containing $g_1\phi(c_2)g_1^{-1}$, which is $\bigg\{\pm \part\bigg\}_{t\in \mathbb{R}}$.

         Now assume that, for $i\in \{1,\cdots, p-3\}$, 
         there exists a path $\{\phi_t: \pi_1(\Sigma_i)\to \psl\}\interval$ with paths $\{g_{k,t}\}\interval$ in $\psl$ for all $k\in \{1,\cdots, i\}$, satisfying the induction hypothesis. 
         We find the path $\{g_{i+1, t}\}\interval$ in $\psl$ and extend $\path$ from $\pi_1(\Sigma_i)$  to $\pi_1(\Sigma_{i+1})$ as follows. 
         Since $\phi|_{\pi_1(P_i)}$ and $\phi|_{\pi_1(P_{i+1})}$ are abelian, the images $\phi\big(\pi_1(P_i)\big)$ and $\phi\big(\pi_1(P_{i+1})\big)$ respectively lie in one-parameter subgroups $S_i$ and $S_{i+1}$ of $\psl$, where either $S_i =S_{i+1}$ or $S_i\cap S_{i+1} = \{\pm \mathrm I\}$. 
         If $S_i = S_{i+1}$, then for each $t\in [0,1]$, we let $g_{i+1,t}\doteq g_{i,t}$ and let $\phi_t|_{\pi_1(P_{i+1})}\doteq g_{i+1,t}\big(\phi|_{\pi_1(P_{i+1})}\big)g_{i+1,t}^{-1}$, and the image $\phi_1\big(\pi_1(P_{i+1})\big)$ lies in $g_{i+1,1}S_{i+1}g_{i+1,1}^{-1}
         = g_{i,1}S_ig_{i,1}^{-1} = \bigg\{\pm \part\bigg\}_{t\in \mathbb{R}}$. 
         If $S_i\cap S_{i+1} = \{\pm \mathrm I\}$, then similar to the base case, 
         let $g_{i+1}\in \psl$ be such that $g_{i+1}\phi(c_{i+2})g_{i+1}^{-1}$ is either $\pm \parpp$ or $\pm \parmp$. 
         We let $\{g_{i+1,t}\}\interval$ be a path in $\psl$ connecting $\pm\mathrm{I}$ to $g_{i+1}$, 
         and for each $t\in [0,1]$, let $\phi_t|_{\pi_1(P_{i+1})}\doteq g_{i+1,t}\big(\phi|_{\pi_1(P_{i+1})}\big)g_{i+1,t}^{-1}$. 
         As $\phi(d_i)\in \phi\big(\pi_1(P_i)\big)\cap  \phi\big(\pi_1(P_{i+1})\big)\subset S_i\cap S_{i+1} = \{\pm \mathrm I\}$, we have $\phi(d_i) = \pm \mathrm I$. 
         Then for all $t\in [0,1]$,
         $g_{i,t} \phi(d_i)g_{i,t}^{-1} = g_{i+1,t} \phi(d_i)g_{i+1,t}^{-1} = \pm \mathrm I$, 
       hence $\phi_t|_{\pi_1(\Sigma_i)}$ and $\phi_t|_{\pi_1(P_{i+1})}$ agree at $d_i$. 
       Moreover, 
       $\phi_t|_{\pi_1(P_{i+1})}= g_{i+1,t}\big(\phi|_{\pi_1(P_{i+1})}\big)g_{i+1,t}^{-1}$ is abelian on $\pi_1(P_{i+1})$, and its image 
        lies in $\bigg\{\pm \part\bigg\}_{t\in \mathbb{R}}$, as it is a one-parameter subgroup of $\psl$ containing $g_{i+1}\phi(c_{i+2})g_{i+1}^{-1}$.
       This extends the path $\path$ from $\pi_1(\Sigma_i)$ to $\pi_1(\Sigma_{i+1})$, with the image $\phi_1\big(\pi_1(\Sigma_{i+1})\big)$ lying in the Borel subgroup of $\psl$. 
       In particular, when $i = p-3$, this defines the path $\path$ on $\pi_1(\Sigma) = \pi_1(\Sigma_{p-2})$; 
       and since $\phi_t|_{\pi_1(P_k)} = g_{k,t} \big(\phi|_{\pi_1(P_k)}\big)g_{k,t}^{-1}$ for all $k\in \{1,\cdots, p-2\}$, $\phi_t(c_i)$ is conjugate to $\phi(c_i)$ for all $i\in \{1,\cdots, p\}$ and  $t\in [0,1]$. 
       Therefore, the signs $s(\phi_t) = s(\phi)$ and the relative Euler classes $e(\phi_t) = e(\phi)$, and $\path\subset \mathcal{R}^s_0(\Sigma)$.
         \\

         Finally, we prove $p_+(s)\geqslant 1$ and $p_-(s)\geqslant 1$.
         For each $i\in \{1,\cdots, p\}$, we have  $\phi_1(c_i)
    = \pm \begin{bmatrix}
        1 & t_i \\
        0 & 1
    \end{bmatrix}$ for some $t_i\neq 0;$ and since $\phi_1(c_1)\cdots \phi_1(c_p) = \pm\mathrm{I}$, we have
    $\sum_{i=1}^p t_i= 0$. Therefore, there must be an $i_1\in \{1,\cdots, p\}$ with $t_{i_1}>0$ and an $i_2\in \{1,\cdots, p\}$ with $t_{i_2}<0$, i.e., $\phi(c_{i_1})\in \Par^+$ and $\phi(c_{i_2})\in \Par^-$.
    \end{proof}

    \begin{lemma}\label{lem_claimforNAsphere2}
        Let $\Sigma = \Sigma_{0,p}$ with $p\geqslant 4$.
        For $s\in \{\pm 1\}^p$ with $p_+(s)\geqslant 2$ and $p_-(s)\geqslant 2$, let $\phi\in \mathcal{R}^s_0(\Sigma)$ 
        such that $\phi |_{\pi_1(P_i)}$ is abelian for all $i\in \{1,\cdots, p-2\}$. 
        Then one of the following two cases holds:
      \begin{enumerate}[(1)]
      \item $\phi|_{\pi_1(\Sigma\setminus P_1)}\in \mathcal{R}^{s'}_0(\Sigma\setminus P_1)$ for some $s'\in \{\pm 1\}^{p-1}$ with $p_+(s')\geqslant 2$ and $p_-(s')\geqslant 2.$
        
       \item  $\phi|_{\pi_1(P_1\cup P_2)}\in \mathcal{R}^{s''}_0(P_1\cup P_2)$ for some $s''\in \{\pm 1\}^4$ with  $p_+(s'')=p_-(s'') = 2.$
       \end{enumerate}
    \end{lemma}
    \begin{proof}
 Since $\phi|_{\pi_1(P_2)}$ is abelian, the elements $\phi(c_3)$ and $\phi(d_2)$ of $\phi\big(\pi_1(P_2)\big)$ commute  and hence lie in the same one-parameter subgroup; and since $\phi(c_3)\in \Par$, either $\phi(d_2) = \pm \mathrm I$ or $\phi(d_2)\in \Par$. Below we will consider these two cases separately.
        \bigskip

        If $\phi(d_2) = \pm \mathrm I$, then we will show that  (1) holds. 
        The fundamental group $\pi_1(\Sigma\setminus P_1)$ has the preferred peripheral elements $c_3,\cdots, c_p,$ and $d_1^{-1}$.
        As $\phi$ is type-preserving, $\phi(c_3),\cdots, \phi(c_p)\in \Par$; and as $\phi(d_1^{-1})\phi(c_3) = \phi(d_2)^{-1} = \pm \mathrm I$, 
        $\phi(d_1^{-1}) = \phi(c_3)^{-1}\in \Par$. Therefore, $\phi|_{\pi_1(\Sigma\setminus P_1)}$ is type-preserving. Moreover, since $\phi|_{\pi_1(P_i)}$ is abelian for $i\in \{2,\cdots, p-2\}$, by Proposition \ref{prop_reducible} and Proposition \ref{prop_additivity}, 
        $e\big(\phi|_{\pi_1(\Sigma\setminus P_1)}\big) = 0$.
        For the sign $s' = s\big(\phi|_{\pi_1(\Sigma\setminus P_1)}\big)$, up to a permutation of the peripheral elements, we have 
        $s'_1 = +1$ if $\phi(d_1^{-1})\in \Parp$, $s'_1 = -1$ if $\phi(d_1^{-1})\in \Parm$; and for $i\in \{2,\cdots, p-1\}$, $s'_i = s_{i+1}$ determined by the sign of $\phi(c_{i+1})$.
        As $\phi(c_3)$ and $\phi(d_1^{-1}) = \phi(c_3)^{-1}$ have opposite signs, $s'_1$ and $s'_2$ have opposite signs. Therefore, to show $p_+(s')\geqslant 2$ and $p_-(s')\geqslant 2$, it suffices to show that the $(p-2)$-tuple $(s'_3,\cdots, s'_{p-1})$ contains both $+1$ and $-1$. To this end, we will apply Lemma \ref{lem_abelsphere} to the restriction $\phi|_{\pi_1(\Sigma\setminus (P_1\cup P_2\cup P_3))}$. 
        The fundamental group $\pi_1\big(\Sigma\setminus (P_1\cup P_2\cup P_3)\big)$ has the preferred peripheral elements $c_5,\cdots, c_p,$ and $d_3^{-1}$; and since $\phi(c_4)\phi(d_3) = \phi(d_2) = \pm \mathrm I$,
        $\phi(d_3^{-1}) = \phi(c_4)\in \Par$. Therefore, the restriction $\phi|_{\pi_1(\Sigma\setminus (P_1\cup P_2\cup P_3))}$ is type-preserving, and its sign is
        determined by the signs of $\phi(c_5),\cdots, \phi(c_p),$ and $\phi(d_3^{-1}) = \phi(c_4)$.
        Since $\phi|_{\pi_1(P_i)}$ is abelian for $i\in \{4,\cdots, p-2\}$, by Lemma \ref{lem_abelsphere}, $p_+\big(s\big(\phi|_{\pi_1(\Sigma\setminus (P_1\cup P_2\cup P_3))}\big)\big)\geqslant 1$ and $p_-\big(s\big(\phi|_{\pi_1(\Sigma\setminus (P_1\cup P_2\cup P_3))}\big)\big)\geqslant 1$, i.e., at least one of the parabolic elements $\phi(c_4),\cdots, \phi(c_p)$ is positive and at least one of them is negative. Consequently, the tuple $(s'_3,\cdots, s'_{p-1})$ contains both $+1$ and $-1$, and hence $p_+(s')\geqslant 2$ and $p_-(s')\geqslant 2$. 
        \\

     The rest of the proof consideres the case that $\phi(d_2)\in \Par$. As the fundamental group $\pi_1(P_1\cup P_2)$ has the preferred peripheral elements $c_1, c_2, c_3,$ and $d_2$, the restriction $\phi|_{\pi_1(P_1\cup P_2)}$ is type-preserving; and since $\phi|_{\pi_1(P_1)}$ and $\phi|_{\pi_1(P_2)}$ are abelian, we can show as above that
        $e\big(\phi|_{\pi_1(P_1\cup P_2)}\big) = 0$. For the sign $s'' = s\big(\phi|_{\pi_1(P_1\cup P_2)}\big)$, up to a permutation of the peripheral elements, we have 
        $s''_i = s_i$ for $i\in \{1,2,3\}$, given by the sign of $\phi(c_i)$; and $s''_4 = +1$ if $\phi(d_2)\in \Par^+$, and $s''_4 = -1$ if $\phi(d_2)\in \Par^-$.
        By Lemma \ref{lem_abelsphere}, $p_+(s'')\geqslant 1$ and $p_-(s'')\geqslant 1$, hence $p_+(s'')$ equals either $1,$ $2$ or $3.$ 
        
        If $p_+(s'') = 2$, then (2) holds. In the case that $p_+(s'') = 1$ or $p_+(s'') = 3$, we will prove that (1) holds. We will first prove this for the case $p_+(s'') = 3$. Then by Proposition \ref{prop_pglpsl}, the same result holds for the case $p_+(s'') = 1$.
\medskip

         Assuming $p_+(s'') = 3$ and $p_-(s'') = 1$, we first prove that $\phi|_{\pi_1(\Sigma\setminus P_1)}\in \mathcal{R}^{s'}_0(\Sigma\setminus P_1)$ for some $s'\in \{\pm 1\}^{p-1}$. 
        Since $\phi|_{\pi_1(P_1)}$ is abelian, the elements $\phi(c_1)$ and $\phi(d_1)$ of $\phi\big(\pi_1(P_1)\big)$ commute and lie in the same one-parameter subgroup; and as $\phi(c_1)\in \Par$, either $\phi(d_1) = \pm \mathrm I$ or $\phi(d_1)\in \Par$. If $\phi(d_1) = \pm \mathrm{I}$, then as $\phi(c_1)\phi(c_2) = \phi(d_1)^{-1} =\pm \mathrm I$, $\phi(c_1)$ and $\phi(c_2) = \phi(c_1)^{-1}$ have opposite signs; and as $\phi(c_3)\phi(d_2) = \phi(d_1) = \pm \mathrm I$, $\phi(c_3)$ and $\phi(d_2) = \phi(c_3)^{-1}$ have opposite signs. This implies that $p_+(s'') = 2$, contradicting the assumption. Therefore, $\phi(d_1)\in \Par$ and $\phi|_{\pi_1(\Sigma\setminus P_1)}$ is type-preserving; and by Proposition \ref{prop_reducible} and Proposition \ref{prop_additivity}, $\phi|_{\pi_1(\Sigma\setminus P_1)}\in \mathcal{R}^{s'}_0(\Sigma\setminus P_1)$ for some $s'\in \{\pm 1\}^{p-1}$.
        Up to a permutation of the peripheral elements, 
        $s'_1 = +1$ if $\phi(d_1^{-1})\in \Parp$, $s'_1 = -1$ if $\phi(d_1^{-1})\in \Parm$; and for $i\in \{2,\cdots, p-1\}$, $s'_i = s_{i+1}$ determined by the sign of $\phi(c_{i+1})$. 
        \medskip
        
        Next, we will prove $p_+(s')\geqslant 2$ and $p_-(s')\geqslant 2$ separately for the cases  $\phi(d_2)\in \Parp$ and $\phi(d_2)\in \Parm$. 
\medskip

        In the case that  $\phi(d_2)\in \Parp$, we will first show $s'_1 = -s'_2$, then show that the tuple $(s'_3,\cdots, s'_{p-1})$ contains both $+1$ and $-1$ in its components, which together imply $p_+(s')\geqslant 2$ and $p_-(s')\geqslant 2$.

        To show that $s'_1 = -s'_2$, we prove that the parabolic elements $\phi(d_1^{-1})$ and $\phi(c_3)$ have opposite signs. As $\phi|_{\pi_1(P_2)}$ is abelian, 
        $\phi(d_1^{-1})$ and $\phi(c_3)$ lie in a common parabolic one-parameter subgroup, hence there are 
        $\xi_1, \xi_2\in \mathbb{R}\setminus \{0\}$ such that, up to a $\psl$-conjugation, $\phi(d_1^{-1}) = 
        \pm \begin{bmatrix}
            1 & \xi_1 \\
            0 & 1
        \end{bmatrix}
        $
        and $\phi(c_3) = 
        \pm \begin{bmatrix}
            1 & \xi_2 \\
            0 & 1
        \end{bmatrix} 
        $. Since
        $\phi(d_2^{-1}) = \phi(d_1^{-1})\phi(c_3) 
         = 
         \pm \begin{bmatrix}
            1 & \xi_1+ \xi_2 \\
            0 & 1
        \end{bmatrix}
        \in \Parm$, 
        we have $\xi_1+ \xi_2 < 0$, and hence $\xi_1<0$ or $\xi_2<0$. Hence, at least one of $\phi(d_1^{-1})$ and $\phi(c_3)$ is in $\Parm$, which implies that if $\phi(d_1^{-1})\in \Parp$, then $\phi(c_3)\in \Parm$. If otherwise that $\phi(d_1^{-1})\in \Parm$, then as $\phi|_{\pi_1(P_1)}$ is abelian, we can show by the similar argument as above that at least one of $\phi(c_1)$ and $\phi(c_2)$ is in $\Parm$. Since $p_-(s'') = 1$, only one in $(s''_1,s''_2,s''_3)= (s_1,s_2,s_3)$ equals $-1$, hence $\phi(c_3)\in \Parp$. Therefore, $\phi(d_1^{-1})$ and $\phi(c_3)$ have opposite signs, and therefore $s'_1 = -s'_2$.

        To show that $(s'_3, \cdots, s'_{p-1})$ contains both $+1$ and $-1$, we recall that 
        $(s'_3, \cdots, s'_{p-1}) = (s_4, \cdots, s_p)$.
      We first find the $+1$ in $(s_4, \cdots, s_p)$. Recall that $\phi|_{\pi_1(\Sigma\setminus (P_1\cup P_2))}$ is type-preserving, and its sign is determined by the signs of $\phi(c_4),\phi(c_5),\cdots, \phi(c_p)$, and $\phi(d_2^{-1})$. 
        As $\phi|_{\pi_1(P_i)}$ is abelian for all $i\in \{3,\cdots, p-2\}$, by Lemma \ref{lem_abelsphere}, $p_+\big(s\big(  \phi|_{\pi_1(P_1\cup P_2)} \big)\big)\geqslant 1$. 
        Since $\phi(d_2^{-1})\in \Par^-$, this implies that at least one of $\phi(c_4),\cdots, \phi(c_p)$ lies in $\Par^+$. 
        Next we find the  $-1$ in $(s_4,\cdots, s_p)$. Since $(s_1,s_2,s_3) = (s''_1, s''_2, s''_3)$, and since $p_-(s'') =  1$ and $s''_4 = +1$, only one among $s_1,s_2,s_3$ are $-1$, and the other two are $+1$. Since $p_-(s)\geqslant 2$, at least one of $s_4,\cdots, s_p$ is $-1$. 
        
        Now, as both $(s'_1, s'_2)$ and $(s'_3,\cdots, s'_{p-1})$ contain at least a $+1$ and a $-1$, we conclude that $p_+(s')\geqslant 2$ and $p_-(s')\geqslant 2$. 
\medskip

        In the case that  $\phi(d_2)\in \Parm$, we will first show $s'_1 = s'_2 = +1$, then show that at least two among $s'_3,\cdots, s'_{p-1}$ equal $-1$, which together imply $p_+(s')\geqslant 2$ and $p_-(s')\geqslant 2$. 
        Since 
        $p_-(s'') = 1$ and $s''_4 = -1$, one has $s''_i = +1$ for $i\in \{1,2,3\}$, i.e., $\phi(c_i)\in \Parp$ for $i\in \{1,2,3\}$. As $\phi|_{\pi_1(P_1)}$ is abelian and $\phi(c_1), \phi(c_2)\in \Parp$, there are 
        $\xi_1, \xi_2 > 0$ such that, up to a $\psl$-conjugation, $\phi(c_1) = 
        \pm \begin{bmatrix}
            1 & \xi_1 \\
            0 & 1
        \end{bmatrix}
        $
        and $\phi(c_2) = 
        \pm \begin{bmatrix}
            1 & \xi_2 \\
            0 & 1
        \end{bmatrix} 
        $. Then $\xi_1 + \xi_2>0$, and
        $\phi(d_1^{-1}) = \phi(c_1)\phi(c_2) 
         = 
         \pm \begin{bmatrix}
            1 & \xi_1+ \xi_2 \\
            0 & 1
        \end{bmatrix}
        \in \Parp$.
        Since $\phi(d_1^{-1}), \phi(c_3)\in \Parp$, we have $s'_1 = s'_2 = +1$. Moreover, since $p_-(s)\geqslant 2$ and $(s_1,s_2,s_3) = (+1,+1,+1)$, at least two in the tuple $(s'_3,\cdots, s'_{p-1}) = (s_4,\cdots, s_p)$ equal $-1$. 
        As $(s'_1, s'_2)$ contains two $+1$'s and $(s'_3,\cdots, s'_{p-1})$ contains two $-1$'s, we conclude that $p_+(s')\geqslant 2$ and $p_-(s')\geqslant 2$. 
        \medskip
        
        Finally, in the case that $p_+(s'') = 1$, we prove that (1) holds. 
        Let $h\in \pgl\setminus \psl$. By Proposition \ref{prop_pglpsl}, the conjugation $h\phi h^{-1}$ is in $\mathcal{R}^{-s}_0(\Sigma)$ with $p_+(-s)\geqslant 2$ and $p_-(-s)\geqslant 2$,
        and its restriction $h\phi h^{-1}|_{\pi_1(P_1\cup P_2)}$ is in $\mathcal{R}^{-s''}_0(P_1\cup P_2)$ with $p_+(-s'') = 3$. As shown above, $h\phi h^{-1}|_{\pi_1(\Sigma\setminus P_1)}$ is in $\mathcal{R}^{s'}_0(\Sigma\setminus P_1)$ for some $s'\in \{\pm 1\}^{p-1}$ satisfying $p_+(s')\geqslant 2$ and $p_-(s')\geqslant 2$. Then by Proposition \ref{prop_pglpsl} again, $\phi|_{\pi_1(\Sigma\setminus P_1)} = h^{-1}(h\phi h^{-1})h|_{\pi_1(\Sigma\setminus P_1)}$ is in $\mathcal{R}^{-s'}_0(\Sigma\setminus P_1)$  with $-s'$ satisfying $p_+(-s')\geqslant 2$ and $p_-(-s')\geqslant 2$.
        This completes the proof.
    \end{proof}

\begin{proof}[Proof of Proposition \ref{prop_NAsphere2}]
        We will proceed by induction on the number of punctures $p$. 
        \bigskip
        
        For the base case $p = 4$, 
        there are in total six possible signs $s\in\{\pm 1\}^4$ satisfying $p_+(s) = p_-(s) = 2$, with four of which satisfying $s_3 = -s_2$ and $s_4 = -s_1$, and two of which satisfying  $s_3 = s_2$ and $s_4 = s_1$. We will consider these cases separately in the proof. 
    \medskip

        In the case that  $s_3 = -s_2$ and $s_4 = -s_1$, we will first connect $\phi$ to the representation $\phi_1$ given by 
        $$\big(\phi_1(c_1), \phi_1(c_2), \phi_1(c_3), \phi_1(c_4)\big)
        = (\phi(c_1), \phi(c_2), \phi(c_2)^{-1}, \phi(c_1)^{-1}\big),$$
        then connect $\phi_1$ to a $\psi\in \mathrm{NA}^s_0(\Sigma_{0,4})$.

        The path $\path$ connecting $\phi$ to $\phi_1$ is constructed as follows.
        As $\phi$ is type-preserving and $\phi |_{\pi_1(P_1)}$ and $\phi |_{\pi_1(P_2)}$ are abelian, by Lemma \ref{lem_abelsphere}, we can assume that the image $\phi\big(\pi_1(\Sigma_{0,4})\big)$ lies in a one-parameter subgroup $\bigg\{\pm \part\bigg\}_{t\in \mathbb{R}}$. Then for each $i\in \{1,2,3,4\}$, there exists a $\xi_i\in \mathbb{R}\setminus\{0\}$ with $sgn(\xi_i) = sgn(s_i)$ such that $\phi(c_i) = \pm \begin{bmatrix}
        1 & \xi_i \\
        0 & 1
        \end{bmatrix}$. 
        We define a path  
        $\{\pm C_t\}\interval$ connecting $\pm C_0 = \phi(c_3)$ to $\pm C_1 = \phi(c_2)^{-1}$
        by setting
        $\pm C_t \doteq 
        \pm \begin{bmatrix}
        1 & \xi_{3,t} \\
        0 & 1
        \end{bmatrix},$ 
        where
        $\xi_{3,t} =  (1-t)\xi_3 - t\xi_2$.
        Since $\{\xi_{3,t}\}\interval$ is a straight line segment connecting $\xi_3$ to $-\xi_2$ and $sgn(\xi_3) = sgn(-\xi_2)$,
        for each $t\in [0,1]$,  $sgn(\xi_{3,t}) = sgn(\xi_3)$ and $\pm C_t \in \Par^{sgn(s_3)}$.
        Similarly, we define a path
        $\{\pm D_t\}\interval$ in $\Par^{sgn(s_4)}$ connecting $\pm D_0 = \phi(c_4)$ to $\pm D_1 = \phi(c_1)^{-1}$
        by setting
        $\pm D_t \doteq 
        \pm \begin{bmatrix}
        1 & \xi_{4,t} \\
        0 & 1
        \end{bmatrix},$ 
        where
        $\xi_{4,t} =  (1-t)\xi_4 - t\xi_1$. Since $\phi(c_1)\phi(c_2)\phi(c_3)\phi(c_4) = \pm \mathrm I$, we have $\sum^4_{i=1}\xi_i = 0$, which implies for each $t\in [0,1]$ that 
        $\xi_{3,t} + \xi_{4,t} 
        =  \xi_3 + \xi_4 - \big(\sum_{i=1}^4\xi_i\big)t 
        = \xi_3 + \xi_4.$
        Then for each $t\in [0,1]$, we have
        $ \pm C_tD_t = 
        \pm\begin{bmatrix}
        1 & \xi_{3,t} + \xi_{4,t} \\
        0 & 1
        \end{bmatrix}
        = \pm\begin{bmatrix}
        1 & \xi_3 + \xi_4 \\
        0 & 1
        \end{bmatrix} = \big(\phi(c_1)\phi(c_2)\big)^{-1}$, hence we can define a representation $\phi_t$ of $\pi_1(\Sigma_{0,4})$ by setting $$\big(\phi_t(c_1), \phi_t(c_2), \phi_t(c_3), \phi_t(c_4)\big)
        \doteq (\phi(c_1), \phi(c_2),\pm C_t, \pm D_t\big).$$
        Since $\{\pm C_t\}\interval\subset \Par^{sgn(s_3)}$ and $\{\pm D_t\}\interval\subset \Par^{sgn(s_4)}$, the signs $s(\phi_t) = s(\phi)$ and the relative Euler classes $e(\phi_t) = e(\phi)$ for all $t\in [0,1]$. This defines a path of representations
        $\{\phi_t\}\interval$ in $\mathcal{R}^s_0(\Sigma_{0,4})$
        connecting 
        $\phi_0 = \phi$ to 
        $\phi_1$ with 
        $$\big(\phi_1(c_1), \phi_1(c_2), \phi_1(c_3), \phi_1(c_4)\big)
        = (\phi(c_1), \phi(c_2), \phi(c_2)^{-1}, \phi(c_1)^{-1}\big).$$
        
        It remains to find a path connecting $\phi_1$ to a $\psi \in \mathrm{NA}^s_0(\Sigma_{0,4})$. 
        As $\phi_1|_{\pi_1(P_1)}$ is abelian with $\phi_1(c_1)\neq \pm \mathrm I$ and $\phi_1(c_2)\neq \pm \mathrm I$,
        by Lemma \ref{lem_nonabelpants}, there is a one-parameter subgroup $\{h_t\}_{t\in \mathbb{R}}$ of $\psl$ with $h_0 = \pm\mathrm{I}$ such that for all $t\neq 0$, $\phi_1(c_1)$ and $ h_t\phi_1(c_2)h_t^{-1}$ do not commute. 
        For each $t\in [0,1]$, we define 
        $\psi_t$ 
        by setting
        $$\big(\psi_t(c_1), \psi_t(c_2), \psi_t(c_3), \psi_t(c_4)\big)
        \doteq (\phi(c_1), h_t\phi(c_2)h_t^{-1}, h_t\phi(c_2)^{-1}h_t^{-1}, \phi(c_1)^{-1}\big).$$
         Since
         $\psi_t(c_i)$ and $\phi_1(c_i)$ are conjugate for all $t\in [0,1]$ and for all $i\in \{1,2,3,4\}$, the signs $s(\psi_t) = s(\phi_1)$ and the relative Euler classes $e(\psi_t) = e(\phi_1)$ for all $t\in [0,1]$, and hence the path $\{\psi_t\}\interval\subset \mathcal{R}^s_0(\Sigma_{0,4})$. Moreover, by the definition of $h_1$, 
         $\psi_1|_{\pi_1(P_1)}$ and $\psi_1|_{\pi_1(P_2)}$ are non-abelian, i.e., $\psi_1\in \mathrm{NA}^s_0(\Sigma_{0,4})$. 
         Letting $\psi = \psi_1$, the composition of 
        $\{\phi_t\}\interval$ and $\{\psi_t\}\interval$ connects $\phi$ to $\psi\in \mathrm{NA}^s_0(\Sigma_{0,4})$ within $\mathcal{R}^s_0(\Sigma_{0,4})$.
        \medskip
        
        In the case that  $s_3 = s_2$ and $s_4 = s_1$,
         we will first connect $\phi$ to the representation $\phi_1$ given by 
        $$\big(\phi_1(c_1), \phi_1(c_2), \phi_1(c_3), \phi_1(c_4)\big)
        = (\phi(c_1), \phi(c_2), \phi(c_1)^{-1}, \phi(c_2)^{-1}\big),$$
        then connect $\phi_1$ to a $\psi\in \mathrm{NA}^s_0(\Sigma_{0,4})$.

        The path $\path$ connecting $\phi$ to $\phi_1$ is constructed as follows.
    Again by Lemma \ref{lem_abelsphere}, we can assume for each $i\in \{1,2,3,4\}$ that $\phi(c_i) = \pm \begin{bmatrix}
        1 & \xi_i \\
        0 & 1
        \end{bmatrix}$ for some $\xi_i\in \mathbb{R}\setminus\{0\}$ with $sgn(\xi_i) = sgn(s_i)$. 
        Similar to the previous case, we define a path  
        $\{\pm C_t\}\interval$ 
        connecting $\pm C_0 = \phi(c_3)$ to $\pm C_1 = \phi(c_1)^{-1}$
        by setting
        $\pm C_t \doteq 
        \pm \begin{bmatrix}
        1 & \xi_{3,t} \\
        0 & 1
        \end{bmatrix},$ 
        where
        $\xi_{3,t} = (1-t)\xi_3 - t\xi_1$. 
        Since $s_3 = s_2$ and $p_+(s) = p_-(s) = 2$, we have $s_3 = -s_1$, i.e., $sgn(\xi_3) = sgn(-\xi_1)$; and since $\{\xi_{3,t}\}\interval$ is a straight line segment connecting $\xi_3$ and $-\xi_1$, for each $t\in [0,1]$, $sgn(\xi_{3,t}) = sgn(\xi_3)$ and $\pm C_t \in \Par^{sgn(s_3)}$.
        Similarly, we define 
        $\{\pm D_t\}\interval$ in $\Par^{sgn(s_4)}$ connecting $\pm D_0 = \phi(c_4)$ to $\pm D_1 = \phi(c_2)^{-1}$
        by setting
        $\pm D_t \doteq 
        \pm \begin{bmatrix}
        1 & \xi_{4,t} \\
        0 & 1
        \end{bmatrix},$ 
        where
        $\xi_{4,t} =  (1-t)\xi_4 - t\xi_2$. Then by the similar argument as above, we have
        $\pm C_tD_t = \big(\phi(c_1)\phi(c_2)\big)^{-1}$ for each $t\in [0,1]$, hence we can define a representation $\phi_t$ of $\pi_1(\Sigma_{0,4})$ by setting $$\big(\phi_t(c_1), \phi_t(c_2), \phi_t(c_3), \phi_t(c_4)\big)
        \doteq (\phi(c_1), \phi(c_2),\pm C_t, \pm D_t\big).$$
        Since $\{\pm C_t\}\interval\subset \Par^{sgn(s_3)}$ and $\{\pm D_t\}\interval\subset \Par^{sgn(s_4)}$, the signs $s(\phi_t) = s(\phi)$ and the relative Euler classes $e(\phi_t) = e(\phi)$ for all $t\in [0,1]$. This defines a path
        $\{\phi_t\}\interval$ in $\mathcal{R}^s_0(\Sigma_{0,4})$
        connecting 
        $\phi_0 = \phi$ to 
        $\phi_1$ with
        $$\big(\phi_1(c_1), \phi_1(c_2), \phi_1(c_3), \phi_1(c_4)\big)
        = (\phi(c_1), \phi(c_2), \phi(c_1)^{-1}, \phi(c_2)^{-1}\big).$$
        
        It remains to find a path connecting $\phi_1$ to a $\psi \in \mathrm{NA}^s_0(\Sigma_{0,4})$. 
        Again by Lemma \ref{lem_nonabelpants}, there is a one-parameter subgroup $\{h_t\}_{t\in \mathbb{R}}$ of $\psl$ with $h_0 = \pm\mathrm{I}$ such that for all $t\neq 0$, $\phi_1(c_1)$ and $ h_t\phi_1(c_2)h_t^{-1}$ do not commute. 
        For each $t\in [0,1]$, we let 
        $g_t \doteq h_t\phi(c_2)h_t^{-1}$, and define 
        $\psi_t$ 
        by setting
        $$\big(\psi_t(c_1), \psi_t(c_2), \psi_t(c_3), \psi_t(c_4)\big)
        = (\phi(c_1), g_t, g_t^{-1}\phi(c_1)^{-1}g_t, g_t^{-1}\big).$$
        Notice that $\phi(c_1)$ and $g_0 = \phi(c_2)$ commute, hence  $g_0^{-1}\phi(c_1)^{-1}g_0 = \phi(c_1)^{-1}$ and $\psi_0 = \phi_1$.
         By the similar argument as above,
         as $\psi_t(c_i)$ and $\phi_1(c_i)$ are conjugate for all $t\in [0,1]$ and for all $i\in \{1,2,3,4\}$, 
         $\{\psi_t\}\interval\subset \mathcal{R}^s_0(\Sigma_{0,4});$ and by the definition of $h_1$, 
         $\psi_1\in \mathrm{NA}^s_0(\Sigma_{0,4})$. 
         Letting $\psi = \psi_1$, the composition of 
        $\{\phi_t\}\interval$ and $\{\psi_t\}\interval$ connects $\phi$ to $\psi\in \mathrm{NA}^s_0(\Sigma_{0,4})$ within $\mathcal{R}^s_0(\Sigma_{0,4})$. This completes the proof for the base case $p = 4$.
        \\
        
        Now let $\Sigma = \Sigma_{0,p}$ with $p\geqslant 5$, and as the induction hypothesis assume that the proposition holds for $\Sigma_{0,p-1}$. 
        By Proposition \ref{prop_NAsphere1}, it suffices to find a path in $\mathcal{R}^s_0(\Sigma)$ that connects $\phi$ to a representation $\rho\in \mathcal{R}^s_0(\Sigma)$ such that $\rho|_{\pi_1(P_i)}$ is non-abelian for some $i\in \{1,\cdots, p-2\}$, which we will do separately in the Cases (1) and (2) of Lemma \ref{lem_claimforNAsphere2}.

         \medskip
        
        If (1) of Lemma \ref{lem_claimforNAsphere2} holds, then as $\Sigma\setminus P_1\cong \Sigma_{0, p-1}$ and $\phi|_{\pi_1(P_i)}$ is abelian for $i\in \{2,\cdots, p-2\}$, by the induction hypothesis, there is a path $\{\phi_t: \pi_1(\Sigma\setminus P_1)\to \psl\}\interval$ in $\mathcal{R}^{s'}_0(\Sigma\setminus P_1)$ that connects
        $\phi_0 = \phi|_{\pi_1(\Sigma\setminus P_1)}$ to $\phi_1 \in \mathrm{NA}^{s'}_0(\Sigma\setminus P_1)$. We extend this path to $\pi_1(\Sigma)$ as follows: Notice that $d_1$ is the common boundary of $P_1$ and $\Sigma\setminus P_1$. 
        As the path lies in $\mathcal{R}^{s'}_0(\Sigma\setminus P_1)$, for all $t\in [0,1]$, $\phi_t(d_1)$ and $\phi(d_1)$ are parabolic elements of the same sign, hence are conjugate. 
        Then by Lemma \ref{lem_conjugacypath_const}, there is a path $\{g_t\}\interval$ in $\psl$ with $g_0 = \pm \mathrm{I}$ such that $\phi_t(d_1) = g_t\phi(d_1)g_t^{-1}$. By letting $\phi_t|_{\pi_1(P_1)}\doteq g_t\big(\phi|_{\pi_1(P_1)}\big)g_t^{-1}$ for all $t\in [0,1]$, we can extend the path $\{\phi_t\}\interval$ from $\pi_1(\Sigma\setminus P_1)$ to $\pi_1(\Sigma)$. 
        Then for all $t\in [0,1]$, as $c_1, c_2\in \pi_1(P_1)$, $\phi_t(c_i) = g_t\phi(c_i)g_t^{-1}$ for each $i\in \{1,2\}$; and for each $i\in \{3,\cdots, p\}$, as $c_i$ is a peripheral element of $\pi_1(\Sigma\setminus P_1)$ and  $\phi_t|_{\pi_1(\Sigma\setminus P_1)}\in \mathcal{R}^{s'}_0(\Sigma\setminus P_1)$, $\phi_t(c_i)$ and $\phi(c_i)$ are parabolic elements of the same sign. 
        Hence, the signs $s(\phi_t) = s(\phi)$ and the relative Euler classes $e(\phi_t) = e(\phi)$ for all $t\in [0,1]$, and $\{\phi_t\}\interval\subset \mathcal{R}^s_0(\Sigma)$. As $\phi_1|_{\pi_1(\Sigma\setminus P_1)}\in \mathrm{NA}^{s'}_0(\Sigma\setminus P_1)$, $\phi_1|_{\pi_1(P_i)}$ is non-abelian for each $i\in \{2,\cdots,  p-2\}$, and  we  can let $\rho = \phi_1$.
        \medskip
        
        If (2) of Lemma \ref{lem_claimforNAsphere2}  holds, then as $P_1\cup P_2\cong \Sigma_{0,4}$ and $\phi |_{\pi_1(P_1)}$ and $\phi |_{\pi_1(P_2)}$ are abelian, 
        as shown in the base case $p=4$,
        there is a path 
        $\big\{\phi_t : \pi_1(P_1\cup P_2)\to \psl\big\}\interval$ in $\mathcal{R}^{s''}_0(P_1\cup P_2)$ that connects
        $\phi_0 = \phi|_{\pi_1(P_1\cup P_2)}$ to $\phi_1 \in \mathrm{NA}^{s''}_0(P_1\cup P_2)$. We extend this path to $\pi_1(\Sigma)$ as follows: Notice that $d_2$ is the common boundary of $P_1\cup P_2$ and $\Sigma\setminus (P_1\cup P_2)$. 
        As the path lies in $\mathcal{R}^{s''}_0(P_1\cup P_2)$, for all $t\in [0,1]$, $\phi_t(d_2)$ and $\phi(d_2)$ are parabolic elements of the same sign, hence  are conjugate. Then by Lemma \ref{lem_conjugacypath_const}, there is a path $\{g_t\}\interval$ in $\psl$ with $g_0 = \pm \mathrm{I}$ such that $\phi_t(d_2) = g_t\phi(d_2)g_t^{-1}$. By letting $\phi_t|_{\pi_1(\Sigma\setminus (P_1\cup P_2))}\doteq g_t\big(\phi|_{\pi_1(\Sigma\setminus (P_1\cup P_2))}\big)g_t^{-1}$ for all $t\in [0,1]$, we extend the path $\{\phi_t\}\interval$ from $\pi_1(P_1\cup P_2)$ to $\pi_1(\Sigma)$. 
        Then for each $i\in \{1,2,3\}$, as $c_i$ is a peripheral element of $\pi_1(P_1\cup P_2)$ and $\phi_t|_{\pi_1(P_1\cup P_2)}\in \mathcal{R}^{s''}_0(P_1\cup P_2)$, $\phi_t(c_i)$ and $\phi(c_i)$ are parabolic elements of the same sign; and for each $i\in \{4,\cdots, p\}$, $\phi_t(c_i) = g_t\phi(c_i)g_t^{-1}$. 
        Hence, the signs $s(\phi_t) = s(\phi)$ and the relative Euler classes $e(\phi_t) = e(\phi)$ for all $t\in [0,1]$, and $\{\phi_t\}\interval\subset \mathcal{R}^s_0(\Sigma)$. As $\phi_1|_{\pi_1(P_1\cup P_2)}\in \mathrm{NA}^{s''}_0(P_1\cup P_2)$, $\phi_1|_{\pi_1(P_1)}$ and $\phi_1|_{\pi_1(P_2)}$ are non-abelian, and we can let $\rho = \phi_1$. 
    \end{proof}
    
    \subsection{Proof of Theorem \ref{thm_nonabel}}\label{Pf}
    
    \begin{proof}[Proof of Theorem \ref{thm_nonabel}]
    If $p_0(s)\geqslant 1$, then the result directly follows from Proposition \ref{prop_NAhyp}; and if $g\geqslant 1$ and $p_0(s) = 0$, then the result follows from
    Proposition \ref{prop_NAtorus}. If $g = 0,$ $p_0(s) = 0$ and $n\neq 0$, then for all $\phi\in \Rns$,
    $\phi|_{\pi_1(P_i)}$ is non-abelian for some $i\in \{1,\cdots, p-2\}$, because otherwise Proposition \ref{prop_reducible} and Proposition $\ref{prop_additivity}$ would imply that $n=e(\phi) = 0$. 
    Then the result follows from Proposition \ref{prop_NAsphere1}. 
    It remains to consider the case that only Condition (4) holds.
    Let $\phi\in \mathcal{R}^s_0(\Sigma_{0,p})$ for some $s\in \{\pm 1\}^p$ with $p_+(s)\neq 1$ and $p_-(s)\neq 1$.
    If $\phi|_{\pi_1(P_i)}$ is non-abelian for some $i\in \{1,\cdots, p-2\}$, then by Proposition \ref{prop_NAsphere1}, there is a path deforming $\phi$ into $\mathrm{NA}^s_0(\Sigma_{0,p})$. If $\phi|_{\pi_1(P_i)}$ is abelian for all $i\in \{1,\cdots, p-2\}$, then by Lemma \ref{lem_abelsphere} and Condition (4), we have $p_+(s)\geqslant 2$ and $p_-(s)\geqslant 2$; and by Proposition \ref{prop_NAsphere2}, there is a path deforming $\phi$ into $\mathrm{NA}^s_0(\Sigma)$, completing the proof.
    \end{proof}

\section{Surfaces with Euler characteristic $-2$}
    In this section, we prove Theorem \ref{thm_main4} for punctured surfaces $\Sigma$ of Euler characteristic $-2$, namely, the two-hole torus and the four-hole sphere, and describe the connected components of $\R$. 
    For both of the cases, we need the following Proposition \ref{prop_sameEuler}.

    \begin{proposition}\label{prop_sameEuler}
            Let $\Sigma = \Sigma_{g,p}$, let $\Sigma'\subset\Sigma$ be a subsurface with the complement  $\Sigma\setminus \Sigma'$ either homeomorphic to  $\Sigma_{1,1}$ or homeomorphic to $\Sigma_{0,3}$, and let $d\in \pi_1(\Sigma)$ be represented by a common boundary component of $\Sigma'$  and $\Sigma\setminus \Sigma'$. For $n\in\mathbb{Z}$ and $s \in \{-1,0,+1\}^p$, let $\phi, \psi$ be representations in $\HPns(\Sigma)$ sending $d$ to hyperbolic elements.
            Suppose 
            $s\big(\phi|_{\pi_1(\Sigma')}\big) = s\big(\psi|_{\pi_1(\Sigma')}\big)=s'$ for some $s'$, $e\big(\phi|_{\pi_1(\Sigma')}\big) = e\big(\psi|_{\pi_1(\Sigma')}\big)=n'$ for some  $n'$, and there is a path
            in $\HP^{s'}_{n'}(\Sigma')$ connecting $\phi|_{\pi_1(\Sigma')}$ and $\psi|_{\pi_1(\Sigma')}$.
            Then there is a path in $\HPns(\Sigma)$ connecting $\phi$ to $\psi$.
        \end{proposition}
        
        \begin{proof} 
        For the case that $\Sigma\setminus \Sigma'\cong \Sigma_{1,1}$, we will first construct a path $\path$ in $\HPns(\Sigma)$ connecting $\phi_0 = \phi$ to a $\phi_1\in \HPns(\Sigma)$ such that $\phi_1|_{\pi_1(\Sigma')} = \psi|_{\pi_1(\Sigma')}$, then construct a path $\{\phi_t\}_{t\in [1,2]}$ in $\HPns(\Sigma)$ connecting $\phi_1$ to $\phi_2 =\psi$. Then the path $\{\phi_t\}_{t\in [0,2]}$ connects $\phi$ to $\psi$.
        
       The path $\{\phi_t\}_{t\in [0,1]}$ is constructed as follows.
       By the assumption, there exists a path $\{\phi_t: \pi_1(\Sigma')\to \psl\}\interval$ in $\HP^{s'}_{n'}(\Sigma')$ connecting $\phi_0= \phi|_{\pi_1(\Sigma')}$ to $\phi_1= \psi|_{\pi_1(\Sigma')}$, which is to be extended to $\pi_1(\Sigma)$.
        For each $t\in [0,1]$, since $\phi(d)\in \Hyp$ and $s(\phi_t) = s(\phi)$, $\phi_t(d)\in \Hyp$. Let $T\doteq \Sigma\setminus \Sigma'$. Then by Proposition \ref{prop_additivity}, 
        we have $\phi|_{\pi_1(T)}, \psi|_{\pi_1(T)} \in \mathcal{W}_{n - n'}(T)$; and by Theorem \ref{thm_goldman}, $n - n'\in \{-1,0,1\}$.
        For each $t\in [0,1]$, we let $\widetilde{\phi_t(d)}$ be the lift of $\phi_t(d)$ to $\Hyp_{n-n'}$. Then the path $\{\widetilde{\phi_t(d)}\}\interval$ lies in $\Hyp_{n-n'}$, which, by Theorem \ref{thm_liftcommu}, is contained in the image of the lifted commutator map.  
        Since $e(\phi|_{\pi_1(T)}) = n - n'$, the lifted commutator map sends the pair $(\phi(a), \phi(b))$ to $\widetilde{\phi(d)}\in \Hyp_{n-n'}$; and
        by Proposition \ref{prop_plpliftcommu}, the path $\{\widetilde{\phi_t(d)}\}\interval$ lifts to the path $\{(\pm A_t, \pm B_t)\}\interval$ in $\psl\times\psl$ with $(\pm A_0, \pm B_0) = (\phi(a), \phi(b)
       )$. For each $t\in [0,1]$, we let $\big(\phi_t(a), \phi_t(b)\big) \doteq (\pm A_t, \pm B_t)$, extending the representation $\phi_t$ from $\pi_1(\Sigma')$ to $\pi_1(\Sigma)$. 
     Since all peripheral elements of $\pi_1(\Sigma)$ are contained in $\pi_1(\Sigma')$, and since $\phi_t|_{\pi_1(\Sigma')}\in \HP^{s'}_{n'}(\Sigma')$ and the signs $s(\phi_t|_{\pi_1(\Sigma')}) = s(\phi|_{\pi_1(\Sigma')})$ for each $t\in [0,1]$, the signs $s(\phi_t) = s(\phi)$ and the relative Euler classes $e(\phi_t) = e(\phi)$. Therefore, the path $\path$ lies in $\HPns(\Sigma)$, and connects $\phi$ to $\phi_1$.
       
       The  path $\{\phi_t\}_{t\in [1,2]}$ is constructed as follows.
       The evaluation map $\ev: \mathcal{W}_{n-n'}(T)\to \Hyp_{n-n'}$
       sends both $\phi_1|_{\pi_1(T)}$ and $\psi|_{\pi_1(T)}$ to $\widetilde{\phi_1(d)}\in \Hyp_{n-n'}$. By Theorem \ref{thm_liftcommu}, the fiber $\ev^{-1}(\widetilde{\phi_1(d)})$ is connected, hence there exists a path $\{\phi_t: \pi_1(T)\to \psl\}_{t\in [1,2]}$ in $\ev^{-1}(\widetilde{\phi_1(d)})$ connecting $\phi_1|_{\pi_1(T)}$ to $\phi_2 = \psi|_{\pi_1(T)}$. As the path lies in $\ev^{-1}(\widetilde{\phi_1(d)})$, for all $t\in [1,2]$, $\phi_t(d) = \phi_1(d)$. Therefore, by letting $\phi_t|_{\pi_1(\Sigma')} \doteq \phi_1|_{\pi_1(\Sigma')} = \psi|_{\pi_1(\Sigma')}$ for all $t\in [1,2]$, we extend the path $\{\phi_t\}_{t\in [1,2]}$ from $\pi_1(T)$ to $\pi_1(\Sigma)$. For all $t\in [1,2]$,
       since $\phi_t(c_i) = \phi_1(c_i)$ for all $i\in \{1,\cdots, p\}$, the signs $s(\phi_t) = s(\phi_1) = s(\phi)$ and the relative Euler classes $e(\phi_t) = e(\phi)$. Therefore, the path $\{\phi_t\}_{t\in [1,2]}$ lies in $\HPns(\Sigma)$, and connects $\phi_1$ to $\psi$.
       \\

       For the case that $\Sigma\setminus \Sigma'\cong \Sigma_{0,3}$,  by a permutation of peripheral elements if necessary, assume that the peripheral elements $c_1,c_2$ are in $\pi_1(\Sigma\setminus\Sigma'),$ and $\pi_1(\Sigma\setminus\Sigma') = \langle c_1, c_2\rangle$.
       As in the previous case, we will first construct a path $\path$ in $\HPns(\Sigma)$ connecting $\phi_0 = \phi$ to a $\phi_1\in \HPns(\Sigma)$ such that $\phi_1|_{\pi_1(\Sigma')} = \psi|_{\pi_1(\Sigma')}$, then construct a path $\{\phi_t\}_{t\in [1,2]}$ in $\HPns(\Sigma)$ connecting $\phi_1$ to $\phi_2 =\psi$. Then the path $\{\phi_t\}_{t\in [0,2]}$ connects $\phi$ to $\psi$.

       The  path $\{\phi_t\}_{t\in [0,1]}$ is constructed as follows. 
       By the assumption, there exists a path $\{\phi_t: \pi_1(\Sigma')\to \psl\}\interval$ in $\HP^{s'}_m(\Sigma')$ connecting $\phi_0= \phi|_{\pi_1(\Sigma')}$ to $\phi_1= \psi|_{\pi_1(\Sigma')}$, which is to be extended to $\pi_1(\Sigma)$. For each $t\in [0,1]$, as in the previous case, we let $\widetilde{\phi_t(d)}$ be the lift of $\phi_t(d)$ in $\Hyp_{n-n'}$, then the path $\{\widetilde{\phi_t(d)}\}\interval$ lies in $\Hyp_{n-n'}$. Let $P\doteq \Sigma\setminus \Sigma'$, and let 
        $\ev: \mathrm{HP}_{n-n'}^{(s_1,s_2,0)}(P)\to\Hyp_{n-n'}$ be the evaluation map defined in Subsection \ref{subsection_PLP}.
       Since $n - n' \in \{-1,0,1\}$, by Proposition \ref{prop_evimage},  $\ev$ is surjective onto $\Hyp_{n-n'}$; and by Theorem \ref{thm_PLP}, 
       the path $\{\widetilde{\phi_t(d)}\}\interval$ lifts to the path $\{\psi_t: \pi_1(P)\to \psl\}\interval$  in $\mathrm{HP}_{n-n'}^{(s_1,s_2,0)}(P)$ with $\psi_0 = \phi|_{\pi_1(P)}$. 
       For each $t\in [0,1]$, since $\psi_t$ and $\phi_t$ agree at the common boundary $d$ of $\Sigma'$ and $P$, by letting $\phi_t|_{\pi_1(P)}\doteq \psi_t$, we extend the path $\path$ from $\pi_1(\Sigma')$ to $\pi_1(\Sigma).$ 
       For each $t\in [0,1]$, since 
       $s(\phi_t|_{\pi_1(\Sigma')})
       = s(\phi|_{\pi_1(\Sigma')})$
       and $s(\phi_t|_{\pi_1(P)})
       = s(\phi|_{\pi_1(P)})$,
       the signs $s(\phi_t) = s(\phi)$ and the relative Euler classes $e(\phi_t) = e(\phi)$. Therefore, the path $\path$ lies in $\HPns(\Sigma)$, and connects $\phi$ to $\phi_1$.

       The  path $\{\phi_t\}_{t\in [1,2]}$ is constructed as follows.
       The evaluation map $\mathrm{HP}_{n-n'}^{(s_1,s_2,0)}(P)\to \Hyp_{n-n'}$
       sends both $\phi_1|_{\pi_1(P)}$ and $\psi|_{\pi_1(P)}$ to $\widetilde{\phi_1(d)}\in \Hyp_{n-n'}$. By Theorem \ref{thm_PLP}, the fiber $\ev^{-1}(\widetilde{\phi_1(d)})$ is connected, hence there exists a path $\{\phi_t: \pi_1(P)\to \psl\}_{t\in [1,2]}$ in $\ev^{-1}(\widetilde{\phi_1(d)})$ connecting $\phi_1|_{\pi_1(P)}$ to $\phi_2 = \psi|_{\pi_1(P)}$. As the path lies in $\ev^{-1}(\widetilde{\phi_1(d)})$, for all $t\in [1,2]$, $\phi_t(d) = \phi_1(d)$. Therefore, by letting $\phi_t|_{\pi_1(\Sigma')} \doteq \phi_1|_{\pi_1(\Sigma')} = \psi|_{\pi_1(\Sigma')}$ for all $t\in [1,2]$, we extend the path $\{\phi_t\}_{t\in [1,2]}$ from $\pi_1(P)$ to $\pi_1(\Sigma)$. For all $t\in [1,2]$,
       as $s(\phi_t|_{\pi_1(P)})=(s_1,s_2,0)$ and $\phi_t(c_i) = \phi_1(c_i)$ for all $i\in \{3,\cdots, p\}$, the signs $s(\phi_t) = s(\phi_1) = s(\phi)$ and the relative Euler classes $e(\phi_t) = e(\phi)$. Therefore, the path $\{\phi_t\}_{t\in [1,2]}$ lies in $\HPns(\Sigma)$, and connects $\phi_1$ to $\psi$.
    \end{proof}

    \subsection{Two-hole torus}\label{subsection_torustwo}

    The fundamental group of the two-hole torus $\Sigma_{1,2}$ has the following presentation
    $$\pi_1(\Sigma_{1,2}) = \langle 
     a, b, c_1, c_2\ |\ [a, b]\cdot c_1 c_2\rangle,$$
     where $c_1$ and $c_2$ are the preferred peripheral elements.
     Let $P\cup T$ be a pants decomposition of $\Sigma_{1,2}$, where $P \cong \Sigma_{0,3}$ with $\pi_1(P) = \langle c_1, c_2 \rangle$ and $T \cong\Sigma_{1,1}$ with $\pi_1(T) = \langle a, b \rangle$, 
     separated by a decomposition curve $d = [a,b] = (c_1c_2)^{-1}$. 
     \\

   To prove the results in this subsection, we need the following lemma.
      \begin{lemma}\label{lem_sameEuler_torus}
           
            For $n\in\mathbb{Z}$ and $s \in \{-1,0,+1\}^2$, let $\phi,\psi\in \HPns(\Sigma_{1,2})$.
            If $\phi(d), \psi(d)\in \Hyp$ and  $e\big(\phi|_{\pi_1(P)}\big) = e\big(\psi|_{\pi_1(P)}\big)$, then there exists a path in $\HPns(\Sigma_{1,2})$ connecting $\phi$ to $\psi$. 
        \end{lemma}
    \begin{proof}
       The fundamental group $\pi_1(P)$ has the preferred peripheral elements $c_1,c_2$ and $d$. As representations in $\HPns(\torustwo)$, $\phi$ and $\psi$ send $c_1$ and $c_2$ into $\overline{\Hyp}$; and by the assumption, $\phi$ and $\psi$ send $d$ into $\Hyp$. 
       Hence, both $\phi|_{\pi_1(P)}$ and $\psi|_{\pi_1(P)}$ are in $\mathrm{HP}(P)$ with sign $s'= (s_1,s_2, 0)\in \{-1,0,+1\}^3$. Let $n' = e\big(\phi|_{\pi_1(P)}\big) = e\big(\psi|_{\pi_1(P)}\big)$. Then both 
       $\phi|_{\pi_1(P)}$ and $\psi|_{\pi_1(P)}$ are in $\HP^{s'}_{n'}(P)$. By Theorem \ref{thm_goldman} and Theorem \ref{thm_pants2}, $\HP^{s'}_{n'}(P)$ is connected, hence there exists a path $\{\phi_t: \pi_1(P)\to \psl\}\interval$ in $\HP^{s'}_{n'}(P)$ connecting $\phi_0= \phi|_{\pi_1(P)}$ to $\phi_1= \psi|_{\pi_1(P)}$. Moreover, 
       by Proposition \ref{prop_additivity}, $e\big(\phi|_{\pi_1(T)}\big) = e\big(\psi|_{\pi_1(T)}\big) = n - n'$. Therefore, Proposition \ref{prop_sameEuler} completes the proof by letting $\Sigma' = P$.
    \end{proof}
    
    We first prove Theorem \ref{thm_main4} for $\Sigma = \Sigma_{1,2}$, which is stated as Theorem \ref{thm_torustype1} below. 
     \begin{theorem}\label{thm_torustype1}
        Let $n\in \mathbb{Z}$, and let $s\in\{-1,0,+1\}^2$ with $p_0(s)\geqslant 1$.
        Then $\HPns(\torustwo)$ is non-empty if and only if 
        $$\chi(\Sigma_{1,2}) + p_+(s)\leqslant n \leqslant -\chi(\Sigma_{1,2}) - p_-(s).$$ 
        Moreover, each non-empty $\HPns(\torustwo)$ is connected.
    \end{theorem}
    \begin{proof}
    As the case that $p_0(s) = 2$ is included in Theorem \ref{thm_goldman}, we only need to consider the case that $p_0(s) = 1$. We will prove the result for the case that $p_+(s)=1$ and $p_-(s)=0.$ Then the result for the remaining case that $p_-(s)=1$ and $p_+(s)=0$ follows from the previous case and Proposition \ref{prop_pglpsl}.
    \\
    

     For the case that $p_+(s)=1$ and $p_-(s)=0,$ the equality in the result becomes $-1\leqslant n\leqslant 2$; and up to a permutation of the peripheral elements we can assume that  $s=(+1,0)$. We first prove that, if $-1\leqslant n \leqslant 2$, then $\HPns(\Sigma_{1,2})$ is non-empty. To find a representation $\phi$ in $\HPns(\torustwo)$, we will define $\phi$ on $\pi_1(P)$ first, then extend it to $\pi_1(\torustwo)$. 
          If $n = -1$, then let $m = 0$; and if $0\leqslant n\leqslant 2$, then let $m = 1$. 
          As the fundamental group $\pi_1(P)$ has the preferred peripheral elements $c_1,c_2$ and $d$, by Theorem \ref{thm_pants2}, there exists a representation $\phi\in \HP_m^{(+1,0,0)}(P)$, sending $c_1$ into $\Par^+$ and sending $c_2, d$ into $\Hyp$. 
          Since $\phi(d)\in \Hyp$, there exists a unique lift $\widetilde{\phi(d)}$ in $\Hyp_{n-m}$.  As $n-m\in \{-1,0,1\}$, by Theorem \ref{thm_liftcommu}, $\Hyp_{n-m}$ is in the image of the lifted commutator map, hence there exists a $\psl$-pair $(\pm A, \pm B)$ whose lifted commutator equals $\widetilde{\phi(d)}$. We let $\big(\phi(a), \phi(b)\big) \doteq (\pm A, \pm B)$. As the commutator $[\phi(a), \phi(b)]$ equals the projection $\phi(d)$ of $\widetilde{\phi(d)}$, this extends the representation $\phi$ from $\pi_1(P)$ to $\pi_1(\torustwo)$. 
       Since the evaluation map $\ev: \mathcal{W}(T)\to \displaystyle\bigcup_{k = -1}^1\Hyp_k$
       sends $\phi|_{\pi_1(T)}$ into $\Hyp_{n-m}$, we have $e(\phi|_{\pi_1(T)}) = n-m$;
       and since $e(\phi|_{\pi_1(P)}) = m$, by Proposition \ref{prop_additivity}, $e(\phi) = m + (n-m) = n$.
       Moreover, since $\phi(c_1)\in \Par^+$ and $\phi(c_2)\in \Hyp$, we have $s(\phi) = (+1,0)$. As a consequence,  $\phi\in \HPns(\torustwo)$.
   \medskip

          
          Next we prove that, if $\HPns(\Sigma_{1,2})$ is non-empty, then $-1\leqslant n\leqslant 2$. For a representation $\phi\in \HPns(\Sigma_{1,2})$, we will  first construct a path $\path$ in $\HPns(\torustwo)$ connecting $\phi$ to a $\phi_1\in \HPns(\torustwo)$ with $\phi_1(d)$ hyperbolic. 
        Indeed, since $s = (+1,0)$, $\phi(c_1)$ and $\phi(c_2)$ are elements of $\psl$ of different types, hence $\phi|_{\pi_1(P)}$ is non-abelian; and since 
        $[\phi(a), \phi(b)] =\big(\phi(c_1)\phi(c_2)\big)^{-1} 
         \neq \pm\mathrm{I}$, $\phi|_{\pi_1(T)}$ is non-abelian. Then by Lemma \ref{lem_inthyptorus}, there exists a path $\path$ starting from $\phi_0 = \phi$ and with $\phi_1(d)$ hyperbolic. Moreover, for each $t\in [0,1]$ and each $i\in\{1,2\},$ $\phi_t(c_i)$ is conjugate to $\phi(c_i)$. Hence the signs $s(\phi_t) = s(\phi)$ and the relative Euler classes $e(\phi_t) = e(\phi)$, and the path $\path$ lies in $\HPns(\torustwo).$  Next, we will deduce from $e(\phi_1) = n$ that $-1\leqslant n\leqslant 2$.  Indeed, since $\phi_1(d)$ is hyperbolic, $\phi_1|_{\pi_1(P)}\in \HP(P)$ and $\phi_1|_{\pi_1(T)}\in \mathcal{W}(T)$; and since
        $\phi_1(c_1)\in \Par^+$ and $\phi_1(c_2), \phi_1(d)\in \Hyp$, we have $s(\phi_1|_{\pi_1(P)}) = (+1,0,0)$. Then by Theorem \ref{thm_pants2},  we have $0\leqslant e(\phi_1|_{\pi_1(P)})\leqslant 1$; and by Theorem \ref{thm_goldman}, we have $-1\leqslant e(\phi_1|_{\pi_1(T)}) \leqslant 1$. Finally, by  Proposition \ref{prop_additivity} that $n = e(\phi_1|_{\pi_1(P)})+e(\phi_1|_{\pi_1(T)})$, we have $-1\leqslant n\leqslant 2.$
    \medskip

    
         It remains to show that any non-empty $\HPns(\Sigma_{1,2})$ is connected. 
         Let $\phi,\psi \in \HPns(\Sigma_{1,2})$. 
         As shown in the previous paragraph, $\phi|_{\pi_1(P)}$, $\phi|_{\pi_1(T)}$, $\psi|_{\pi_1(P)}$ and $\psi|_{\pi_1(T)}$ are non-abelian.
         Hence, by Lemma \ref{lem_inthyptorus}, we can assume that $\phi(d)$ and $\psi(d)$ are hyperbolic, and $\phi|_{\pi_1(P)}, \psi|_{\pi_1(P)}\in \HP(P)$. 
        Under this assumption, for each $n$ satisfying $-1\leqslant n\leqslant 2$,
        we will find a path in $\HPns(\torustwo)$ connecting $\phi$ to $\psi$.
        \smallskip
        
        In the case that $n = -1$ or $n = 2$, we use  Lemma \ref{lem_sameEuler_torus}.
        If $n = -1$, then we claim that $e(\phi|_{\pi_1(P)}) = e(\psi|_{\pi_1(P)})= 0$. 
        Indeed, as $s(\phi|_{\pi_1(P)}) = (+1,0,0)$, Theorem \ref{thm_pants2} implies that either $e(\phi|_{\pi_1(P)}) = 0$ or $e(\phi|_{\pi_1(P)}) = 1$. If $e(\phi|_{\pi_1(P)}) = 1$, then by Proposition \ref{prop_additivity}, $e(\phi|_{\pi_1(T)})= -2$, contradicting Theorem \ref{thm_goldman}. 
        Therefore, we have $e(\phi|_{\pi_1(P)}) = 0$. The proof for $e(\psi|_{\pi_1(P)}) = 0$ follows verbatim. Then by Lemma \ref{lem_sameEuler_torus}, there exists a path in  $\HPns(\torustwo)$ connecting $\phi$ and $\psi$. If $n = 2$, similar to the case $n = -1$, we can show $e(\phi|_{\pi_1(P)}) = e(\psi|_{\pi_1(P)})= 1$ using Theorem \ref{thm_goldman}, Theorem \ref{thm_pants2}, and Proposition \ref{prop_additivity}. Then by Lemma \ref{lem_sameEuler_torus}, there is a path in $\HPns(\torustwo)$ connecting $\phi$ and $\psi$.
\smallskip

        In the case that  $n = 0$ or $n = 1$, we will find the path as follows. 
        We will first construct paths $\path$ and $\{\psi_t\}\interval$ in $\HPns(\torustwo)$, respectively starting from $\phi$ and $\psi$ and with $e(\phi_1|_{\pi_1(P)}) = e(\psi_1|_{\pi_1(P)}) = 1$. Then by Lemma \ref{lem_sameEuler_torus}, there is a path connecting $\phi_1$ and $\psi_1$; and the composition of these paths connects  $\phi$ to $\psi$. 
        
       Below we construct the path $\path$; and the construction of $\{\psi_t\}\interval$ follows verbatim. If $e(\phi|_{\pi_1(P)}) = 1$, then we let $\{\phi_t\}_{t\in [0,1]}$ be the constant path $\phi_t \equiv \phi$. 
        If $e(\phi|_{\pi_1(P)}) = 0$,
        then we will first construct a path $\{\phi_t: \pi_1(P)\to \psl\}_{t\in [0,1]}$ using Lemma \ref{lem_plpchar}, and then extend it to $\pi_1(\torustwo)$ using Proposition \ref{prop_plpliftcommu}. The path $\{\phi_t: \pi_1(P)\to \psl\}_{t\in [0,1]}$ is constructed as follows.
         Recall that the character of $\phi|_{\pi_1(P)}$ is defined by the triples
         $$\chi(\phi|_{\pi_1(P)}) = \Big(\mathrm{tr}\big(\widetilde{\phi(c_1)}\big), 
    \mathrm{tr}\big(\widetilde{\phi(c_2)}\big), 
    \mathrm{tr}\big(\widetilde{\phi(c_1)}\widetilde{\phi(c_2)}\big)
     \Big)$$ 
     in $\mathbb{R}^3$, where $\widetilde{\phi(c_1)}$ is the lift of $\phi(c_1)$ in $\Par^+_0$ and 
    $\widetilde{\phi(c_2)}$ is the lift of $\phi(c_2)$ in $\Hyp_0$. 
    Since $\phi(d) \in \Hyp$ and $e(\phi|_{\pi_1(P)}) = 0$, $\widetilde{\phi(c_1)}\widetilde{\phi(c_2)} = \widetilde{\phi(d)^{-1}}\in \Hyp_0$, and hence
    $\chi(\phi|_{\pi_1(P)}) = (2,v,w)\in \{2\}\times (2,\infty)^2$. 
    Define the straight line segment
    $\big\{(2,v,w_t)\big\}\interval$ in $\mathbb{R}^3$, where $w_t= (1-2t)w$ connecting $w$ to $-w$. 
    Since
    $v>2$, the path $\big\{(2,v,w_t)\big\}\interval$ lie in
    $\mathbb{R}^3
        \setminus \big([-2,2]^3\cap \kappa^{-1}([-2,2])\big)$, where $\kappa$ is as defined in Proposition \ref{prop_char}. 
    Let $A_{0}$ and $B_{0}$ respectively
    be the projections of $\widetilde{\phi(c_1)}$ and $\widetilde{\phi(c_2)}$ to $\SL$, noting that the image of $(A_{0}, B_{0})$ under the character map $\chi$ equals $\chi(\phi|_{\pi_1(P)}) = (2,v, w)$.
    Since $\phi(c_1)\in \Par^+$ and $\phi(c_2)\in \Hyp$, $A_{0}$ and $B_{0}$ do not commute. Hence, by Lemma \ref{lem_plpchar}, 
        there exists a path $\{(A_{t}, B_{t})\}\interval$ of non-commuting pairs of elements of $\SL$ starting from  $(A_{0}, B_{0})$, and having 
        $\chi(A_{t}, B_{t}) = (2, v,w_t)$ for all $t\in [0,1]$. 
        Let $\pm A_{t}$ and $\pm B_{t}$ respectively be the projections of 
        $A_{t}$ and $B_{t}$ to $\psl$, and define the representation $\phi_t:\pi_1(P)\to \psl$ by letting $\big(\phi_t(c_1), \phi_t(c_2)\big) \doteq (\pm A_{t}, \pm B_{t})$. This defines the path $\{\phi_t\}\interval$ of non-abelian representations starting from $\phi_0 = \phi|_{\pi_1(P)}$ that satisfies $\chi(\phi_t) = \chi(A_{t}, B_{t}) = (2,v,w_t)$ for all $t\in [0,1]$. In particular, $\chi(\phi_1) =  (2,v,-w)\in \{2\}\times (2,\infty)\times (-\infty, -2)$, hence $e(\phi_1)$ is odd. As a consequence, $e(\phi_1)$ is either $1$ or $-1$. We will show that $e(\phi_1) = 1$ as follows.
        Since $\tr A_t = 2$ for all $t\in [0,1]$, the path $\{\phi_t(c_1)\}\interval$ in $\psl$ is contained in $\Par\cup \{\pm \mathrm{I}\}$. It follows that this path never intersects $\pm\mathrm{I}$, as otherwise $A_{t} = \mathrm{I}$ for some $t\in [0,1]$ and would commute with $B_{t}$, contradicting the non-abelianness of $\phi_t$. Since $\phi_0(c_1) = \phi(c_1)$ lies in $\Par^+$ which is connected, we have
        $\phi_t(c_1)\in \Par^+$
        for all $t\in [0,1]$, and in particular, $\phi_1(c_1)\in \Par^+$.
        As $v>2$ and $-z<-2$, we have $\phi_1(c_2), \phi_1(d)\in \Hyp$, and hence the sign $s(\phi_1) = (+1,0,0)$. Consequently, Theorem \ref{thm_pants2} implies that $e(\phi_1) = 1$.
         
         The path $\{\phi_t: \pi_1(P)\to \psl\}\interval$ is extended  to  $\pi_1(\Sigma_{1,2})$ as follows. 
         Consider the path $\big\{z^n\big(\widetilde{\phi_t(c_1)}
         \widetilde{\phi_t(c_2)}\big)^{-1}\big\}\interval$ in $\univcover$, where $z$ is the generator of the center of $\univcover$. 
         We first claim that this path lies in the image of the lifted commutator map. Indeed, as $e(\phi_0) = 0$, we have $\widetilde{\phi_0(c_1)}
         \widetilde{\phi_0(c_2)}\in \Hyp_0$; and as its inverse lying in the same one-parameter subgroup of $\univcover$, $\big(\widetilde{\phi_0(c_1)}
         \widetilde{\phi_0(c_2)}\big)^{-1}\in \Hyp_0$. 
         Similarly, as $e(\phi_1) = 1$, we have $\widetilde{\phi_1(c_1)}
         \widetilde{\phi_1(c_2)}\in \Hyp_1 = z\Hyp_0$ and 
         $\big(\widetilde{\phi_1(c_1)}
         \widetilde{\phi_1(c_2)}\big)^
         {-1}\in z^{-1}\Hyp_0 =  \Hyp_{-1}$. 
         Therefore, 
        the path 
         $\big\{z^n\big(\widetilde{\phi_t(c_1)}
         \widetilde{\phi_t(c_2)}\big)^{-1}\big\}\interval$ connects 
         $z^n\big(\widetilde{\phi_0(c_1)}
         \widetilde{\phi_0(c_2)}\big)^
         {-1}\in \Hyp_n$ to $z^n\big(\widetilde{\phi_1(c_1)}
         \widetilde{\phi_1(c_2)}\big)^
         {-1} \in \Hyp_{n-1}$. In the case $n = 0$, the path $\big\{z^n\big(\widetilde{\phi_{t_1}(c_1)}
         \widetilde{\phi_{t_1}(c_2)}\big)^{-1}\big\}\interval = \big\{\big(\widetilde{\phi_{t_1}(c_1)}
         \widetilde{\phi_{t_1}(c_2)}\big)^{-1}\big\}\interval$ connects $\Hyp_0$ to $\Hyp_{-1}$.
         Since the path $\{w_t\}\interval$ of traces is a straight line segment connecting $w > 2$ to $-w < -2$, there exists a  unique $t_1\in (0,1)$ such that $w_{t_1} = -2$, and a unique $t_2\in (0,1)$ such that $w_{t_2} = 2$. Having the trace $-2$, $\big(\widetilde{\phi_{t_1}(c_1)}
         \widetilde{\phi_{t_1}(c_2)}\big)^{-1}$ is adjacent to $\Hyp_{-1}$ and $\Ell_{-1}$;
         and by Lemma \ref{lem_evimage}, $\big(\widetilde{\phi_{t_1}(c_1)}
         \widetilde{\phi_{t_1}(c_2)}\big)^{-1}\neq z^{-1}$. Therefore,  $\big(\widetilde{\phi_{t_1}(c_1)}
         \widetilde{\phi_{t_1}(c_2)}\big)^{-1}\in \Par^+_{-1}$. Similarly, $\big(\widetilde{\phi_{t_2}(c_1)}
         \widetilde{\phi_{t_2}(c_2)}\big)^{-1}$ is adjacent to $\Ell_{-1}$ and $\Hyp_0$;
         and as $\phi_{t_2}$ is non-abelian, 
         $\phi_{t_2}(c_1)\phi_{t_2}(c_2)\neq \pm \mathrm I$ and hence 
         $\big(\widetilde{\phi_{t_2}(c_1)}
         \widetilde{\phi_{t_2}(c_2)}\big)^{-1}\neq \mathrm I$. Therefore, $\big(\widetilde{\phi_{t_2}(c_1)}
         \widetilde{\phi_{t_2}(c_2)}\big)^{-1}\in \Par^-_0$. 
        As a consequence, 
         the path $\{\big(\widetilde{\phi_0(c_1)}
         \widetilde{\phi_0(c_2)}\big)^
         {-1}\}\interval$ lies in $\Hyp_{-1}\cup \Par^+_{-1}\cup \Ell_{-1}\cup \Par^-_0\cup \Hyp_0$. In the case $n = 1$, the path $\big\{z^n\big(\widetilde{\phi_{t_1}(c_1)}
         \widetilde{\phi_{t_1}(c_2)}\big)^{-1}\big\}\interval = \big\{z\big(\widetilde{\phi_{t_1}(c_1)}
         \widetilde{\phi_{t_1}(c_2)}\big)^{-1}\big\}\interval$ connects $\Hyp_1$ to $\Hyp_0$, having the trace $-w_t$ for each $t\in [0,1]$. Similar to the previous case that $n = 0$,
         $z\big(\widetilde{\phi_{t_1}(c_1)}
         \widetilde{\phi_{t_1}(c_2)}\big)^{-1}\in \Par^-_1$ as it is adjacent to $\Ell_1$ and $\Hyp_1$, and $z\big(\widetilde{\phi_{t_2}(c_1)}
         \widetilde{\phi_{t_2}(c_2)}\big)^{-1}\in \Par^+_0$ as it is adjacent to $\Hyp_0$ and $\Ell_1$, and 
         the path $\{z\big(\widetilde{\phi_0(c_1)}
         \widetilde{\phi_0(c_2)}\big)^
         {-1}\}\interval$ lies in $\Hyp_0\cup \Par^+_0\cup \Ell_1\cup \Par^-_1\cup \Hyp_1$. 
         Therefore, in both of these two cases of $n$, by Theorem \ref{thm_liftcommu}, the path 
          $\big\{z^n\big(\widetilde{\phi_t(c_1)}
         \widetilde{\phi_t(c_2)}\big)^{-1}\big\}\interval$ 
         lies in the image of the lifted commutator, proving the claim. Next, by Proposition \ref{prop_plpliftcommu}, there exists a path $\{(\pm A'_t, \pm B'_t)\}\interval$ of non-commuting pairs of elements of $\psl$ starting from  $(\pm A'_{0}, \pm B'_{0}) = (\phi(a), \phi(b))$, and having 
        the lifted commutator $[\widetilde{A'_t}, \widetilde{B'_t}] = z^n\big(\widetilde{\phi_t(c_1)}
         \widetilde{\phi_t(c_2)}\big)$ for all $t\in [0,1]$. 
         Finally, for each $t\in [0,1]$, we let $\big(\phi_t(a), \phi_t(b)\big)\doteq (\pm A'_t, \pm B'_t)$, 
         extending the path $\path$ from $\pi_1(P)$ to $\pi_1(\torustwo)$. For all $t\in [0,1]$, $\phi_t(c_1)\in \Par^+$; and as $\tr\widetilde{\phi_t(c_2)} = v>2$, $\phi_t(c_2)\in \Hyp$. As a consequence, the path $\path$ is contained in $\HPns(\torustwo)$.
        \end{proof}


    Next, we prove Theorem \ref{thm_main1} for $\Sigma = \Sigma_{1,2}$, which is stated as Theorem \ref{thm_torustype2} below. 
    \begin{theorem}\label{thm_torustype2}
        Let $n\in \mathbb{Z}$ and let $s\in\{\pm 1\}^2$.
        Then $\mathcal{R}^s_n(\torustwo)$ is non-empty if and only if 
        $$\chi(\Sigma_{1,2}) + p_+(s)\leqslant n \leqslant -\chi(\Sigma_{1,2}) - p_-(s).$$ 
        Moreover, each non-empty $\mathcal{R}^s_n(\torustwo)$ is connected.
    \end{theorem}
    \begin{proof}
        We first observe that, in all the four possible  cases $(+1, +1),$ $ (+1, -1),$ $ (-1,+1)$ and $(-1,-1)$ of $s$, the inequality in the result is equivalent to 
         $\frac{1}{2}(s_1+s_2) - 1 \leqslant n \leqslant \frac{1}{2}(s_1+s_2) + 1$.
         \medskip

          We begin by proving that, if $\frac{1}{2}(s_1+s_2) -1 \leqslant n \leqslant \frac{1}{2}(s_1+s_2) + 1$, then $\mathcal{R}^s_n(\torustwo)$ is non-empty. To find a representation $\phi$ in $\mathcal{R}^s_n(\torustwo)$, we will define $\phi$ on  $\pi_1(P)$ first, then extend it to $\pi_1(\torustwo)$. Let $m = \frac{1}{2}(s_1+s_2)$.
          As the fundamental group $\pi_1(P)$ has the preferred peripheral elements $c_1,c_2$ and $d$, Theorem \ref{thm_pants2} implies that there exists a representation $\phi\in \HP_m^{(s_1,s_2,0)}(P)$, having $\phi(c_1)\in\Par^{sgn(s_1)}$, $\phi(c_2)\in\Par^{sgn(s_2)}$, and $\phi(d)\in\Hyp$.
          Let $\widetilde{\phi(d)}$ be the lift of $\phi(d)$ in $\Hyp_{n-m}$. As $n-m\in \{-1,0,1\}$, by Theorem \ref{thm_liftcommu}, $\Hyp_{n-m}$ is in the image of the lifted commutator map, hence there exists a $\psl$-pair $(\pm A, \pm B)$ whose lifted commutator equals $\widetilde{\phi(d)}$. We let $\big(\phi(a), \phi(b)\big) \doteq (\pm A, \pm B)$.
       As the commutator $[\phi(a), \phi(b)]$ equals the projection $\phi(d)$ of $\widetilde{\phi(d)}$, this extends the representation $\phi$ from $\pi_1(P)$ to $\pi_1(\torustwo)$. 
       Since the evaluation map $\ev: \mathcal{W}(T)\to \displaystyle\bigcup_{k = -1}^1\Hyp_k$
       sends $\phi|_{\pi_1(T)}$ into $\Hyp_{n-m}$, we have $e(\phi|_{\pi_1(T)}) = n-m$;
       and since $e(\phi|_{\pi_1(P)}) = m$, by Proposition \ref{prop_additivity}, $e(\phi)= n$.
       Moreover, as $\phi(c_1)\in \Par^{sgn(s_1)}$ and $\phi(c_2)\in \Par^{sgn(s_2)}$, we have $s(\phi) = (s_1,s_2)$. As a consequence, $\phi\in \mathcal{R}^s_n(\torustwo)$.
          \medskip
        

        We now prove that, if $\mathcal{R}^s_n(\torustwo)$ is non-empty, then $\frac{1}{2}(s_1+s_2) -1 \leqslant n \leqslant \frac{1}{2}(s_1+s_2) + 1$. For a representation $\phi\in \mathcal{R}^s_n(\torustwo)$, we first construct a path $\path$ in $\mathcal{R}^s_n(\torustwo)$ connecting $\phi$ to a $\phi_1\in \mathcal{R}^s_n(\torustwo)$ with $\phi_1(d)$ hyperbolic. 
        Indeed, by Theorem \ref{thm_nonabel}, we can assume that $\phi|_{\pi_1(P)}$ and $\phi|_{\pi_1(T)}$ are non-abelian. Then by Lemma \ref{lem_inthyptorus}, there exists a path $\path$ starting from $\phi_0 = \phi$ with $\phi_1(d)$ hyperbolic. Moreover, for each $t\in [0,1]$ and each $i\in\{1,2\},$ $\phi_t(c_i)$ is conjugate to $\phi(c_i)$. Therefore, the signs $s(\phi_t) = s(\phi)$ and the relative Euler classes $e(\phi_t) = e(\phi)$, and the path $\path$ lies in $\mathcal{R}^s_n(\torustwo).$  Next, we deduce from $e(\phi_1) = n$ that $\frac{1}{2}(s_1+s_2) -1 \leqslant n \leqslant \frac{1}{2}(s_1+s_2) + 1$. Indeed, 
        since $\phi_1(d)$ is hyperbolic, we have $\phi_1|_{\pi_1(P)}\in \HP(P)$ and $\phi_1|_{\pi_1(T)}\in \mathcal{W}(T)$; and since $\phi_1(c_1)\in\Par^{sgn(s_1)}$ and $\phi_1(c_2)\in\Par^{sgn(s_2)}$, we have $s(\phi_1|_{\pi_1(P)}) = (s_1,s_2,0)$.
    Then by Theorem \ref{thm_pants2}, we have $e(\phi_1|_{\pi_1(P)}) = \frac{1}{2}(s_1+s_2)$; and by Theorem \ref{thm_goldman}, we have $-1\leqslant e(\phi_1|_{\pi_1(T)}) \leqslant 1$. Finally, by Proposition \ref{prop_additivity} that  $n = e(\phi_1|_{\pi_1(P)})+e(\phi_1|_{\pi_1(T)})$, we have $\frac{1}{2}(s_1+s_2) - 1 \leqslant n\leqslant \frac{1}{2}(s_1+s_2) + 1$. 
        \medskip
        
         It remains to show that any non-empty $\HPns(\Sigma_{1,2})$ is connected. 
         Let $\phi,\psi \in \HPns(\Sigma_{1,2})$. 
         By Theorem \ref{thm_nonabel}, we can assume that $\phi|_{\pi_1(P)},$ $\phi|_{\pi_1(T)},$ $\psi|_{\pi_1(P)}$ and $\psi|_{\pi_1(T)}$ are non-abelian; and by Lemma \ref{lem_inthyptorus}, we can assume that $\phi(d)$ and $\psi(d)$ are hyperbolic, and hence $\phi|_{\pi_1(P)}, \psi|_{\pi_1(P)}\in \HP(P)$. 
        Then by Theorem \ref{thm_pants2}, we have
         $e(\phi|_{\pi_1(P)}) = e(\psi|_{\pi_1(P)}) = \frac{1}{2}(s_1+s_2)$, and by Lemma \ref{lem_sameEuler_torus}, there exists a path in $\HPns(\Sigma_{1,2})$ connecting $\phi$ and $\psi$.
    \end{proof}

\subsection{Four-hole sphere}\label{subsection_pantstwo}

        The fundamental group of the four-hole sphere $\Sigma_{0,4}$ has the following presentation
     $$\pi_1(\Sigma_{0,4}) = \langle 
     c_1, c_2, c_3, c_4\ |\ c_1c_2c_3c_4 \rangle,$$
     where $c_1, c_2, c_3$, and $c_4$ are the preferred peripheral elements.
     Let $P_1\cup P_2$ be a pants decomposition of $\Sigma_{0,4}$, where $P_1 \cong\Sigma_{0,3}$ with $\pi_1(P_1) = \langle c_1, c_2 \rangle$ and $P_2 \cong \Sigma_{0,3}$ with $\pi_1(P_2) = \langle c_3, c_4 \rangle$, separated by a decomposition curve $d = c_3c_4 = (c_1c_2)^{-1}$.
     \\

   To prove the results in this subsection, we need the following lemma.
      \begin{lemma}\label{lem_sameEuler_pants}
           
            For $n\in\mathbb{Z}$ and $s \in \{-1,0,+1\}^4$, let $\phi,\psi\in \HPns(\Sigma_{0,4})$.
            If $\phi(d), \psi(d)\in \Hyp$ and  $e\big(\phi|_{\pi_1(P_1)}\big) = e\big(\psi|_{\pi_1(P_1)}\big)$, then there exists a path in $\HPns(\Sigma_{0,4})$ connecting $\phi$ to $\psi$. 
        \end{lemma}
    \begin{proof}
       The fundamental group $\pi_1(P_1)$ has the preferred peripheral elements $c_1,c_2$ and $d$. As representations in $\HPns(\pantstwo)$, $\phi$ and $\psi$ send $c_1$ and $c_2$ into $\overline{\Hyp}$; and by the assumption, $\phi$ and $\psi$ send $d$ into $\Hyp$. 
       Hence, both $\phi|_{\pi_1(P_1)}$ and $\psi|_{\pi_1(P_1)}$ are in $\mathrm{HP}(P_1)$ with sign $s'= (s_1,s_2, 0)\in \{-1,0,+1\}^3$. Let $n' = e\big(\phi|_{\pi_1(P_1)}\big) = e\big(\psi|_{\pi_1(P_1)}\big)$. Then both 
       $\phi|_{\pi_1(P_1)}$ and $\psi|_{\pi_1(P_1)}$ are in $\HP^{s'}_{n'}(P_1)$. By Theorem \ref{thm_goldman} and Theorem \ref{thm_pants2}, $\HP^{s'}_{n'}(P_1)$ is connected, hence there exists a path $\{\phi_t: \pi_1(P_1)\to \psl\}\interval$ in $\HP^{s'}_{n'}(P_1)$ connecting $\phi_0= \phi|_{\pi_1(P_1)}$ to $\phi_1= \psi|_{\pi_1(P_1)}$. Moreover, 
       by Proposition \ref{prop_additivity}, $e\big(\phi|_{\pi_1(P_2)}\big) = e\big(\psi|_{\pi_1(P_2)}\big) = n - n'$. Therefore, Proposition \ref{prop_sameEuler} completes the proof by letting $\Sigma' = P_1$.
    \end{proof}
    
    For $\Sigma = \Sigma_{0,4}$, Theorem \ref{thm_main4} considers the signs $s\in \{-1,0,+1\}^4$ with $p_0(s)\in \{1,2,3,4\}$; and the case that $p_0(s) = 4$ is proved in Theorem \ref{thm_goldman}.
    The following Theorem \ref{thm_pantstype1} and  Theorem \ref{thm_pantstype2}
    prove Theorem \ref{thm_main4} for $\Sigma=\Sigma_{0,4}$ respectively for the cases that  $p_0(s) = 3$ and $p_0(s) \in \{1,2\}$.

     \begin{theorem}\label{thm_pantstype1}
        Let $n\in \mathbb{Z}$, and let $s\in\{-1,0,+1\}^4$ with $p_0(s) \geqslant 3$.
        Then $\HPns(\pantstwo)$ is non-empty if and only if 
        $$\chi(\Sigma_{0,4}) + p_+(s)\leqslant n \leqslant -\chi(\Sigma_{0,4}) - p_-(s).$$ 
        Moreover, each non-empty $\HPns(\pantstwo)$ is connected.

    \end{theorem}
    \begin{proof}
        As the case that $p_0(s) = 4$ is included in Theorem \ref{thm_goldman}, we only need to consider the case that $p_0(s) = 3$.
        We will prove the result for the case that $p_+(s)=1$ and $p_-(s)=0.$ Then the result for the remaining case that $p_-(s)=1$ and $p_+(s)=0$ follows from the previous case and Proposition \ref{prop_pglpsl}.
    \\
    

     For the case that $p_+(s)=1$ and $p_-(s)=0,$ the equality in the result becomes $-1\leqslant n\leqslant 2$; and up to a permutation of the peripheral elements we can assume that  $s=(+1,0,0,0)$. We first prove that, if $-1\leqslant n \leqslant 2$, then $\HPns(\Sigma_{0,4})$ is non-empty. To find a representation $\phi$ in $\HPns(\pantstwo)$, we will define $\phi$ on $\pi_1(P_1)$ first , then extend it to $\pi_1(\pantstwo)$. 
          If $n = -1$, then let $m = 0$; and if $0\leqslant n\leqslant 2$, then let $m = 1$. 
          As the fundamental group $\pi_1(P_1)$ has the preferred peripheral elements $c_1,c_2$ and $d$, by Theorem \ref{thm_pants2}, there exists a representation $\phi\in \HP_m^{(+1,0,0)}(P_1)$, sending $c_1$ into $\Par^+$ and sending $c_2, d$ into $\Hyp$. 
          Since $\phi(d)\in \Hyp$, there exists a unique lift $\widetilde{\phi(d)}$ in $\Hyp_{n-m}$.  As $n-m\in \{-1,0,1\}$, 
          by Proposition \ref{prop_evimage}, the evaluation map $\ev: \mathcal{W}_{n-m}(P_2)\to \Hyp_{n-m}$ is surjective. Hence,
          there exists a representation $\psi\in \mathcal{W}_{n-m}(P_2)$ such that $\ev(\psi) = \widetilde{\phi(d)}$. As $\psi$ and $\phi$ agree at the common boundary $d$ of $P_1$ and $P_2$, letting $\phi|_{\pi_1(P_2)}\doteq \psi$ extends the representation $\phi$ from $\pi_1(P_1)$ to $\pi_1(\pantstwo)$.
          Moreover, since $e(\phi|_{\pi_1(P_1)}) = m$ and $e(\phi|_{\pi_1(P_2)}) = n-m$, by Proposition \ref{prop_additivity}, $e(\phi) = m + (n-m) = n$; and since $\phi(c_1)\in \Par^+$ and $\phi(c_i)\in \Hyp$ for $i\in \{2,3,4\}$, we have $s(\phi) = (+1,0,0,0)$. As a consequence,  $\phi\in \HPns(\pantstwo)$.
   \medskip
   
          
          Next we prove that, if $\HPns(\Sigma_{0,4})$ is non-empty, then $-1\leqslant n\leqslant 2$. For a representation $\phi\in \HPns(\Sigma_{0,4})$, we will  first construct a path $\path$ in $\HPns(\pantstwo)$ connecting $\phi$ to a $\phi_1\in \HPns(\pantstwo)$ with $\phi_1(d)$ hyperbolic. 
        Indeed, by Theorem \ref{thm_nonabel}, we can assume that $\phi|_{\pi_1(P_1)}$ and $\phi|_{\pi_1(P_2)}$ are non-abelian, and that $\phi(c_4)$ is hyperbolic. Hence, by Lemma \ref{lem_inthyppants}, there exists a path $\path$ starting from $\phi_0 = \phi$ and with $\phi_1(d)$ hyperbolic. Moreover, for each $t\in [0,1]$ and each $i\in\{1,2,3,4\},$ $\phi_t(c_i)$ is conjugate to $\phi(c_i)$. Hence the signs $s(\phi_t) = s(\phi)$ and the relative Euler classes $e(\phi_t) = e(\phi)$, and the path $\path$ lies in $\HPns(\pantstwo).$  Next, we will deduce from $e(\phi_1) = n$ that $-1\leqslant n\leqslant 2$.  Indeed, since $\phi_1(d)$ is hyperbolic, $\phi_1|_{\pi_1(P_1)}\in \HP(P_1)$ and $\phi_1|_{\pi_1(P_2)}\in \mathcal{W}(P_2)$; and since
        $\phi_1(c_1)\in \Par^+$ and $\phi_1(c_2), \phi_1(d)\in \Hyp$, we have $s(\phi_1|_{\pi_1(P_1)}) = (+1,0,0)$. Then by Theorem \ref{thm_pants2},  we have $0\leqslant e(\phi_1|_{\pi_1(P_1)})\leqslant 1$; and by Theorem \ref{thm_goldman}, we have $-1\leqslant e(\phi_1|_{\pi_1(P_2)}) \leqslant 1$. Finally, by  Proposition \ref{prop_additivity} that $n = e(\phi_1|_{\pi_1(P_1)})+e(\phi_1|_{\pi_1(P_2)})$, we have $-1\leqslant n\leqslant 2.$
    \medskip
    
             It remains to show that any non-empty $\HPns(\Sigma_{0,4})$ is connected. 
         Let $\phi,\psi \in \HPns(\Sigma_{0,4})$. 
         By Theorem \ref{thm_nonabel}, we can assume that $\phi|_{\pi_1(P_1)}$, $\phi|_{\pi_1(P_2)}$, $\psi|_{\pi_1(P_1)}$ and $\psi|_{\pi_1(P_2)}$ are non-abelian; and by Lemma \ref{lem_inthyppants}, we can assume that $\phi(d)$ and $\psi(d)$ are hyperbolic, and $\phi|_{\pi_1(P_1)}, \psi|_{\pi_1(P_1)}\in \HP(P_1)$. Under these assumptions, for each $n$ satisfying $-1\leqslant n\leqslant 2$,
        we will find a path in $\HPns(\pantstwo)$ connecting $\phi$ to $\psi$.
        \smallskip
        
        In the case that $n = -1$ or $n = 2$, we use  Lemma \ref{lem_sameEuler_pants}.
        If $n = -1$, then we claim that $e(\phi|_{\pi_1(P_1)}) = e(\psi|_{\pi_1(P_1)})= 0$. 
        Indeed, as $s(\phi|_{\pi_1(P_1)}) = (+1,0,0)$, Theorem \ref{thm_pants2} implies that either $e(\phi|_{\pi_1(P_1)}) = 0$ or $e(\phi|_{\pi_1(P_1)}) = 1$. If $e(\phi|_{\pi_1(P_1)}) = 1$, then by Proposition \ref{prop_additivity}, $e(\phi|_{\pi_1(P_2)})= -2$, contradicting Theorem \ref{thm_goldman}. 
        Therefore, we have $e(\phi|_{\pi_1(P_1)}) = 0$. The proof for $e(\psi|_{\pi_1(P_1)}) = 0$ follows verbatim. Then by Lemma \ref{lem_sameEuler_pants}, there exists a path in  $\HPns(\pantstwo)$ connecting $\phi$ and $\psi$. If $n = 2$, similar to the case $n = -1$, we can show $e(\phi|_{\pi_1(P_1)}) = e(\psi|_{\pi_1(P_1)})= 1$ using Theorem \ref{thm_goldman}, Theorem \ref{thm_pants2}, and Proposition \ref{prop_additivity}. Then by Lemma \ref{lem_sameEuler_pants}, there is a path in $\HPns(\pantstwo)$ connecting $\phi$ and $\psi$.
\smallskip

        In the case that  $n = 0$ or $n = 1$, we will find the path as follows. 
        We will first construct paths $\{\phi_t\}_{t\in [0,2]}$ and $\{\psi_t\}_{t\in [0,2]}$ in $\HPns(\pantstwo)$, respectively starting from $\phi$ and $\psi$ and with $e(\phi_2|_{\pi_1(P_2)}) = e(\psi_2|_{\pi_1(P_2)}) = 0$. Then by Lemma \ref{lem_sameEuler_pants}, there is a path connecting $\phi_2$ and $\psi_2$; and the composition of these paths connects  $\phi$ to $\psi$. 

    Below we construct the path $\{\phi_t\}_{t\in [0,2]}$; and the construction of $\{\psi_t\}_{t\in [0,2]}$ follows verbatim. If $e(\phi|_{\pi_1(P_2)}) = 0$, then we let $\{\phi_t\}_{t\in [0,2]}$ be the constant path $\phi_t \equiv \phi$. 
        If otherwise that $e(\phi|_{\pi_1(P_2)}) \in \{\pm 1\}$,
        then the path $\{\phi_t\}_{t\in[0,2]}$ is constructed as follows. 
        First, using Lemma \ref{lem_plpchar}, we will construct paths $\{\phi_{1,t}\}_{t\in[1,2]}$ and $\{\phi_{2,t}\}_{t\in[1,2]}$ of representations of $\pi_1(P_1)$ and $\pi_1(P_2)$, respectively. Next, we will find a path $\{g_t\}_{t\in [1,2]}\subset \psl$ using Lemma \ref{lem_conjugacypath}, and
        glue $\{\phi_{1,t}\}_{t\in[1,2]}$ and $\{g_t\phi_{2,t}g_t^{-1}\}_{t\in[1,2]}$ along the decomposition curve $d$ to define the path $\{\phi_t\}_{t\in[1,2]}$. 
        Finally, we will construct a path $\path$ connecting $\phi$ to $\phi_1$, and compose it with $\{\phi_t\}_{t\in[1,2]}$.
        
         The paths $\{\phi_{1,t}\}_{t\in[1,2]}$ and $\{\phi_{2,t}\}_{t\in[1,2]}$ are constructed  as follows. To use Lemma \ref{lem_plpchar}, we first compute the relative Euler classes of the restrictions $\phi|_{\pi_1(P_1)}$ and $\phi|_{\pi_1(P_2)}$, which determine the signs of the third component of the characters $\chi(\phi|_{\pi_1(P_1)})$ and $\chi(\phi|_{\pi_1(P_2)})$. 
         
         In the case that $n = 0$, we show that $e(\phi|_{\pi_1(P_1)}) =1$ and $e(\phi|_{\pi_1(P_2)}) = -1$.
         Indeed, as $s(\phi|_{\pi_1(P_1)}) = (+1,0,0)$, by Theorem \ref{thm_pants2},
         either $e(\phi|_{\pi_1(P_1)}) = 0$ or $e(\phi|_{\pi_1(P_1)}) =1$.
         Then by Proposition \ref{prop_additivity}, we have either $e(\phi|_{\pi_1(P_1)})=0$ and $e(\phi|_{\pi_1(P_2)}) =0$, or $e(\phi|_{\pi_1(P_1)})=1$ and $e(\phi|_{\pi_1(P_2)}) = -1$. 
         Since $e(\phi|_{\pi_1(P_2)})\big)\in \{\pm 1\}$, we conclude that $e(\phi|_{\pi_1(P_1)})=1$ and $e(\phi|_{\pi_1(P_2)})= -1$.
         
         In the case that $n = 1$, we show that $e(\phi|_{\pi_1(P_1)}) =0$ and $e(\phi|_{\pi_1(P_2)}) = 1$.
         Indeed, as $s(\phi|_{\pi_1(P_1)}) = (+1,0,0)$, by Theorem \ref{thm_pants2},
         either $e(\phi|_{\pi_1(P_1)}) = 0$ or $e(\phi|_{\pi_1(P_1)}) =1$.
         Then by Proposition \ref{prop_additivity}, we have either $e(\phi|_{\pi_1(P_1)})=0$ and $e(\phi|_{\pi_1(P_2)})=1$, or $e(\phi|_{\pi_1(P_1)})=1$ and $e(\phi|_{\pi_1(P_2)})= 0$. 
         Since $e(\phi|_{\pi_1(P_2)})\big)\in \{\pm 1\}$, we conclude that $e(\phi|_{\pi_1(P_1)})=0$ and $e(\phi|_{\pi_1(P_2)})= 1$. 

         Let $w\doteq -|\tr\big(\widetilde{\phi(c_1)}\widetilde{\phi(c_2)}\big)|\in (-\infty, -2)$.
         Since $e(\phi|_{\pi_1(P_1)}) = 1$ when $n = 0$ and $e(\phi|_{\pi_1(P_1)}) = 0$ when $n = 1$, we have
         $$\chi\big(\phi|_{\pi_1(P_2)}\big) = (2,v_1, (-1)^n w)$$
         that lies in $\{2\}\times (2,\infty)\times (-\infty, -2)$ when $n = 0$, and in $\{2\}\times (2,\infty)^2$ when $n = 1$.
        Moreover, since
         $e(\phi|_{\pi_1(P_2)})$ is odd, we have
         $$\chi\big(\phi|_{\pi_1(P_2)}\big)
         =  (u,v_2, w)\in (2,\infty)^2\times (-\infty, -2).$$
         For each $t\in [1,2]$, let $w_t\doteq w(3 - 2t)$ connecting $w$ to $-w$, and consider the straight line paths $\{(2,v_1, (-1)^nw_t)\}_{t\in [1,2]}$ and $\{(u,v_2, w_t)\}_{t\in [1,2]}$. 
         Since $v_1, v_2\in (2,\infty)$, these paths lie in 
         $\mathbb{R}^3 
         \setminus \big([-2,2]^3\cap \kappa^{-1}([-2,2])\big)$. 
        Let $A_1,\ B_1$, $C_1$ and $D_1$ respectively
        be the projections of $\widetilde{\phi(c_1)}$,  $\widetilde{\phi(c_2)}$, $\widetilde{\phi(c_3)}$ and $\widetilde{\phi(c_4)}$ to $\SL$.
        Notice that the image of $(A_1, B_1)$ under the character map $\chi$ equals $\chi(\phi|_{\pi_1(P_1)}) = (2,v_1, (-1)^nw)$, and the image of $(C_1, D_1)$ under $\chi$ equals $\chi(\phi|_{\pi_1(P_2)}) = (u,v_2, w)$.
         Moreover, as $\phi|_{\pi_1(P_1)}$ and $\phi|_{\pi_1(P_2)}$ are non-abelian, $(A_1, B_1)$ and $(C_1, D_1)$ are pairs of non-commuting $\SL$-elements.
         Hence, by Lemma \ref{lem_plpchar}, there is a path $\{(A_{t}, B_{t})\}_{t\in [1,2]}$ of non-commuting $\SL$-pairs starting from  $(A_1, B_1)$, and having 
        $\chi(A_{t}, B_{t}) = (2, v_1, (-1)^nw_t)$ for all $t\in [1,2]$; and there is a path $\{(C_{t}, D_{t})\}_{t\in [1,2]}$ of non-commuting $\SL$-pairs starting from  $(C_1, D_1)$, and having 
        $\chi(C_{t}, D_{t}) = (u, v_2,w_t)$ for all $t\in [1,2]$. 
        For each $t\in [1,2]$, let $\pm A_{t}, \pm B_{t}, \pm C_{t},\text{ and } \pm D_{t}$ respectively be the projections of 
        $A_{t}, B_{t}, C_{t},\text{ and }  D_{t}$  to $\psl$. 
        By letting $\big(\phi_{1,t}(c_1), \phi_{1,t}(c_2)\big) \doteq (\pm A_{t}, \pm B_{t})$, we define the path $\{\phi_{1,t}\}_{t\in [1,2]}$ of non-abelian representations starting from $\phi_{1,1} = \phi|_{\pi_1(P_1)}$, such that for all $t\in [0,1]$, $\chi(\phi_{1,t}) = \chi(A_{t}, B_{t}) = (2,v_1,(-1)^nw_t)$. 
        Similarly, by letting $\big(\phi_{2,t}(c_3), \phi_{2,t}(c_4)\big) \doteq (\pm C_{t}, \pm D_{t})$, we define the path $\{\phi_{2,t}\}_{t\in [1,2]}$ of non-abelian representations starting from $\phi_{2,t} = \phi|_{\pi_1(P_2)}$, such that for all $t\in [0,1]$, $\chi(\phi_{2,t}) = \chi(C_{t}, D_{t}) = (u,v_2,w_t)$. 
        Then, since $\chi(\phi_{2,2}) =  (u,v_2,-w)\in (2,\infty)^3$, we have $e(\phi_{2,2}) = 0$.

         To define the path $\{g_t\}_{t\in [1,2]}$  using Lemma \ref{lem_conjugacypath}, we first show that for each $t\in [0,1]$, the elements $z^n\big(\ev(\phi_{1,t})\big)^{-1} = 
         z^n\big(\widetilde{\phi_{1,t}(c_1)}\widetilde{\phi_{1,t}(c_2)}\big)^{-1}$ 
         and 
         $\ev(\phi_{2,t}) = 
            \widetilde{\phi_{2,t}(c_1)}\widetilde{\phi_{2,t}(c_2)}$ are conjugate in $\univcover$, so that  their projections $\phi_{1,t}(d)$ and $\phi_{2,t}(d)$ are conjugate in $\psl$. Notice that the trace $\tr\big(z^n\big(\ev(\phi_{1,t})\big)^{-1}\big) = \tr\big(\ev(\phi_{2,t})\big) = w_t$. Since $w_t$  increases monotonically in $t\in [1,2]$ from $w< -2 $ to $-w > 2$, there are unique $t_1, t_2 \in (1,2)$ such that $w_{t_1} = -2$ and $w_{t_2} = 2$.
         For $n = 0$,
         since $e(\phi_{1,1}) = 1$ and $e(\phi_{1,2}) = 0$, we have $\ev(\phi_{1,1})\in \Hyp_1$ and $\ev(\phi_{1,2})\in \Hyp_0$, and the path $\{\big(\ev(\phi_{1,t})\big)^{-1}\}_{t\in[1,2]}$ connects $\ev(\phi_{1,1})^{-1}\in \Hyp_{-1}$ to $\ev(\phi_{1,2})^{-1}\in \Hyp_0$. 
         For $n = 1$,
         since $e(\phi_{1,1}) = 0$ and $e(\phi_{1,2}) = 1$, we have $\ev(\phi_{1,1})\in \Hyp_0$ and $\ev(\phi_{1,2})\in \Hyp_1$, and the path $\{z\big(\ev(\phi_{1,t})\big)^{-1}\}_{t\in[1,2]}$ connects $z\big(\ev(\phi_{1,1})\big)^{-1}\in \Hyp_1$ to $z\big(\ev(\phi_{1,2})\big)^{-1}\in \Hyp_0$. Therefore, letting $m = (-1)^{n-1}$, we have
         $z^n\big(\ev(\phi_{1,t})\big)^{-1}\in \Hyp_m$ for  $t\in [0, t_1)$, $z^n\big(\ev(\phi_{1,t})\big)^{-1}\in \Ell_m$ for $t\in (t_1, t_2)$, and $z^n\big(\ev(\phi_{1,t})\big)^{-1}\in \Hyp_0$ for $t\in (t_2, 1]$. 
         Then $z^n\big(\ev(\phi_{1,t_1})\big)^{-1}$ is adjacent to $\Hyp_m$ and $\Ell_m$; and as $\phi_{1, t_1}$ is non-abelian, $\phi_{1, t_1}(d)\neq \pm \mathrm I$ and its lift $z^n\big(\ev(\phi_{1,t_1})\big)^{-1}\neq z^m$. Therefore,  
         $z^n\big(\ev(\phi_{1,t_1})\big)^{-1}\in \Par_m^{sgn(-m)}$.
         Similarly, as $z^n\big(\ev(\phi_{1,t_2})\big)^{-1}$ is adjacent to $\Ell_m$ and $\Hyp_0$, 
         $z^n\big(\ev(\phi_{1,t_2})\big)^{-1}\in \Par_0^{sgn(m)}\cup \{\mathrm I\}$; and as $\phi_{1, t_2}$ is non-abelian, $\phi_{1, t_2}(d)\neq \pm \mathrm I$ and its lift $z^n\big(\ev(\phi_{1,t_2})\big)^{-1}\neq \mathrm I$. Therefore, $z^n\big(\ev(\phi_{1,t_2})\big)^{-1}\in \Par^{sgn(m)}_0$. 
         On the other hand, since $e(\phi_{2,1}) = m$ and $e(\phi_{2,2}) = 0$, the path $\{\ev(\phi_{2,t})\}_{t\in[1,2]}$ connects $\ev(\phi_{2,1})\in \Hyp_m$ to $\ev(\phi_{2,2})\in \Hyp_0$, and the proof of $\ev(\phi_{2,t_1
         })\in \Par_m^{sgn(-m)}$ and $\ev(\phi_{2,t_2
         })\in \Par_0^{sgn(m)}$ follows verbatim.
         Consequently, for each $t\in [1,2]$, as elements both lying in the same connected component of $\Hyp_m\cup \Hyp_0$, 
         $\Par_m^{sgn(-m)}\cup \Par_0^{sgn(m)}$, or $\Ell_m$ with the same trace, $z^n\big(\ev(\phi_{1,t})\big)^{-1}$ and $\ev(\phi_{2,t})$ are conjugate in $\univcover$; and hence their projections $\phi_{1,t}(d)$ and $\phi_{2,t}(d)$ are conjugate in $\psl$.

        We now define $\{g_t\}_{t\in [1,2]}\subset \psl$, and
        glue the paths $\{\phi_{1,t}\}_{t\in [1,2]}$ and $\{g_t\phi_{2,t}g_t^{-1}\}_{t\in [1,2]}$ along $d$. 
        Since $\phi_{1,t}(d)$ and $\phi_{2,t}(d)$ are conjugate for all $t\in [1,2]$, by Lemma \ref{lem_conjugacypath}, there exists a path $\{g_t\}_{t\in [1,2]}$ such that 
         $\phi_{1,t}(d) = g_t \phi_{2,t}(d) g_t^{-1}$ for all $t\in [1,2]$. Then we let $\phi_t|_{\pi_1(P_1)} \doteq \phi_{1, t}$ and $\phi_t|_{\pi_1(P_2)} \doteq g_t\phi_{2,t}g_t^{-1}$ for each $t\in [1,2]$, defining the path $\{\phi_t\}_{t\in [1,2]}$.

         Finally, we construct the path $\{\phi_t\}_{t\in [0,1]}$ connecting $\phi$ and $\phi_1$. Let $\{g_t\}\interval$ be a path connecting $g_0 = \pm \mathrm I$ to $g_1$ within the one-parameter subgroup of $\psl$ generated by $g_1$. For each $t\in [0,1]$, we let $\phi_t|_{\pi_1(P_1)}\doteq \phi|_{\pi_1(P_1)}$ and $\phi_t|_{\pi_1(P_2)}\doteq g_t\phi|_{\pi_1(P_2)}g_t^{-1}$.
         As $g_1\phi(d)g_1^{-1} = g_1\phi_{2,1}(d)g_1^{-1} = \phi_{1,1}(d) = \phi(d)$, $g_1$ commutes with $\phi(d)$. Hence for all $t\in [0,1]$, we have $g_t \phi(d) g_t^{-1} = \phi(d)$. 
         This defines a path $\path$ of representations on $\pi_1(\Sigma_{0,4})$. By composing with $\{\phi_t\}_{t\in [1,2]}$,
          we have a path $\{\phi_t\}_{t\in [0,2]}$ starting from  $\phi_0 = \phi$, having  $e(\phi_2|_{\pi_1(P_2)}) = 0$.  Since for all $t\in [0,2]$ and $i\in \{1,2,3,4\}$,  $\phi_t(c_i)$ and $\phi(c_i)$ are conjugate in $\psl$, the signs $s(\phi_t) = s(\phi)$ and the relative Euler classes $e(\phi_t) = e(\phi)$. This implies that the path
          $\{\phi_t\}_{t\in [0,2]}$ lies in $\HP^{s}_n(\pantstwo)$, completing  the proof.
    \end{proof}

        \begin{theorem}\label{thm_pantstype2}
        Let $n\in \mathbb{Z}$, and let $s\in\{-1,0,+1\}^4$ with $p_0(s) = 1$ or $p_0(s) = 2$.
        Then $\HPns(\pantstwo)$ is non-empty if and only if 
        $$\chi(\Sigma_{0,4}) + p_+(s)\leqslant n \leqslant -\chi(\Sigma_{0,4}) - p_-(s).$$ 
        Moreover, each non-empty $\HPns(\pantstwo)$ is connected.
        \end{theorem}
    \begin{proof}
    	Up to a permutation of the peripheral elements, we can assume that  $s=(s_1,s_2,s_3,0)$ with $s_1, s_2\in \{\pm 1\}$ and $s_3 \in \{-1,0,+1\}$. We will prove the result for the case that $s_3 \in \{0, +1\}$. Then the result for the remaining case that $s_3 = -1$ follows from the case that $s_3 = +1$ and Proposition \ref{prop_pglpsl}. 
	In the case that $s_3 \in \{0, +1\}$, in all the four possible  cases $(+1, +1),$ $ (+1, -1),$ $ (-1,+1)$ and $(-1,-1)$ of $(s_1, s_2)$, the inequality in the result is equivalent to $\frac{1}{2}(s_1+s_2)  + s_3 - 1 \leqslant n \leqslant \frac{1}{2}(s_1+s_2) + 1.$
         \medskip

          We begin by proving that, if $\frac{1}{2}(s_1+s_2) +s_3-1 \leqslant n \leqslant \frac{1}{2}(s_1+s_2) + 1$, then $\HPns(\pantstwo)$ is non-empty. To find a representation $\phi$ in $\HPns(\pantstwo)$, we will define $\phi$ on $\pi_1(P_1)$ first, then extend it to $\pi_1(\pantstwo)$. Let $m = \frac{1}{2}(s_1+s_2)$.
          As the fundamental group $\pi_1(P_1)$ has the preferred peripheral elements $c_1,c_2$ and $d$, Theorem \ref{thm_pants2} implies that there exists a representation $\phi\in \HP_m^{(s_1,s_2,0)}(P_1)$, having $\phi(c_1)\in\Par^{sgn(s_1)}$, $\phi(c_2)\in\Par^{sgn(s_2)}$, and $\phi(d)\in\Hyp$.
          Let $\widetilde{\phi(d)}$ be the lift of $\phi(d)$ in $\Hyp_{n-m}$. 
          As $s_3\in \{0,+1\}$ and 
          $s_3 - 1\leqslant n - m \leqslant 1$, by Proposition \ref{prop_evimage} and Corollary \ref{cor_evimage}, the evaluation map $\ev: \HP_{n-m}^{(s_3,0,0)}(P_2)\to \Hyp_{n-m}$ is surjective. Hence, there exists a representation $\psi\in \HP_{n-m}^{(s_3,0,0)}(P_2)$ such that $\ev(\psi) = \widetilde{\phi(d)}$. As
          $\psi$ and $\phi$ agree at the common boundary $d$ of $P_1$ and $P_2$, letting $\phi|_{\pi_1(P_2)}\doteq \psi$ extends the representation $\phi$ from $\pi_1(P_1)$ to $\pi_1(\pantstwo)$.
      As $e(\phi|_{\pi_1(P_1)}) = m$ and $e(\phi|_{\pi_1(P_2)}) = n-m$,
       by Proposition \ref{prop_additivity}, $e(\phi)= m + (n-m) = n$. 
       Moreover, we have $\phi(c_1)\in \Par^{sgn(s_1)}$, $\phi(c_2)\in \Par^{sgn(s_2)}$, and $\phi(c_4)\in \Hyp$; and have $\phi(c_3)\in \Hyp$ when $s_3 = 0$, and $\phi(c_3)\in \Par^+$ when $s_3 = 1$. This implies that the sign $s(\phi) = (s_1,s_2,s_3,0)$. As a consequence, $\phi\in \HPns(\pantstwo)$.
          \medskip
        

        We now prove that, if $\HPns(\pantstwo)$ is non-empty, then $\frac{1}{2}(s_1+s_2) + s_3-1 \leqslant n \leqslant \frac{1}{2}(s_1+s_2) + 1$. For a representation $\phi\in \HPns(\pantstwo)$, we first construct a path $\path$ in $\HPns(\pantstwo)$ connecting $\phi$ to a $\phi_1\in \HPns(\pantstwo)$ with $\phi_1(d)$ hyperbolic. 
        Indeed, by Theorem \ref{thm_nonabel}, we can assume that $\phi|_{\pi_1(P_1)}$ and $\phi|_{\pi_1(P_2)}$ are non-abelian, and $\phi(c_4)$ is hyperbolic. Hence, by Lemma \ref{lem_inthyppants}, there exists a path $\path$ starting from $\phi_0 = \phi$ with $\phi_1(d)$ hyperbolic. Moreover, for each $t\in [0,1]$ and each $i\in\{1,2,3,4\},$ $\phi_t(c_i)$ is conjugate to $\phi(c_i)$. Therefore, the signs $s(\phi_t) = s(\phi)$ and the relative Euler classes $e(\phi_t) = e(\phi)$, and the path $\path$ lies in $\HPns(\pantstwo).$  Next, we deduce from $e(\phi_1) = n$ that $\frac{1}{2}(s_1+s_2) + s_3 -1 \leqslant n \leqslant \frac{1}{2}(s_1+s_2) + 1$. Indeed, 
        for each $j\in \{1,2\}$, since $\phi_1(d)$ is hyperbolic, we have $\phi_1|_{\pi_1(P_j)}\in \HP(P_j)$; and the sign $s(\phi_1) = s = (s_1,s_2,s_3,0)$ implies that $s(\phi_1|_{\pi_1(P_1)}) = (s_1,s_2,0)$ and $s(\phi_1|_{\pi_1(P_2)}) = (s_3,0,0)$.
    Then by Theorem \ref{thm_goldman} and Theorem \ref{thm_pants2}, we have $e(\phi_1|_{\pi_1(P_1)}) = \frac{1}{2}(s_1+s_2)$ and $ s_3 -1\leqslant e(\phi_1|_{\pi_1(P_2)}) \leqslant 1$. Finally, by Proposition \ref{prop_additivity} that  $n = e(\phi_1|_{\pi_1(P_1)})+e(\phi_1|_{\pi_1(P_2)})$, we have $\frac{1}{2}(s_1+s_2) + s_3- 1 \leqslant n\leqslant \frac{1}{2}(s_1+s_2) + 1$. 
        \medskip
        
         It remains to show that any non-empty $\HPns(\Sigma_{0,4})$ is connected. 
         Let $\phi,\psi \in \HPns(\Sigma_{0,4})$. 
         By Theorem \ref{thm_nonabel}, we can assume that $\phi|_{\pi_1(P_1)},$ $\phi|_{\pi_1(P_2)},$ $\psi|_{\pi_1(P_1)}$ and $\psi|_{\pi_1(P_2)}$ are non-abelian; and by Lemma \ref{lem_inthyppants}, we can assume that $\phi(d)$ and $\psi(d)$ are hyperbolic, and hence $\phi|_{\pi_1(P_1)}, \psi|_{\pi_1(P_1)}\in \HP(P_1)$. 
        Then by Theorem \ref{thm_pants2}, we have
         $e(\phi|_{\pi_1(P_1)}) = e(\psi|_{\pi_1(P_1)}) = \frac{1}{2}(s_1+s_2)$, and by Lemma \ref{lem_sameEuler_pants}, there exists a path in $\HPns(\Sigma_{0,4})$ connecting $\phi$ and $\psi$.
    \end{proof}

    Finally, we consider the type-preserving representations of $\pi_1(\pantstwo)$.
    For $n\in \mathbb{Z}$ and $s\in \{\pm 1\}^4$, corresponding to the Cases (2) and (3) in Theorem \ref{thm_main2}, we will call the pair $(n,s)$ \emph{exceptional} if it satisfies one of the following conditions:
    \begin{enumerate}[(1)]
        \item  $n = 0,$ and either $p_-(s) = 1$ or $p_+(s)=1$, and
        \item  $n = 1$ and $p_-(s) = 0$,  or $n = -1$ and $ p_+(s) = 0$.
    \end{enumerate} 

    The following Theorem \ref{thm_pantstype4} proves Theorem \ref{thm_main2} for $\Sigma = \Sigma_{0,4}$ for the non-exceptional pairs, and the case for the exceptional pairs will be discussed in Section 7.

    \begin{theorem}\label{thm_pantstype4}
        Let $n\in \mathbb{Z}$, $s\in\{\pm 1\}^4$ such that the pair $(n,s)$ is not exceptional.
        Then $\mathcal{R}^s_n(\pantstwo)$ is non-empty if and only if 
       $$\chi(\Sigma_{0,4}) + p_+(s)\leqslant n \leqslant -\chi(\Sigma_{0,4}) - p_-(s).$$ 
        Moreover, each non-empty $\mathcal{R}^s_n(\pantstwo)$ above is connected.
    \end{theorem}
    \begin{remark}
        Theorem \ref{thm_pantstype4} can also be found in \cite{yang}, where the connected components of the space of  $\psl$-characters  of irreducible type-preserving representations were characterized.
    \end{remark}
    
To prove Theorem \ref{thm_pantstype4}, we need the following lemma.
    \begin{lemma}\label{lem_inthyp_sphere}
        For $p\geqslant 4$, let $\displaystyle\bigcup^{p-2}_{i=1} P_i$ be the chosen almost-path decomposition of $\Sigma_{0,p}$, with the decomposition curves $d_i = c_{i+2}\cdots c_p = (c_1\cdots c_{i+1})^{-1}$ for $i\in \{1\cdots p-3\}$.
        For $n\in \mathbb{Z}$ and $s\in \{\pm 1\}^p$, let $\phi\in \mathrm{NA}^s_n(\Sigma_{0,p})$. 
        \begin{enumerate}[(1)]
        \item If $s_1 = -s_2$, then $\phi(d_1)$ is hyperbolic.
        \item If $p_+(s) = 0$ or $p_-(s) = 0$, then either $\phi(d_i)$ is elliptic for all $i\in \{1,\cdots, p-3\}$, 
        or there exists an $i\in \{1,\cdots, p-3\}$ such that $\phi(d_i)$ is hyperbolic.
        \end{enumerate}
    \end{lemma}
    \begin{proof}
    For (1), 
    as $\phi(c_1)\in \Par^{sgn(s_1)}$, 
    up to a $\psl$-conjugation, we can assume that 
    $\phi(c_1) = \pm \begin{bmatrix}
        1 & s_1 \\
        0 & 1
    \end{bmatrix}$. For $\phi(c_2) = \pm \abcd \in \Par^{sgn(s_2)}$ with $a+d = 2$, we have $c \neq 0$, 
     as otherwise $\phi(c_2) = \pm \begin{bmatrix}
        1 & b \\
        0 & 1
    \end{bmatrix}$ which commutes with $\phi(c_1)$, contradicting 
    that $\phi|_{\pi_1(P_1)}$ is non-abelian. Then by Lemma \ref{lem_offdiag}, we have $sgn(c) = -sgn(s_2) = sgn(s_1)$, and hence $s_1c>0$. As a consequence, 
     $\phi(d_1) = \big(\phi(c_1)\phi(c_2)\big)^{-1}
     = 
     \pm \begin{bmatrix}
         d & -b - s_1d \\
         -c & a + s_1c
     \end{bmatrix}$ has the  trace $a+d+s_1c > 2$ in absolute value,  hence is hyperbolic. 
     \medskip
     
    For (2), we will consider the cases that $p = 4$ and $p\geqslant 5$ separately .
    \smallskip
    
    In the case that $p = 4$, we show that $\phi(d_1)$ is either hyperbolic or elliptic. 
    Since $\phi|_{\pi_1(P_1)}$ is non-abelian, we have $\phi(d_1) \neq \pm\mathrm{I}$. Now suppose that $\phi(d_1)$ is parabolic. If $\phi(d_1)\in \Par^+$, then we let $s' = +1$; and if $\phi(d_1)\in \Par^-$, then we let $s' = -1$.
       Since $\pi_1(P_1)$ has the preferred peripheral elements $c_1,c_2, d_1$ that are mapped to parabolic elements, 
       $\phi|_{\pi_1(P_1)}$ is type-preserving with the sign $s(\phi|_{\pi_1(P_1)}) = (s_1,s_2,s')$; and since $\pi_1(P_2)$ has the preferred peripheral elements $c_3,c_4, d_1^{-1}$ that are mapped to parabolic elements, 
       $\phi|_{\pi_1(P_2)}$ is type-preserving with the sign $s(\phi|_{\pi_1(P_2)}) = (s_3,s_4,-s')$. Since $p_+(s)=0$ or $p_-(s)=0,$ we have $s_1 = s_2 = s_3 = s_4$, and one of  $s(\phi|_{\pi_1(P_1)})$ and $s(\phi|_{\pi_1(P_2)})$ equals $(s_1, s_1, s_1)$ while the other equals $(s_1, s_1, -s_1)$.
       Then by Theorem \ref{thm_pants}, one of $e(\phi|_{\pi_1(P_1)})$ and $e(\phi|_{\pi_1(P_2)})$ is $s_1$, while the other is $0$; and by Lemma \ref{lem_R0pants}, one of $\phi|_{\pi_1(P_1)}$ and  $\phi|_{\pi_1(P_2)}$ is abelian, contradicting that $\phi\in \mathrm{NA}^s_n(\Sigma_{0,p})$. Therefore, $\phi(d_1)$ is not parabolic, and hence either hyperbolic or elliptic.
    \smallskip
    
    In the case that $p\geqslant 5$, we first observe that $\phi(d_i)\neq \pm\mathrm{I}$ for all $i\in \{1\cdots p - 3\}$, as otherwise $\phi|_{\pi_1(P_i)}$ is abelian for some $i\in \{1,\cdots, p-2\}$. 
Therefore, it suffices to show that, if 
     $\phi(d_i)$ is parabolic for some $i\in \{1\cdots p - 3\}$, then 
     $\phi(d_j)$ is hyperbolic for some $j\in \{1\cdots p - 3\}\setminus \{i\}$. 
         Suppose that $\phi(d_i)$ is parabolic.
         Let $\Sigma_1 \doteq \displaystyle\bigcup_{k =1}^i P_k$ and let $\Sigma_2 \doteq \displaystyle\bigcup_{k = i+1}^{p-2} P_k$ with the common boundary $d_i=\Sigma_1\cap\Sigma_2$. Then, since $p - 3\geqslant 2$, at least one of $\Sigma_1$ and $\Sigma_2$ has more than three punctures so that at least one of $d_{i-1}\in \pi_1(\Sigma_1)$ and $d_{i+1}\in \pi_1(\Sigma_2)$ is defined. 
         We will show that at least one of $\phi(d_{i-1})$ and $\phi(d_{i+1})$ is hyperbolic using part (1).
    If $\phi(d_i)\in \Par^+$, then we let $s' = +1$; and if $\phi(d_i)\in \Par^-$, then we  let $s' = -1$.
     Then the restrictions $\phi|_{\pi_1(\Sigma_1)}$ and 
     $\phi|_{\pi_1(\Sigma_2)}$ are type-preserving, respectively with signs 
     $(s_1,\cdots, s_{i+1}, s')$ and $(-s', s_{i+2},\cdots, s_p)$.
     Moreover, since $p_+(s) = 0$ or $p_-(s) = 0$, we have $s_1 =\cdots =  s_p$, which equals either $s'$ or $-s'$. We consider the following three cases: $i=1,$ $i=p-3$ and $2\leqslant i\leqslant p-4.$ 
     If $i = 1$, then $\Sigma_1 = P_1$, and the sign $s(\phi|_{\pi_1(P_1)}) = (s_1,s_2,s')$. 
     It follows that $s' = s_1 = s_2$, as otherwise $\phi|_{\pi_1(P_1)}$ is abelian by Theorem \ref{thm_pants} and Lemma \ref{lem_R0pants}, contradicting that $\phi\in \mathrm{NA}^s_n(\Sigma_{0,p})$. Then $-s'$ and $s_3 = s_2$ have opposite signs; 
     and applying part (1) to $\phi|_{\pi_1(\Sigma_2)}$, 
     we have $\phi(d_2)= \big(\phi(d_1^{-1})\phi(c_3)\big)^{-1}\in \Hyp$. 
     Similarly, if $i = p-3$, then $\Sigma_2 = P_{p-2}$, and the sign $s(\phi|_{\pi_1(P_{p-2})}) = (-s',s_{p-1},s_p)$. 
     It follows that $-s' = s_{p-1} = s_p$, as otherwise $\phi|_{\pi_1(P_{p-2})}$ is abelian by Theorem \ref{thm_pants} and Lemma \ref{lem_R0pants}, contradicting that $\phi\in \mathrm{NA}^s_n(\Sigma_{0,p})$. Then $s'$ and $s_{p-2} = s_{p-1}$ have opposite signs; and applying part (1) to $\phi|_{\pi_1(\Sigma_1)}$, we have $\phi(d_{p-4}) = \phi(c_{p-2})\phi(d_{p-3})\in \Hyp$.
     Finally, if $2\leqslant i\leqslant p-4$, then $s_{i+1} = s_{i+2}$ implies that  either
     $s_{i+1} = -s'$ or $s_{i+2} = s'$, 
     hence by applying part (1) to either $\phi|_{\pi_1(\Sigma_1)}$ or $\phi|_{\pi_1(\Sigma_2)}$, we conclude that either $\phi(d_{i-1}) = \phi(c_{i+1})\phi(d_i)$ or $\phi(d_{i+1}) = \big(\phi(d_i^{-1})\phi(c_{i+2})\big)^{-1}$ is hyperbolic.
    \end{proof}

    \begin{proof}[Proof of Theorem \ref{thm_pantstype4}]
        We first observe that the inequality in the result is equivalent to the following equality
        $n = \frac{1}{2}(s_1+s_2+s_3+s_4).$
                 \medskip

        We begin by proving that, if $n = \frac{1}{2}(s_1+s_2+s_3+s_4)$, then $\mathcal{R}^s_n(\pantstwo)$ is non-empty. To find a representation $\phi$ in $\mathcal{R}^s_n(\pantstwo)$, we will define $\phi$ on $\pi_1(P_1)$ first, then extend it to $\pi_1(\pantstwo)$. Let $m = \frac{1}{2}(s_1+s_2)$.
          As the fundamental group $\pi_1(P_1)$ has the preferred peripheral elements $c_1,c_2$ and $d$, Theorem \ref{thm_pants2} implies that there exists a representation $\phi\in \HP_m^{(s_1,s_2,0)}(P_1)$, having $\phi(c_1)\in\Par^{sgn(s_1)}$, $\phi(c_2)\in\Par^{sgn(s_2)}$, and $\phi(d)\in\Hyp$.
          Let $\widetilde{\phi(d)}$ be the lift of $\phi(d)$ in $\Hyp_{n-m}$. As  
          $n - m = \frac{1}{2}(s_3 + s_4)$, by Corollary \ref{cor_evimage}, the evaluation map $\ev: \HP_{n-m}^{(s_3,s_4,0)}(P_2)\to \Hyp_{n-m}$ is surjective. Hence, there exists a representation $\psi\in \HP_{n-m}^{(s_3,s_4,0)}(P_2)$ such that $\ev(\psi) = \widetilde{\phi(d)}$. As $\psi$ and $\phi$ agree at the common boundary $d$ of $P_1$ and $P_2$, letting $\phi|_{\pi_1(P_2)}\doteq \psi$ extends the representation $\phi$ from $\pi_1(P_1)$ to $\pi_1(\pantstwo)$. 
       Moreover, as 
       $e(\phi|_{\pi_1(P_1)}) = m$ and $e(\phi|_{\pi_1(P_2)}) = n-m$, by Proposition \ref{prop_additivity}, 
       the relative Euler class 
       $e(\phi) = m + (n-m) = n$;
       and as $\phi(c_i)\in \Par^{sgn(s_i)}$ for $i\in \{1,2,3,4\}$, the sign 
       $s(\phi) = s$. As a consequence, $\phi\in \mathcal{R}^s_n(\pantstwo)$.
          \medskip
        

        Next we prove that, if $\mathcal{R}^s_n(\pantstwo)$ is non-empty, then $n = \frac{1}{2}(s_1+s_2+s_3+s_4)$. 
        If $p_+(s)\geqslant 1$ and $p_-(s)\geqslant 1$, then up to a permutation of the peripheral elements, we can assume that $s_1 = +1$ and $s_2 = -1$.
        For a representation $\phi\in \mathcal{R}^s_n(\pantstwo)$, 
        since the pair $(n,s)$ is not exceptional,
           by Theorem \ref{thm_nonabel}, we can assume that $\phi|_{\pi_1(P_1)}$ and $\phi|_{\pi_1(P_2)}$ are non-abelian; and by Lemma \ref{lem_inthyp_sphere}, $\phi(d)$ is hyperbolic. 
    If otherwise that $p_+(s) = 0$ or $p_-(s) = 0$, then by Theorem \ref{thm_nonabel} again, we can assume that $\phi|_{\pi_1(P_1)}$ and $\phi|_{\pi_1(P_2)}$ are non-abelian; and by Lemma \ref{lem_inthyp_sphere}, $\phi(d)$ is either hyperbolic or elliptic.
        If $\phi(d)$ is elliptic, 
        then $\phi$ is so called ``totally non-hyperbolic" (see the beginning of Section 7 for the definition), and Proposition \ref{Prop_Exc2} would imply that $\phi \in \mathcal{R}^s_{s_1}(\Sigma_{0,4})$ and the pair $(n, s) = (s_1, s)$ is exceptional, which is a contradiction. Therefore, $\phi(d)$ is hyperbolic;
    and we have $\phi|_{\pi_1(P_1)}\in \HP(P_1)$ with sign $s(\phi|_{\pi_1(P_1)}) = (s_1,s_2,0)$, and $\phi|_{\pi_1(P_2)}\in \HP(P_2)$ with sign $s(\phi|_{\pi_1(P_2)}) = (s_3,s_4,0)$.
    Then by Theorem \ref{thm_pants2}, we have $e(\phi|_{\pi_1(P_1)}) = \frac{1}{2}(s_1+s_2)$ and $e(\phi|_{\pi_1(P_1)}) = \frac{1}{2}(s_3+s_4)$. Finally, by Proposition \ref{prop_additivity}, we have $n = e(\phi|_{\pi_1(P_1)})+e(\phi|_{\pi_1(P_2)}) = \frac{1}{2}(s_1+s_2+s_3+s_4)$. 
        \medskip
        
         It remains to show that each non-empty $\mathcal{R}^s_n(\Sigma_{0,4})$ is connected. 
         Let $\phi,\psi \in \mathcal{R}^s_n(\Sigma_{0,4})$. 
         By Theorem \ref{thm_nonabel}, we can assume that $\phi|_{\pi_1(P_1)},$ $\phi|_{\pi_1(P_2)},$ $\psi|_{\pi_1(P_1)}$ and $\psi|_{\pi_1(P_2)}$ are non-abelian. Then as shown above, $\phi(d)$ and $\psi(d)$ are hyperbolic, and hence $\phi|_{\pi_1(P_1)}, \psi|_{\pi_1(P_1)}\in \HP(P_1)$ with signs $s(\phi|_{\pi_1(P_1)}) = s(\psi|_{\pi_1(P_1)}) = (s_1,s_2,0)$. 
        Finally, by Theorem \ref{thm_pants2}, we have
         $e(\phi|_{\pi_1(P_1)}) = e(\psi|_{\pi_1(P_1)}) = \frac{1}{2}(s_1+s_2)$; and by Lemma \ref{lem_sameEuler_pants}, there exists a path in $\HPns(\Sigma_{0,4})$ connecting $\phi$ and $\psi$.
    \end{proof}

\section{General surfaces}
     In this section, we complete the proof of Theorem \ref{thm_main1} and Theorem \ref{thm_main4} for general punctured surfaces. Corresponding to the Cases (2) and (3) in Theorem \ref{thm_main2}, for $\Sigma = \Sigma_{g,p}$, $n\in \mathbb{Z}$ and $s\in \{-1,0,+1\}^p$, we will call the triple $(\Sigma, n,s)$ \emph{exceptional} if $g = 0$, $p_0(s) = 0$, and the pair $(n,s)$ satisfies one of the following conditions:
    \begin{enumerate}[(1)]
        \item  $n = 0,$ and either $p_-(s) = 1$ or $p_+(s)=1$, and
        \item  $n = 1$ and $p_-(s) = 0$,  or $n = -1$ and $ p_+(s) = 0$.
    \end{enumerate}
    The generalized Milnor–Wood inequality does not hold for these exceptional triples $(\Sigma, n, s)$; but  the corresponding spaces $\Rns$ are non-empty, forming the \emph{exceptional components} 
    of $\R$. This which will be discussed in Section 7.
    \\
    
     The main result in this section describes $\HPns(\Sigma)$ for all non-exceptional $(\Sigma, n, s)$, which is stated as Theorem \ref{thm_general} below.
\begin{theorem}\label{thm_general}
        Let $\Sigma = \Sigma_{g,p}$, $n\in \mathbb{Z}$ and 
        $s \in \{-1,0,+1\}^p$ such that the triple $(\Sigma, n,s)$ is not exceptional.
        Then $\HPns(\Sigma)$ is non-empty if and only if 
        $$\chi(\Sigma) + p_+(s)\leqslant n\leqslant -\chi(\Sigma) - p_-(s).$$
        Moreover, each non-empty $\HPns(\Sigma)$ above is connected.
    \end{theorem}
    In the proof of Theorem \ref{thm_general}, we will use the following notation defined in Section 4. Recall that we have fixed an almost-path decomposition $\apdecomp$ of $\Sigma$, by connecting the vertices $P_i$ and $P_{i+1}$ by an edge for each $i\in\{1,\dots,g+p-3\}$ and 
    connecting the vertices $T_j$ and $P_j$ by an edge for each $j\in\{1,\dots,g\}$ in the dual graph. The decomposition curves are the following separating curves in $\Sigma$: For $i\in \{1,\cdots, g+p-3\}$, $d_i$ is the common boundary of $P_i$ and $P_{i+1}$; and for $j\in \{1,\cdots, g\}$, $d'_j$ is the common boundary of $T_j$ and $P_j$. 
    Moreover, we denote by $\nonabel$ the subspace of $\HPns(\Sigma)$ consisting of representations whose restriction to  each $\pi_1(P_i)$ for  $i \in\{ 1,\cdots, g+p-2\}$,  and to each $\pi_1(T_j)$ for $j \in\{ 1,\cdots, g\}$, is non-abelian.
    \\
    
    The proof of Theorem \ref{thm_general} relies on the following Lemma \ref{lem_inthyp}.
    \begin{lemma}\label{lem_inthyp}
         Let $\Sigma = \Sigma_{g,p}$ 
         with $\chi(\Sigma)\leqslant -2$, together with the chosen almost-path decomposition $\apdecomp$. Let $n\in \mathbb{Z}$ and
        $s\in \{-1,0,+1\}^p$ such that the triple $(\Sigma, n,s)$ is not exceptional. 
        If $p_0(s)\geqslant 1$, then assume that $s_1 = 0$; and if $p_0(s) = 0$, $p_+(s)\geqslant 1$ and $p_-(s)\geqslant 1$, then assume that $s_1 = -s_2$.
        Then for every $\phi\in \nonabel$, there is a path $\{\phi_t\}_{t\in [0,1]}$ in $\nonabel$ starting from $\phi_0 = \phi$ such that: 
        \begin{enumerate}[(1)]       \item For all $t\in [0,1]$ and all $i\in \{1,\cdots , p\}$, $\phi_t(c_i)$ is conjugate to $\phi(c_i)$ in $\psl$.
        \item $\phi_1$ maps all the decomposition curves to hyperbolic elements.
        \end{enumerate}
    \end{lemma}
    \begin{proof}

    We will proceed by induction
    on the Euler characteristic of $\Sigma$. For the base case $\chi(\Sigma) = -2$, if $g\geqslant 1$, then the result follows by Lemma \ref{lem_inthyptorus}; and if $g = 0$ and $p_0(s)\geqslant 1$, then the result follows by Lemma \ref{lem_inthyppants}. 
    In the case that $g = 0$ and $p_0(s) = 0$, we claim  that $\phi(d_1)$ is hyperbolic, then we can let $\path$ be the constant path $\phi_t\equiv \phi$. 
    Indeed, if $p_+(s)\geqslant 1$ and $p_-(s)\geqslant 1$, as $\phi\in \nonabel$ and $s_1 = -s_2$, Lemma \ref{lem_inthyp_sphere} implies that $\phi(d_1)$ is hyperbolic. If otherwise that $p_+(s) = 0$ or $p_-(s) = 0$, then by Lemma \ref{lem_inthyp_sphere} again, $\phi(d_1)$ is either hyperbolic or elliptic; and if $\phi(d_1)$ is elliptic, then $\phi$ is so called ``totally non-hyperbolic" (see the beginning of Section 7 for the definition) and Proposition \ref{Prop_Exc2} would imply that $\phi \in \mathcal{R}^s_{s_1}(\Sigma)$ and the triple $(\Sigma, n, s)=(\Sigma, s_1, s)$ is exceptional, which is a contradiction. Therefore, $\phi(d_1)$ is hyperbolic.
    \bigskip
    
    Now assume that the result holds for all punctured surfaces with the Euler characteristic strictly greater than $\chi(\Sigma)$.
        For $\phi\in \nonabel$, we will construct the path $\path$ starting from $\phi$ as follows. We will first construct a path $\{\psi_t\}\interval$ in $\nonabel$ starting from $\psi_0 = \phi$ and satisfying (1),
        such that $\psi_1$ maps at least one decomposition curve to a hyperbolic element.
       Then, using the induction hypothesis, we will construct a path $\{\rho_t\}\interval$ in $\nonabel$ starting from $\rho_0 = \psi_1$ that satisfies (1) and (2).
        Then the composition of these two paths defines the path $\path$.
        \medskip
        
        Below we construct the path $\{\psi_t\}\interval$]  for the following three cases separately: $g\geqslant 1$, $g =  0$ and $p_0(s)\geqslant 1$, and $g =0$ and  $p_0(s) = 0$.

        In the case that $g\geqslant 1$, we consider the subsurface $P_1\cup T_1\subset \Sigma$ 
        where $P_1$ and $T_1$ are 
        separated by the decomposition curve $d'_1$. 
    Since the restrictions $\phi|_{\pi_1(P_1)}$ and $\phi|_{\pi_1(T_1)}$ are non-abelian, by Lemma \ref{lem_inthyptorus}, there exists a path 
    $\{\psi_t: \pi_1(P_1\cup T_1)\to \psl\}\interval$ of representations 
    starting from $\psi_0 = \phi|_{\pi_1(P_1\cup T_1)}$ such that $\psi_t|_{\pi_1(P_1)}$ and $\psi_t|_{\pi_1(T_1)}$ are non-abelian for all $t\in [0,1]$, and
    $\psi_1(d'_1)$ is hyperbolic. Moreover, for all $t\in [0,1]$, $\psi_t(c_1)$ is conjugate to $\phi(c_1)$, and $\psi_t(d_1)$ is conjugate to $\phi(d_1)$. We extend this path to $\pi_1(\Sigma)$ as follows.
    Since $\psi_t(d_1)$ is conjugate to $\phi(d_1)$ for all $t\in [0,1]$, by Lemma \ref{lem_conjugacypath_const},
    there exists a path $\{g_t\}\interval$ in $\psl$ such that $g_0 = \pm\mathrm{I}$, and 
    for all $t\in [0,1]$, $\psi_t(d_1) = g_t \phi(d_1) g_t^{-1}$. 
    Then for each $t\in [0,1]$, as $d_1$ is the common boundary of $P_1\cup T_1$ and its complement $\Sigma' = \Sigma \setminus (P_1\cup T_1)$, letting $\psi_t|_{\pi_1(\Sigma')}\doteq g_t\big(\phi|_{\pi_1(\Sigma')}\big)g_t^{-1}$ 
    extends the representation $\psi_t$ from $\pi_1(P_1\cup T_1)$ to $\pi_1(\Sigma)$. 
    For all $t\in [0,1]$, as a $\psl$-conjugation of a non-abelian representation, 
    each $\psi_t|_{\pi_1(P_i)}$, $i \in\{2,\cdots, g+p-2\}$ and each $\psi_t|_{\pi_1(T_j)}$, $j \in\{2,\cdots, g\}$, is non-abelian.
    Moreover, for each $t\in [0,1]$ and each $i\in\{1,\cdots, p\},$ $\psi_t(c_i)$ is conjugate to $\phi(c_i)$, hence the signs $s(\psi_t) = s(\phi)$ and the relative Euler classes $e(\psi_t) = e(\phi)$. As a consequence, the path $\{\psi_t\}\interval$ lies in $\nonabel.$
     
     In the case that $g = 0$ and $p_0(s)\geqslant 1$, we consider the subsurface $P_1\cup P_2\subset \Sigma$, 
     where $P_1$ and $P_2$ are separated by the decomposition curve $d_1$. 
    The restrictions $\phi|_{\pi_1(P_1)}$ and $\phi|_{\pi_1(P_2)}$ are non-abelian; and since $s_1 = 0$, $\phi(c_1)$ is hyperbolic.
    Hence, by Lemma \ref{lem_inthyppants}, there exists a path 
    $\{\psi_t: \pi_1(P_1\cup P_2)\to \psl\}\interval$ of representations 
    starting from $\psi_0 = \phi|_{\pi_1(P_1\cup P_2)}$ such that $\psi_t|_{\pi_1(P_1)}$ and $\psi_t|_{\pi_1(P_2)}$ are non-abelian for all $t\in [0,1]$, and
    $\psi_1(d_1)$ is hyperbolic. Moreover, for all $t\in [0,1]$, $\psi_t(c_i)$ is conjugate to $\phi(c_i)$ for each $i\in \{1,2,3\}$, and $\psi_t(d_2)$ is conjugate to $\phi(d_2)$. We extend this path to $\pi_1(\Sigma)$ as follows.
    Since $\psi_t(d_2)$ is conjugate to $\phi(d_2)$ for all $t\in [0,1]$, by Lemma \ref{lem_conjugacypath_const},
    there exists a path $\{g_t\}\interval$ in $\psl$ such that $g_0 = \pm\mathrm{I}$, and 
    for all $t\in [0,1]$, $\psi_t(d_2) = g_t \phi(d_2) g_t^{-1}$. 
    Then for each $t\in [0,1]$, as $d_2$ is the common boundary of $P_1\cup P_2$ and its complement $\Sigma' = \Sigma \setminus (P_1\cup P_2)$, letting $\psi_t|_{\pi_1(\Sigma')}\doteq g_t\big(\phi|_{\pi_1(\Sigma')}\big)g_t^{-1}$ 
    extends the representation $\psi_t$ from $\pi_1(P_1\cup P_2)$ to $\pi_1(\Sigma)$. 
    For all $t\in [0,1]$, as a $\psl$-conjugation of a non-abelian representation, 
    each $\psi_t|_{\pi_1(P_i)}$, $i \in\{3,\cdots, p-2\}$ is non-abelian.
    Moreover, for each $t\in [0,1]$ and each $i\in\{1,\cdots, p\},$ $\psi_t(c_i)$ is conjugate to $\phi(c_i)$, hence the signs $s(\psi_t) = s(\phi)$ and the relative Euler classes $e(\psi_t) = e(\phi)$. As a consequence, the path $\{\psi_t\}\interval$ lies in $\nonabel.$
    
     In the case that $g = 0$ and $p_0(s) = 0$, we will prove that there is an $i\in \{1,\cdots, p-2\}$ such that $\phi(d_i)$ is hyperbolic, then we can let  $\path$ be the constant path $\phi_t\equiv \phi$. 
    If  $p_+(s)\geqslant 1$ and $p_-(s)\geqslant 1$, then as $\phi\in \nonabel$ and $s_1 = -s_2$, Lemma \ref{lem_inthyp_sphere} implies that $\phi(d_1)$ is hyperbolic. 
    If otherwise that $p_+(s) = 0$ or $p_-(s) = 0$, then by Lemma \ref{lem_inthyp_sphere} again,
       either $\phi(d_i)\in \Ell$ for all $i\in \{1,\cdots p-3\}$, or there exists an $i\in \{1,\cdots, p-3\}$ such that $\phi(d_i)\in \Hyp$; and if $\phi(d_i)\in \Ell$ for all $i\in \{1,\cdots p-3\}$, then $\phi$ is so called ``totally non-hyperbolic" (see the beginning of Section 7 for the definition) and Proposition \ref{Prop_Exc2} would imply that $\phi \in \mathcal{R}^s_{s_1}(\Sigma)$ and the triple $(\Sigma, n, s)=(\Sigma, s_1, s)$ is exceptional, which is a contradiction. Therefore, there exists an $i\in \{1,\cdots, p-2\}$ such that $\phi(d_i)$ is hyperbolic.
     \medskip
     
    Next, we construct the path $\{\rho_t\}\interval$ using the induction hypothesis.
    Let $d$ be a decomposition curve such that $\psi_1(d)\in \Hyp$.
    Then $d$ separates $\Sigma$ into two subsurfaces $\Sigma_1$ and $\Sigma_2$, with $\chi(\Sigma_1)> \chi(\Sigma)$ and $\chi(\Sigma_2)> \chi(\Sigma)$.
    Since $\psi_1(c_i)\in \overline{\Hyp}$ for 
    $i\in \{1,\cdots, p\}$ and $\psi_1(d)\in \Hyp$, we have 
    $\psi_1|_{\pi_1(\Sigma_1)}
    \in \HP(\Sigma_1)$ and $\psi_1|_{\pi_1(\Sigma_2)}
    \in \HP(\Sigma_2)$. 
    Let the relative Euler classes 
    $n' \doteq e\big(\psi_1|_{\pi_1(\Sigma_1)}\big)$ and 
    $n'' \doteq e\big(\psi_1|_{\pi_1(\Sigma_2)}\big)$, and the signs $s' = s\big(\psi_1|_{\pi_1(\Sigma_1)}\big)$ and 
    $s'' = s(\psi_1|_{\pi_1(\Sigma_2)})$. 
    As $\psi_1(d)\in \Hyp$, we have $p_0(s')\geqslant 1$ and $p_0(s'')\geqslant 1$, hence neither $(\Sigma_1, n', s')$ nor $(\Sigma_2, n'', s'')$ is exceptional. Moreover, since $\psi_1\in \nonabel$, we have 
    $\psi_1|_{\pi_1(\Sigma_1)}\in \mathrm{NA}^{s'}_{n'}(\Sigma_1)$ 
    and 
    $\psi_1|_{\pi_1(\Sigma_2)}\in \mathrm{NA}^{s''}_{n''}(\Sigma_2)$. 
    Therefore, by the induction hypothesis, there exist paths
    $\{\rho_{1,t}\}_{t\in [0,1]}\subset \mathrm{NA}^{s'}_{n_1}(\Sigma_1)$ and $\{\rho_{2,t}\}_{t\in [0,1]}\subset \mathrm{NA}^{s''}_{n_2}(\Sigma_2)$, respectively starting from $\psi_1|_{\pi_1(\Sigma_1)}$
    and $\psi_1|_{\pi_1(\Sigma_2)}$, and both satisfying Conditions (1) and (2) in the lemma. 
    For each $j\in 
    \{1,2\}$, as $\rho_{j,t}(d)$ is conjugate to $\psi_1(d)$ for all $t\in [0,1]$, by Lemma \ref{lem_conjugacypath_const}, there exists a path $\{g_{j,t}\}\interval$ in $\psl$ starting from $g_{j,0} = \pm\mathrm{I}$, such that $g_{j,t}\rho_{j,t}(d)(g_{j,t})^{-1} = \psi_1(d)$.
    Then for each $t\in [0,1]$, as $g_{1,t}\rho_{1,t}(g_{1,t})^{-1}$ and $g_{2,t}\rho_{2,t}(g_{2,t})^{-1}$ agree at the common boundary $d$ of $\Sigma_1$ and $\Sigma_2$,
    letting $\rho_t|_{\pi_1(\Sigma_1)}\doteq g_{1,t}\rho_{1,t}(g_{1,t})^{-1}$ and 
    $\rho_t|_{\pi_1(\Sigma_2)}\doteq g_{2,t}\rho_{2,t}(g_{2,t})^{-1}$ defines the representation $\rho_t$. This defines a path $\{\rho_t\}\interval$ starting from $\rho_0 = \psi_1$.

    Finally, we show that $\{\rho_t\}\interval$ satisfies Conditions (1) and (2), and that $\{\rho_t\}\interval\subset \nonabel$.
    For each $t\in [0,1]$ and $i\in \{1,\cdots, p\}$, as $\rho_t(c_i)$ is conjugate to either 
    $\rho_{1,t}(c_i)$ or $\rho_{2,t}(c_i)$ in $\psl$, $\rho_t(c_i)$ is conjugate to $\psi_1(c_i)$, and hence conjugate to $\phi(c_i)$. This verifies Condition (1). For Condition (2), we notice that every decomposition curve in $\Sigma$ is either the common boundary $d$ of $\Sigma_1$ and $\Sigma_2$, or a decomposition curve of $\Sigma_j$ for $j\in \{1,2\}$. Since $\psi_1(d)\in \Hyp$ and $s(\rho_{1,1}) = s(\psi_1|_{\pi_1(\Sigma_1)})$,
    $\rho_{1,1}(d) \in \Hyp$, and hence its $\psl$-conjugation $\rho_1(d)\in \Hyp$. Moreover, for each $j\in \{1,2\}$, as a conjugation of $\rho_{j,1}$, $\rho_1$ sends each decomposition curve of $\Sigma_j$ into $\Hyp$. Therefore, $\{\rho_t\}\interval$ satisfies Condition (2).
    It remains to show that $\{\rho_t\}\interval\subset \nonabel$.
    For each $t\in [0,1]$ and each $i \in\{1,\cdots, g+p-2\}$, the restriction $\rho_t|_{\pi_1(P_i)}$ is non-abelian, as it is a $\psl$-conjugation of a non-abelian representation $\rho_{j,t}|_{\pi_1(P_i)}$ for some $j\in \{1,2\}$. Similarly, for each $t\in [0,1]$ and each $j \in\{1,\cdots, g\}$, as a $\psl$-conjugation of a non-abelian representation, the restriction $\rho_t|_{\pi_1(T_j)}$ is non-abelian. Moreover, for each $t\in [0,1]$ and each $i\in\{1,\cdots, p\},$ $\rho_t(c_i)$ is conjugate to $\psi_1(c_i)$, hence the signs $s(\rho_t) = s(\psi_1) = s
    $ and the relative Euler classes $e(\rho_t) = e(\psi_1) = n$. As a consequence, the path $\{\rho_t\}\interval$ lies in $\nonabel.$
    \end{proof}

    \begin{proof}[Proof of Theorem \ref{thm_general}] 
    Up to a permutation of peripheral elements, 
    if $p_0(s) \geqslant 1$, then we can assume that  $s_i = 0$ for $i\in \{1,\cdots, p_0(s)\}$; and if $p_0(s) = 0$, $p_+(s)\geqslant 1$ and $p_-(s)\geqslant 1$, then we can assume that $s_1 = +1$ and $s_2 = -1$. Let $\Sigma = \apdecomp$ be the chosen almost-path decomposition.
    \\

    We first prove that, if $\chi(\Sigma) + p_+(s)\leqslant n\leqslant -\chi(\Sigma) - p_-(s)$, then $\HPns(\Sigma)$ is non-empty. We will prove this by induction on the Euler characteristic $\chi(\Sigma)$. The base cases $\chi(\Sigma) = -1$ and $\chi(\Sigma) = -2$ are respectively proved in Section 3 and Section 5. 
    Let $\Sigma = \Sigma_{g,p}$ be a punctured surface with Euler characteristic $\chi(\Sigma)<-2$, and assume that the statement above holds for any surface $\Sigma'$ with Euler characteristic $\chi(\Sigma')  = \chi(\Sigma) + 1$.
    To construct a representation $\phi$ in $\HPns(\Sigma)$, we will find a representation of the fundamental group of a subsurface $\Sigma'$ of $\Sigma$ using the induction hypothesis and then extend it to $\pi_1(\Sigma)$, for the cases $g\geqslant 1$ and $g = 0$ separately.
    
    In the case that $g\geqslant 1$, we let $\Sigma'\doteq \Sigma\setminus T_1$.  
    The fundamental group $\pi_1(\Sigma')$ has the preferred peripheral elements $c_1,\cdots, c_p$ and $d'_1$.
    Let $s' \doteq (s_1,\cdots,s_p, 0) \in \{-1,0,+1\}^{p+1}$. Then $p_+(s') = p_+(s)$, $p_-(s') = p_-(s)$ and $\chi(\Sigma) = \chi(\Sigma')-1$, and we can rewrite the generalized Milnor-Wood inequality into 
    $\chi(\Sigma') + p_+(s') - 1\leqslant n\leqslant -\chi(\Sigma') - p_-(s') + 1.$ 
    If $n = \chi(\Sigma') + p_+(s') - 1$, then let $m = -1$; if $n = -\chi(\Sigma') - p_-(s') + 1$, then let $m = +1$; and if otherwise that 
    $\chi(\Sigma') + p_+(s')\leqslant n \leqslant -\chi(\Sigma') - p_-(s')$, then let $m = 0$.
    With this choice of $m,$ we have 
    $$\chi(\Sigma') + p_+(s')\leqslant n - m\leqslant -\chi(\Sigma') - p_-(s');$$
    and by the induction hypothesis, there exists a representation $\phi \in \HP^{s'}_{n - m}(\Sigma')$ that maps $d'_1$ to a hyperbolic element.
    We extend $\phi$ to $\pi_1(\Sigma)$ as follows. As $\phi(d'_1)\in \Hyp$, there is a unique lift
    $\widetilde{\phi(d'_1)}$ of $\phi(d'_1)$ in $\Hyp_{m}$, which by Theorem \ref{thm_liftcommu} is contained in the image of the lifted commutator. 
    Therefore, there exists a $\psl$-pair $(\pm A, \pm B)$ whose lifted commutator equals $\widetilde{\phi(d'_1)}$. Then as the commutator $[\pm A, \pm B] = \phi(d'_1)$,  we can extend $\phi$ from $\pi_1(\Sigma')$ to $\pi_1(\Sigma)$ by letting $\big(\phi(a_1), \phi(b_1)\big) \doteq (\pm A, \pm B)$. 
    Since the evaluation map $\ev: \mathcal{W}(T_1)\to \displaystyle\bigcup_{k =-1}^1\Hyp_k$
       maps $\phi|_{\pi_1(T_1)}$ into $\Hyp_{m}$, we have $e(\phi|_{\pi_1(T_1)}) = m$; and since $e(\phi|_{\pi_1(\Sigma')}) = n-m$, by Proposition \ref{prop_additivity}, $e(\phi) = m + (n-m) = n$.
    Moreover, we have $s(\phi) = s$, and hence
    $\phi \in \HPns(\Sigma)$.

    In the case that $g = 0$, we let $\Sigma'\doteq \Sigma\setminus P_1$.
    The fundamental group $\pi_1(\Sigma')$ has the preferred peripheral elements $c_3,\cdots, c_p$ and $d_1^{-1}$.
    Let $s' \doteq (s_3,\cdots,s_p, 0)\in \{-1,0,+1\}^{p-1}$. 
    To satisfy the induction hypothesis, we first choose an $m$ satisfying
    $$\chi(\Sigma') + p_+(s')\leqslant n - m\leqslant -\chi(\Sigma') - p_-(s')$$
    separately in the cases $p_0(s)\geqslant 1$ and $p_0(s) = 0$.
    
    For $p_0(s)\geqslant 1$, we have $s_1 = 0$. 
    If $s_2 = 0$, then we have $p_+ (s) = p_+ (s')$ and $p_- (s) = p_- (s')$, and  the generalized Milnor-Wood inequality becomes  
    $
    \chi(\Sigma') + p_+(s') - 1\leqslant n\leqslant -\chi(\Sigma') - p_-(s') + 1;
    $
    if $s_2 = +1$, then we have $p_+ (s) = p_+ (s')+1$ and $p_- (s) = p_- (s')$, and  the generalized Milnor-Wood inequality becomes
    $
    \chi(\Sigma') + p_+(s')\leqslant n\leqslant -\chi(\Sigma') - p_-(s') + 1;
    $
  and if $s_2 = -1$, then we have $p_+ (s) = p_+ (s')$ and $p_- (s) = p_- (s')+1$, and  the generalized Milnor-Wood inequality becomes 
     $\chi(\Sigma') + p_+(s') - 1\leqslant n\leqslant -\chi(\Sigma') - p_-(s').
    $
    This leads us to choose an $m$ as follows: If $n = \chi(\Sigma') + p_+(s') - 1$, then let $m = -1$; if $n = -\chi(\Sigma') - p_-(s') + 1$, then let $m = +1$; and if otherwise that 
    $\chi(\Sigma') + p_+(s')\leqslant n \leqslant -\chi(\Sigma') - p_-(s')$, then let $m = 0$. 
    With this choice of $m,$ for all $s_2\in \{-1,0,1\}$ we have $\chi(\Sigma') + p_+(s')\leqslant n - m\leqslant -\chi(\Sigma') - p_-(s').$

    For $p_0(s) = 0$, we let $m = \frac{1}{2}(s_1+s_2)$. Then for $n$ safisfying $\chi(\Sigma') + p_+(s) - 1\leqslant n\leqslant -\chi(\Sigma') - p_-(s) + 1$, we have  $\chi(\Sigma') + p_+(s')\leqslant n-m\leqslant -\chi(\Sigma') - p_-(s').$
    Indeed, as $p_+(s) - p_+(s')$ and $p_-(s) - p_-(s')$ are respectively the numbers of $+1$'s and $-1$'s in $(s_1, s_2)$,
        in all the four possible  cases $(+1, +1),$ $ (+1, -1),$ $ (-1,+1)$ and $(-1,-1)$ of $(s_1, s_2)$, 
        we have $\frac{1}{2}(s_1+s_2) = p_+(s) - p_+(s') - 1 = 1 - p_-(s) +p_-(s')$. 
        From $m = p_+(s) - p_+(s') - 1$, we have $n - m \geqslant \chi(\Sigma')+ p_+(s')$; and from $m = 1 -p_-(s) + p_-(s')$, we have $n - m \leqslant -\chi(\Sigma')- p_-(s')$.

   Now, with this choice of $m$ and by the induction hypothesis, there exists a representation $\phi \in \HP^{s'}_{n - m}(\Sigma')$ that maps $d_1$ to a hyperbolic element. 
    We extend $\phi$ to $\pi_1(\Sigma)$ as follows. 
    We recall that the fundamental group $\pi_1(P_1)$ has the preferred peripheral elements $c_1,c_2,$ and $d_1$. As $\phi(d_1)\in \Hyp$, there is a unique lift
    $\widetilde{\phi(d_1)}$ of $\phi(d_1)$ in $\Hyp_{m}$; and by Proposition \ref{prop_evimage} and Corollary \ref{cor_evimage}, the evaluation map $\ev: \HP_{m}^{(s_1, s_2, 0)}(P_1)\to \Hyp_{m}$ is surjective. Hence, there exists a representation $\psi\in \HP_{m}^{(s_1, s_2, 0)}(P_1)$ such that $\ev(\psi) = \widetilde{\phi(d_1)}$. As $\psi$ and $\phi$ agree at the common boundary $d_1$ of $P_1$ and $\Sigma'$, letting $\phi|_{\pi_1(P_1)}\doteq \psi$ extends the representation $\phi$ from $\pi_1(\Sigma')$ to $\pi_1(\Sigma)$. 
    As $\ev(\phi|_{\pi_1(P_1)})\in \Hyp_m$, we have $e(\phi|_{\pi_1(P_1)}) = m$;
       and as $e(\phi|_{\pi_1(\Sigma')}) = n - m$, by Proposition \ref{prop_additivity}, $e(\phi)= m + (n-m)= n$.
       Moreover, as $s(\phi|_{\pi_1(\Sigma')}) = s'$ and $s(\phi|_{\pi_1(P_1)}) = (s_1, s_2, 0)$, we have $s(\phi) = s$. As a consequence, $\phi\in \HPns(\Sigma)$.
    \\

    Next we prove that, if $\HPns(\Sigma)$ is non-empty, then $\chi(\Sigma) + p_+(s)\leqslant n\leqslant -\chi(\Sigma) - p_-(s)$. Let $\phi\in \HPns(\Sigma)$. As the triple $(\Sigma, n, s)$ is not exceptional,
    by Theorem \ref{thm_nonabel} and Lemma \ref{lem_inthyp}, we can assume that $\phi$ maps all decomposition curves to hyperbolic elements. 
       Then for all $i\in \{1,\cdots, g+p-2\}$, $\phi|_{\pi_1(P_i)}\in \HP(P_i)$; and by Theorem \ref{thm_goldman} and Theorem \ref{thm_pants2} together with Remark \ref{rmk_MWpants}, we have 
     $$\chi(P_i)+ p_+\big(s\big(\phi|_{\pi_1(P_i)}\big)\big) \leqslant e\big(\phi|_{\pi_1(P_i)}\big) \leqslant -\chi(P_i) - p_-\big(s\big(\phi|_{\pi_1(P_i)}\big)\big).$$
     Moreover, for all $j\in \{1,\cdots, g\}$, $\phi|_{\pi_1(T_j)}\in \mathcal{W}(T_j)$; and by Theorem \ref{thm_goldman}, we have 
     $$\chi(T_j) \leqslant e\big(\phi|_{\pi_1(T_j)}\big) \leqslant -\chi(T_j).$$
      From these inequalities and Proposition \ref{prop_additivity}, we have the generalized Milnor-Wood inequality
      $$\chi(\Sigma) + p_+(s)\leqslant n\leqslant -\chi(\Sigma) - p_-(s).$$
     Indeed, for all $i\in \{1,\cdots, g+p-2\}$, each preferred peripheral element of $\pi_1(P_i)$ is either a decomposition curve or $c_k$ for some $k\in \{1,\cdots, p\}$; and its image is parabolic only if it is $c_k$ for some $k\in \{1,\cdots, p\}$, as $\phi$ maps all decomposition curves to hyperbolic elements.
     Therefore, we have
     $p_+(s) = \displaystyle\sum_{i = 1}^{g+p-2} p_+ \big(s\big(\phi|_{\pi_1(P_i)}\big)\big)$ and $p_-(s) = \displaystyle\sum_{i = 1}^{g+p-2} p_- \big(s\big(\phi|_{\pi_1(P_i)}\big)\big)$.
    Finally, since 
    $\chi(\Sigma) = \displaystyle\sum_{i = 1}^{g+p-2} \chi(P_i) + \displaystyle\sum_{j = 1}^{g} \chi(T_j)$ and 
    $n = \displaystyle\sum_{i = 1}^{g+p-2} e\big(\phi|_{\pi_1(P_i)}\big) + \displaystyle\sum_{j = 1}^{g} e\big(\phi|_{\pi_1(T_j)}\big)$
    by Proposition \ref{prop_additivity}, we obtain the desired generalized Milnor-Wood inequality.

\bigskip

    It remains to show that any non-empty $\HPns(\Sigma)$ above is connected. Let $\phi,\psi \in \HPns(\Sigma)$. 
         As the triple $(\Sigma, n, s)$ is not exceptional, by Theorem \ref{thm_nonabel} and Lemma \ref{lem_inthyp}, we can assume that $\phi$ and $\psi$ map all decomposition curves to hyperbolic elements. 
    Then by 
    Lemma \ref{lem_graph}, 
    it suffices to consider the case that 
    there exists a subsurface $\Sigma' \doteq  T_j\cup P_j$ for some $j\in \{1,\cdots, g\}$ or $\Sigma' \doteq P_i\cup P_{i+1}$ for some $i\in \{1,\cdots, g+p-2\}$,
    such that $e(\phi|_{\pi_1(\Sigma')}) = e(\psi|_{\pi_1(\Sigma')}) = n'$ for some $n',$ and for all 
    $k\in \{1,\cdots, g+p-2\}\setminus \{i, i+1\}$ and $l\in \{1,\cdots, g\}\setminus \{j\}$, $e\big(\phi|_{\pi_1(P_k)}\big)  = e\big(\psi|_{\pi_1(P_k)}\big)$ and $e\big(\phi|_{\pi_1(T_l)}\big)  =e\big(\psi|_{\pi_1(T_l)}\big).$
        In this case, we find a path in $\HPns(\Sigma)$ connecting $\phi$ and $\psi$ as follows. 
        We will show by induction that for all $k\in \{2,\cdots, -\chi(\Sigma)\}$, there is a connected subsurface $\Sigma_k$ of $\Sigma$ of Euler characteristic $\chi(\Sigma_k) = -k$ bounded by the decomposition curves and the boundary components of $\Sigma$; 
        and that there is a path $\path$ in $\HP(\Sigma_k)$
        connecting $\phi|_{\pi_1(\Sigma_k)}$ to $\psi|_{\pi_1(\Sigma_k)}$, such that for all $t\in [0,1]$,
        $s(\phi_t|_{\pi_1(\Sigma_k)}) = s(\phi|_{\pi_1(\Sigma_k)})$ and $e(\phi_t|_{\pi_1(\Sigma_k)}) = e(\phi|_{\pi_1(\Sigma_k)})$. Then since $\Sigma = \Sigma_{-\chi(\Sigma)}$, there is a path $\path$ in $\HPns(\Sigma)$ connecting $\phi$ and $\psi$. 
        For the base case $k = 2$, we let $\Sigma_2 = \Sigma'$. Since $s(\phi) = s(\psi)$, and since both $\phi$ and $\psi$ map the only decomposition curve to hyperbolic elements, we have $s(\phi|_{\pi_1(\Sigma')}) = s(\psi|_{\pi_1(\Sigma')}) = s'$ for some $s'$. 
        Then by  Theorem \ref{thm_torustype1}, Theorem \ref{thm_pantstype1} and Theorem \ref{thm_pantstype2}, there is a path $\path$ in $\HP^{s'}_{n'}(\Sigma')$ connecting $\phi|_{\pi_1(\Sigma')}$ to $\psi|_{\pi_1(\Sigma')}$. 
     Now assume that, for some
     $k\in \{2, \cdots, -\chi(\Sigma)\}$, 
     there is a connected subsurface $\Sigma_k\subset \Sigma$ with Euler characteristic $\chi(\Sigma_k) = -k$ and  
        with a path $\path$ in $\HP(\Sigma_k)$
        connecting $\phi|_{\pi_1(\Sigma_k)}$ and $\psi|_{\pi_1(\Sigma_k)}$ such that for all $t\in [0,1]$,
        $s(\phi_t|_{\pi_1(\Sigma_k)}) = s(\phi|_{\pi_1(\Sigma_k)})$ and $e(\phi_t|_{\pi_1(\Sigma_k)}) = e(\phi|_{\pi_1(\Sigma_k)})$. 
        We will find a subsurface $\Sigma_{k+1}$ of $\Sigma$ with Euler characteristic $\chi(\Sigma_{k+1}) = -(k+1)$, and find the path in $\HP(\Sigma_{k+1})$ connecting $\phi|_{\pi_1(\Sigma_{k+1})}$ and $\psi|_{\pi_1(\Sigma_{k+1})}$. 
        Since $\chi(\Sigma_k) > \chi(\Sigma)$, $\Sigma\setminus \Sigma_k$ is non-empty, and there exists a $\Sigma'' \doteq P_i$ for some $i\in \{1,\cdots, g+p-2\}$ or $\Sigma'' \doteq  T_j$ for some $j\in \{1,\cdots, g\}$  that shares a common boundary component with $\Sigma_k$. 
        Then $\phi$ and $\psi$ map this common boundary component to a hyperbolic element, and $e\big(\phi|_{\pi_1(\Sigma'')}\big)  = e\big(\psi|_{\pi_1(\Sigma'')}\big)$. Hence, by applying Proposition \ref{prop_sameEuler} to $\Sigma_{k+1}\doteq \Sigma_k\cup \Sigma''$, we obtain a path connecting $\phi|_{\pi_1(\Sigma_{k+1})}$ to $\psi|_{\pi_1(\Sigma_{k+1})}$, such that for all $t\in [0,1]$,
        $s(\phi_t|_{\pi_1(\Sigma_{k+1})}) = s(\phi|_{\pi_1(\Sigma_{k+1})})$ and $e(\phi_t|_{\pi_1(\Sigma_{k+1})}) = e(\phi|_{\pi_1(\Sigma_{k+1})})$. This completes the proof.
        \end{proof}
        
        \begin{proof}[Proof of Corollary \ref{thm_main3} (1) and Corollary \ref{total} (1)] 
        By Theorem \ref{thm_general}, for $g\geqslant 1,$ the connected components of $\mathcal R(\Sigma)$ are in one-to-one correspondence to the pairs $(n,s)$ that satisfy the generalized Milnor-Wood inequality  $\chi(\Sigma) + p_+(s)\leqslant n\leqslant -\chi(\Sigma) - p_-(s)$. 
        To count the number of such pairs, we let $k = p_+(s)$. Then $0\leqslant k\leqslant p$ and the generalized Milnor-Wood inequality becomes 
    $n + p + \chi(\Sigma)\leqslant k \leqslant n -\chi(\Sigma),$ which together imply the formula in Corollary \ref{thm_main3} (1). Finally, combining with the original Milnor-Wood inequality $\chi(\Sigma)\leqslant n\leqslant -\chi(\Sigma)$, we obtain the formula in Corollary \ref{total} (1). \end{proof}

\section{Exceptional components}

In this section  we complete the proof of Theorem \ref{thm_main2}. As Case (1) is proved in Theorem \ref{thm_general}, it remains to prove Cases (2) and (3). 
 Let $\Sigma = \Sigma_{0,p}$ with $p\geqslant 3$ 
and with  the following presentation of the fundamental group $$\pi_1(\Sigma) = \langle c_1, c_2,\cdots, c_p | c_1c_2\cdots c_p \rangle.$$
We will use  the chosen almost-path decomposition $\Sigma = \displaystyle\bigcup^{p-2}_{i = 1} P_i$, with the decomposition curves  $d_i = c_{i+2}\cdots c_p = (c_1\cdots c_{i+1})^{-1}$ for $i\in \{1,\cdots, p-3\}$.
For $n\in \mathbb{Z}$ and $s\in \{\pm 1\}^p$ corresponding to either Case (2) or Case (3) in Theorem \ref{thm_main2}, 
the triple $(\Sigma, n, s)$ is exceptional, and does not satisfy the generalized Milnor-Wood inequality. 
For each of these two cases, we will prove that the corresponding spaces $\Rns$ remain non-empty and connected, forming  the \emph{exceptional components} 
of $\R$. Moreover, we will show that these exceptional components consist of type-preserving representations that send all the decomposition curves $\{d_1,\dots, d_{p-3}\}$ of the chosen almost-path decomposition to non-hyperbolic elements of $\psl$, which we will call the \emph{totally non-hyperbolic representations}. We will also show that these totally non-hyperbolic representations send not only the decomposition curves but also all the simple closed curves on $\Sigma$ to non-hyperbolic elements, hence coincide with the notion of the ``totally non-hyperbolic representations" in \cite{deroin_tholozan}.


    \subsection{Abelian representations}
    In this subsection, we will prove Theorem \ref{thm_main2} for Case (2). Moreover, we will show that the corresponding components $\mathcal{R}^s_n(\Sigma)$ consist entirely of abelian representations. This also explains why Theorem \ref{thm_nonabel}, and consequently, Theorem \ref{thm_general} fail in this case. 
    These results are stated as Theorem \ref{thm_Exc1} below.

    \begin{theorem}\label{thm_Exc1}
        Let $\Sigma = \Sigma_{0,p}$ with $p\geqslant 3$, and let $s\in \{\pm 1\}^p$ with $\ p_+(s) = 1 \mbox{ or }p_-(s) = 1.$ Then 
        $\mathcal{R}^s_0(\Sigma)$ is non-empty, connected, and consists of abelian representations. 
    \end{theorem}
    \begin{proof}
       As the case that $p = 3$ is proved by Theorem \ref{thm_pants} and Lemma \ref{lem_R0pants}, we only need to consider the case that $p\geqslant 4$. We will prove the result for the case that $p_+(s)=1$. Then the result for the remaining case that $p_-(s)=1$ follows from the previous case and Proposition \ref{prop_pglpsl}.
       \\

       Up to a permutation of the peripheral elements we can assume that $s_1 = +1$ and $s_2=\cdots=s_p= -1$. To show that $\mathcal{R}^s_0(\Sigma)$ is non-empty, we define a  representation $\phi$ by letting 
        $\phi(c_1) = \pm \begin{bmatrix}
                        1 & p-1\\
                        0 & 1
                    \end{bmatrix}$ and 
        $\phi(c_i) = \pm \parmp$ for each $i\in \{2,\cdots, p\}$, which is an abelian representation and is in $\mathcal{R}^s_0(\Sigma)$ by Proposition \ref{prop_reducible}.
        \medskip 
        
        Next, we prove that every representation $\phi$ in $\mathcal{R}^s_0(\Sigma)$ is abelian.
        Let $ \displaystyle\bigcup_{i=1}^{p-2} P_i$ be the chosen almost-path decomposition of $\Sigma$.
        We will first show that $\phi |_{\pi_1(P_i)}$ is abelian for all $i\in\{1,\dots,p-2\}$, then use this to show that $\phi$ is abelian on the entire $\pi_1(\Sigma)$.
        
        We use contradiction. Suppose otherwise that $\phi|_{\pi_1(P_i)}$ is non-abelian for some $i\in \{1,\cdots, p-2\}$. 
        Then by Proposition \ref{prop_NAsphere1}, we can assume that $\phi\in \mathrm{NA}^s_0(\Sigma)$; and since $s_1 = -s_2$, by Lemma \ref{lem_inthyp_sphere}, $\phi(d_1)$ is hyperbolic. 
        As the fundamental group $\pi_1(P_1)$ has the preferred peripheral elements $c_1, c_2$ and $d_1$ with $\phi(c_1)\in \Par^+, \phi(c_2)\in \Par^-$ and $\phi(d_1)\in \Hyp$, we have 
        $\phi|_{\pi_1(P_1)}\in \HP(P_1)$ with sign $s\big(\phi|_{\pi_1(P_1)}\big) = (+1, -1, 0)$; and by Theorem \ref{thm_pants2}, it follows that the relative Euler class $e\big(\phi|_{\pi_1(P_1)}\big) = 0$.
        Similarly, as the fundamental group $\pi_1(\Sigma\setminus P_1)$ has the preferred peripheral elements $c_3,\cdots, c_p$ and $d_1^{-1}$ with $\phi(c_i)\in \Par^-$ for 
    $i\in \{3,\cdots, p\}$ and $\phi(d_1^{-1})\in \Hyp$, we have 
    $\phi|_{\pi_1(\Sigma \setminus P_1)}\in \HP(\Sigma \setminus P_1)$ with sign $s\big(\phi|_{\pi_1(\Sigma\setminus P_1)}\big) = (-1,\cdots, -1, 0)$.
    As $p_0\big(s\big(\phi|_{\pi_1(\Sigma\setminus P_1)}\big)\big) = 1$, the triple $\big(\Sigma\setminus P_1, e(\phi|_{\pi_1(\Sigma \setminus P_1)}), s(\phi|_{\pi_1(\Sigma \setminus P_1)})\big)$ is not exceptional, hence by Theorem \ref{thm_general}  satisfies the generalized Milnor-Wood inequality
     $$\chi(\Sigma \setminus P_1) + p_+\big(s(\phi|_{\pi_1(\Sigma \setminus P_1)})\big)\leqslant e(\phi|_{\pi_1(\Sigma \setminus P_1)})\leqslant -\chi(\Sigma \setminus P_1) - p_-\big(s(\phi|_{\pi_1(\Sigma \setminus P_1)})\big).$$
     As $\chi(\Sigma\setminus P_1) = 3-p$ 
     and $p_-\big(s(\phi|_{\pi_1(\Sigma \setminus P_1)})\big) = p-2$, it follows that
     $e(\phi|_{\pi_1(\Sigma \setminus P_1)})\leqslant -1$. Then as
     $e\big(\phi|_{\pi_1(P_1)}\big) = 0$, by Proposition \ref{prop_additivity}, we have 
     $0 = e\big(\phi|_{\pi_1(P_1)}\big) + e\big(\phi|_{\pi_1(\Sigma \setminus P_1)}\big)\leqslant -1$, which is a contradiction. 
     
     We now prove that $\phi$ is abelian on $\pi_1(\Sigma)$. For each $i\in \{1,\cdots, p-2\}$, as $\phi|_{\pi_1(P_i)}$  is abelian, the image $\phi\big(\pi_1(P_i)\big)$ lies in a one-parameter subgroup $S_i$ of $\psl$; and for each $i\in \{1,\cdots, p-3\}$, either $S_i =S_{i+1}$ or $S_i\cap S_{i+1} = \{\pm \mathrm I\}$. We will show that 
         $S_i =S_{i+1}$ for all $i\in \{1,\cdots, p-3\}$, and conclude that the entire image $\phi\big(\pi_1(\Sigma)\big)$ lies in the one-parameter subgroup $S_1$, and hence $\phi$ abelian. 
         Suppose that $S_i \neq S_{i+1}$ for some $i\in \{1,\cdots, p-3\}$. Then $S_i\cap S_{i+1} = \{\pm \mathrm I\}$, and as an element in $\phi\big(\pi_1(P_i)\big)\cap  \phi\big(\pi_1(P_{i+1})\big)\subset S_i\cap S_{i+1}$, $\phi(d_i) = \pm \mathrm I$. 
         Since there exists an $i\in \{1,\cdots, p-3\}$ where $\phi(d_i) = \pm \mathrm I$, we can choose the largest $i_0\in\{1,\cdots, p-3\}$  such that $\phi(d_{i_0}) = \pm \mathrm I$. 
         If $i_0 = p-3$, then as $\phi(c_{p-1})\phi(c_p) = \phi(d_{p-3}) = \pm \mathrm I$, we have $\phi(c_{p-1}) = \phi(c_p)^{-1}$, contradicting that both $\phi(c_{p-1})$ and $\phi(c_p)$ are in $\Par^-$.
         If otherwise that $i_0 \leqslant p-4$, we let $\Sigma' \doteq \displaystyle\bigcup^{p-2}_{i = i_0+2} P_i$. 
         The fundamental group $\pi_1(\Sigma')$ has the preferred peripheral elements $c_{i_0+3},\cdots, c_p$ and $d_{i_0+1}^{-1}$, and for $k\in \{i_0+3,\cdots, p\}$, we have $\phi(c_k)\in \Par^-$. Moreover, since $\phi(d_{i_0}^{-1}) = \pm \mathrm I$ and $\phi(c_{i_0+2})\in \Par^-$, we have 
         $\phi(d_{i_0 + 1}^{-1}) = \phi(d_{i_0}^{-1})\phi(c_{i_0+2})\in \Par^-$. Therefore, $\phi|_{\pi_1(\Sigma')}$ is type-preserving, and $p_+\big(s\big(\phi|_{\pi_1(\Sigma')}\big)\big) = 0$. 
         However, as $\phi |_{\pi_1(P_i)}$ is abelian for all $i\in \{i_0+2,\cdots, p-2\}$, $p_+\big(s\big(\phi|_{\pi_1(\Sigma')}\big)\big)\geqslant 1$ by Lemma \ref{lem_abelsphere}, which is a contradiction. Therefore, $S_i =S_{i+1}$ for all $i\in \{1,\cdots, p-3\}$; and as a consequence, $\phi$ is abelian.
         \\
         
        Finally, we prove that $\mathcal{R}^s_0(\Sigma)$ is connected. 
        Let $\phi_1, \phi_2\in \mathcal{R}^s_0(\Sigma)$. For each $j\in \{1,2\}$, since $\phi_j(c_1)\in \Par^+$, there is a $g_j\in \psl$ such that $g_j\phi_j(c_1)g_j^{-1} = \pm \parpp$.
        Let $\{g_{j,t}\}\interval$ be a path in $\psl$ connecting $\pm\mathrm{I}$ to $g_j$. 
         Then the path $\{\phi_{j,t}\}\interval$ defined by $\phi_{j,t}\doteq g_{j,t}\phi_j g_{j,t}^{-1}$ connects $\phi_j$ to $\phi_{j,1} = g_j\phi_jg_j^{-1}$.
        Since $\phi_j \in \mathcal{R}^s_0(\Sigma)$ is abelian, its $\psl$-conjugation $\phi_{j,1}$ is also abelian; and the image $\phi_{j,1}\big(\pi_1(\Sigma)\big)$
        lies in the one-parameter subgroup of $\psl$ that contains $g_j\phi_j(c_1)g_j^{-1}$, which is $\bigg\{\pm \part\bigg\}_{t\in \mathbb{R}}$. 
        Moreover, for each $j\in \{1,2\}$ and $t\in [0,1]$, since $\phi_{j,t}(c_i)$ and $\phi_j(c_i)$ are conjugate for all $i\in \{1,\cdots, p\}$, the signs $s(\phi_{j,t}) = s(\phi_j)$ and the relative Euler classes $e(\phi_{j,t}) = e(\phi_j)$. Therefore, the paths $\{\phi_{1,t}\}\interval$ and $\{\phi_{2,t}\}\interval$ are in $\mathcal{R}^s_0(\Sigma).$
        Now for each $i\in \{1,\cdots, p\}$ and each $j\in \{1,2\}$, 
        we let 
        $g_j\phi_j(c_i)g_j^{-1} = 
        \pm \begin{bmatrix}
        1 & \xi_{i,j} \\
        0 & 1
        \end{bmatrix}$ for some $\xi_{i,j}\in \mathbb{R}\setminus \{0\}$.
        To connect $g_1\phi_1g_1^{-1}$ and $g_2\phi_2g_2^{-1},$ we
        consider the path $\{\psi_t\}_{t\in [1,2]}$ given by $\psi_t(c_i) = \pm  \begin{bmatrix}
        1 & (2-t)\xi_{i,1}+(t-1)\xi_{i,2} \\
        0 & 1
        \end{bmatrix}$ for each $i\in \{1,\cdots, p\}$. 
        Since $s(g_1\phi_1g_1^{-1})=s(g_2\phi_2g_2^{-1}),$ 
        for each $i\in \{1,\cdots, p\}$, 
        we have $sgn(\xi_{i,1})=sgn(\xi_{i,2})$, 
        and hence for all $t\in [1,2],$ we have $sgn\big((2-t)\xi_{i,1}+(t-1)\xi_{i,2}\big)=sgn(\xi_{i,1})$. 
        Therefore, for all $t\in [1,2]$, the signs $s(\psi_t)=s(g_1\phi_1g_1^{-1})$ 
         and the relative Euler classes $e(\psi_t)=e(g_1\phi_1g_1^{-1})$, and the path $\{\psi_t\}_{t\in [1,2]}$ is in $\mathcal R_0^s(\Sigma)$ connecting $g_1\phi_1g_1^{-1}$ and $g_2\phi_2g_2^{-1}.$ As a consequence, $\mathcal R_0^s(\Sigma)$ is connected. 
\end{proof}

\begin{remark}
        Theorem \ref{thm_Exc1} shows that $\mathcal{R}^s_0(\Sigma)$ with $p_+(s) = 1$ or $p_-(s) = 1$ consists of totally non-hyperbolic representations,
        as abelian type-preserving representations map all the simple closed curves (hence all the decomposition curves) to parabolic elements.
    \end{remark}

\subsection{Non-abelian representations}

    In this subsection, we will prove Theorem \ref{thm_main2} for Case (3). Moreover, we will show that the corresponding components $\mathcal{R}^s_n(\Sigma)$ consist entirely of totally non-hyperbolic representations. This also explains why Lemma \ref{lem_inthyp}, and consequently, Theorem \ref{thm_general} fail in this case. 
    These results are stated as Theorem \ref{thm_Exc2} below.

    \begin{theorem}\label{thm_Exc2}
        Let $\Sigma = \Sigma_{0,p}$ with $p\geqslant 3$, and let $s\in \{\pm 1\}^p$ with $\ p_-(s) = 0 \mbox{ or }p_+(s) = 0.$ Then 
        $\mathcal{R}^s_{s_1}(\Sigma)$ is non-empty, connected, and consists of totally non-hyperbolic representations. 
    \end{theorem}

    We will first extend the lifted product map 
    $\ev: \overline{\Hyp}\times \overline{\Hyp}\to \widetilde{\SL}$
    defined in Subsection \ref{subsection_PLP} to $\psl\times \psl$ in Subsection \ref{subsection_PLP_Ell}, using which we will prove Theorem \ref{thm_Exc2} in Subsection \ref{subsection_exc2proof}.

\subsubsection{The evaluation map and the path-lifting properties}\label{subsection_PLP_Ell}

    In this subsection, we define the \emph{lifted product map}  
        $$\ev: \psl\times \psl\to \widetilde{\SL}$$
    by sending the pair $(g_1,g_2)$ to the product $\widetilde{g_1}
    \widetilde{g_2}$ of the unique lifts $\widetilde{g_1}$ and $
    \widetilde{g_2}$ of  $g_1$ and $g_2$  in $\overline{\Hyp_0}\cup \Ell_1$.
    Recall that the product $\widetilde{g_1}\widetilde{g_2}$ in the universal cover $\widetilde{\SL}$ is defined as follows. For $i\in \{1,2\},$ let $\{\widetilde{g_{i,t}}\}\interval$ be a path connecting $\widetilde{g_{i,0}} = \mathrm{I}$ and $\widetilde{g_{i,1}} = \widetilde{g_i}$ in the one-parameter subgroup of $\univcover$ generated by $\widetilde{g_i},$ and let $\{g_{i,t}\}\interval$ be the projection of $\{\widetilde{g_{i,t}}\}\interval$ to $\psl$.
    Let $\{\widetilde{g_{1,t}g_{2,t}}\}\interval$ be the lift of the path $\{g_{1,t}g_{2,t}\}\interval$ in $\psl$ to  $\widetilde{\SL}$ starting from $\mathrm I.$ Then  $\widetilde{g_1}\widetilde{g_2}$ is the endpoint $\widetilde{g_{1,1}g_{2,1}}$ of the lifted path. Here we notice that   $\ev$ is not continuous on the whole $\psl\times\psl.$ See also \cite{deroin_tholozan} for a similar discussion.
    \\
     
     To define the evaluation maps and study their path-lifting properties, we need the following Lemma \ref{lem_offdiag_Ell} and Lemma \ref{lem_evimage_Ell}. 

\begin{lemma}\label{lem_offdiag_Ell} 
        For $n\in \mathbb{Z}$ and an elliptic element $\widetilde{A}\in \mathrm{Ell}_n$ of $\univcover,$ let $A$ be its projection to $\SL$, and for $i,j\in\{1,2\}$ let $a_{ij}$ be the $(i,j)$-entry of $A.$ 
        Then $a_{12}\neq0$ and $a_{21}\neq 0.$  Moreover:
        \begin{enumerate}[(1)]
            \item 
            If $n$ is odd, then $sgn(n)=sgn(a_{12}) = -sgn(a_{21})$.
            
            \item 
            If $n$ is even, then $sgn(n)=-sgn(a_{12}) = sgn(a_{21})$.
        \end{enumerate}
    \end{lemma}
    \begin{proof} 
    Let $\widetilde{\exp}: \sl\to \univcover$ be the exponential map of $\univcover$.
        By the definition of $\Ell_n$, $\widetilde{A}\in \Ell_n$ is conjugate to $\widetilde{\exp}\begin{bmatrix}
        0 & \alpha \\
       -\alpha & 0 
    \end{bmatrix}$ for some $\alpha \in \big((n-1)\pi, n\pi\big)$ if $n>0$, and for
    some $\alpha \in \big(n\pi, (n+1)\pi\big)$ if $n<0$.
    Therefore, letting $\exp: \sl\to \SL$ be the exponential map of $\SL$,
    the projection $A$ of $\widetilde{A}$ is conjugate to 
    $\exp\begin{bmatrix}
        0 & \theta \\
       -\theta & 0 
    \end{bmatrix}
    = \begin{bmatrix}
        \cos\theta & \sin\theta \\
       -\sin\theta & \cos\theta
    \end{bmatrix}$ in $\SL$ for some $\theta\in (-\pi,0)\cup(0, \pi)$ 
    with $\theta \equiv \alpha \pmod{2\pi}$. 
    It follows that
    $$A = \abcd
        \begin{bmatrix}
            \cos\theta & \sin\theta \\
            -\sin\theta & \cos\theta
        \end{bmatrix}
        \abcd^{-1}
        = \begin{bmatrix}
            \cos\theta - (ac+bd)\sin\theta & (a^2 + b^2)\sin\theta \\
            -(c^2 + d^2)\sin\theta & \cos\theta + (ac+bd)\sin\theta
        \end{bmatrix}$$ for some matrix $\abcd\in \SL$. 
        Since $ad - bc = 1$, at least one of $a$ and $b$ is nonzero, hence $a^2 + b^2 > 0$; and at least one of $c$ and $d$ is nonzero, hence $c^2 + d^2 > 0$. 
        In the case that $n$ is odd, if $n>0$ then $\theta\in (0,\pi)$, and hence 
        $a_{12} = (a^2+b^2)\sin\theta >0$ and $a_{21} = -(c^2+d^2)\sin\theta < 0$; and 
        if $n<0$ then $\theta\in (-\pi,0)$, and hence 
        $a_{12} = (a^2+b^2)\sin\theta <0$ and $a_{21} = -(c^2+d^2)\sin\theta > 0$. As a consequence, $a_{12}\neq 0$, $a_{21}\neq 0$, and $sgn(n) = sgn(a_{12}) = -sgn(a_{21})$. Similarly, in the case that $n$ is even, if $n>0$ then $\theta\in (-\pi,0)$, and hence 
        $a_{12} = (a^2+b^2)\sin\theta <0$ and $a_{21} = -(c^2+d^2)\sin\theta > 0$; and 
        if $n<0$ then $\theta\in (0,\pi)$, and hence 
        $a_{12} = (a^2+b^2)\sin\theta >0$ and $a_{21} = -(c^2+d^2)\sin\theta < 0$. As a consequence, $sgn(n)=-sgn(a_{12}) = sgn(a_{21})$.
    \end{proof} 
    \begin{lemma}\label{lem_evimage_Ell}Let $s\in \{\pm 1\}$, and let $\ev: \psl\times \psl\to \widetilde{\SL}$ be the lifted product map defined above.
    \begin{enumerate}[(1)] 
    
    \item $\ev\big(\Par^{sgn(s)} \times\Ell\big)\cap \big(\bigcup_{k\in \mathbb{Z}\setminus \{0\}}\Ell_k\big) = \Ell_1$.
    \item $\ev\big(\Par^{sgn(s)} \times\Par^{sgn(s)} \big)\cap \big(\bigcup_{k\in \mathbb{Z}\setminus \{0\}}\Ell_k\big) = \Ell_s$.
    \end{enumerate}
    \end{lemma}
    \begin{proof} 
    For (1), we first show 
    $\ev\big(\Par^{sgn(s)} \times\Ell\big)\cap \big(\bigcup_{k\in \mathbb{Z}\setminus \{0\}}\Ell_k\big) \subset \Ell_1$ as follows. 
    Let $(g_1, g_2)\in \Par^{sgn(s)}\times \Ell$, and suppose that $\ev(g_1, g_2) \in \Ell_k$ for some $k\in \mathbb{Z}\setminus\{0\}$. 
    Let $\widetilde{g_1}$ be the lift of $g_1$ in $\Par^{sgn(s)}_0$, and let $\widetilde{g_2}$ be the lift of $g_2$ in $\Ell_1$.
    For $i\in \{1,2\},$ let $\{\widetilde{g_{i,t}}\}\interval$ be a path  connecting $\mathrm{I}$ and $\widetilde{g_i}$ in the one-parameter subgroup of $\univcover$ generated by $\widetilde{g_i},$ and let $\{g_{i,t}\}\interval$ be the projection of $\{\widetilde{g_{i,t}}\}\interval$ to $\psl$.
    Then the lift $\{\widetilde{g_{1,t}g_{2,t}}\}\interval$ of $\{g_{1,t}g_{2,t}\}_{t\in [0,1]}$ in $\widetilde{\SL}$ connects $\widetilde{g_{1,0}g_{2,0}} = \mathrm{I}$ to $\widetilde{g_{1,1}g_{2,1}} = \widetilde{g_1}\widetilde{g_2}$. Up to a $\psl-$conjugation, we can assume that $g_1 = \pm \begin{bmatrix}
        1 & s \\
        0 & 1
    \end{bmatrix}$; and by a re-parametrization if necessary, we can assume that 
        the projection $\{A_t\}\interval$ of the path $\{\widetilde{g_{1,t}}\}\interval$  in $\univcover$ to $\SL$ equals 
        $$A_t
        = \begin{bmatrix}
            1 & st \\
            0 & 1
        \end{bmatrix}$$
        for all $t\in [0,1].$ 
        Similarly, we can also assume that the projection $\{B_t\}\interval$ of the path $\{\widetilde{g_{2,t}}\}\interval$  in $\univcover$ to $\SL$ equals $$B_t=\abcd\begin{bmatrix}
            \cos(\theta t) & \sin(\theta t)\\
            -\sin(\theta t) & \cos(\theta t)
        \end{bmatrix}\abcd^{-1}$$
        for some $\theta\in (0,\pi)$ and $\abcd\in\SL$. 
        Then the projection $\{A_tB_t\}\interval$ of the path $\{\widetilde{g_{1,t}g_{2,t}}\}\interval$  in $\univcover$ to $\SL$ equals 
          \begin{equation*}
          \begin{split}
          &A_tB_t=\\
           &\begin{bmatrix}
            \cos(\theta t) - (ac + bd)\sin(\theta t) - (c^2+d^2)st\sin(\theta t) 
               & (a^2+b^2)\sin(\theta t) + st \cos(\theta t) + (ac+bd)st\sin(\theta t) \\
               -(c^2 + d^2)\sin(\theta t) & \cos(\theta t) + (ac + bd)\sin(\theta t)
            \end{bmatrix}.
            \end{split}
            \end{equation*}
            Since $ad - bc = 1$, at least one of $c$ and $d$ is nonzero, hence $c^2 + d^2 > 0$; and since $\theta\in (0, \pi)$, we have the $(2,1)$-entry $(A_tB_t)_{21} = -(c^2+d^2)\sin(\theta t) <0$ for all $t\in (0,1]$. 
            Then by Lemma \ref{lem_offdiag_Ell}, the path $\{\widetilde{g_{1,t}g_{2,t}}\}\interval$ intersects neither $\Ell_{-1}$ nor $\Ell_2$, hence lies in $\overline{\Hyp_0}\cup \overline{\Hyp_1}\cup \Ell_1$ as it is the connected component 
            of $\univcover\setminus \big(\Ell_{-1}\cup \Ell_2\big)$ that contains $\widetilde{g_{1,0}g_{2,0}} = \mathrm I$.
            As $\widetilde{g_{1,1}g_{2,1}} = \widetilde{g_1}\widetilde{g_2}$ is elliptic by the assumption, we have $\widetilde{g_{1,1}g_{2,1}}\in \Ell_1$. 
            \medskip

            Next, to show that 
    $\Ell_1\subset \ev\big(\Par^{sgn(s)} \times\Ell\big)\cap \big(\bigcup_{k\in \mathbb{Z}\setminus \{0\}}\Ell_k\big)$, 
    we consider $$g_1=\pm \begin{bmatrix}
        1 & s \\
        0 & 1
    \end{bmatrix}\quad\text{and}\quad g_2
        = \pm \begin{bmatrix}
        \cos\theta & \frac{1}{2}\sin\theta \\
        -2\sin\theta & \cos\theta
        \end{bmatrix}$$ for $\theta\in (0,\pi)$.  Then $\ev(g_1,g_2) = \widetilde{g_1}\widetilde{g_2}$ projects to the matrix 
         $$A_\theta = \begin{bmatrix}
        \cos\theta - 2s\sin\theta &  \frac{1}{2}\sin\theta + s\cos\theta  \\
        -2\sin\theta & \cos\theta
        \end{bmatrix}$$
           in $\SL$ with $\tr(A_\theta)=2(\cos\theta-s\sin\theta).$ 
            If $s = +1$, 
            then as $\theta$ varies from $0$ to $\frac{\pi}{2},$ the trace $\tr(A_\theta)=2(\cos\theta-\sin\theta)$ decreases from $2$ to $-2$; and 
            if $s = -1$, 
            then as $\theta$ varies from $\frac{\pi}{2}$ to $\pi,$ the trace $\tr(A_\theta)=2(\cos\theta+\sin\theta)$ decreases from $2$ to $-2$. Moreover, whenever $\tr(A_\theta)\in (-2,2)$, the corresponding $\widetilde{g_1}\widetilde{g_2}$ is elliptic as it projects to $A_\theta$ in $\SL$, hence lies in $\Ell_1$ as shown in the previous paragraph. 
            Since two elements in $\Ell_1$ are conjugate if and only if they have the same trace, for any $\widetilde{C} \in \Ell_1$, there exists a $\theta\in (0,\pi)$ where the corresponding $\widetilde{g_1}\widetilde{g_2}$ is conjugate to $\widetilde{C}$ by some element $\widetilde{h}\in \univcover$. Let $h$ be the projection of $\widetilde{h}$ to $\psl$. Then  $\widetilde{h}\widetilde{g_1}\widetilde{h}^{-1}$ is the unique lift of $hg_1h^{-1}$ in $\Par^{sgn(s)}_0$, $\widetilde{h}\widetilde{g_2}\widetilde{h}^{-1}$ is the unique lift of $hg_2h^{-1}$ in $\Ell_1$,  and we have $\ev(hg_1h^{-1}, hg_2h^{-1}) = 
        \widetilde{C}$. Therefore, $\Ell_1\subset \ev\big(\Par^{sgn(s)} \times\Ell\big)$, and consequently, 
            $\Ell_1\subset \ev\big(\Par^{sgn(s)} \times\Ell\big)\cap \big(\bigcup_{k\in \mathbb{Z}\setminus \{0\}}\Ell_k\big)$.
        \\

    For (2), we first show 
    $\ev\big(\Par^{sgn(s)} \times\Par^{sgn(s)} \big)\cap \big(\bigcup_{k\in \mathbb{Z}\setminus \{0\}}\Ell_k\big) \subset \Ell_s$ as follows. 
    Let $g_1, g_2\in \Par^{sgn(s)}$, and suppose that $\ev(g_1, g_2) \in \Ell_k$ for some $k\in \mathbb{Z}\setminus\{0\}$. 
    For $i\in \{1,2\},$ let $\widetilde{g_i}$ be the lift of $g_i$ in $\Par^{sgn(s)}_0$, let $\{\widetilde{g_{i,t}}\}\interval$ be a path connecting $\mathrm{I}$ and $\widetilde{g_i}$ in the one-parameter subgroup of $\univcover$ generated by $\widetilde{g_i},$ and let $\{g_{i,t}\}\interval$ be the projection of $\{\widetilde{g_{i,t}}\}\interval$ to $\psl$. Then the lift $\{\widetilde{g_{1,t}g_{2,t}}\}\interval$ of $\{g_{1,t}g_{2,t}\}_{t\in [0,1]}$ in $\widetilde{\SL}$ connects $\widetilde{g_{1,0}g_{2,0}} = \mathrm{I}$ to $\widetilde{g_{1,1}g_{2,1}} = \widetilde{g_1}\widetilde{g_2}$. Up to a $\psl-$conjugation, we can assume that $g_1 = \pm \begin{bmatrix}
        1 & s\\
        0 & 1
    \end{bmatrix}$; and by a re-parametrization if necessary, we can assume that 
        the projection $\{A_t\}\interval$ of the path $\{\widetilde{g_{1,t}}\}\interval$  in $\univcover$ to $\SL$ equals 
        $$A_t
        = \begin{bmatrix}
        1 & st\\
        0 & 1
    \end{bmatrix}$$
        for all $t\in [0,1].$ 
        Similarly, we can also assume that the projection $\{B_t\}\interval$ of the path $\{\widetilde{g_{2,t}}\}\interval$  in $\univcover$ to $\SL$ equals $$B_t=\abcd\begin{bmatrix}
            1 & st \\
            0 & 1
        \end{bmatrix}\abcd^{-1}$$
        for some $\abcd\in\SL$. 
        Then the projection $\{A_tB_t\}\interval$ of the path $\{\widetilde{g_{1,t}g_{2,t}}\}\interval$  in $\univcover$ to $\SL$ equals 
          $$A_tB_t=  
        \begin{bmatrix}
                1 - sact - c^2t^2
                & (1 + a^2)st + ac t^2  \\
                -sc^2t 
                & 1 + sact
            \end{bmatrix}.$$
           We observe that $c\neq 0$, as otherwise $\widetilde{g_1}\widetilde{g_2}= \begin{bmatrix}
                1 
                & (1 + a^2)s  \\
               0
                & 1
            \end{bmatrix} \not\in \Ell_k$.
            Then for all $t\in (0,1]$, the sign of the $(2,1)$-entry $(A_tB_t)_{21} = -sc^2t$ equals $sgn(-sc^2t) = -sgn(s)$. 
            Then by Lemma \ref{lem_offdiag_Ell}, the path $\{\widetilde{g_{1,t}g_{2,t}}\}\interval$ intersects neither $\Ell_{-s}$ nor $\Ell_{2s}$, 
            hence lies in $\overline{\Hyp_0}\cup \overline{\Hyp_s}\cup \Ell_s$ as it is the connected component 
            of $\univcover\setminus \big(\Ell_{-s}\cup \Ell_{2s}\big)$ that contains $\widetilde{g_{1,0}g_{2,0}} = \mathrm I$. 
            Since $\widetilde{g_{1,1}g_{2,1}} = \widetilde{g_1}\widetilde{g_2}$ is elliptic by the assumption, we have $\widetilde{g_{1,1}g_{2,1}}\in \Ell_s$. 
            \medskip
            
            Next, to show 
            $\Ell_s\subset \ev\big(\Par^{sgn(s)} \times\Par^{sgn(s)} \big)\cap \big(\bigcup_{k\in \mathbb{Z}\setminus \{0\}}\Ell_k\big)$, 
    we consider $$g_1=\pm \begin{bmatrix}
        1 & s \\
        0 & 1
    \end{bmatrix}\quad\text{and}\quad g_2
        = \pm \begin{bmatrix}
            1 & 0 \\
            -s\lambda & 1
        \end{bmatrix}$$ for some $\lambda >0$. 
        Then $\ev(g_1,g_2) = \widetilde{g_1}\widetilde{g_2}$ projects to the matrix 
         $$A_\theta = \begin{bmatrix}
        1-\lambda & s \\
        -s\lambda & 1
        \end{bmatrix}$$
           in $\SL$. 
            As $\lambda$ varies from $0$ to $4,$ the trace $\tr(A_\theta)=2-\lambda$ decreses from $2$ to $-2$.
            Moreover, whenever $\tr(A_\theta)\in (-2,2)$, the corresponding $\widetilde{g_1}\widetilde{g_2}$ is elliptic as it projects to $A_\theta$ in $\SL$, hence lies in $\Ell_s$ as shown in the previous paragraph. 
            Since two elements in $\Ell_s$ are conjugate if and only if they have the same trace, for any $\widetilde{C} \in \Ell_s$, there exists a $\lambda\in (0,4)$ where the corresponding $\widetilde{g_1}\widetilde{g_2}$ is conjugate to $\widetilde{C}$ by some element $\widetilde{h}\in \univcover$. Let $h$ be the projection of $\widetilde{h}$ to $\psl$. Then $\widetilde{h}\widetilde{g_i}\widetilde{h}^{-1}$ is the unique lift of $hg_ih^{-1}$ in $\Par_0$ for $i \in \{1,2\},$ and  we have $\ev(hg_1h^{-1}, hg_2h^{-1}) = 
            \widetilde{C}$. Therefore, $\Ell_s\subset \ev\big(\Par^{sgn(s)} \times\Par^{sgn(s)}\big)$, and consequently, 
            $\Ell_s\subset \ev\big(\Par^{sgn(s)} \times\Par^{sgn(s)}\big)\cap \big(\bigcup_{k\in \mathbb{Z}\setminus \{0\}}\Ell_k\big)$.  
    \end{proof}

    For $s\in \{\pm 1\}$, 
    let $\PE^{s}(\Sigma_{0,3})$ be the space of $\psl$-representations of $\pi_1(\Sigma_{0,3})$ that send the preferred peripheral element 
    $c_1$ into $\Par^{sgn(s)}$, and the preferred peripheral elements $c_2,c_3$ into $\Ell$. 
    Similarly, let $\PP^{s}(\Sigma_{0,3})$ be the space of $\psl$-representations of $\pi_1(\Sigma_{0,3})$ that send $c_1, c_2$ into $\Par^{sgn(s)}$, and $c_3$ into $\Ell$. 
    Under the identification of $\psl$-representations of $\pi_1(\Sigma_{0,3})$ with pairs of elements of $\psl,$  both $\PE^s(\Sigma_{0,3})$ and $\PP^s(\Sigma_{0,3})$ can be considered as subspaces of $\psl\times \psl$. By Lemma \ref{lem_evimage_Ell}, 
    we have $\ev\big(\PE^s(\Sigma_{0,3})\big) = \Ell_1$ and 
    $\ev(\PP^s\big(\Sigma_{0,3})\big) = \Ell_s$, and
    the \emph{evaluation maps} 
     $$\ev: \mathrm{PE}^s(\Sigma_{0,3})\to \Ell_1$$
     and 
     $$\ev: \mathrm{PP}^s(\Sigma_{0,3})\to \Ell_s$$
     are respectively defined as the restrictions of the lifted product map. Here we notice that, even though the lifted product map is not continuous on the whole $\psl\times \psl$, its restrictions to $\PE^s(\pants)$ and $\PP^s(\pants)$ are continuous.
        \begin{proposition} \label{prop_PLP_Ell}
        Let $s\in \{\pm 1\}$.
        \begin{enumerate}[(1)]
    \item The evaluation map $\ev: \PE^s(\Sigma_{0,3})\to \Ell_1$ satisfies the path-lifting property. Moreover, for each $\widetilde{C}$ in $\Ell_1$, the fiber $\ev^{-1}\big(\widetilde{C}\big)$ is connected in $\mathrm{PE}^s(\Sigma_{0,3})$.

    \item The evaluation map $\ev: \PP^s(\Sigma_{0,3})\to \Ell_s$ satisfies the path-lifting property. Moreover, for each $\widetilde{C}$ in $\Ell_s$, the fiber $\ev^{-1}\big(\widetilde{C}\big)$ is connected in $\PP^s(\Sigma_{0,3})$.
    \end{enumerate}
    \end{proposition}
    \begin{proof}
    For (1), by Lemma \ref{lem_plp} and Lemma \ref{lem_evPsubm}, it suffices to show that for all $\widetilde{C}\in \Ell_1$, the fiber $\ev^{-1}(\widetilde{C})$ is connected in $\PE^s(\Sigma_{0,3})$. Let $\widetilde{C}\in \Ell_1$, and let $\phi_1, \phi_2$ be representations in $\PE^s(\Sigma_{0,3})$ such that $\ev(\phi_1) = \ev(\phi_2) = \widetilde{C}$. For each $j \in \{ 1,2\}$,  we let the character of $\phi_j$ be the triple $$
    \chi(\phi_j) \doteq 
    \Big(
    \mathrm{tr}\big(\widetilde{\phi_j(c_1)}\big), 
    \mathrm{tr}\big(\widetilde{\phi_j(c_2)}\big), 
    \mathrm{tr}\big(\widetilde{\phi_j(c_1)}\widetilde{\phi_j(c_2)}\big)\Big)
    $$
    in $\mathbb{R}^3$, 
    where $\widetilde{\phi_j(c_1)}$ and 
    $\widetilde{\phi_j(c_2)}$ are respectively the lifts of $\phi_j(c_1)$ and $\phi_j(c_2)$ in $\overline{\Hyp_0}\cup \Ell_1$,
    and their traces are defined as the traces of their projections to $\SL$. 
    We will first construct a path 
    $\{\psi_t\}_{t\in [1,2]}$ in $\PE^s(\pants)$ connecting $\psi_1 = \phi_1$ to a $\psi_2\in \PE^s(\pants)$ with 
    $\chi(\psi_2) = \chi(\phi_2)$, followed by constructing a path 
 $\{\rho_t\}_{t\in [0,1]}$ 
 in $\PE^s(\pants)$ connecting $\rho_0 = \psi_2$ to $\rho_1 = \phi_2$. Then the composition of these two paths defines a path $\{\phi'_t\}_{t\in [1,2]}$ in $\PE^s(\Sigma)$ connecting $\phi_1$ to $\phi_2$. Finally, using $\{\phi'_t\}_{t\in [1,2]}$, we will construct a path $\{\phi_t\}_{t\in [1,2]}$ in $\ev^{-1}(\widetilde{C})$ connecting $\phi_1$ to $\phi_2$.
    \medskip
    
    First, we construct the path $\{\psi_t\}_{t\in [1,2]}$ as follows.
    Since $\widetilde{\phi_j(c_1)}\in \Par^{sgn(s)}_0$ and $\widetilde{\phi_j(c_2)}\in \Ell_1$ for $j\in \{1,2\}$, 
    we have the characters $\chi(\phi_1) = (2,y_1,z)$ and 
    $\chi(\phi_2) = (2,y_2,z)$ with $y_1, y_2 \in (-2,2)$ and $z = \tr \widetilde{C}$. Define the straight line segment
    $\big\{(2,y_t,z)\big\}_{t\in [1,2]}$ in $\mathbb{R}^3$ with $y_t= (2-t)y_1+(t-1)y_2$ connecting $y_1$ to $y_2$.
    We will lift $\big\{(2,y_t,z)\big\}_{t\in [1,2]}$ to $\SL\times\SL$ using
    Lemma \ref{lem_plpchar}, then its projection to $\psl\times \psl$ defines the path $\{\psi_t\}_{t\in [1,2]}$.
    Let $A_1$ and $B_1$ respectively
    be the projections of $\widetilde{\phi_1(c_1)}$ and $\widetilde{\phi_1(c_2)}$ to $\SL$, noting that the image of $(A_1, B_1)$ under the character map $\chi$ equals $\chi(\phi_1) = (2,y_1, z)$.
    Since $\phi_1(c_1)\in \Par^{sgn(s)}$ and $\phi_1(c_2)\in \Ell$, $A_1$ and $B_1$ do not commute. 
    We also claim that the path $\big\{(2,y_t,z)\big\}_{t\in [1,2]}$ lies in
    $\mathbb{R}^3
        \setminus \big([-2,2]^3\cap \kappa^{-1}([-2,2])\big)$, where $\kappa$ is as defined in Proposition \ref{prop_char}. Indeed, 
        up to an $\SL$-conjugation, we can assume that 
        $A_1 = \begin{bmatrix}
        1 & s \\
        0 & 1
        \end{bmatrix}$ and $B_1 = \abcd$ for some $\abcd\in \SL$. Then $A_1B_1
        = \begin{bmatrix}
        a+sc & b+sd \\
        c & d
        \end{bmatrix}$ with the trace
        $\tr(A_1B_1) = a+d + sc = z$. As $\widetilde{\phi_1(c_2)}\in \Ell_1$, Lemma \ref{lem_offdiag_Ell} implies $c<0$; and since $y_1 = a+d$, we have $y_1 > z$ if $s = +1$, and $y_1 < z$ if $s = -1$. 
        Similarly, we can show that $y_2 > z$ if $s = +1$, and $y_2 < z$ if $s = -1$. 
        Hence, if $s = +1$, then $y_t = (2-t)y_1+(t-1)y_2 > z$ for all $t\in [1,2]$; and if $s = -1$, then $y_t = (2-t)y_1+(t-1)y_2 < z$ for all $t\in [1,2]$. Thus, $y_t \neq z$ for all $t\in [1,2]$, and
        we have $\kappa(2,y_t,z)
        = y_t^2 + z^2 - 2y_tz + 2 > 2$. Therefore, the path $\big\{(2,y_t,z)\big\}_{t\in [1,2]}$ lies in
    $\mathbb{R}^3
        \setminus \big([-2,2]^3\cap \kappa^{-1}([-2,2])\big)$. 
        Then by Lemma \ref{lem_plpchar}, 
        there exists a path $\{(A_t, B_t)\}_{t\in [1,2]}$ of non-commuting pairs of elements of $\SL$ starting from  $(A_1, B_1)$, having 
        $\chi(A_t, B_t) = (2, y_t,z)$ for all $t\in [1,2]$. 
        Let $\pm A_t$ and $\pm B_t$ respectively be the projections of 
        $A_t$ and $B_t$ to $\psl$, and define $\psi_t:\pi_1(\pants)\to \psl$ by letting $\big(\psi_t(c_1), \psi_t(c_2)\big) = (\pm A_t, \pm B_t)$. Then the path $\{\psi_t\}_{t\in [1,2]}$ of $\psl-$representations starting from $\psi_1 = \phi_1$ satisfies $\chi(\psi_t) = \chi(A_t, B_t) = (2,y_t, z)$ for all $t\in [1,2]$, which implies $\chi(\psi_2) =(2,y_2,z)$. 
        %
        Since $\tr A_t = 2$ for all $t\in [1,2]$, the path $\{\psi_t(c_1)\}_{t\in [1,2]}$ in $\psl$ is contained in $\Par\cup \{\pm \mathrm{I}\}$; and it never passes $\pm\mathrm{I}$, as otherwise $A_t = \mathrm{I}$ for some $t\in [1,2]$ and would commute with $B_t$. Since $\psi_1(c_1)$ lies in $\Par^{sgn(s)}$ which is a connected component of $\Par$, we have
        $\psi_t(c_1)\in \Par^{sgn(s)}$
        for all $t\in [1,2]$; since $y_1, y_2\in (-2,2)$,
        $y_t = (2-t)y_1+(t-1)y_2 \in (-2,2)$, and hence $\psi_t(c_2)\in \Ell$ for all $t\in [1,2]$; and since $|z|<2,$ $\psi_t(c_3) = \big(\psi_t(c_1)\psi_t(c_2)\big)^{-1}\in \Ell$ for all $t\in [1,2]$. This implies that 
        the path $\{\psi_t\}_{t\in [1,2]}$ lies in $\PE^s(\pants)$.
        
    Next, we construct the path $\{\rho_t\}_{t\in [0,1]}$ connecting $\psi_2$ to $\phi_2$.
        Let $A$ and $B$ respectively
    be the projections of $\widetilde{\phi_2(c_1)}$ and $\widetilde{\phi_2(c_2)}$ to $\SL$, noting that the image of $(A, B)$ under the character map $\chi$ equals $\chi(\phi_2) = (2,y_2, z)$. As shown 
    in the previous paragraph, we have $y_2 > z$ if $s = +1$ and $y_2 < z$ if $s = -1$, hence 
    $\kappa(2,y_2,z)
        = y_2^2 + z^2 - 2y_2z + 2 \neq 2$. 
    Since both $(A_2, B_2)$ and $(A, B)$ are in $\chi^{-1}(2,y_2,z)\subset \SL\times \SL$ with $\kappa(2, y_2, z)\neq 2$, by Proposition \ref{prop_char}, these pairs of elements of $\SL$ are $\GL$-conjugate.
        It follows that they are $\SL$-conjugate as otherwise the parabolic elements $\psi_2(c_1)$ and $\phi_2(c_1)$, which are respectively the projections of $A_2$ and $A$, would have opposite signs. (See the proof of Proposition \ref{prop_pglpsl}.)
        Let $P\in\SL$ be such that 
        $(A, B)
        = 
        \big(PA_2P^{-1}, PB_2P^{-1}\big),$ and let $g=\pm P$ be the projection of $P$ to $\psl.$ Then we have $\phi_2=g \psi_2 g^{-1}.$ 
        Let $\{g_t\}\interval$ be a path in $\psl$ connecting $g_0 =\pm \mathrm I$ and $g_1 = g$. Then
        the path $\{\rho_t\}_{t\in [0,1]}$ defined by $\rho_t \doteq g_t\psi_2 g_t^{-1}$ connects $\psi_2$ to $\phi_2$. Moreover, as
        $\psl$-conjugations of
        $\psi_2(c_1)\in \Par^{sgn(s)}$ remain in $\Par^{sgn(s)}$ and $\psl$-conjugations of
        $\psi_2(c_2), \psi_2(c_3)\in \Ell$ remain in $\Ell$, the path 
        $\{\rho_t\}\interval$ lies in $\PE^s(\pants)$. 
        As 
        $\{\psi_t\}$ connects $\phi_1$ and $\psi_2$ and $\{\rho_t\}$ connects $\psi_2$ and $\phi_2,$ the composition of these two paths with a suitable re-parametrization defines a path $\{\phi'_t\}_{t\in [1,2]}$ in $\PE^s(\pants)$ connecting $\phi'_1 = \phi_1$ and $\phi'_2 = \phi_2$. 
        
        Finally, 
        we construct the path $\{\phi_t\}_{t\in [1,2]}$ in $\ev^{-1}(\widetilde{C})$ as follows. 
        For each $t\in [1,2]$, as 
        both $\widetilde{\phi'_t(c_1)}\widetilde{\phi'_t(c_2)} = \ev(\phi'_t)$ and $\widetilde{C}$ lie in $\Ell_1$ with the traces $\mathrm{tr}\big(\widetilde{\phi'_t(c_1)}\widetilde{\phi'_t(c_2)}\big) = z = \tr\widetilde{C}$, they are conjugate in $\univcover$; and hence their projections $\phi'_t(c_1)\phi'_t(c_2)$ and $\pm C$ to $\psl$ are $\psl$-conjugate. 
        Guaranteed by  Lemma \ref{lem_conjugacypath_const}, let $\{h_t\}_{t\in [1,2]}$ be a continuous path in $\psl$ such that 
        $\pm C = h_t\big(\phi'_t(c_1)\phi'_t(c_2)\big)h_t^{-1}$ for all $t\in [1,2]$. Since $\phi'_1, \phi'_2\in \ev^{-1}(\widetilde{C})$, for $j\in\{1,2\},$ 
        $h_j\big(\pm C\big)h_j^{-1}=h_j\big(\phi'_j(c_1)\phi'_j(c_2)\big)h_j^{-1}= \pm C$, and hence $h_j$ commutes with $\pm C$. This implies that $h_1^{-1}h_t\big(\phi'_t(c_1)\phi'_t(c_2)\big)h_t^{-1}h_1 = h_1\big(\pm C\big) h_1^{-1}= \pm C$ for all $t\in [1,2]$; and by replacing $h_t$ by  $h_1^{-1}h_t$ if necessary, we can assume that $h_1 = \pm \mathrm I$. 
        Then the path $\{h_t\phi'_th_t^{-1}\}_{t\in [1,2]}$ connects  $h_1\phi'_1h_1^{-1} = \phi_1$ to $h_2\phi'_2h_2^{-1} = 
         h_2\phi_2h_2^{-1}$. Moreover, since $\phi'_t$ lies in $\PE^s(\pants)$ for all $t\in [1,2]$, its $\psl$-conjugation $h_t\phi'_th_t^{-1}$ also lies in $\PE^s(\pants)$ as shown in the previous paragraph. Thus, the evaluation map $\ev$ is defined at $h_t\phi'_th_t^{-1}$ for all $t\in [1,2]$ with $\ev\big(h_t\phi'_th_t^{-1}\big)=\widetilde{C}$ which is the unique lift of $\pm C$ to $\Ell_1$, i.e., $\{h_t\phi'_th_t^{-1}\}_{t\in[1,2]}\subset \ev^{-1}(\widetilde C)$.
        It remains to find a path in $\ev^{-1}(\widetilde{C})$ connecting $h_2\phi_2h_2^{-1}$ and $\phi_2$. Let $\{h_{2,t}\}\interval$ be a path connecting $h_{2,0} = h_2$ to $h_{2,1} = \pm\mathrm{I}$ in the one-parameter subgroup of $\psl$ generated by $h_2$. Since $h_2$ and $\pm C$ commute, $\pm C$ lies in this one-parameter subgroup, and we have $h_{2,t}
        \phi_2(c_1)
        \phi_2(c_2)
        h_{2,t}^{-1} = h_{2,t}
       \big(\pm C\big)
        h_{2,t}^{-1} = \pm C$ for all $t\in [0,1]$. 
        As a consequence, the path 
        $\{h_{2,t}
        \phi_2
            h_{2,t}^{-1}\}_{t\in [0,1]}$ connects $h_2\phi_2h_2^{-1}$ to $\phi_2$ within $\PE^s(\pants)$ with $\ev(h_{2,t}
        \phi_2
            h_{2,t}^{-1}) = \widetilde{C}$ for all $t\in [0,1]$, i.e., $\{h_{2,t}
        \phi_2
            h_{2,t}^{-1}\}_{t\in [0,1]}\subset \ev^{-1}(\widetilde{C})$. 
            Finally, the composition of 
            $\{h_t\phi'_th_t^{-1}\}_{t\in[1,2]}$ and  
            $\{h_{2,t}
        \phi_2
            h_{2,t}^{-1}\}_{t\in [0,1]}$ with a suitable re-parametrization defines a path $\{\phi_t\}_{t\in [1,2]}$ in 
            $\ev^{-1}(\widetilde{C})$ connecting $\phi_1$ and $\phi_2$. This concludes that  the fiber $\ev^{-1}(\widetilde{C})$ is connected in $\PE^s(\pants)$.   
    \\

    For (2), by Lemma \ref{lem_plp} and Lemma \ref{lem_evPsubm}, it suffices to show that for all $\widetilde{C}\in \Ell_s$, the fiber $\ev^{-1}(\widetilde{C})$ is connected in $\PP^s(\Sigma_{0,3})$. 
    Let $\widetilde{C}\in \Ell_s$, and let $\phi_1, \phi_2\in \PP^s(\Sigma_{0,3})$ such that $\ev(\phi_1) = \ev(\phi_2) = \widetilde{C}$. 
    We will construct a path  
     $\{\phi'_t\}_{t\in [1,2]}$ 
     in $\PP^s(\pants)$ connecting $\phi'_1 = \phi_1$ to $\phi'_2 = \phi_2$, then the construction of a path $\{\phi_t\}_{t\in [1,2]}$ in $\ev^{-1}(\widetilde{C})$ connecting $\phi_1$ to $\phi_2$ follows verbatim that of (1). 
     Since $\widetilde{\phi_j(c_1)}$ and $\widetilde{\phi_j(c_2)}$ are in $\Par^{sgn(s)}_0$ for $j\in \{1,2\}$,
    we have the characters $\chi(\phi_1) = \chi(\phi_2) = (2,2,z)$ with $z = \tr \widetilde{C}$. Moreover, since $|z|<2$, 
    $\kappa(2, 2, z) = (z-2)^2 + 2 \neq 2$. 
    For $j\in \{1,2\}$,
    let $A_j$ and $B_j$ respectively be the projections of $\widetilde{\phi_j(c_1)}$ and $\widetilde{\phi_j(c_2)}$ in $\Par^{sgn(s)}_0$ to $\SL$.
    Since both $(A_1, B_1)$ and $(A_2, B_2)$ are in $\chi^{-1}(2,2,z)\subset \SL\times \SL$ with $\kappa(2, 2, z)\neq 2$, by Proposition \ref{prop_char}, $(A_1, B_1)$ and $(A_2, B_2)$ are $\GL$-conjugate; and they are $\SL$-conjugate as the parabolic elements $\phi_1(c_1)$ and $\phi_2(c_1)$ 
    have the same sign. (See the proof of Proposition \ref{prop_pglpsl}.) Let $P\in\SL$ be such that 
        $(A, B)
        = 
        \big(PA_2P^{-1}, PB_2P^{-1}\big),$ and let $g=\pm P$ be the projection of $P$ to $\psl.$ Then we have $\phi_2=g \phi_1 g^{-1}.$ 
        Let $\{g_t\}_{t\in [1,2]}$ be a path in $\psl$ connecting $g_1 =\pm \mathrm I$ and $g_2 = g$. Then
        the path $\{\phi'_t\}_{t\in [1,2]}$ defined by $\phi'_t \doteq g_t\phi_1 g_t^{-1}$ connects $\phi_1$ to $\phi_2$ in $\PP^s(\pants)$.
    \end{proof}

\subsubsection{Proof of Theorem \ref{thm_Exc2}}\label{subsection_exc2proof}
The proof of Theorem \ref{thm_Exc2} relies on the following Lemma \ref{lem_Exc2} and Proposition \ref{Prop_Exc2}.
\begin{lemma}\label{lem_Exc2}
    Let $\Sigma = \Sigma_{0,p}$ with $p\geqslant 3$, and let $s\in \{\pm 1\}^p$ with $p_-(s) = 0$ or $p_+(s) = 0$.
    Let $\phi$ and $\psi$ be totally non-hyperbolic representations with sign $s$ such that $\phi|_{\pi_1(P_i)}$ and $\psi|_{\pi_1(P_i)}$ are non-abelian for all $i \in \{1,\cdots,p-2\}$. Then $e(\phi) = e(\psi) = n$ for some $n\in \mathbb{Z}$, and there is a path in $\nonabel$ connecting $\phi$ and $\psi$.
\end{lemma}
\begin{proof}
        If $p = 3$, then by Theorem \ref{thm_pants}, $\phi$ and $\psi$ are in $\mathcal{R}^s_{s_1}(\Sigma)$ which is connected; and by Proposition \ref{prop_reducible}, 
        $\mathcal{R}^s_{s_1}(\Sigma)$ consists of non-abelian representations, hence the lemma follows. 
        If otherwise that $p\geqslant 4$, then by Lemma \ref{lem_inthyp_sphere} and the assumption that $\phi$ and $\psi$ are totally non-hyperbolic, 
        for all $i\in \{1,\cdots, p-3\}$, $\phi(d_i)$ and $\psi(d_i)$ are elliptic. 
        For each $i\in \{1, \cdots, p-3\}$, let $\Sigma_i\doteq \displaystyle\bigcup_{k = 1}^{i} P_k$.
        We will first show by induction that for all $i\in \{1,\cdots, p-3\}$, there is a path $\{\phi_t: \pi_1(\Sigma_i)\to \psl\}_{t\in [0,i]}$ of representations of $\pi_1(\Sigma_i)$ connecting $\phi|_{\pi_1(\Sigma_i)}$ to $\psi|_{\pi_1(\Sigma_i)}$, such that for all $t\in [0,i]$, $\phi_t(c_k)\in \Par^{sgn(s_k)}$ for all $k\in \{1,\cdots, i+1\}$, and $\phi_t(d_l)\in \Ell$ for all $l\in \{1,\cdots, i\}$. 
        Then as $\Sigma_{p-3} = \Sigma\setminus P_{p-2}$,
        there is a path $\{\phi_t: \pi_1(\Sigma\setminus P_{p-2})\to \psl\}_{t\in [0,p-3]}$ connecting $\phi|_{\pi_1(\Sigma\setminus P_{p-2})}$ to $\psi|_{\pi_1(\Sigma\setminus P_{p-2})}$, such that for all $t\in [0,p-3]$, 
        $\phi_t(c_k)\in \Par^{sgn(s_k)}$ for all $k\in \{1,\cdots, p-2\}$, and $\phi_t(d_l)\in \Ell$ for all $l\in \{1,\cdots, p-3\}$. Finally, using $\{\phi_t\}_{t\in [0,p-3]}$ on $\pi_1(\Sigma\setminus P_{p-2})$, we will construct a path $\{\phi_t\}_{t\in [0,p-2]}$ on $\pi_1(\Sigma)$ connecting $\phi$ to $\psi$.
        \medskip

        As the base case of the induction, the path $\{\phi_t: \pi_1(P_1)\to \psl\}\interval$ is defined as follows. 
        The fundamental group $\pi_1(P_1)$ has the preferred peripheral elements $c_1,c_2$ and $d_1$. 
        For each $k\in \{1,2\}$, we have $\phi(c_k), \psi(c_k)\in \Par^{sgn(s_k)}$, where $s_1 = s_2$ as otherwise $p_+(s) \neq 0$ and $p_-(s) \neq 0$.
        Moreover, we have $\phi(d_1), \psi(d_1)\in \Ell$, hence both $\phi|_{\pi_1(P_1)}$ and $\psi|_{\pi_1(P_1)}$ are in $\PP^{s_1}(P_1)$. 
        By Lemma \ref{lem_evimage_Ell}, 
        the image of the evaluation map $\ev: \PP^{s_1}(P_1)\to \Ell_{s_1}$ is $\Ell_{s_1}$, which is connected; and by Theorem \ref{prop_PLP_Ell}, its domain $\PP^{s_1}(P_1)$ is connected. 
         Therefore, there exists a path $\{\phi_t\}\interval$ in $\PP^{s_1}(P_1)$ connecting $\phi_0= \phi|_{\pi_1(P_1)}$ to $\phi_1= \psi|_{\pi_1(P_1)}$, with 
        $\phi_t(c_1)\in \Par^{sgn(s_1)}, \phi_t(c_2)\in \Par^{sgn(s_2)}$, and $\phi_t(d_1)\in \Ell$ for all $t\in [0,1]$.
         \medskip

         Now assume that, for some
     $i\in \{1, \cdots, p-4\}$, there is a path $\{\phi_t: \pi_1(\Sigma_i)\to \psl\}_{t\in [0,i]}$ connecting $\phi|_{\pi_1(\Sigma_i)}$ to $\psi|_{\pi_1(\Sigma_i)}$ that satisfies the condition above. We will first construct a path $\{\phi_t: \pi_1(\Sigma_{i+1})\to \psl\}_{t\in [0,i]}$ 
    connecting $\phi|_{\pi_1(\Sigma_{i+1})}$ to a $\phi_i$ such that $\phi_i|_{\pi_1(\Sigma_i)} = \psi|_{\pi_1(\Sigma_i)}$, then construct a path $\{\phi_t: \pi_1(\Sigma_{i+1})\to \psl\}_{t\in [i,i+1]}$ connecting $\phi_i$ to $\psi|_{\pi_1(\Sigma_{i+1})}$. Then the path $\{\phi_t\}_{t\in [0,i+1]}$ connects $\phi|_{\pi_1(\Sigma_{i+1})}$ to $\psi|_{\pi_1(\Sigma_{i+1})}$.

        The  path $\{\phi_t: \pi_1(\Sigma_{i+1})\to \psl\}_{t\in [0,i]}$  is constructed as follows. 
       By the assumption, there exists a path $\{\phi_t: \pi_1(\Sigma_i)\to \psl\}_{t\in [0,i]}$ connecting $\phi_0= \phi|_{\pi_1(\Sigma_i)}$ to $\phi_i= \psi|_{\pi_1(\Sigma_i)}$, which is to be extended to $\pi_1(\Sigma_{i+1})$. 
       Recall that the fundamental group $\pi_1(P_{i+1})$ has the preferred peripheral elements $c_{i+2}, d_{i+1}$, and $d_i^{-1}$. 
       For each $t\in [0,i]$, we let $\widetilde{\phi_t(d_i)}$ be the lift of $\phi_t(d_i)\in \Ell$ in $\Ell_1$, then the path $\{\widetilde{\phi_t(d_i)}\}_{t\in [0,i]}$ lies in $\Ell_1$. Let 
        $\ev: \PE^{s_{i+2}}(P_{i+1})\to\Ell_1$ be the evaluation map defined in Subsection \ref{subsection_PLP_Ell}.
       By Lemma \ref{lem_evimage_Ell}, $\ev$ is surjective onto $\Ell_1$; and by Proposition \ref{prop_PLP_Ell}, 
       the path $\{\widetilde{\phi_t(d_i)}\}_{t\in[0,i]}$ lifts to the path $\{\psi_t
       \}_{t\in[0,i]}$  in $\PE^{s_{i+2}}(P_{i+1})$ with $\psi_0 = \phi|_{\pi_1(P_{i+1})}$. 
       For each $t\in [0,i]$, since $\psi_t$ and $\phi_t$ agree at the common boundary $d_i$ of $\Sigma_i$ and $P_{i+1}$, by letting $\phi_t|_{\pi_1(P_{i+1})}\doteq \psi_t$, we extend the path $\{\phi_t\}_{t\in [0,i]}$ from $\pi_1(\Sigma_i)$ to $\pi_1(\Sigma_{i+1})$ with $\phi_0 = \phi|_{\pi_1(\Sigma_{i+1})}$. For each $t\in [0,i]$, by the assumption, 
       $\phi_t(c_k)\in \Par^{sgn(s_k)}$ for all $k\in \{1,\cdots, i+1\}$, and $\phi_t(d_l)\in \Ell$ for all $l\in \{1,\cdots, i\}$. Moreover, as the path $\{\phi_t|_{\pi_1(P_{i+1})}\}_{t\in [0,i]}$ lies in $\PE^{s_{i+2}}(P_{i+1})$,
       we have $\phi_t(c_{i+2})\in \Par^{sgn(s_{i+2})}$ and $\phi_t(d_{i+1})\in \Ell$ for each $t\in [0,i]$. 

       The  path $\{\phi_t: \pi_1(\Sigma_{i+1})\to \psl\}_{t\in [i,i+1]}$  is constructed as follows.
       The evaluation map $\PE^{s_{i+2}}(P_{i+1})\to \Ell_1$
       sends both $\phi_i|_{\pi_1(P_{i+1})}$ and $\psi|_{\pi_1(P_{i+1})}$ to $\widetilde{\phi_i(d_i)}\in \Ell_1$. By Proposition \ref{prop_PLP_Ell}, the fiber $\ev^{-1}(\widetilde{\phi_i(d_i)})$ is connected, hence there exists a path $\{\phi_t: \pi_1(P_{i+1})\to \psl\}_{t\in [i,i+1]}$ in $\ev^{-1}(\widetilde{\phi_i(d_i)})$ connecting $\phi_i|_{\pi_1(P_{i+1})}$ to $\phi_{i+1} = \psi|_{\pi_1(P_{i+1})}$. As the path lies in $\ev^{-1}(\widetilde{\phi_i(d_i)})$, for all $t\in [i,i+1]$, $\phi_t(d_i) = \phi_i(d_i)$. Therefore, by letting $\phi_t|_{\pi_1(\Sigma_i)} \doteq \phi_i|_{\pi_1(\Sigma_i)} = \psi|_{\pi_1(\Sigma_i)}$ for all $t\in [i,i+1]$, we extend the path $\{\phi_t\}_{t\in [i,i+1]}$ from $\pi_1(P_{i+1})$ to $\pi_1(\Sigma_{i+1})$ with $\phi_{i+1} = \psi|_{\pi_1(\Sigma_{i+1})}$. For each $t\in [i,i+1]$, $\phi_t(c_k) = \phi_i(c_k)\in \Par^{sgn(s_k)}$ for all $k\in \{1,\cdots, i+1\}$, and $\phi_t(d_l) = \phi_i(d_l)\in \Ell$ for all $l\in \{1,\cdots, i\}$. Moreover, as the path $\{\phi_t|_{\pi_1(P_{i+1})}\}_{t\in [i,i+1]}$ lies in $\PE^{s_{i+2}}(P_{i+1})$,
       we have $\phi_t(c_{i+2})\in \Par^{sgn(s_{i+2})}$ and $\phi_t(d_{i+1})\in \Ell$ for each $t\in [i,i+1]$.
       \medskip

       Next, we construct the path $\{\phi_t\}_{t\in [0,p-2]}$ connecting $\phi$ to $\psi$. 
       By the induction above, there exists a path $\{\phi_t: \pi_1(\Sigma\setminus P_{p-2})\to \psl\}_{t\in[0,p-3]}$ connecting $\phi_0= \phi|_{\pi_1(\Sigma\setminus P_{p-2})}$ to $\phi_{p-3}= \psi|_{\pi_1(\Sigma\setminus P_{p-2})}$, which is to be extended to $\pi_1(\Sigma)$. Recall that the fundamental group $\pi_1(P_{p-2})$ has the preferred peripheral elements $c_{p-1}, c_p$, and $d_{p-3}^{-1}$.
       Let $\{\widetilde{\phi_t(d_{p-3})}\}_{t\in [0,p-3]}$ be the lift of the path $\{\phi_t(d_{p-3})\}_{t\in [0,p-3]}\subset \Ell$ in $\Ell_{s_p}$, and let
        $\ev: \PP^{s_p}(P_{p-2})\to\Ell_{s_p}$ be the evaluation map defined in Subsection \ref{subsection_PLP_Ell}.
       For each $k\in \{p-1,p\}$, we have $\phi(c_k), \psi(c_k)\in \Par^{sgn(s_k)}$, where $s_{p-1} = s_p$ as otherwise $p_+(s) \neq 0$ and $p_-(s) \neq 0$. 
       Moreover, we have $\phi(d_{p-3}^{-1}), \psi(d_{p-3}^{-1})\in \Ell$, hence both $\phi|_{\pi_1(P_{p-2})}$ and $\psi|_{\pi_1(P_{p-2})}$ are in $\PP^{s_p}(P_{p-2})$. By Lemma \ref{lem_evimage_Ell}, $\ev$ is surjective onto $\Ell_{s_p}$; and by Proposition \ref{prop_PLP_Ell}, 
       the path $\{\widetilde{\phi_t(d_{p-3})}\}_{t\in[0,p-3]}$ lifts to the path $\{\psi_t
       \}_{t\in[0,p-3]}$  in $\PP^{s_p}(P_{p-2})$ with $\psi_0 = \phi|_{\pi_1(P_{p-2})}$. 
       For each $t\in [0,p-3]$, since $\psi_t$ and $\phi_t$ agree at the common boundary $d_{p-3}$ of $\Sigma\setminus P_{p-2}$ and $P_{p-2}$, by letting $\phi_t|_{\pi_1(P_{p-2})}\doteq \psi_t$, we extend the path $\{\phi_t\}_{t\in [0,p-3]}$ from $\pi_1(\Sigma\setminus P_{p-2})$ to $\pi_1(\Sigma)$ with $\phi_0 = \phi$. 
       It remains to find a path $\{\phi_t\}_{t\in [p-3,p-2]}$ connecting $\phi_{p-3}$ to $\psi$. Notice that $\ev(\phi_{p-3}|_{\pi_1(P_{p-2})}) = \ev(\psi|_{\pi_1(P_{p-2})}) = \widetilde{\phi_{p-3}(d_{p-3})}\in \Ell_{s_p}$. By Proposition \ref{prop_PLP_Ell}, the fiber $\ev^{-1}\big(\widetilde{\phi_{p-3}(d_{p-3})}\big)$ is connected, hence there is a path $\{\phi_t: \pi_1(P_{p-2})\to \psl\}_{t\in [p-3,p-2]}$ in $\ev^{-1}(\widetilde{\phi_{p-3}(d_{p-3})})$ connecting $\phi_{p-3}|_{\pi_1(P_{p-2})}$ to $\phi_{p-2} = \psi|_{\pi_1(P_{p-2})}$. As the path lies in $\ev^{-1}(\widetilde{\phi_{p-3}(d_{p-3})})$, for all $t\in [p-3,p-2]$, $\phi_t(d_{p-3}) = \phi_{p-3}(d_{p-3})$. Therefore, by letting $\phi_t|_{\pi_1(\Sigma\setminus P_{p-2})} \doteq \phi_{p-3}|_{\pi_1(\Sigma\setminus P_{p-2})} = \psi|_{\pi_1(\Sigma\setminus P_{p-2})}$ for all $t\in [p-3,p-2]$, we extend the path $\{\phi_t\}_{t\in [p-3,p-2]}$ from $\pi_1(P_{p-2})$ to $\pi_1(\Sigma)$  with $\phi_{p-2} = \psi$. This defines the path $\{\phi_t\}_{t\in [0,p-2]}$ connecting $\phi$ to $\psi$.

       Finally, we show that $e(\phi) = e(\psi) = n$ for some $n\in \mathbb{Z}$, and that $\{\phi_t\}_{t\in [0,p-2]}\subset \nonabel$.
       For each $k\in \{1,\cdots, p-2\}$, 
       we have $\phi_t(c_k)\in \Par^{sgn(s_k)}$ for $t\in [0,p-3]$, also 
       $\phi_t(c_k) = \phi_{p-3}(c_k)\in \Par^{sgn(s_k)}$ for $t\in [p-3,p-2]$.
       Moreover, 
       since $s_{p-1} = s_p$ and
       $\{\phi_t|_{\pi_1(P_{p-2})}\}_{t\in [0,p-2]}\subset \PP^{s_p}(P_{p-2})$,
       $\phi_t(c_{p-1})\in \Par^{sgn(s_{p-1})}$ and $\phi_t(c_p)\in \Par^{sgn(s_p)}$.
       Consequently, 
       for each $t\in [0,p-2]$, 
       the signs $s(\phi_t) = s(\phi)$ and the relative Euler classes $e(\phi_t) = e(\phi)$, and hence $e(\psi) = e(\phi) =n$ for some $n\in \mathbb{Z}$. Moreover, 
          $\phi_t(c_1)\in \Par$ and $\phi_t(d_1)\in \Ell$ do not commute, hence $\phi_t|_{\pi_1(P_1)}$ is non-abelian; and 
          for each $i\in \{2,\cdots, p-2\}$, 
          $\phi_t(c_{i+1})\in \Par$ and $\phi_t(d_{i-1})\in \Ell$ do not commute, hence $\phi_t|_{\pi_1(P_i)}$ is non-abelian. As a consequence, the path $\{\phi_t\}_{t\in [0,p-2]}$ lies in $\nonabel$.    
\end{proof}
\begin{proposition}\label{Prop_Exc2}
    Let $\Sigma = \Sigma_{0,p}$ with $p\geqslant 3$, and let $s\in \{\pm 1\}^p$ with $p_-(s) = 0$ or $p_+(s) = 0$. Then the space $\mathcal{R}^s_{s_1}(\Sigma)$ is non-empty. Moreover, for a totally non-hyperbolic representation $\phi$ with sign $s$, if $\phi|_{\pi_1(P_i)}$ is non-abelian for all $i \in \{1,\cdots,p-2\}$, then $\phi$ is in $\mathcal{R}^s_{s_1}(\Sigma)$.
\end{proposition}
\begin{proof}
    As the case that $p = 3$ is proved by Theorem \ref{thm_pants}, it suffices to consider the case that $p\geqslant 4$. We will prove the result for the case that $p_-(s)=0$. Then the result for the remaining case that $p_+(s)=0$ follows from the previous case and Proposition \ref{prop_pglpsl}.
    \\
        
    As $s_i = +1$ for all $i\in \{1,\cdots, p\}$, we first show that $\mathcal{R}^s_1(\Sigma)$ is non-empty. We define a  representation $\psi$ by letting 
        $\psi(c_1) = \pm \begin{bmatrix}
                        3 & p-2\\
                        \frac{4}{2-p} & -1
                    \end{bmatrix}$, 
        $\psi(c_2) = \pm \begin{bmatrix}
                        1 & 0\\
                        \frac{4}{2-p} & 1
                    \end{bmatrix}$, 
        and 
        $\psi(c_i) = \pm \parpp$ for each $i\in \{3,\cdots, p\}$. 
        By Lemma \ref{lem_offdiag}, $\psi(c_i)\in \Par^+$ for all $i\in \{1,\cdots, p\}$, and hence $s(\psi) = s$. 
        To show that $e(\psi) = 1$, we consider the restrictions 
        $\psi|_{\pi_1(P_1)}$ and $\psi|_{\pi_1(\Sigma\setminus P_1)}$. The fundamental group $\pi_1(P_1)$ has the preferred peripheral elements $c_1,c_2,$ and $d_1$, and
        the fundamental group $\pi_1(\Sigma\setminus P_1)$ has the preferred peripheral elements $c_3,\cdots, c_p$ and $d_1^{-1}$. As $\psi(d_1) = \big(\psi(c_1)\psi(c_2)\big)^{-1} = \pm \begin{bmatrix}
                        1 & p-2\\
                        0 & 1
                    \end{bmatrix}\in \Par^+$, 
        $\psi|_{\pi_1(P_1)}$ and $\psi|_{\pi_1(\Sigma\setminus P_1)}$ are type-preserving. 
        As $s(\psi|_{\pi_1(P_1)}) = (+1,+1,+1)$, Theorem \ref{thm_pants} implies $e(\psi|_{\pi_1(P_1)}) = 1$; and as $\psi|_{\pi_1(\Sigma\setminus P_1)}$ is abelian, Proposition \ref{prop_reducible} implies that $e(\psi|_{\pi_1(\Sigma\setminus P_1)}) = 0$. Hence by Proposition \ref{prop_additivity}, $e(\psi) = e(\psi|_{\pi_1(P_1)})
        + (\psi|_{\pi_1(\Sigma \setminus P_1)}) = 1.$
        \\
        
        Next, let $\phi$ be a totally non-hyperbolic representation with sign $s$, and suppose that $\phi|_{\pi_1(P_i)}$ is non-abelian for all $i \in \{1,\cdots,p-2\}$. To show that $\phi\in \mathcal{R}^s_1(\Sigma)$,
        we will first
        find a totally non-hyperbolic representation 
        $\rho$
        in $\mathrm{NA}^s_1(\Sigma)$. Then by applying Lemma \ref{lem_Exc2} to $\phi$ and $\rho$, we conclude that $\phi\in \mathrm{NA}^s_1(\Sigma)\subset\mathcal{R}^s_1(\Sigma)$.
        To this end, we will 
        construct a path $\{\psi_t\}_{t\in[0,2]}$ in $\mathcal{R}^s_1(\Sigma)$ connecting $\psi_0 =\psi$ to a $\psi_2 \in \mathrm{NA}^s_1(\Sigma)$ as follows. 
        First, we will define paths $\{\psi_t: \pi_1(\Sigma\setminus P_1)\to \psl\}_{t\in[1,2]}$ and $\{\psi'_t: \pi_1( P_1)\to \psl\}_{t\in[1,2]}$ of representations of $\pi_1(\Sigma\setminus P_1)$ and $\pi_1(P_1)$, respectively. Next, we will find a path $\{g_t\}_{t\in [1,2]}\subset \psl$ using Lemma \ref{lem_conjugacypath}, and
        glue $\{\psi_t: \pi_1(\Sigma\setminus P_1)\to \psl\}_{t\in[1,2]}$ and $\{g_t\psi'_tg_t^{-1}: \pi_1(P_1)\to \psl\}_{t\in[1,2]}$ along the decomposition curve $d_1$ to extend the path $\{\psi_t\}_{t\in[1,2]}$ to $\pi_1(\Sigma)$. 
        Finally, we will construct a path $\{\psi_t\}_{t\in[0,1]}$ connecting $\psi$ to $\psi_1$, and compose it with $\{\psi_t\}_{t\in[1,2]}$.
        \smallskip
        
        We first define a path 
        $\{\psi_t: \pi_1(\Sigma\setminus P_1)\to \psl\}_{t\in [1,2]}$ as follows: 
        For each $t\in [1,2]$,
        let 
        $\psi_t(c_k) \doteq \pm \parpp$ for each $k\in \{3,\cdots, p-1\}$, and let 
        $\psi_t(c_p) \doteq \pm
        \begin{bmatrix}
            1 - \frac{t-1}{\sqrt{p-3}} & 1 \\
            - \frac{(t-1)^2}{p-3} & 1+ \frac{t-1}{\sqrt{p-3}}
        \end{bmatrix}$.
        Then the values $\{\psi_t(c_k)\}_{k\in\{3,\dots,p\}}$ determine the representation $\psi_t$ of $\pi_1(\Sigma\setminus P_1)$ with $\psi_1 = \psi|_{\pi_1(\Sigma\setminus P_1)}$. 
        For each $t\in (1,2]$ and each $i\in \{1,\cdots, p-3\}$, $\psi_t(d_i)$ is elliptic, as
        $$\psi_t(d_i) = \psi_t(c_{i+2})\cdots \psi_t(c_p) 
            = \pm 
        \begin{bmatrix}
            1-\frac{t-1}{\sqrt{p-3}} - \frac{(p-2-i)(t-1)^2}{p-3} & p-1-i+\frac{(p-2-i)(t-1)}{\sqrt{p-3}}  \\
            -\frac{(t-1)^2}{p-3} & 1+ \frac{t-1}{\sqrt{p-3}}
        \end{bmatrix}
        $$
         has the trace $2 - \frac{(p-2-i)}{p-3}(t-1)^2 < 2$ in absolute value.
        
        Next, we construct a path 
        $\{\psi'_t: \pi_1( P_1)\to \psl\}_{t\in [1,2]}$ using
    Lemma \ref{lem_plpchar} as follows. 
        Since $e\big(\psi|_{\pi_1(P_1)}\big) = 1$, it follows that
        the character $\chi(\psi|_{\pi_1(P_1)}) = (2,2,-2)$. Define the straight line segment $\{(2,2,z_t)\}_{t\in [1,2]}$ in $\mathbb{R}^3$, where $z_t 
        = (t-1)^2 - 2 
        $
        connecting $-2$ to $-1$.
    Let $A_1$ and $B_1$ respectively
    be the lifts of $\psi(c_1)$ and $\psi(c_2)$ in $\SL$ with trace $2$, noting that the image of $(A_1, B_1)$ under the character map $\chi$ equals $\chi(\psi|_{\pi_1(P_1)}) = (2,2,-2)$.
    As $e(\psi|_{\pi_1(P_1)}) = 1$, by Proposition \ref{prop_reducible}, $\psi|_{\pi_1(P_1)}$ is non-abelian.
    In particular, $\psi(c_1)$ and $\psi(c_2)$ do not commute in $\psl$, and their lifts $A_1$ and $B_1$ do not commute in $\SL$. Moreover, as $\kappa(2, 2, z_t) = (z_t - 2)^2 + 2 > 2$ for all $t\in [1,2]$, where $\kappa$ is as defined in Proposition \ref{prop_char},  the path $\big\{(2,2,z_t)\big\}_{t\in [1,2]}$ lies in
    $\mathbb{R}^3
    \setminus \big([-2,2]^3\cap \kappa^{-1}([-2,2])\big)$. 
    Then by Lemma \ref{lem_plpchar}, 
        there exists a path $\{(A_t, B_t)\}_{t\in [1,2]}$ of non-commuting pairs of elements of $\SL$ starting from  $(A_1, B_1)$, having 
        $\chi(A_t, B_t) = (2,2,z_t)$ for all $t\in [1,2]$. 
        Let $\pm A_t$ and $\pm B_t$ respectively be the projections of 
        $A_t$ and $B_t$ to $\psl$, and define a representation $\psi'_t:\pi_1(P_1)\to \psl$ by letting $\big(\psi'_t(c_1), \psi'_t(c_2)\big) = (\pm A_t, \pm B_t)$. Then the path $\{\psi'_t\}_{t\in [1,2]}$ of $\psl-$representations starting from $\psi'_1 = \psi|_{\pi_1(P_1)}$ satisfies $\chi(\psi'_t) = \chi(A_t, B_t) = (2,2,z_t)$ for all $t\in [1,2]$. 
        Since $\tr A_t = \tr B_t = 2$ for all $t\in [1,2]$, for each $i\in \{1,2\}$, the path $\{\psi'_t(c_i)\}_{t\in [1,2]}$ in $\psl$ is contained in $\Par\cup \{\pm \mathrm{I}\}$; and it never passes $\pm\mathrm{I}$, as otherwise $A_t = \mathrm{I}$ or $B_t = \mathrm{I}$ for some $t\in [1,2]$ which contradicts that $A_t$ and $B_t$ do not commute. Since $\psi'_1(c_1)$ and $\psi'_1(c_2)$ lie in $\Par^+$ which is a connected component of $\Par$, we have
        $\psi'_t(c_1), \psi'_t(c_2)\in \Par^+$
        for all $t\in [1,2]$.

        We now define $\{g_t\}_{t\in [1,2]}\subset \psl$ using Lemma \ref{lem_conjugacypath}, and
        glue the paths $\{\psi_t\}_{t\in [1,2]}$ and $\{g_t\psi'_tg_t^{-1}\}_{t\in [1,2]}$ along $d_1$. To this end, we first prove that for each $t\in [1,2]$, $\psi_t(d_1)$ and $\psi'_t(d_1)$ are conjugate in $\psl$. 
        For $t = 1$, $\psi_1(d_1) = \psi'_1(d_1) = \psi(d_1)$, hence are clearly conjugate;
        and since $\psi_1(d_1) = \psi'_1(d_1) =\pm \begin{bmatrix}
            1 & p-2\\
            0 & 1
        \end{bmatrix}\in \Par^+$, 
        there exist unique lifts $\{\widetilde{\psi_t(d_1)}\}_{t\in [1,2]}$
        and $\{\widetilde{\psi'_t(d_1)}\}_{t\in [1,2]}$ of $\{\psi_t(d_1)\}_{t\in [1,2]}$ and $\{\psi'_t(d_1)\}_{t\in [1,2]}$, respectively, both starting from $\Par^+_0$. 
        Notice that the trace $\tr\big(\widetilde{\psi_t(d_1)}\big) = \tr\big(\widetilde{\psi'_t(d_1)}\big) = -z_t$.
        For all $t\in (0,1]$, $\widetilde{\psi_t(d_1)}$ and $\widetilde{\psi'_t(d_1)}$ are elliptic as their trace $-z_t\in (-2,2)$, hence
         lie in $\Ell_1$ as it is the only lift of $\Ell$ in $\univcover$ that is adjacent to $\Par_0^+$. Consequently, for each $t\in (1,2]$, as elements both lying in  $\Ell_1$ with the same trace, $\widetilde{\psi_t(d_1)}$ and $\widetilde{\psi'_t(d_1)}$ are conjugate in $\univcover$; and hence their projections $\psi_t(d_1)$ and $\psi'_t(d_1)$ are conjugate in $\psl$. Then by Lemma \ref{lem_conjugacypath}, there is a path $\{g_t\}_{t\in [1,2]}$ such that 
         $\psi_t(d_1) = g_t \psi'_t(d_1) g_t^{-1}$ for all $t\in [1,2]$. By letting $\psi_t|_{\pi_1(P_1)} \doteq g_t\psi'_tg_t^{-1}$ for each $t\in [1,2]$, we extend the path $\{\psi_t\}_{t\in [1,2]}$ to $\pi_1(\Sigma)$.

         Finally, we construct the path $\{\psi_t\}_{t\in [0,1]}$ connecting $\psi$ and $\psi_1$. Let $\{g_t\}\interval$ be a path connecting $g_0 = \pm \mathrm I$ to $g_1$ within the one-parameter subgroup of $\psl$ generated by $g_1$. For each $t\in [0,1]$, we let $\psi_t|_{\pi_1(P_1)}\doteq g_t\psi|_{\pi_1(P_1)}g_t^{-1}$ and $\psi_t|_{\pi_1(\Sigma\setminus P_1)}\doteq \psi|_{\pi_1(\Sigma\setminus P_1)}$.
         As $g_1\psi(d_1)g_1^{-1} = g_1\psi'_1(d_1)g_1^{-1} = \psi_1(d_1) = \psi(d_1)$, $g_1$ commutes with $\psi(d_1)$, and hence $g_t \psi(d_1) g_t^{-1} = \psi(d_1)$ for all $t\in [0,1]$. 
         This defines a path $\{\psi_t\}\interval$ of representations of $\pi_1(\Sigma)$. By composing with $\{\psi_t\}_{t\in [1,2]}$,
          we have a path $\{\psi_t\}_{t\in [0,2]}$ starting from  $\psi_0 = \psi$. 
          For each $t\in [0,2]$ and each $i\in \{1,\cdots, p\}$, $\psi_t(c_i)\in \Par^+$, hence
          the signs $s(\psi_t) = s$ and the relative Euler classes $e(\psi_t) = e(\psi) = 1$, i.e., the path
          $\{\psi_t\}_{t\in [0,2]}$ lies in $\mathcal{R}^s_1(\Sigma)$. Moreover, 
          $\psi_2(c_1)\in \Par$ and $\psi_2(d_1)\in \Ell$ do not commute, hence $\psi_2|_{\pi_1(P_1)}$ is non-abelian; and 
          for each $i\in \{2,\cdots, p-2\}$, 
          $\psi_2(c_{i+1})\in \Par$ and $\psi_2(d_{i-1})\in \Ell$ do not commute, hence $\psi_2|_{\pi_1(P_i)}$ is non-abelian. 
          As $\psi_2$ is a totally non-hyperbolic representation in $\mathrm{NA}^s_1(\Sigma)$, we let $\rho = \psi_2$. This completes the proof.
              
\end{proof}

    \begin{proof}[Proof of Theorem \ref{thm_Exc2}]
        As the case that $p = 3$ is proved by Theorem \ref{thm_pants}, it suffices to consider the case that $p\geqslant 4$. We will prove the result for the case that $p_-(s)=0$. Then the result for the remaining case that $p_+(s)=0$ follows from the previous case and Proposition \ref{prop_pglpsl}. 
        \medskip
        
        By Proposition \ref{Prop_Exc2}, $\mathcal{R}^s_1(\Sigma)$ is non-empty. 
        We will first prove that $\mathcal{R}^s_1(\Sigma)$ consists of totally non-hyperbolic representations, then conclude that $\mathcal{R}^s_1(\Sigma)$ is connected.
        \smallskip
        
        We first prove by contradiction that every representation $\phi$ in $\mathcal{R}^s_1(\Sigma)$ is totally non-hyperbolic. Suppose otherwise that $\phi(d_i)$ is hyperbolic for some $i\in \{1,\cdots, p-3\}$. 
        Then $d_i$ separates $\Sigma$ into two subsurfaces $\Sigma_1$ and $\Sigma_2$. 
        As $\phi(d_i)\in \Hyp$ and $\phi(c_k)\in \Par^+$ for each
        $k\in \{1,\cdots, p\}$, we have 
        $\phi|_{\pi_1(\Sigma_1)}
        \in \HP(\Sigma_1)$ and $\phi|_{\pi_1(\Sigma_2)}
        \in \HP(\Sigma_2)$.
        For each $j\in \{1,2\}$,
        the triple $\big(\Sigma_j, e(\phi|_{\pi_1(\Sigma_j)}), s(\phi|_{\pi_1(\Sigma_j)})\big)$ is not exceptional as $p_0\big(s\big(\phi|_{\pi_1(\Sigma_j)}\big)\big) = 1$, and hence Theorem \ref{thm_general} implies
         $$\chi(\Sigma_j) + p_+\big(s(\phi|_{\pi_1(\Sigma_j)})\big)\leqslant e(\phi|_{\pi_1(\Sigma_j)})\leqslant -\chi(\Sigma_j) - p_-\big(s(\phi|_{\pi_1(\Sigma_j)})\big).$$
         As $\phi$ maps the common boundary $d_i$ of $\Sigma_1$ and $\Sigma_2$ 
         to a hyperbolic element, we have $p_+(s) = p_+ \big(s\big(\phi|_{\pi_1(\Sigma_1)}\big)\big) + p_+ \big(s\big(\phi|_{\pi_1(\Sigma_2)}\big)\big)$ and $p_-(s) = p_- \big(s\big(\phi|_{\pi_1(\Sigma_1)}\big)\big) + p_- \big(s\big(\phi|_{\pi_1(\Sigma_2)}\big)\big)$. Moreover, $e(\phi) = e(\phi|_{\pi_1(\Sigma_1)})
         + e(\phi|_{\pi_1(\Sigma_2)})$ by Proposition \ref{prop_additivity}, and $\chi(\Sigma) = \chi(\Sigma_1)+\chi(\Sigma_2)$, which together imply
        $$\chi(\Sigma)+ p_+(s) \leqslant e(\phi) \leqslant -\chi(\Sigma) - p_-(s).$$
        Since $\chi(\Sigma) = 2-p$, $p_+(s) = p$ and $e(\phi) = 1$, we would obtain that $2\leqslant 1$, which is a contradiction.
        \smallskip

        It remains to show that $\mathcal{R}^s_1(\Sigma)$ is connected. 
        Let $\phi, \psi\in \mathcal{R}^s_1(\Sigma)$.
        As $p_+(s)\neq 1$ and $p_-(s)\neq 1$, by Theorem \ref{thm_nonabel}, we can assume that $\phi, \psi\in \mathrm{NA}^s_1(\Sigma)$; and they are totally non-hyperbolic as shown above. Therefore, by Lemma \ref{lem_Exc2}, there exists a path in $\mathcal{R}^s_1(\Sigma)$ connecting $\phi$ and $\psi$. 
        \end{proof}
    \begin{remark}
        In the proof of Theorem \ref{thm_Exc2}, by replacing the decomposition curve $d_i$ with any essential simple closed curve $\gamma$ on $\Sigma$, we can prove that every representation in $\mathcal{R}^s_{s_1}(\Sigma)$ sends all the simple closed curves on $\Sigma$ to non-hyperbolic elements.
    \end{remark}
    
        \begin{proof}[Proof of Corollary \ref{thm_main3} (2) and (3) and Corollary \ref{total} (2)] 
        By Theorem \ref{thm_general},
        the non-exceptional components of $\mathcal R(\Sigma)$ 
        are in one-to-one correspondence to the pairs $(n,s)$ that satisfy the generalized Milnor-Wood inequality  
        $\chi(\Sigma) + p_+(s)\leqslant n\leqslant -\chi(\Sigma) - p_-(s)$, the number of which  shares the same formula as in Corollary \ref{thm_main3} (1). 
    For the exceptional pairs $(n,s)$, if $n = 0$, then by Theorem \ref{thm_Exc1}, there are $2p$ exceptional components of $\Rn$ corresponding to the signs $s\in \{\pm 1\}^p$ with $\ p_+(s) = 1 \mbox{ or }p_-(s) = 1$; and if $n =1$ or $n=-1$, then by Theorem \ref{thm_Exc2}, there is one exceptional component of $\Rn$ for each case. Together with the number of non-exceptional components, we obtain the formula in Corollary \ref{thm_main3} (2) and (3). Finally, combining with the original Milnor-Wood inequality $\chi(\Sigma)\leqslant n\leqslant -\chi(\Sigma)$, we obtain the formula in Corollary \ref{total} (2). \end{proof}
\section{Appendix: submersive property of the evaluation maps
}
\subsection{Proof of Lemma \ref{lem_evTsubm}}
\begin{proof}
    In this proof, for $2\times 2$-matrices $X$ and $Y$, we denote $[X, Y]\doteq XY - YX$.
    Let $\xi, \eta \in \sl$. 
    The differential of $R$ at a point $(A,B)\in \SL\times \SL$ equals
    $$d{R}_{(A,B)}(\xi A, \eta B)
    = [\xi, ABA^{-1}]B^{-1}
    + 
    A[\eta, BA^{-1}B^{-1}],$$ 
    and by composing this with the linear isomorphism 
    $\xi\mapsto \xi(ABA^{-1}B^{-1})^{-1}$, we obtain the map $R_*: \sl\times\sl\to \sl$ defined by
    $$
    R_*(\xi, \eta)
    = [\xi, ABA^{-1}]AB^{-1}A^{-1}
    + 
    A[\eta, BA^{-1}B^{-1}]BAB^{-1}A^{-1}.$$
    To show that $dR_{(A,B)}$ is surjective, for each generator $
    \sldiag,
    \slupper$, and $\sllower$ of $\sl$,
    we will find a preimage $(\xi, \eta)\in \sl\times \sl$ under $R_*$.
    \\
    
    For convenience in computation, we let 
    $A = \abcd$ and 
    $B^{-1} = \pqrs$.
    Moreover, 
    up to conjugation and multiplication by $-\mathrm I$, we can assume that the matrix $ABA^{-1}$ is one of the following: 
    $\xdiag$ for some $x > 1$, 
    $\begin{bmatrix}
        1 & t\\
        0 & 1
    \end{bmatrix}$ for some $t\in \mathbb{R}$,
    $\elltheta$ for some $\theta\in (0, \pi)$.
    \\
    
    First, we let $ABA^{-1} = \xdiag$ for $x > 1$. 
    By computing $R_*(\xi, 0) = \xi - ABA^{-1}\xi AB^{-1}A^{-1},$ we obtain 
    $R_*\bigg(
    \begin{bmatrix}
        0 & \frac{1}{1-x^2}\\
        0 & 0
    \end{bmatrix}, 0\bigg)
    = \slupper$, 
    $R_*\bigg(
    \begin{bmatrix}
        0 & 0\\
        \frac{x^2}{x^2-1} & 0
    \end{bmatrix}, 0\bigg)
    = \sllower$ 
    and 
    $R_*\bigg(\sldiag, 0\bigg)
    = 0$. 
    As $\slupper$ and $\sllower$ are in the image of $R_*$, it remains to find $\sldiag$ in the image, which exists if and only if the image contains a matrix with non-zero diagonal entries.
    By computing $R_*(0,\eta) = A\eta A^{-1} - ABA^{-1}B^{-1}\eta BAB^{-1}A^{-1}$, we obtain 
\begin{eqnarray*}
    R_*\bigg(0, \slupper\bigg)
    &=& \begin{bmatrix}
        -ac + pr          &  a^2 - p^2x^2   \\
        -c^2 + r^2x^{-2}  &  ac - pr
      \end{bmatrix},\\
      R_*\bigg(0, \sllower\bigg)&=& \begin{bmatrix}
        bd - qs          &  -b^2 - q^2x^2   \\
        d^2 - s^2x^{-2}  &  -bd + qs
      \end{bmatrix}, \mbox{ and }\\
      R_*\bigg(0, \sldiag\bigg)&=& 
      \begin{bmatrix}
        (ad + bc) - (ps + qr)   &   -2ab + 2pqx^2  \\
        2cd - 2rsx^{-2}         &  - (ad + bc) + (ps + qr)
      \end{bmatrix}.
\end{eqnarray*}
      If all the three matrices have zero diagonal entries, then we have $ac = pr$, $bd = qs$, and $ad + bc = ps + qr$. Combining with $\det(A) = ad - bc= 1$ and $\det(B) = ps - qr= 1$, we also obtain $ad = ps$ and $bc = qr$. 
      By solving this system of equations, we obtain 
      $a = tp, 
        b = tq, 
        c = t^{-1}r, d = t^{-1}s$ for some $t\neq 0$, which implies that
      $A = \abcd = 
      \begin{bmatrix}
        t & 0 \\
        0 & t^{-1}
      \end{bmatrix}
      \pqrs = \begin{bmatrix}
        t & 0 \\
        0 & t^{-1}
      \end{bmatrix} B^{-1}$. 
      Then $AB = \begin{bmatrix}
        t & 0 \\
        0 & t^{-1}
      \end{bmatrix}$ and 
      $ABA^{-1} = \xdiag$ 
      together imply 
      $A = \begin{bmatrix}
        xt^{-1} & 0 \\
        0 & tx^{-1}
      \end{bmatrix}$ and 
      $B = \begin{bmatrix}
        t^2x^{-1} & 0 \\
        0 & xt^{-2}
      \end{bmatrix}$ that commute with each other. As a consequence, $dR_{(A,B)}$ is surjective if and only if $A$ and $B$ do not commute.
      \\

    Next, we let $ABA^{-1} = \part$ for $t\in\mathbb{R}$. 
    By computing $R_*(\xi, 0) = \xi - ABA^{-1}\xi AB^{-1}A^{-1},$ we obtain 
    $R_*\bigg(\slupper, 0\bigg)
    = 0$, 
    $R_*\bigg(\begin{bmatrix}
        \frac{1}{2} & 0\\
        -\frac{1}{t} & -\frac{1}{2}
    \end{bmatrix}, 0\bigg)
    = \sldiag$ 
    and 
    $R_*\bigg(
    \begin{bmatrix}
        \frac{1}{2t} & 0\\
        0 & -\frac{1}{2t}
    \end{bmatrix}, 0\bigg)
    = \slupper$.
    As $\sldiag$ and $\slupper$ are in the image of $R_*$, it remains to find $\sllower$ in the image, which exists if and only if the image contains a matrix with non-zero $(2,1)$-entry. By computing $R_*(0,\eta) = A\eta A^{-1} - ABA^{-1}B^{-1}\eta BAB^{-1}A^{-1}$, we obtain 
    $R_*\bigg(0, \slupper\bigg) = \begin{bmatrix}
        -ac + pr + r^2t          &  a^2 - p^2 - r^2t^2   \\
        -c^2 + r^2  &  ac - pr - r^2t
      \end{bmatrix},$
      with its $(2,1)$-entry equals  $-c^2 + r^2$;
    and by similar computations, we obtain
    $R_*\bigg(0, \sllower\bigg)$ 
    where its $(2,1)$-entry equals  $d^2 - s^2$,
    and $R_*\bigg(0, \sldiag\bigg)$ where its $(2,1)$-entry equals  $2(cd - rs)$. 
    If all the matrices in the image have zero $(2,1)$-entries, then we have 
    $c^2 = r^2$, $d^2 = s^2$ and $cd = rs$, which implies
      $r = c$ and $s = d$, or $r = -c$ and $s = -d$. 
      If $r = c$ and $s = d$,
      then by computing \\
      $$AB^{-1}A^{-1} = 
       \abcd\pqrs\inverse
       =
      \begin{bmatrix}
        a(pd - qc) & -abp - b^2c + a^2q + abd \\
        c(pd - qc) & -bcp - bdc + acq +ad^2
      \end{bmatrix}
      = \begin{bmatrix}
            1 & -t \\
            0 & 1
      \end{bmatrix}
      $$
      with $pd - qc = ps - qr = \det(B^{-1})= 1$,
      we obtain
      $a = d = p = s = 1 \mbox{ and }c = r = 0$, i.e., $A = 
      \begin{bmatrix}
      1 & b \\ 
      0 & 1
      \end{bmatrix}$ and 
      $B = 
      \begin{bmatrix}
      1 & -q \\ 
      0 & 1
      \end{bmatrix}$ that commute with each other. 
      Similarly, if $r = -c$ and $s = -d$, then by computing $AB^{-1}A^{-1}$ as above, we obtain $A = 
      \begin{bmatrix}
      1 & b \\ 
      0 & 1
      \end{bmatrix}$ and 
      $B = 
      -\begin{bmatrix}
      1 & q \\ 
      0 & 1
      \end{bmatrix}$ that commute with each other. 
      As a consequence, $dR_{(A,B)}$ is surjective if and only if $A$ and $B$ do not commute.
      \medskip
      
    Finally, we let $ABA^{-1} = \elltheta$ for $\theta\in (0, \pi)$. Let $x \doteq \cos\theta$ and $y\doteq \sin\theta$. 
    By computing $R_*(\xi, 0) = \xi - ABA^{-1}\xi AB^{-1}A^{-1}$ with $ABA^{-1} = \xell$, we obtain 
    $$R_*\bigg(\sldiag, 0\bigg)
    = 2xy\begin{bmatrix}
        0 & 1 \\
        1 & 0
    \end{bmatrix} 
    + 2y^2\sldiag$$ and $$R_*\bigg(\slupper, 0\bigg)
    = R_*\bigg(\sllower, 0\bigg) = y^2\begin{bmatrix}
        0 & 1 \\
        1 & 0
    \end{bmatrix} 
    - xy\sldiag.$$
    Then $\begin{bmatrix}
        0 & 1 \\
        1 & 0
    \end{bmatrix}$ and $\sldiag$ are in the image of $R_*$ and it remains to find $\slupper$ in the image, which exists if and only if the image contains a matrix whose $(2,1)$ and $(1,2)$-entries have different values. By computing $R_*(0,\eta) = A\eta A^{-1} - ABA^{-1}B^{-1}\eta BAB^{-1}A^{-1}$, we obtain  
    $$R_*\bigg(0, \slupper\bigg) 
    = \begin{bmatrix}
        - ac + (px + ry)(rx - py)    &  a^2 - (px + ry)^2   \\
        -c^2 + (rx - py)^2  &  ac - (px + ry)(rx - py)
      \end{bmatrix},$$ 
      which is generated by $\begin{bmatrix}
        0 & 1 \\
        1 & 0
    \end{bmatrix}$ and $\sldiag$ if and only if $a^2 - (px + ry)^2 = -c^2 + (rx - py)^2$. 
     By similar computation, $R_*\bigg(0, \sllower\bigg)$  is generated by $\begin{bmatrix}
        0 & 1 \\
        1 & 0
    \end{bmatrix}$ and $\sldiag$ if and only if 
    $$-b^2 + (qx + sy)^2 = d^2 - (sx - qy)^2,$$
    and $R_*\bigg(0, \sldiag\bigg)$ is generated by $\begin{bmatrix}
        0 & 1 \\
        1 & 0
    \end{bmatrix}$ and $\sldiag$ if and only if $$- ab + (px + ry)(qx + sy) = cd - (rx - py)(sx - qy).$$
      If all the matrices in the image have the same $(1,2)$ and $(2,1)$-entries, then we have the system of equations
      \begin{eqnarray*}
           a^2 + c^2 &=& (px + ry)^2 + (rx - py)^2, \\
             b^2 + d^2 &=& (qx + sy)^2+  (sx - qy)^2, \text{ and }\\
           ab + cd &=& (px + ry)(qx + sy) + (rx - py)(sx - qy).
      \end{eqnarray*} 
      Since $x^2+y^2 = \cos^2\theta
      + \sin^2\theta = 1,$ we have
      $(px + ry)^2 + (rx + py)^2
      = (x^2+y^2)(p^2+ r^2) = p^2 + r^2,$ hence the first equation becomes 
      \begin{eqnarray}
          a^2 + c^2 = p^2 + r^2.\label{eq1}
      \end{eqnarray}
      Similarly, combining $x^2+y^2 = 1$ with the second and third equation, we obtain
      \begin{eqnarray}
          b^2 + d^2 &=& q^2 + s^2\label{eq2}, \text{ and }\\
          ab + cd &=& pq + rs.\label{eq3}
      \end{eqnarray}
      From the determinant $\det(A) = \det(B) = 1$, we obtain
      \begin{eqnarray}
           ad - bc &=& ps - qr = 1. \label{det}
      \end{eqnarray} 
      Moreover, by computing each entry of $AB^{-1}A^{-1} = \begin{bmatrix}
          x & -y \\
          y & x
      \end{bmatrix}$ with $A = \abcd$ and $B^{-1} = \pqrs$, we have $2x = \tr(AB^{-1}A^{-1}) = \tr(B^{-1}) = p+s$ and 
           $2y = (AB^{-1}A^{-1})_{21} 
           - (AB^{-1}A^{-1})_{12} =  r-q$. 
      Combining this with 
      $\eqref{eq1} + \eqref{eq2} + 2\times\eqref{det}$, we obtain 
      $$(a+d)^2 + (b-c)^2 = (p+s)^2 + (q-r)^2 = (2x)^2 + (2y)^2 = 4,$$
      and by subtracting $4\times \eqref{det},$ we obtain \\
      $$(a-d)^2 + (b+c)^2 = (a+d)^2 + (b-c)^2 - 4(ad - bc)
     = 0$$
     and 
     $$(p-s)^2 + (q+r)^2 = (p+s)^2 + (q-r)^2 - 4(ps - qr)
     = 0.$$
     This implies $a = d, b = -c$ and $p = s, q = -r$, i.e., 
      $A = 
      \begin{bmatrix}
        a & b \\
        -b & a
      \end{bmatrix}\mbox{ with }a^2 + b^2 = 1$ and 
    $B = 
      \begin{bmatrix}
        p & -q \\
        q & p
      \end{bmatrix}\mbox{ with }p^2 + q^2 = 1$ 
      that commute with each other. 
      As a consequence, $dR_{(A,B)}$ is surjective if and only if $A$ and $B$ do not commute.
\end{proof}
\subsection{Proof of Lemma \ref{lem_evPsubm}}
\begin{proof}
    Let $\xi, \eta \in \sl$. 
    The differential of $m$ at a point $(A,B)\in \SL\times \SL$ equals
    $$d{m}_{(A,B)}(\xi A, \eta B)
    = \xi AB  + A\eta B,$$ 
    and by composing this with the linear isomorphism 
    $\xi\mapsto \xi(AB)^{-1}$, we obtain the map $m_*: \sl\times\sl\to \sl$ defined by
    $$
    m_*(\xi, \eta)
    = \xi  + A\eta A^{-1}.$$
    To show that each restriction of $dm_{(A,B)}$ is surjective, for each generator $
    \sldiag,
    \slupper$, and $\sllower$ of $\sl$,
    we will find a preimage $(\xi, \eta)\in \sl\times \sl$ under $m_*$ such that $(\xi A, \eta B)$ lies in the tangent space of the restricted domain of $m$.
    \\
    
    For (1), for each $\xi \in \sl$, we have $m_*(\xi, 0) = m_*(0,A^{-1}\xi A) = \xi$.
    Moreover, for the restriction of $m$ to each of $U\times U, \Par\times U, U\times \Par$,
    $(\xi A, 0)$ or 
    $(0, A^{-1}\xi A B)$ lies in the tangent space of the domain, as $\xi A\in T_AU\cong T_A\SL$ and $(A^{-1}\xi A) B\in T_BU\cong T_B\SL$. Therefore, $m$ is submersive on these domains.
    \medskip
    
    For (2), as the tangent space of the domain is 
    $T_A\Par\times T_B\Par$, we first compute the elements of $T_X\Par$ for each $X = \abcd\in \Par$ in $\SL$. 
    Each element of $T_X\SL$ can be written as $$\xi X = \bigg(x\slupper + y\sllower + z\sldiag\bigg) \abcd
    = \begin{bmatrix}
            cx+az & dx+bz \\
            ay-cz & by-dz
        \end{bmatrix}$$\\
    for some $x,y,z\in \mathbb{R}.$
    Since $\Par$ is the level set $\tr^{-1}(2)\cup \tr^{-1}(-2)$  of the trace map, the differential of the trace map $d(\tr)_X:T_X\SL\to \mathbb{R}$ at $X$, which maps $\xi X\in \sl$ to the trace $\tr(\xi X)$, is zero on $T_X\Par$.
    This implies that $\xi X\in T_X\Par$ if and only if
    $d(\tr)_X(\xi X) = cx + by + (a-d)z = 0.$

    We now compute the image of $dm_{(A,B)}$.
    Up to conjugation and multiplication by $-\mathrm I$, we can assume that 
    $A = \begin{bmatrix}
        1 & s \\
        0 & 1
    \end{bmatrix}$ for some $s\in \{\pm 1\}$.
    We have $\xi A = x\slupper + y\sllower + z\sldiag\in T_A\Par$ if and only if $y = 0$, i.e.,
    $T_A\Par$ is generated by
        $\slupper A$ and $\sldiag A$.
    Hence for $\xi = \slupper$ or $\xi = \sldiag$, 
    $m_*(\xi, 0) = \xi$ and $(\xi A, 0)\in T_A\Par\times T_B\Par$.
    As $\slupper$ and $\sldiag$
    are in the image of $m_*$, it remains to find a preimage $(\xi,\eta)\in \sl\times \sl$ of $\sllower$ such that $(\xi A, \eta B)\in T_A\Par\times T_B\Par$; and it suffices to find a $(\xi,\eta)$ with $(\xi A, \eta B)\in T_A\Par\times T_B\Par$ whose image $m_*(\xi,\eta)$ has non-zero $(2,1)$-entry. By computing $R_*(0,\eta) =  A\eta A^{-1}$ for $
    \eta = \begin{bmatrix}
        z  & x \\
        y  & -z
      \end{bmatrix}\in \sl$,
    we obtain 
$m_*(0,\eta)
      = A\eta A^{-1}
      = \begin{bmatrix}
        z-y  & x - y + 2z \\
          y    &  y-z
      \end{bmatrix}$.
      Therefore, there is such $(\xi, \eta)$ if and only if 
      there is an $\eta B = \begin{bmatrix}
        z  & x \\
        y  & -z
      \end{bmatrix} B\in T_B\Par$ with $y \neq 0$.
      Letting $B = \abcd\in \SL$, having $T_B\Par = \bigg\{ \begin{bmatrix}
        z  & x \\
        0  & -z
      \end{bmatrix} B: x,z\in \mathbb{R}^2\bigg\}$ is equivalent to having $cx + by + (a-d)z = 0$ for all $(x,0,z)\in \mathbb{R}^3$, i.e., $c = a-d = 0$ and $B = \pm\begin{bmatrix}
        1 & b \\
        0 & 1
    \end{bmatrix}$ for some $b\in \mathbb{R}$ which commutes with $A$.
    As a consequence, $dm_{(A,B)}$ is surjective if and only if $A$ and $B$ do not commute.
\end{proof}

\noindent
Inyoung Ryu\\
Department of Mathematics\\  Texas A\&M University\\
College Station, TX 77843, USA\\
(riy520@tamu.edu)
\\

\noindent
Tian Yang\\
Department of Mathematics\\  Texas A\&M University\\
College Station, TX 77843, USA\\
(tianyang@math.tamu.edu)

\end{document}